\documentclass[11pt,reqno]{amsart}

\usepackage{amsthm,amsfonts,amssymb,amsmath,amsxtra}
\usepackage[all]{xy}
\SelectTips{cm}{}
\usepackage{tikz}
\usepackage{verbatim}

\usepackage[margin=1.1in]{geometry}
\usepackage{mathrsfs}
\usepackage[backref=page]{hyperref} %

\RequirePackage{xspace}
\RequirePackage{etoolbox}
\RequirePackage{varwidth}
\RequirePackage{enumitem}
\RequirePackage{tensor}
\RequirePackage{mathtools}
\RequirePackage{longtable}
\RequirePackage{multirow}
\RequirePackage{makecell}
\RequirePackage{bigints}

\newcommand{\fke}{\ensuremath{\mathfrak{e}}\xspace}

\newcommand{\fkt}{\ensuremath{\mathfrak{t}}\xspace}

\newcommand{\fkY}{\ensuremath{\mathfrak{Y}}\xspace}
\newcommand{\fkZ}{\ensuremath{\mathfrak{Z}}\xspace}

\newcommand{\BA}{\ensuremath{\mathbb{A}}\xspace}

\newcommand{\BC}{\ensuremath{\mathbb{C}}\xspace}
\newcommand{\BD}{\ensuremath{\mathbb{D}}\xspace}
\newcommand{\BE}{\ensuremath{\mathbb{E}}\xspace}
\newcommand{\BF}{\ensuremath{\mathbb{F}}\xspace}
\newcommand{\BG}{\ensuremath{\mathbb{G}}\xspace}

\newcommand{\BL}{\ensuremath{\mathbb{L}}\xspace}
\newcommand{\BM}{\ensuremath{\mathbb{M}}\xspace}
\newcommand{\BN}{\ensuremath{\mathbb{N}}\xspace}

\newcommand{\BP}{\ensuremath{\mathbb{P}}\xspace}
\newcommand{\BQ}{\ensuremath{\mathbb{Q}}\xspace}

\newcommand{\BT}{\ensuremath{\mathbb{T}}\xspace}

\newcommand{\BV}{\ensuremath{\mathbb{V}}\xspace}
\newcommand{\BW}{\ensuremath{\mathbb{W}}\xspace}
\newcommand{\BX}{\ensuremath{\mathbb{X}}\xspace}
\newcommand{\BY}{\ensuremath{\mathbb{Y}}\xspace}
\newcommand{\BZ}{\ensuremath{\mathbb{Z}}\xspace}

\newcommand{\bM}{\ensuremath{\mathbf{M}}\xspace}

\newcommand{\bY}{\ensuremath{\mathbf{Y}}\xspace}
\newcommand{\bZ}{\ensuremath{\mathbf{Z}}\xspace}

\newcommand{\CC}{\ensuremath{\mathcal{C}}\xspace}

\newcommand{\CE}{\ensuremath{\mathcal{E}}\xspace}
\newcommand{\CF}{\ensuremath{\mathcal{F}}\xspace}

\newcommand{\CM}{\ensuremath{\mathcal{M}}\xspace}
\newcommand{\CN}{\ensuremath{\mathcal{N}}\xspace}
\newcommand{\CO}{\ensuremath{\mathcal{O}}\xspace}

\newcommand{\CT}{\ensuremath{\mathcal{T}}\xspace}

\newcommand{\CV}{\ensuremath{\mathcal{V}}\xspace}
\newcommand{\CW}{\ensuremath{\mathcal{W}}\xspace}

\newcommand{\CY}{\ensuremath{\mathcal{Y}}\xspace}
\newcommand{\CZ}{\ensuremath{\mathcal{Z}}\xspace}

\newcommand{\uF}{\underline{F}}
\newcommand{\uV}{\underline{V}}

\newcommand{\Ad}{\mathrm{Ad}}

\DeclareMathOperator{\Aut}{Aut}

\DeclareMathOperator{\charac}{char}

\newcommand{\corr}{\mathrm{corr}}

\newcommand{\del}{\operatorname{\partial Orb}}

\DeclareMathOperator{\End}{End}

\newcommand{\Fil}{\ensuremath{\mathrm{Fil}}\xspace}

\DeclareMathOperator{\Gal}{Gal}

\newcommand{\GL}{\mathrm{GL}}

\DeclareMathOperator{\Gr}{Gr}

\newcommand{\GS}{\mathrm{GS}}

\newcommand{\GU}{\mathrm{GU}}

\DeclareMathOperator{\Hom}{Hom}

\newcommand{\id}{\ensuremath{\mathrm{id}}\xspace}

\DeclareMathOperator{\Int}{\ensuremath{\mathrm{Int}}\xspace}

\DeclareMathOperator{\Ker}{Ker}

\DeclareMathOperator{\Lie}{Lie}

\newcommand{\OGr}{\mathrm{OGr}}
\DeclareMathOperator{\Orb}{Orb}

\DeclareMathOperator{\rank}{rank}

\DeclareMathOperator{\Res}{Res}
\newcommand{\rs}{\ensuremath{\mathrm{rs}}\xspace}

\DeclareMathOperator{\Spec}{Spec}
\DeclareMathOperator{\Spf}{Spf}

\newcommand{\SO}{{\mathrm{SO}}}

\newcommand{\nspt}{\mathrm{nspl}}
\newcommand{\spt}{\mathrm{spl}}
\newcommand{\SU}{{\mathrm{SU}}}

\newcommand{\val}{{\mathrm{val}}}

\newcommand{\Ver}{{\mathrm{Vert}}}

\DeclareMathOperator{\tr}{tr}

\newcommand{\U}{\mathrm{U}}

\DeclareMathOperator{\vol}{vol}

\newcommand{\wit}{\widetilde}
\newcommand{\wh}{\widehat}

\newcommand{\pair}[1]{\langle {#1} \rangle}

\newcommand{\ov}{\overline}
\newcommand{\incl}{\hookrightarrow}
\newcommand{\lra}{\longrightarrow}

\newcommand{\bs}{\backslash}

\newcommand{\lv}{\lvert}
\newcommand{\rv}{\rvert}

\newcommand{\ep}{\varepsilon}

\newenvironment{altenumerate}
   {\begin{list}
      {(\theenumi) }
      {\usecounter{enumi}
       \setlength{\labelwidth}{0pt}
       \setlength{\labelsep}{0pt}
       \setlength{\leftmargin}{0pt}
       \setlength{\itemsep}{\the\smallskipamount}
       \renewcommand{\theenumi}{\roman{enumi}}
      }}
   {\end{list}}
\newenvironment{altenumerate2}
   {\begin{list}
      {\textup{(\theenumii)} }
      {\usecounter{enumii}
       \setlength{\labelwidth}{0pt}
       \setlength{\labelsep}{0pt}
       \setlength{\leftmargin}{2em}
       \setlength{\itemsep}{\the\smallskipamount}
       \renewcommand{\theenumii}{\alph{enumii}}
      }}
   {\end{list}}

\newenvironment{altitemize}
   {\begin{list}
      {$\bullet$}
      {\setlength{\labelwidth}{0pt}
	   \setlength{\itemindent}{5pt}
       \setlength{\labelsep}{5pt}
       \setlength{\leftmargin}{0pt}
       \setlength{\itemsep}{\the\smallskipamount}
      }}
   {\end{list}}

\setitemize[0]{leftmargin=*,itemsep=\the\smallskipamount}
\setenumerate[0]{leftmargin=*,itemsep=\the\smallskipamount}

\renewcommand{\to}{%
   \ifbool{@display}{\longrightarrow}{\rightarrow}%
   }
\let\shortmapsto\mapsto
\renewcommand{\mapsto}{%
   \ifbool{@display}{\longmapsto}{\shortmapsto}%
   }
\newlength{\olen}
\newlength{\ulen}
\newlength{\xlen}
\newcommand{\xra}[2][]{%
   \ifbool{@display}%
      {\settowidth{\olen}{$\overset{#2}{\longrightarrow}$}%
       \settowidth{\ulen}{$\underset{#1}{\longrightarrow}$}%
       \settowidth{\xlen}{$\xrightarrow[#1]{#2}$}%
       \ifdimgreater{\olen}{\xlen}%
          {\underset{#1}{\overset{#2}{\longrightarrow}}}%
          {\ifdimgreater{\ulen}{\xlen}%
             {\underset{#1}{\overset{#2}{\longrightarrow}}}
             {\xrightarrow[#1]{#2}}}}%
      {\xrightarrow[#1]{#2}}
   }
\makeatother
\newcommand{\xyra}[2][]{%
   \settowidth{\xlen}{$\xrightarrow[#1]{#2}$}%
   \ifbool{@display}%
      {\settowidth{\olen}{$\overset{#2}{\longrightarrow}$}%
       \settowidth{\ulen}{$\underset{#1}{\longrightarrow}$}%
       \ifdimgreater{\olen}{\xlen}%
          {\mathrel{\xymatrix@M=.12ex@C=3.2ex{\ar[r]^-{#2}_-{#1} &}}}%
          {\ifdimgreater{\ulen}{\xlen}%
             {\mathrel{\xymatrix@M=.12ex@C=3.2ex{\ar[r]^-{#2}_-{#1} &}}}
             {\mathrel{\xymatrix@M=.12ex@C=\the\xlen{\ar[r]^-{#2}_-{#1} &}}}}}%
      {\mathrel{\xymatrix@M=.12ex@C=\the\xlen{\ar[r]^-{#2}_-{#1} &}}}%
   }
\makeatletter
\newcommand{\xla}[2][]{%
   \ifbool{@display}%
      {\settowidth{\olen}{$\overset{#2}{\longleftarrow}$}%
       \settowidth{\ulen}{$\underset{#1}{\longleftarrow}$}%
       \settowidth{\xlen}{$\xleftarrow[#1]{#2}$}%
       \ifdimgreater{\olen}{\xlen}%
          {\underset{#1}{\overset{#2}{\longleftarrow}}}%
          {\ifdimgreater{\ulen}{\xlen}%
             {\underset{#1}{\overset{#2}{\longleftarrow}}}
             {\xleftarrow[#1]{#2}}}}%
      {\xleftarrow[#1]{#2}}
   }
\newcommand{\isoarrow}{%
   \ifbool{@display}{\overset{\sim}{\longrightarrow}}{\xrightarrow\sim}%
   }
\renewcommand{\lra}{%
   \ifbool{@display}{\longleftrightarrow}{\leftrightarrow}%
   }

\newcommand{\OFb}{{O_{\breve F}}}

\newcommand{\wt}{\wit}

\DeclareFontFamily{U}{matha}{\hyphenchar\font45}
\DeclareFontShape{U}{matha}{m}{n}{
      <5> <6> <7> <8> <9> <10> gen * matha
      <10.95> matha10 <12> <14.4> <17.28> <20.74> <24.88> matha12
      }{}
\DeclareSymbolFont{matha}{U}{matha}{m}{n}
\DeclareFontFamily{U}{mathx}{\hyphenchar\font45}
\DeclareFontShape{U}{mathx}{m}{n}{
      <5> <6> <7> <8> <9> <10>
      <10.95> <12> <14.4> <17.28> <20.74> <24.88>
      mathx10
      }{}
\DeclareSymbolFont{mathx}{U}{mathx}{m}{n}

\DeclareMathSymbol{\obot}         {2}{matha}{"6B}

\newtheorem{theorem}[subsubsection]{Theorem}
\newtheorem{proposition}[subsubsection]{Proposition}
\newtheorem{lemma}[subsubsection]{Lemma}
\newtheorem {conjecture}[subsubsection]{Conjecture}
\newtheorem{corollary}[subsubsection]{Corollary}

\theoremstyle{definition}
\newtheorem{definition}[subsubsection]{Definition}

\newtheorem{remark}[subsubsection]{Remark}
\newtheorem{remarks}[subsubsection]{Remarks}
\newtheorem{question}[subsubsection]{Question}
\numberwithin{equation}{subsection}

\usepackage{xcolor}

\newcommand{\fp}{\varphi}

\newcommand{\Sing}{\mathrm{Sing}}
\newcommand{\Exc}{\mathrm{Exc}}
\newcommand{\WT}{\mathrm{WT}}
\newcommand{\bLambda}{\boldsymbol{\Lambda}}

\usepackage{fancyhdr}

\begin{document}

\title[Ramified splitting Arithmetic transfer conjectures]{More regular formal moduli spaces and arithmetic transfer conjectures: the ramified quadratic case}

\author{Y. Luo}
\address{University of Wisconsin-Madison, Department of Mathematics, 480 Lincoln Drive, Madison, WI 53706, USA}
\email{yluo237@wisc.edu}
\author{M. Rapoport}
\address{Mathematisches Institut der Universit\"at Bonn, Endenicher Allee 60, 53115 Bonn, Germany}
\email{rapoport@math.uni-bonn.de}

\author{W. Zhang}
\address{Massachusetts Institute of Technology, Department of Mathematics, 77 Massachusetts Avenue, Cambridge, MA 02139, USA}
\email{weizhang@mit.edu}

 \date{\today}

\begin{abstract}For unitary groups associated to a ramified quadratic extension of a $p$-adic field,
we define various regular formal moduli spaces of $p$-divisible groups with parahoric levels, characterize exceptional special divisors on them, and construct correspondences between them. We formulate arithmetic transfer conjectures,
which are variants of the arithmetic fundamental lemma conjecture in this context. We prove the conjectures in the lowest dimensional cases.
\end{abstract}

\maketitle
\tableofcontents

\section{Introduction}\label{Intro section}

 The arithmetic Gan--Gross--Prasad (GGP) conjecture \cite{GGP} is one of the generalizations of the Gross--Zagier formula \cite{GZ} from modular curves to higher dimensional Shimura varieties.  The third author  proposed a relative trace formula approach to the arithmetic GGP conjecture \cite{Z12}. In this context, he formulated the \emph{arithmetic fundamental lemma} (AFL) conjecture, which is now a theorem, cf. \cite{Zha21,MZ,ZZha}. The AFL conjecture  relates the special value of the derivative of an orbital integral to an arithmetic intersection number on a \emph{Rapoport--Zink formal moduli space of $p$-divisible groups} (RZ space) attached to a unitary group. It is essential for the AFL conjecture that one is dealing with a situation that is unramified in every possible sense (the quadratic extension $F/F_0$ defining the unitary group is unramified, and the special vector has unit length, and the function appearing in the derivative of the orbital integral is the characteristic function of a hyperspecial maximal compact subgroup).

When these  unramifiedness hypotheses are dropped,  the statement of the AFL has to be modified. In the context of the \emph{fundamental lemma} (FL) conjecture of Jacquet--Rallis, this question leads naturally to their \emph{smooth transfer} (ST) conjecture, proved by the third author in the non-archimedean case  \cite{Z14}. In the arithmetic context, this question naturally leads  to the problem of formulating  \emph{arithmetic transfer} (AT) conjectures. There are two ways of relaxing the unramifiedness conditions. One is when the quadratic extension $F/F_0$ is unramified but where the level structure imposed is no longer hyperspecial (and, relatedly, the special vector is no longer of unit length). This case is dealt with in \cite{RSZ2}, \cite{LRZ2} and \cite{ZZha}. The other kind of AT conjectures arises when $F/F_0$ is no longer unramified.  This  was considered in special cases in \cite{RSZ1} and \cite{RSZ2}. In  the present paper, we explore systematically the ramified case. In both the unramified and the ramified cases, a limiting factor is the requirement that the ambient space (a product of RZ spaces) is a regular formal scheme.

The AFL conjecture concerns the closed embedding of RZ spaces
\begin{equation}\label{diaga-intro}
\CN_{n}^{[0]}\hookrightarrow \CN_{n+1}^{[0]} ,
\end{equation} 
where $\CN_{n}^{[0]}\simeq \CZ(u_0)$, for a special vector $u_0$ of unit norm, from which we deduce the special cycle $\CZ(u_0)\subset \CN_{n}^{[0]}\times \CN_{n+1}^{[0]}$. 
In the generic fiber (a rigid-analytic space), the left term in \eqref{diaga-intro} is the member $S_{K_n^{[0]}}$ of the RZ tower of $\CN_{n}^{[0]}$ corresponding to the (hyperspecial) parahoric $K_n^{[0]}$ of $\U(W_0^\flat)$ and the right term is the member $S_{K_{n+1}^{[0]}}$ of the RZ tower of $\CN_{n+1}^{[0]}$ corresponding to the (hyperspecial) parahoric $K_{n+1}^{[0]}$ of $\U(W_0)$.  In  \cite{LRZ2} and the present paper, the inclusion \eqref{diaga-intro}, which is an integral model of the inclusion $S_{K_{n}^{[0]}}\subset S_{K_{n+1}^{[0]}}$,  is replaced by an integral \emph{correspondence}. This correspondence is to be as simple as possible, i.e., of atomic type in the sense of \cite[\S 4.2]{LRZ1}, modelled on the definition $\varphi_{t, t'}=\varphi_t\otimes\varphi_{t'}$ of an atomic function in loc.~cit., in which either $t'=0$ or $t=0$. This means that the correspondence is of either of the following two types: 
\begin{equation}\label{smcorintro}
 \begin{aligned}
 \xymatrix{&&\CN^{[r, t]}_n \times \CN_{n+1}^{[t]}  \ar[rd]   \ar[ld] &
 \\ \CN_n^{[t]}\ar@{^(->}[r]& \CN_n^{[t]} \times \CN_{n+1}^{[t]} &&  \CN_{n}^{[r]} \times \CN_{n+1}^{[t]} 
 }
 \end{aligned}
 \end{equation}
 or 
 \begin{equation}\label{lgcorintro}
 \begin{aligned}
 \xymatrix{&&\CN^{[ t]}_n \times \CN_{n+1}^{[ r, t]}  \ar[rd]   \ar[ld] &
 \\ \CN_n^{[t]}\ar@{^(->}[r]& \CN_n^{[t]} \times \CN_{n+1}^{[t]} &&  \CN_{n}^{[t]} \times \CN_{n+1}^{[r]} .
 }
 \end{aligned}
 \end{equation}
In \eqref{smcorintro}, the map in the second factor is the identity; in \eqref{lgcorintro}, the map in the first factor is the identity. Here we have added on  the left of these diagrams  the closed embeddings given by  the graphs of  closed embeddings $\CN_n^{[t]}\hookrightarrow   \CN_{n+1}^{[t]} $ which exhibit $\CN_n^{[t]}$ as a special cycle in $\CN_{n+1}^{[t]} $. We call the first type a \emph{small correspondence} and the second type a \emph{big} or a \emph{large correspondence} (in the first type, the non-trivial correspondence is on the RZ space of dimension $n$; in the second type,  the non-trivial correspondence is on the RZ space of dimension $n+1$).

When $F/F_0$ is ramified (the case considered in this paper), the regularity condition on the  ambient product of RZ spaces  $\CN_n^{[s]}\times\CN_{n+1}^{[r]}$  imposes that   either $s=n-\ep(n)$ or $r=n+\ep(n)$, where $\ep(n)=0$ if $n$ is even and $\ep(n)=1$ if $n$ is odd.  In this case, one of the factors in the product  is formally smooth. However, the other factor will in general not be regular. When the second factor is not regular, we replace it by an explicit  blow-up which is regular, in fact semi-stable (the \emph{splitting model}, see below).

\subsection{AT conjectures}
Before we give more details on the construction of the correspondences, let us state the general form of our AT conjectures. Let $F/F_0$ be a ramified quadratic extension of $p$-adic local fields ($p\neq 2$). The relevant RZ spaces $\CN_{n, \ep}^{[t]}$ (see \S\ref{sec:RZ-spaces}) depend on two integers $n$ and $t$, and on $\varepsilon\in\{\pm 1\}$.  Here $n$ denotes the dimension, and $t$ (the \emph{type}) is  an even integer   between $0$ and $n$ and defines the level structure, and $\ep$ fixes the isomorphism class of the framing object.  
Here, when $n$ is odd, the isomorphism class of $\CN_{n, \ep}^{[t]}$ is independent of $\varepsilon$.

In our AT conjectures, the spaces are (variants of) RZ spaces and the cycles are closed formal subspaces in a product of these attached to the integers $n$ and $n+1$. The precise definitions of the spaces and the cycles are given in the main body of the paper,  see \S\ref{sec:cycle}. We denote by $G'(F_0)_\rs$ the set of regular semi-simple elements on the $\GL$-side and by $G_{W}(F_0)_\rs$ the set of regular semi-simple elements on the $\U$-side,  comp. \cite[\S 2]{RSZ1}. Here $W$ denotes a hermitian space of dimension $n+1$. Also, we have incorporated the transfer factor in the definition of weighted orbital integrals on the $\GL$-side. 

 \begin{conjecture}\label{conj all}
 Let $n, t, \ep$ be numerical invariants as above ($n\geq 1$, $0\leq t\leq n+1$ is an even integer, and $\ep\in\{\pm 1\}$), and let  $(\CN_{n,n+1; t}, \CZ^{[t]}_n,\varphi)$ be a triple consisting of an ambient space $\CN_{n,n+1; t}$ (a product of RZ spaces of dimension $n$ and $n+1$), a special cycle $ \CZ^{[t]}_n$ on $\CN_{n,n+1; t}$, and a test function $\varphi$ on the $\U$-side, as in the table in \S\ref{ss:table}.

 \begin{altenumerate}
\item
There exists $\varphi'\in C_c^\infty(G')$ with transfer $(\varphi,0)\in C_c^\infty(G_{W_0})\times C_c^\infty(G_{W_1})$ such that, if $\gamma\in G'(F_0)_\rs$ is matched with  $g\in G_{W_1}(F_0)_\rs$, then
 \begin{equation*}
  \left\langle \CZ^{[t]}_n , g\CZ^{[t]}_n \right\rangle_{\CN_{n,n+1; t}} \cdot\log q=- \del\big(\gamma,  \varphi' \big).
\end{equation*}
\item For any $\varphi'\in C_c^\infty(G')$  with transfer $(\varphi,0)\in C_c^\infty(G_{W_0})\times C_c^\infty(G_{W_1})$,  there exists $\fp'_\corr\in C_c^\infty(G')$  such that, if $\gamma\in G'(F_0)_\rs$ is matched with  $g\in G_{W_1}(F_0)_\rs$, then
 \begin{equation*}
  \left\langle \CZ^{[t]}_n , g\CZ^{[t]}_n \right\rangle_{\CN_{n,n+1; t}} \cdot\log q=- \del\big(\gamma,  \varphi' \big)- \Orb\big(\gamma,  \fp'_\corr \big).
\end{equation*}
\end{altenumerate}
Here $W_0$ is the hermitian space of dimension $n+1$ with Hasse invariant $\ep$ and $W_1$ is the opposite space. 
\end{conjecture}
By  \cite[Prop. 5.14]{RSZ1}, part (i) follows from part (ii); by the density conjecture \cite[Conj. 5.16]{RSZ1}, part (ii) follows from part (i). Something analogous holds for all further conjectures later in this paper; in the interest of brevity, we will omit the variants (ii) of these conjectures in the statements below. The conjecture above is the \emph{homogeneous version}.  There is also an inhomogeneous version, which we omit here and below. 

\subsection{Summary of cases} 
\label{ss:table}
The following table summarizes all the cases of AT conjectures in this paper.  Here $t$ is always even and lies in  $[0,n+1]$ or $[0,n]$, depending on whether $t$ appears as the second entry or the first. Moreover, in each row the parity of $n$ is determined by the rule that all types are even. In each case, an \emph{aligned triple} $(\BY, \BX, u)$ that underlies the construction of the correspondence is fixed, cf. Definition \ref{defn:aligned-module}. For simplicity, we drop the invariant $\ep$ from the notation of RZ spaces. 

\medskip
{\setlongtables
\renewcommand{\arraystretch}{1.25}
\begin{longtable}{|c|c|c|c|c|}
\hline
\begin{varwidth}{\linewidth}
   \centering
  Type 
\end{varwidth} 
&
\begin{varwidth}{\linewidth}
   \centering
 Ambient space
  \\ $\CN_{n,n+1;t}$
\end{varwidth} 
   & \begin{varwidth}{\linewidth}
   \centering
   The cycle\\ $\CZ^{[t]}_n$
\end{varwidth} 
   &  \begin{varwidth}{\linewidth}
   \centering
  Test function \\ $\varphi$
\end{varwidth} & \begin{varwidth}{\linewidth}
         \centering
        AT \\ Conjecture
      \end{varwidth}
      
        \\
    \hline
 $(n,n)$& $\CN_{n}^{[n]}\times \CN_{n+1}^{[n]}$&   $\CN_{n}^{[n]} $ 
 & $ \vol( K_n^{[n]})^{-2}  {\bf 1}_{K_n^{[n]}\times K^{[n]}_{n+1}}$&  \cite[Conj. 5.3]{RSZ1}
   \\
\hline
$(n-1,n+1)$& $\CN_{n}^{[n-1]}\times \CN_{n+1}^{[n+1]}$ &   $\CN_{n}^{[n-1]}$ &$\vol(K_n^{[n-1],\circ})^{-2} {\bf 1}_{K_n^{[n-1]} \times K_{n+1}^{[n+1]} }$
&
\cite[Conj. 12.4]{RSZ2}
\\	 \hline
 $(n,t)$, $0\leq t\leq n$ & $\CN_{n}^{[n]}\times \CN_{n+1}^{[t],\spt}$&   $\wh \CN_{n}^{[t],\spt} $&
$ \vol( K_n^{[n,t]})^{-2}  {\bf 1}_{K_n^{[n]}\times K^{[t]}_{n+1}}$ & Conj. \ref{conj (n even,t<n)}  
	\\
 \hline
 $(n-1,t)$, $0\leq t\leq n-1$& $\CN_{n}^{[n-1]}\times \CN_{n+1}^{[t],\spt}$&   $\wh \CN_{n}^{[t],\spt} $ 
& $ \vol( K_n^{[n-1,t]})^{-2} {\bf 1}_{K_n^{[n-1]}\times K^{[t]}_{n+1}}$&  Conj. \ref{conj (n odd,t<n) Z}
	\\
	 \hline
  $(n-1,t)$,  $0\leq t\leq n+1$ & $\CN_{n}^{[n-1]}\times \CN_{n+1}^{[t],\spt}$&   $\wt \CM_{n}^{[t],\pm,\spt} $ 
  &$\vol(K_n^{[n-1],\circ})^{-2} {\bf 1}_{K_n^{[n-1]}}\otimes \varphi^{[n+1,t]}_{n+1}$&  Conj. \ref{conj (n odd,t) Y}
	\\
\hline
 $(t,n)$, $0\leq t\leq n$& $\CN_{n}^{[t],\spt}\times \CN_{n+1}^{[n]}$&   $\wh\CN_{n}^{[t],\spt}$&$\vol(K_n^{[n,t]})^{-2}  {\bf 1}_{ K_{n}^{[t]}  \times K^{[n]}_{n+1} } $
  &
  Conj. \ref{conj (t,n even) sm}  \\
   \hline
 $(t,n)$, $0\leq t\leq n$ & $\CN_{n}^{[t],\spt}\times \CN_{n+1}^{[n]}$&   $\wt \CN_{n}^{[t],\spt}$
  &$\vol(K_n^{[t]})^{-2} {\bf 1}_{K_n^{[t]}}\otimes \varphi^{[t,n]}_{n+1}$
   &
 Conj. \ref{conj (t,n even)} \\ 
    \hline
 $(t,n+1)$, $0\leq t\leq n-1$& $\CN_{n}^{[t],\spt}\times \CN_{n+1}^{[n+1]}$ &   $\wh \CN_{n}^{[t],\spt}$
 &$\vol(K_n^{[n-1,t],\circ})^{-2} {\bf 1}_{K_{n}^{[t]}\times K_{n+1}^{[n+1]} }$&
  Conj. \ref{conj (t,n odd) sm} 
   \\	 \hline
\end{longtable}

We have a few comments. 
\begin{itemize}
\item The cycles in the  first two rows  are graphs of closed embeddings. The cycles in the  third, fourth, sixth and eighth row are  small correspondences. The cycles in the fifth and the seventh row are large correspondences. 
\item If $t$ achieves the upper bound, there are the variants without the superscript $\spt$. In some of these extreme cases, splitting models coincide with the usual ones, and the corresponding conjecture is then  identical to that in \cite{RSZ1,RSZ2}. More precisely, in the sixth row the case for $t=n$ and in the seventh row the case for $t=n$  are both identical with the first row, see Remark \ref{rem: even t=n} and Remark \ref{rem:t=n alt}.
On the other hand,  the case $t=n$ in the third row has a different ambient space from the case in the first row ($\CN_{n+1}^{[n], \spt}$ versus $\CN_{n+1}^{[n]}$), even though the cycles are identical. Nevertheless, we show that the two AT conjectures are  equivalent,  see Proposition \ref{prop: (n even,t=n)}. Similarly, the case $t=n-1$ in the last row  differs from the second row ($\CN_{n}^{[n-1], \spt}$ versus $\CN_{n}^{[n-1]}$); we conjecture that the difference of intersection numbers is an orbital integral function, cf. Conjecture \ref{conj: t=n-1 corr}.

\item Regarding the fifth row, we refer to  Conjecture  \ref{conj (n odd,t) Y}, (ii) and (iii) for refinements taking into account the disjoint sum decomposition of $\wt \CM_{n}^{[t],\pm,\spt} $.

\end{itemize}
 
 \medskip

  \subsection{Low dimensional cases}
 We can prove our conjectures in the first non-trivial case. 
  \begin{theorem}\label{thmn=1-intro}
  Conjecture \ref{conj all} holds when $n=1$. 
  \end{theorem}
 \begin{proof} Indeed, the cases when $n=1$ are all covered by the literature, except case (iii) below, which is dealt with in \S  \ref{sec:(0,0)-case}.
  \begin{altenumerate}

 \item type $(n-1,t)=(0,0)$: $\CN_1^{[0]}\to \CN_1^{[0]}\times \CN_{2}^{[0],\spt}$, cf. Conjecture \ref{conj (n odd,t<n) Z}.   This case follows from  \cite[Thm. 13.4]{RSZ2} when $\ep=1$ (i.e., $\CN_{2}^{[0]}$ is the base change of the Drinfeld space), resp. from  \cite[Thm. 13.2]{RSZ2} when $\ep=-1$ (i.e.,  $\CN_{2}^{[0]}$ is the base change of the Lubin--Tate space at the Iwahori level). 
  \item type $(n-1,t)=(0,2)$:
$\wt{\CN}_1^{[0],\circ}\to \CN_1^{[0]}\times \CN_{2}^{[2]}$, cf. Conjecture \ref{conj (n odd,t) Y}.
This case follows from  \cite[Thm. 1.6]{RSZ2}.

 \item type $(n-1,t)=(0,0)$:
$\wt{\CM}_1^{[0],\spt}\to \CN_1^{[0]}\times \CN_{2}^{[0],\spt}$, cf. Conjecture \ref{conj (n odd,t) Y}. \qedhere
\end{altenumerate}
\end{proof} 

 We list the cases when $n=2$, one of which is known. \begin{altenumerate}
\item
type $(n,t)=(2,2)$: 
$
\wt{\CN}_2^{[2]}\to \CN_2^{[2]}\times \CN_{3}^{[2],\spt}$, cf. Conjecture \ref{conj (n even,t=n)}. This is proved by  Proposition \ref{prop: (n even,t=n)} and \cite{RSZ1}.
\item
type $(n,t)=(2,0)$:  $\wt\CN^{[0]}_{2}  \to \CN^{[2]}_{2}\times\CN^{[0],\spt}_{3} $, cf. Conjecture \ref{conj (n even,t<n)}.
\item type $(t,n)=(0,2)$ when $W_0^\flat$ is split:   $\CN^{[0,2],\spt}_{2}  \to \CN^{[0],\spt}_{2}\times\CN^{[2]}_{3} 
$, cf. Conjecture \ref{conj (t,n even) sm}.
\item type $(t,n)=(0,2)$ when $W_0^\flat$ is split:  $\wt\CN^{[0],\spt}_{2}  \to \CN^{[0],\spt}_{2}\times\CN^{[2]}_{3} 
$, cf. Conjecture \ref{conj (t,n even)}, the case $\ep^\flat=1$. 
\item type $(t,n)=(0,2)$ when $W_0^\flat$ is non-split:   $\wt\CN^{[0],\spt}_{2}  \to \CN^{[0],\spt}_{2}\times\CN^{[2]}_{3} 
$, cf. Conjecture \ref{conj (t,n even)}, the case $\ep^\flat=-1$.
\end{altenumerate}
These cases stand as the next test cases of our conjectures. We hope to return to them in the future.

\subsection{More background on the enumeration of cases} 
For $n\geq 2$, the RZ space $\CN_{n, \ep}^{[t]}$ is formally smooth (\emph{exotic smoothness}) in two instances: when $n$ is even and $t=n$ (in which case $\ep=+1$ is the only possibility), and when $n$ is odd and $t=n-1$. In \cite{RSZ1} and \cite{RSZ2}, the last two  authors and B. Smithing consider  on the geometric side the natural closed embedding of formally smooth RZ spaces $\CN_{n,1}^{[n]}\hookrightarrow \CN_{n+1}^{[n]}$. They also construct in a non-trivial way an embedding $\CN_{n,1}^{[n-1]}\hookrightarrow \CN_{n+1}^{[n+1]}$. They then propose AT conjectures in these two cases and verify them for $n=1$ and $n=2$. The construction of these embeddings is based on the moduli-theoretic definition of these particular RZ spaces. 

Beyond these cases,  there are no natural embeddings; instead, we replace the embeddings  by  correspondences linking the two spaces and obtain  in this way cycles on the product space. This is made possible by the recent moduli-theoretic definition of all relevant RZ spaces due to the first author \cite{Luo24}. Our spaces and cycles are built on the answer to the following question:
\begin{question}\label{consid-intro}
\item [1) ]\label{cond1} \emph{Correspondences}: We would like the ambient space to be a product of RZ spaces of {\em maximal parahoric levels} which is regular. How can this be achieved? 
 
\item  [2) ] \label{cond2} \emph{Cycles}: When is the $\CZ$-divisor $ \CZ(u)^{[t]}$ or the $\CY$-divisor $ \CY(u)^{[t]}$ on $\CN_{n+1}^{[t]}$   isomorphic  to a lower-dimensional RZ space of maximal parahoric level? 
 \end{question}

   \subsection{Cycles}

Let us first consider part 2) of Question \ref{cond2}. We have the following \emph{exceptional isomorphisms} (see \S \ref{sec:special-cycles} for notation and precise statements):
\begin{theorem}\label{thm:ex iso-intro}
Let $u\in \BV(\BX_{n,\ep}^{[t]})$ be a unit length vector. Set $\ep^\flat=\ep\ep(u)\eta((-1)^{n-1})$. 
\begin{altenumerate}
\item Consider the special $\CZ$-cycle $\CZ(u)_{n,\ep}^{[t]}\subset \CN_{n,\ep}^{[t]}$. Then:
	\begin{altitemize} 
	\item When $n$ is even and $t=n$, $\CZ(u)^{[t]}_{n,\ep}$ is empty (note that $\ep=1$ in this case).
	\item In the remaining cases, we have an isomorphism
	\begin{equation*}
		\CZ(u)_{n,\ep}^{[t]}\simeq \CN_{n-1,\ep^\flat}^{[t]},
	\end{equation*}
	except when:
	\item $n$ is odd, $t=n-1$, and $\ep^\flat=-1$, in which case the RHS is not defined  and  the special cycle $\CZ(u)^{[t]}_{n,\ep}$ is the disjoint union of  points $\WT(\Lambda)$ in $\Sing(\CN_{n-1, \ep}^{[t]})$ (the \emph{worst points}, indexed by all  almost $\pi$-modular lattices $\Lambda\subset \BV(\BX_{n,\ep}^{[t]})$  containing $u$). \end{altitemize}

\item(H.~Yao \cite[Thm. 5.5]{Yao24}) Let $n$ be even and $t=n$.
Define $ \CN_{n-1, \ep^\flat}^{[n-2],\circ}$ by the following  fiber product diagram,
\begin{equation*}
	\begin{aligned}
		\xymatrix{
			\CN_{n-1, \ep^\flat}^{[n-2],\circ}\ar@{}[rd]|{\square}\ar[d]\ar@{^(->}[r]&\CN_{n}^{[n-2,n]}\ar[d]\\
			\CN_{n-1, \ep^\flat}^{[n-2]}\ar@{^(->}[r]&\CN_{n}^{[n-2]} .&}
	\end{aligned}
\end{equation*}
Then the morphism $\CN_{n}^{[n-2,n]}\to \CN_{n}^{[n-2]}$ is a trivial double covering, cf. \cite[Prop. 6.4]{RSZ2}. Furthermore,  the composition $ \CN_{n-1,\ep^\flat}^{[n-2],\circ}\to \CN_{n}^{[n-2,n]}\to \CN_{n}^{[n]}$ factors through $\CY(u)^{[n]}_{n}$ and induces an isomorphism
$$
\CN_{n-1,\ep^\flat}^{[n-2],\circ}\simeq \CY(u)^{[n]}_{n}.
$$In particular, there is a natural morphism 
$$
\CY(u)^{[n]}_{n}\to  \CN_{n-1,\ep^\flat}^{[n-2]},
$$
which is a trivial double covering.
Furthermore, $\CY(u)^{[n]}_{n}=\CZ(\pi u)^{[n]}_{n}$. 
\end{altenumerate}
\end{theorem}
 
It is conceivable that exceptional special divisors on $\CN_{n,\ep}^{[t]}$, i.e., the divisors $\CZ(u)^{[t]}_{n, \ep}$ appearing in Theorem \ref{thm:ex iso-intro}, are characterized by the property that they are non-empty regular formal schemes, see Conjecture \ref{conj:reg-exceptional}. Similarly, the divisors $\CY(u)^{[n]}_{n, \ep}$ in Theorem \ref{thm:ex iso-intro} should be characterized by the property that $\CY(u)^{[n]}_{n, \ep}$ is a non-empty regular formal scheme, see Remark \ref{rmk:char-regular-Y}.

At the heart of the proof of this theorem is a thorough analysis of the \emph{}strengthened spin condition which gives a moduli-theoretic definition of  the RZ spaces $ \CN_{n,\ep}^{[t]}$, as established in   \cite{Luo24}. We prove the following theorem (cf. Theorem \ref{thm:ss-comparison}). Again, we refer to the body of the paper for the definitions and the notation.

\begin{theorem}\label{thm:ss-comparison-intro}
	Let 
	 $(Y,\iota_Y,\lambda_Y)$ be a hermitian $O_F$-module of dimension $n-1$ and type $t$ over $S\in ({\rm Sch}/\Spf O_{\breve{F}})$.
	Let $\zeta\in O_{F_0}^\times$ be a unit and define
	\begin{equation*}
		(X,\iota_X,\lambda_X):=(Y\times\ov{\CE},\iota_Y\times\iota_{\ov{\CE}},\lambda_Y\times \zeta\lambda_{\ov{\CE}}).
	\end{equation*}
	Then the following assertions hold:
	\begin{altenumerate}
		\item $(X,\iota_X,\lambda_X)$ is a hermitian $O_F$-module of dimension $n$ and type $t$.
		\item If $(Y,\iota_Y,\lambda_Y)$ satisfies the strengthened spin condition, then so does $(X,\iota_X,\lambda_X)$.
		\item Suppose $t\neq n-1$. Then, if $(X,\iota_X,\lambda_X)$ satisfies the strengthened spin condition, then so does $(Y,\iota_Y,\lambda_Y)$.
	\end{altenumerate}
\end{theorem}

\subsection{Correspondences}\label{ss:correspcyc}
Now let us address part 1) of Question \ref{consid-intro}. As mentioned above, among all RZ spaces of maximal parahoric level, for $n\geq 2$, there are two instances when the RZ space $\CN_{n, \ep}^{[t]}$ is formally smooth (\emph{exotic smoothness}): when $n$ is even and $t=n$ (in which case $\ep=1$ is the only possibility), and when $n$ is odd and $t=n-1$. Outside these cases, $\CN_{n, \ep}^{[t]}$ is not even regular. 
However, there is a certain blow-up $\CN_{n, \ep}^{[t], \spt}$ of $\CN_{n, \ep}^{[t]}$ which is always regular (the \emph{splitting model}). 
Hence the product formal schemes $\CN_{n}^{[n]}\times\CN_{n+1, \ep}^{[t], \spt}$ and $\CN_{n, \ep^\flat}^{[n-1]}\times \CN_{n+1, \ep}^{[t], \spt}$ as well as $\CN_{n, \ep^\flat}^{[t],\spt}\times \CN_{n+1, \ep}^{[n]}$ and $\CN_{n, \ep^\flat}^{[t],\spt}\times \CN_{n+1}^{[n+1]}$ are regular and therefore  can serve as ambient spaces for arithmetic intersections. These four possibilities lead to the AT conjectures of type $(n,t)$ (\S \ref{sec:type-(n,t)}), type $(n-1,t)$ (\S \ref{sec:type-(n-1,t)}), type $(t,n)$ (\S \ref{sec:type-(t,n)}), and type $(t,n+1)$ (\S \ref{sec:type-(t,n+1)}).

The He-Luo-Shi  theory  of splitting models \cite{HLS1} is key here. Recall that they are  defined in two steps. First, one introduces 
 the naive splitting model $\CN_{n,\ep}^{[t],\nspt}$ over $\Spf O_{\breve F}$ parametrizing the collection of data 
\begin{equation*}
	(X, \iota, \lambda, \Fil^0(X),\Fil^0(X^\vee); \rho),
\end{equation*}
where $(X,\iota,\lambda,\Fil^0(X),\Fil^0(X^\vee))$ is a hermitian $O_F$-module of   signature $(1,n-1)$ and type $t$ with \emph{splitting  structure}, and where $\rho$ is a {\it framing} with the fixed framing object.
In a second step,  the splitting model $\CN_{n,\ep}^{[t],\spt}$ over $\Spf O_{\breve F}$ is defined as the flat closure of $\CN_{n,\ep}^{[t],\nspt}$.

 In the $\pi$-modular case, the naive splitting model $\CN_n^{[n],\nspt}$  and the splitting model $\CN_n^{[n],\spt}$ are both isomorphic to the RZ space $\CN_n^{[n]}$. 
 In the remaining cases for $n>1$, the splitting model is different from the RZ space. For instance, the splitting model $\CN_n^{[0],\spt}$ coincides with the Kr\"amer model \cite{Kramer}.  The splitting structure is uniquely determined outside the worst points. In fact,  the splitting model $\CN_{n,\ep}^{[t],\spt}$ is the blow-up of the RZ space $\CN_{n,\ep}^{[t]}$ in the worst points, cf. \cite[Thm. 1.3.1]{HLS1}.
Any splitting model $\CN_{n,\ep}^{[t],\spt}$ is flat and semi-stable, and it is smooth if and only if $t=n$, in which case, as mentioned above,  $\CN_n^{[n],\spt}\simeq \CN_n^{[n]}$.

\begin{remark}
More generally, one could consider an intersection on $\CN_n^{[r],\spt}\times\CN_{n+1}^{[s],\spt}$ for any pair $(r,s)$ of even integers. However, the last product is not regular in general.  One could replace the ambient space by its blow-up, similarly to \cite{ZZha}, and obtain in this way further AT conjectures. However, we will not discuss these cases here.
\end{remark}

\subsection{The large correspondence}
 We can relate  the large correspondences to the guiding principle of 
 \cite[\S1, (1.0.5)]{LRZ2}, i.e., to the construction of    ``pull-back" diagrams of exceptional special divisors along the natural projection maps   $\CN_{n+1}^{[r,s]}\to \CN_{n+1}^{[r]}$ from  RZ spaces of (non-maximal) parahoric levels (our notation here is modelled on that of \cite{LRZ2}):  
\begin{equation}
\label{eq Mnr}
  \begin{aligned}
  \xymatrix{  \wt\CZ_1  \ar@{^(->}[r]\ar[d] \ar@{}[rd]|{\square}  &  \CN_{n+1}^{[r,s]} \ar[d] \ar[r] &  \CN_{n+1}^{[s]} \\ 
\CC^{[r]}=  \CZ(u_0)^{[r]} \text{ or } \CY(u_0) ^{[r]} \ar@{^(->}[r] &   \CN_{n+1}^{[r]} .&}
  \end{aligned}
\end{equation}

We would like to consider the cartesian product $\wt\CZ_1$ as our cycle and the product $\CC^{[r]}\times\CN_{n+1}^{[s]}$ as the ambient space. The regularity of the latter  product requires that at least one of the factors is smooth over $\Spf O_{\breve{F}}$. We distinguish two cases. 

\smallskip

\noindent I. {\em The case when $\CC^{[r]}$ is smooth}.  Then there are three cases.
\begin{altenumerate}
\item
 $\CC^{[r]} = \CZ(u_0)^{[n]}\simeq \CN_{n}^{[n]}$ with  $v(u_0)=0$ (and $n$ is even),
  \item
 $\CC^{[r]} = \CZ(u_0)^{[n-1]}\simeq \CN_{n}^{[n-1]}$ with  $v(u_0)=0$ (and $n$ is odd),
  \item $\CC^{[r]} = \CY(u_0)^{[n+1]}\simeq \CN_{n}^{[n-1]}\coprod \CN_{n}^{[n-1]}$ with  $v(u_0)=0$ (and $n$ is odd).
 \end{altenumerate}
They give rise to the cases in  \S \ref{sec:type-(n,t)} (Conj. \ref{conj (n even,t<n)}) and \S \ref{sec:type-(n-1,t)} (Conj. \ref{conj (n odd,t) Y} (i)). In  \S \ref{sec:type-(n-1,t)} (cf. \S\ref{ss: refine} and Conj. \ref{conj (n odd,t) Y}, (ii), (iii)) we also have refinements of the case (iii) taking into account the individual summands in the disjoint union appearing in (iii).

\smallskip

\noindent II. {\em  The case when  $\CN_{n+1}^{[s]}$ is smooth}. 
Then there are two cases. 
\begin{altenumerate}
\item
 $\CN_{n+1}^{[s]}\simeq \CN_{n+1}^{[n]}$, (when $n$ is even),
  \item
 $\CN_{n+1}^{[s]}\simeq \CN_{n+1}^{[n+1]}$, (when $n$ is odd).
 \end{altenumerate}
They give rise to the cases in \S\ref{sec:type-(t,n)}  (Conj. \ref{conj (t,n even)}) 
and \S \ref{sec:type-(t,n+1)} (Conj. \ref{conj (t,n odd) sm}). 
  
  We conjecture that the formal scheme $\wt\CZ_1$ in \eqref{eq Mnr} is flat over $\Spf O_{\breve{F}}$ in all cases. In the cases we can prove this conjecture, this is done  by relating $\wt\CZ_1$ to RZ spaces, cf. Theorems \ref{th:wh} and \ref{th: wt N=N(t,n-1)} in the cases I., (i) and (ii), and Theorem \ref{sm=big10} in the case  II., (ii) (the latter proof being based on the work of H.~Yao \cite{Yao24}).

The correspondence obtained from \eqref{eq Mnr} is related to the large correspondence 
\begin{equation}\label{diala}
	\begin{aligned}
		\xymatrix{&&\CN^{[r]}_n \times \CN_{n+1}^{[r,s]}  \ar[rd]   \ar[ld] &
			\\ \CC^{[r]}\ar@{^(->}[r]& \CN_n^{[r]} \times \CN_{n+1}^{[r]} &&  \CN_{n}^{[r]} \times \CN_{n+1}^{[s]} .
		}
	\end{aligned}
\end{equation}
Indeed,  taking fiber products, \eqref{diala} leads to the following diagram:
\begin{equation*}
	\begin{aligned}\label{fiblar}
		\xymatrix{
			&\wt{\CZ}_1\ar@{}[d]|{\square}\ar[dl]\ar@{^(->}[r]\ar[rrd]
			&\CN_{n+1}^{[r,s]}\ar[rd]\ar[ld]|!{[l];[rd]}\hole&\\
			\CC^{[r]}\ar@{^(->}[r]&\CN_{n+1}^{[r]}&&\CN_{n+1}^{[s]}.}
	\end{aligned}
\end{equation*}
 The cartesian square  diagram  appearing here coincides with the cartesian square in   \eqref{eq Mnr}.
 This in turn leads to the linking diagram:
\begin{equation}\label{lilar}
	\begin{aligned}
\xymatrix{&\wt{\CZ}_1 \ar[rd]  \ar[ld] &\\ \CC^{[r]}&& \CN_{n+1}^{[s]}. }
\end{aligned}
\end{equation}

\subsection{The small correspondence}
Let us consider the small correspondence 
\begin{equation*}
	\begin{aligned}
		\xymatrix{&&\CN^{[r,s]}_n \times \CN_{n+1}^{[s]}  \ar[rd]   \ar[ld] &
			\\ \CC^{[s]}\ar@{^(->}[r]& \CN_n^{[s]} \times \CN_{n+1}^{[s]} &&  \CN_{n}^{[r]} \times \CN_{n+1}^{[s]} .
		}
	\end{aligned}
\end{equation*}
The small correspondence leads us (by taking fiber products) to the linking diagram:
\begin{equation}
	\begin{aligned}\label{corrsm}
		\xymatrix{&\CN^{[r,s]}_n \ar[rd]   \ar[ld] &&
			\\ \CN_n^{[r]} &&  \CC^{[s]}\simeq \CN_{n}^{[s]}\ar@{^(->}[r]&\CN_{n+1}^{[s]}.
		}
	\end{aligned}
\end{equation}
A small correspondence is always a closed formal subscheme of the corresponding big correspondence. 
Sometimes they coincide (for suitable indices $r, s$), see Theorems \ref{th:wh} and \ref{th: wt N=N(t,n-1)} and  \ref{sm=big10}. But often these two correspondences are different and lead to genuinely different ATC statements. For example, we arrive  in this way at  Conjecture  \ref{conj (t,n even) sm}.
\subsection{The generic fiber} Note that in the generic fiber, the small correspondence \eqref{corrsm} induces a correspondence in the RZ tower of the form 
\begin{equation}\label{diaggenun-intro}
\begin{aligned}
\xymatrix{&S_{\wt K_{n}^{[r,s]}} \ar[rd]  \ar[ld] &\\ S_{K_{n}^{[r]}}&& S_{K_{n+1}^{[s]}}, }
\end{aligned}
\end{equation}  in which $\wt K_{n}^{[r,s]}= K_{n}^{[r,s]}$ is a quasi-parahoric in $\U(W_0^\flat)(F_0)$, containing the  corresponding parahoric subgroup  with index one or two. In other words, the small correspondence is a natural integral model of the correspondence \eqref{diaggenun-intro} in the RZ tower of $\CN_n^{[r]}$. 

This is no longer  true for the generic fiber of the linking diagram \eqref{lilar} in the large correspondence, unless the large and the small correspondences coincide. In this case, the generic fiber of $\wt{\CZ}_1$ is not always a member of an RZ tower, comp. Remarks \ref{notRZ9} and \ref{rem: no RZ}. We note that in the situation of \cite{LRZ2}, a similar phenomenon occurs. Indeed, let us compare the correspondences in the present paper with those in \cite{LRZ2}. 

When $F/F_0$ is unramified (as in \cite{LRZ2}), in order that the ambient space $\CN_n^{[s]}\times\CN_{n+1}^{[r]}$  be regular, only the cases $r=0$ and $s=0$ are relevant. If $r=0$ and $s$ is even, then the generic fiber of $\wt{\CZ}_1$ is the  member $S_{\wt K_{n}^{[r,s]}} $ of the RZ tower corresponding to the open compact subgroup $\wt K_{n}^{[r,s]}\subset \U(W_0^\flat)(F_0)$ equal to  the parahoric $K_{n}^{[0,s]}$.  If  $r=0$  and $s$ is odd, the generic fiber of $\wt{\CZ}_1$ is again a member of the RZ tower but $\wt K_{n}^{[r,s]}$ is a non-parahoric.  In either case $\wt{\CZ}_1$ is given by the formal scheme $\wt\CN_n^{[r]}$ of  \cite[\S 3.5]{LRZ2}. If $s=0$ and $r$ is odd, then  the generic fiber of $\wt{\CZ}_1$ is the  member $S_{\wt K_{n}^{[r,s]}} $ of the RZ tower for $\wt K_{n}^{[r,s]}=K_{n}^{[r-1,0]}$.  In this case  $\wt{\CZ}_1$ is given by $\wt\CM_n^{[r]}$, cf. \cite[\S 3.10]{LRZ2}. If $s=0$ and $r$ is even $\neq 0$, then the generic fiber of $\wt{\CZ}_1=\wt\CM_n^{[r]}$ is not a member of the RZ tower of $\CN_n^{[s]}$, cf. \cite[\S 3.10]{LRZ2} (but the generic fiber of its closed formal subscheme  $\wt\CM_n^{[r], +}$ is).

\subsection{Test functions and the lattice models} 
To determine the correct test functions $\varphi$, we consider the lattice models of the RZ spaces in play, motivated by the global aspects of the conjectures. To explain this, 
we temporarily let $F/F_0$ be a CM extension of a totally real field; we refer to \cite{RSZ3} for  unexplained notation.  Consider the integral model $\CM$ of the Shimura variety associated to a variant of the unitary group $G_{W_0}$ (now over a global field) with  level $K\subset G_{W_0}(\BA_f)$,  and the arithmetic diagonal cycle $z_K$. For a bi-$K$-invariant $f\in \CC^\infty_c(G_{W_0}(\BA_f))$,  the global arithmetic Gan--Gross--Prasad conjecture \cite{Z12,RSZ3} concerns  the arithmetic intersection pairing  of Gillet--Soul\'e  on the arithmetic Chow groups of $\CM$,
$$
 \Int(f)=\bigl(\wh R(f)\wh z_{K}, \wh z_{K} \bigr)_{ \GS} .
 $$
It decomposes into a sum  of local terms $\Int_v(f)$ given by intersection numbers  at all places of the reflex field  above the given place $v$ of $F_0$.   Via a non-archimedean  uniformization, the local intersection numbers are in turn related to  intersection numbers on the relevant RZ spaces studied in this paper. We may transport our local construction of our cycles on RZ spaces to the global integral model $\CM$ and we can find the test function $f$ by considering the generic fiber of $\CM$. Since the Hecke action on the generic fiber is defined through the change of level structures involving Tate modules, it can be detected by considering a lattice model within a fixed rational Tate module.  
 
  For example, the lattice model of $\CN_{n, \ep}^{[t]}$ is $\BN_{n, \ep}^{[t]}$, defined as the set of vertex lattices $\Lambda$ of type $t$ in a hermitian space $W^\flat$ of dimension $n$ and Hasse invariant $\ep$. We then translate the construction of the naive  correspondences to the setting of lattice models.  We construct a  function $\varphi$ (the \emph{test function}) with the characterizing property that the naive (set theoretical) intersection number  on the lattice model is equal to a suitable orbital integral of this function, comp. \S\ref{sec:(n,t)-lat}, \ref{ss: Lat(n-1,t) Z}, \ref{ss: Lat(n-1,t) Y}, \ref{sec: t n lattice}, \ref{sec:lattice (n odd, n+1)}. 
 This  function is then  used to formulate the AT conjecture.
 
\subsection{Acknowledgements}
We thank Andreas Mihatsch for helpful discussions. Y. Luo and M. Rapoport  thank the departments of Mathematics at MIT and Zhejiang University for their hospitality when part of this work was done.
W. Zhang was
supported by NSF grant DMS 2401548 and a Simons Investigator grant.

 \subsection{Notation} \subsubsection{General notation} We let $F/F_0$ be a quadratic extension of finite extensions of $\BQ_p$ ($p\neq 2$), with corresponding ring extension $O_F/O_{F_0}$. We denote the residue field of $F_0$ by $k$ and write $\BF=\bar k$ for a fixed algebraic closure. 
 We write  $\eta=\eta_{F/F_0}$ for the corresponding quadratic character of $F_0^\times$, and ${\rm N}:F^\times\to F_0^\times$ for the norm character. 
 
 For an algebraic variety $X$ over $F_0$, we write $\CC^\infty_c(X)$ for $\CC^\infty_c(X(F_0))$.
  \subsubsection{Flat closure and flat fiber product}
 Let $X$ be a formal scheme locally of finite type over  a complete discrete valuation ring $O$ with uniformizer $\pi$. We define the \emph{flat closure} $X^f\subset X$ to be the closed formal subscheme defined by the ideal sheaf $\CO_{X}[\pi^{\infty}]\subset \CO_X$ ($\pi$-power torsion elements of the structure sheaf). Then $X^f$  satisfies the following universal property: for any formal scheme $Y$ locally of finite type and flat over $\Spf O$, a morphism $Y\rightarrow X$ factors through $X^f$,
\begin{equation*}
\begin{aligned}
\xymatrix{
Y\ar[rr]\ar@{-->}[rd]&&X\\
&X^f .\ar[ur]&
}
\end{aligned}
\end{equation*}
Let $(\mathrm{formal}/\Spf O)$ be the category of formal schemes locally of finite type over $\Spf O$ and let $(\mathrm{fformal}/\Spf O)$ be the full subcategory of $O$-flat formal schemes, and let  $i:(\mathrm{fformal}/\Spf O)\hookrightarrow (\mathrm{formal}/\Spf O)$ be the inclusion.
Then the universal property can be reinterpreted by the   adjunction property 
\begin{equation*}
\mathrm{Mor}_{\mathrm{fformal}}(i(Y)),X)\simeq \mathrm{Mor}_{\mathrm{formal}}(Y,X^f).
\end{equation*}
Therefore, the flat closure preserves limits. Let $X, Y, Z$ be flat $\Spf O$-schemes, with morphisms $X\to Z$ and $Y\to Z$.  We define the flat fiber product by the following cartesian product in $(\mathrm{fformal}/O_F)$,
\begin{equation}\label{equ:flat-prod}
  \begin{aligned}
  \xymatrix{
  (Y\times_X Z)^f \ar[r] \ar[d] \ar@{}[rd]  &Y \ar[d] 
  \\ Z \ar[r] & X.} 
  \end{aligned}
\end{equation}

\section{The setting}\label{sec:thesetting}

Let $p$ be an odd prime number. Let $F/F_0$ be a quadratic extension of $p$-adic local fields. We denote by $q$ the number of elements in the residue field of $F_0$.

We fix uniformizers $\pi_0, \pi$ of $F_0$ and $F$ respectively, such that $\pi_0 = \pi^2$ (resp. $\pi_0 = \pi$) when $F/F_0$ is ramified (resp. unramified). Denote by $x\mapsto \bar{x}$ the action of the nontrivial element in $\Gal(F/F_0)$. We fix an extension of $\eta=\eta_{F/F_0}:F_0^\times\to  \{\pm1\}$ to a character $\wt\eta:F^\times\to \BC^\times$, as follows. When $F/F_0$ is ramified, we require $\wt\eta|_{O_{F^\times}}$ to factor through the unique non-trivial quadratic character of $k^\times$. There are two choices of such extensions depending on the choice of  the value $\wt\eta(\pi)$ (a number such that  $\wt\eta(\pi)^2= \wt\eta(-\pi\ov \pi)=\eta(-1)$). When $F/F_0$ is unramified, we simply take  $\wt\eta(x)=(-1)^{v(x)}$.

Let $n\geq 1$. Let
\begin{equation}
G'(F_0)=\GL_{n} (F)\times\GL_{n+1}(F) .
\end{equation}
For a $F/F_0$-hermitian space $W$ of dimension $n+1$, fix $u\in W$  a non-isotropic vector (the {\it special vector}), and let $W^\flat=\langle u\rangle^\perp$. Set
\begin{equation}
G_W=\U^\flat\times \U 
\end{equation}
and $H= \U^\flat$ with the diagonal embedding of $\U^\flat$ into $G_W$. We have the notion of {\it regular semi-simple} element  $\gamma\in G'(F_0)$, resp. $g\in G_W(F_0)$, as well as the notion of {\it matching} $\gamma\leftrightarrow g$ for regular semi-simple elements, comp. \cite[\S 2]{RSZ1}. These notions are with respect to the action of $H\times H$ on $G_W$, resp., of $H'_{1, 2}=H'_1\times H'_2=\Res_{F/F_0} (\GL_{n})\times ( \GL_{n}\times\GL_{n+1})$ on $\Res_{F/F_0}(\GL_{n} \times\GL_{n+1})$. It is important to note that the latter action is arranged \emph{after} the choice of $u\in W$. 

We let $W_0, W_1$ denote the two isomorphism classes of $F/F_0$-hermitian spaces of dimension $n+1$. For $g\in G_W(F_0)_{\rs}$ and for a function $f\in C_c^\infty(G_W)$, we introduce the orbit integral
\[
\Orb(g,f):=\int_{H(F_0)\times H(F_0)}f(h_1^{-1}gh_2)dh_1dh_2.
\]
Here on $H(F_0)\times H(F_0)$ we take a product measure of identical Haar measure on $H(F_0)$.

For $\gamma\in G'(F_0)_{\rs}$, for a function $f'\in C_c^{\infty}(G')$, and for a complex parameter $s\in\BC$, we use the notation $\Orb(\gamma, f', s)$ for the weighted orbital integral
\begin{equation*}
   \Orb(\gamma, f', s) := 
	   \omega(\gamma)\cdot \int_{H_{1,2}'(F_0)} f'(h_1^{-1}\gamma h_2) \lv\det h_1\rv^s \eta(h_2)\, dh_1\, dh_2,
\end{equation*}
where:
\begin{altitemize}
\item $\lv\phantom{a}\rv$ denotes the normalized absolute value on $F$.
\item We use fixed Haar measures on $H_1'(F_0)$ and $H_2'(F_0)$ and the product Haar measure on $H_{1,2}'(F_0) = H_1'(F_0) \times H_2'(F_0)$ and set 
\[
   \eta(h_2) := \eta(\det h_2')^n \eta(\det h_2'')^{n-1}
	\quad\text{for}\quad
	h_2 = (h_2', h_2'')\in H_2'({F_0}) = \GL_{n-1}(F_0) \times \GL_n(F_0),
\] 

\item $\omega: G'(F_0)_{\rs}\to \BC^\times$ is a \emph{transfer factor}, see \cite[\S 5]{RSZ1}. We will take the following explicit transfer factor:
$$
\omega(\gamma):=\wt{\eta}\Bigl(
\det(\wt{\gamma})^{-(n-1)/2}\det(\wt{\gamma}^ie)_{i=0,\cdots,n-1}
\Bigr),
$$
where for $\gamma=(\gamma_1,\gamma_2)\in G'(F_0)_{\rs}$, we set $\wt{\gamma}=s(\gamma)=(\gamma_1^{-1}\gamma_2)\ov{(\gamma_1^{-1}\gamma_2)}^{-1}\in S_n(F_0)$, and where $e=(0,\cdots,0,1)\in F^n$ is the column vector.
\end{altitemize}

We further define the special values for a regular semi-simple element $\gamma\in G'(F_0)$, 
\begin{equation*}
   \Orb(\gamma, f') := \Orb(\gamma, f', 0)
	\quad\text{and}\quad
	\del(\gamma, f') := \frac{d}{ds} \Big|_{s=0} \Orb({\gamma},  f',s) . 
\end{equation*}
The integral defining $\Orb(\gamma,f',s)$ is absolutely convergent, and $\Orb(\gamma,f')$ has the transformation property\[
   \Orb(h_1^{-1}\gamma h_2,f') =\Orb(\gamma,f') 
	\quad\text{for}\quad 
	(h_1, h_2)\in H_{1, 2}'(F_0)=H_1'(F_0)\times H_2'(F_0) .
\]

\begin{definition}
A function $f'\in C_c^{\infty}(G')$ and a pair of functions $(f_0,f_1)\in C_c^{\infty}(G_{W_0})\times C_c^\infty(G_{W_1})$ are \emph{transfers} of each other (for the fixed choices of Haar measures, our fixed choice of transfer factor, and a fixed choice of special vectors $u_i$ in $W_i$), if for each $i\in \{0,1\}$ and each $g\in G_{W_i}(F_0)_{\rs}$,
\[
\Orb(g,f_i)=\Orb(\gamma,f')
\]
whenever $\gamma\in G'(F_0)_{\rs}$ matches $g$.
\end{definition}
Note that this notion depends on the choices of the transfer factor and the Haar measures. But the truth of the AFL and AT conjecture is independent of these choices.

\section{The AFL conjecture}\label{sec:AFL}
In this section, we recall the statement of the AFL conjecture, now a theorem. We assume that $F/F_0$ is unramified. In this section, we  denote by $W_0$ the split hermitian space of dimension $n+1$ and by $W_1$ the non-split space. We also assume  that the special vector $u_1\in W_1$ has norm a unit in $F_0$.  Under these unramifiedness hypotheses, we have the AFL conjecture.  Before stating it, we recall the following theorem on the Jacquet--Rallis FL.

\begin{theorem}(\cite{Y,BP-FL,Zha21}) Fix a special vector $u_0\in W_0$ of the same length as $u_1$, so that the notion of \emph{transfer} between functions $f'\in \CC_c^\infty(G')$ and pairs of functions $(f_0, f_1)\in \CC_c^\infty(G_{W_0})\times \CC_c^\infty(G_{W_1})$ is defined. Then  the function $\mathbf{1}_{\GL_{n}(O)\times\GL_{n+1}(O)}$  is a transfer of $(\mathbf{1}_{K_0^\flat\times {K}_0}, 0)$. Here $K_0^\flat$, resp. $K_0$, is the stabilizer of a selfdual lattice in $W_0^\flat$, resp. $W_0$. 
\end{theorem}
In the statement above, the Haar measures on $H'_{1, 2}(F_0)$ and $H(F_0)$ are normalized in such a  way that the canonical maximal compact subgroup gets volume one. 
In particular, if  $\gamma\in G'(F_0)_\rs$ is matched with the element $g\in G_{W_1}(F_0)_\rs$, we have  
$$
  \Orb(\gamma, \mathbf{1}_{\GL_{n}(O_F)\times\GL_{n+1}(O_F)}\bigr) = 0.
$$This vanishing of the orbital integral  motivates considering its derivative. The AFL conjecture is the  following statement.
\begin{theorem}\label{conjAFL}
Let $\gamma\in G'(F_0)_\rs$ be matched to the element $g\in G_{W_1}(F_0)_\rs$. Then
$$
-2 \langle\Delta, g\Delta\rangle= \del \bigl(\gamma, \mathbf{1}_{\GL_{n-1}(O_F)\times\GL_n(O_F)}\bigr) .
$$
\end{theorem}
 To define the LHS, we need to introduce certain RZ-spaces. Let $\CN_n$ be the RZ-space over $\Spf O_{\breve F}$ parametrizing tuples $(X, \iota, \lambda, \rho)$, where $X$ is a strict formal $O_{F_0}$-module, where $\iota\colon O_F\to \End(X)$ is an action of $O_F$ which satisfies the {\it Kottwitz condition} of signature $(1, n-1)$, where $\lambda$ is a {\it compatible} principal polarization, and where $\rho$ is a {\it framing} with \emph{framing object} $(\BX_n, \iota_{\BX_n}, \lambda_{\BX_n})$.  To be precise,  set $\BX_1 = \BE$, the unique such triple for $n=1$, and define  inductively 
\begin{equation}\label{BX_n unram}
   \BX_n = \BX_{n-1} \times \ov\BE ,
\end{equation}
 where $\ov\BE$ is the same as $\BE$ but with the conjugate action of $O_F$. We take $\BX_n$ as the framing object for $\CN_n$. For ease of notation, we set $\BY=\BX_{n}$ and $\BX=\BX_{n+1}=\BY\times\bar\BE$.  Let $\Delta\subset \CN_{n}\times \CN_{n+1}$ denote the graph of the closed embedding
\begin{equation}\label{deltaAFL}
   \delta\colon \CN_{n}\to \CN_{n+1}, \quad Y\mapsto Y\times \bar\CE.
\end{equation}

In this case, both $W_1$ and $W_1^\flat$ are non-split, and $\U_1(F_0)=\Aut^\circ(\BX)$ acts on $\CN_{n+1}$ and  $\U^\flat_1(F_0)=\Aut^\circ(\BY)$ acts on $\CN_{n}$. Hence $g\in G_{W_1}(F_0)$ acts on $\CN_{n}\times \CN_{n+1}$, and $g\Delta$ denotes the translate under $g$ of $\Delta$. Finally,
\begin{equation}\label{eqchi}
\langle\Delta, g\Delta\rangle:=\chi\big(\CO_\Delta\otimes^\BL\CO_{g\Delta}\big)\cdot\log q .
\end{equation} 
Since $\CN_{n}\times \CN_{n+1}$ is a regular formal scheme, the complex appearing in \eqref{eqchi} is perfect. For $g\in G_{W_1}(F_0)_\rs$, the quantity on the right is finite, cf. \cite[proof of Lem. 6.1]{M-Th}

The conjecture is known (by global methods) (W. Zhang \cite{Zha21}, A. Mihatsch and W. Zhang \cite{MZ}, Z. Zhang \cite{ZZha}). It is also known (by local methods) when $n=1,2$ (W. Zhang \cite{Z12}, Mihatsch \cite{M-AFL}) and when $g$ is \emph{minuscule} (Rapoport-Terstiege-Zhang/He-Li-Zhu \cite{RTZ,HLZ}).

\section{Hermitian $O_{F}$-modules for  the ramified quadratic case}
From now on we assume that $F/F_0$ is a ramified quadratic extension. In this case, we know fewer RZ-spaces that are formally smooth, or with semi-stable reduction. More precisely, of the first kind we only have the case of exotic smoothness. This occurs for the $\pi$-modular even case and the almost $\pi$-modular odd case and in no other case.  There are no other RZ spaces with semi-stable reduction (\cite{He-Pappas-Rapoport}). But there are RZ-spaces which have ``Kr\"amer-style'' blowings-up which are semi-stable: these are attached to a vertex lattice.

 \subsection{Parahoric subgroups for ramified unitary groups}\label{sec:parahoric}
 
Let $F/F_0$ be a ramified quadratic extension of $p$-adic fields ($p>2$) with uniformizers $\pi$ and $\pi_0$, resp., such that $\pi^2=\pi_0$.

Let $(V,\phi)$ be a $F/F_0$-Hermitian space of dimension $n$. We have associated $F_0$-bilinear forms,
\begin{equation*}
(x,y)=\frac{1}{2}\tr_{F/F_0}(\phi(x,y)),
\quad
\langle x,y\rangle=\frac{1}{2}\tr_{F/F_0}(\pi^{-1}\cdot \phi(x,y)).
\end{equation*}
The form $(\cdot,\cdot)$ is symmetric, while $\langle\cdot,\cdot\rangle$ is alternating. 

For any $O_F$-lattice $\Lambda$ in $V$, we set
\begin{equation*}
	\Lambda^\vee=\{v\in V\mid \phi(v,\Lambda)\subset O_F\}=\{v\in V\mid \langle v,V\rangle\subset O_{F_0}\}.
\end{equation*}
The lattice $\Lambda$  is called a \emph{vertex lattice} if 
\begin{equation*}
	\pi\Lambda^\vee\subseteq \Lambda\subseteq \Lambda^\vee.
\end{equation*}
We call the dimension $\dim_{k}\Lambda^\vee/\Lambda$ the \emph{type} of the vertex lattice and denote it by $t(\Lambda)$. Note that this integer is always even and satisfies $0\leq t\leq n=\rank \Lambda$.
Let $m=\lfloor n/2\rfloor$. We will fix a maximal  chain of vertex lattices, and enumerate the lattices by their type,
\begin{equation*}
	\Lambda_{2m}\subset \Lambda_{2m-2}\subset \ldots \subset\Lambda_{0}= \Lambda_0^\vee,\quad t(\Lambda_{i})=i.
\end{equation*}
Note that when $n$ is even and $(V,\phi)$ is non-split, the $\pi$-modular lattice $\Lambda_{2m}$ is missing. However, we will still index the lattices as $\{2m,\ldots,2,0\}$ to simplify notation.
We  extend the maximal chain of vertex lattices into a polarized chain of lattices, see \cite[Ch. 3]{RZ}.

For each non-empty subset $I=\{t_1,\ldots,t_k\}\subseteq \{2m,\ldots,2,0\}$, ordering the elements as $t_1>t_2>\ldots>t_k$, we have a  sub-chain
\begin{equation}\label{equ:lattice-chain}
	\Lambda_I: \quad \Lambda_{t_1}\subset \ldots \subset \Lambda_{t_k}\subseteq  \Lambda_{t_k}^\vee\subset \Lambda_{t_2}^\vee\subset\ldots \subset\Lambda_{t_1}^\vee.
\end{equation}
This extends periodically to a polarized chain of lattices. 
If $I=\{t\}$, resp. $I=\{s, t\}$, we will use the notation $\Lambda_{[t]}$, resp. $\Lambda_{[s, t]}$ for this sub-chain, or for its periodic extension. 

Consider the subgroup
\begin{equation*}
	K_n^{I}=\{g\in \U(V)(F_0)\mid g\Lambda_{i}=\Lambda_{i},\quad \forall i\in I\}.
\end{equation*}
If $I=\{t\}$, we also write  $K_n^{I}=K_n^{[t]}$. It is a quasi-parahoric subgroup of $\U(V)$.
There is a functorial surjective homomorphism called the \emph{Kottwitz map}, 
\begin{equation*}
	\kappa:\U(V)\rightarrow \pi_1(\U(V))_I^\sigma=\BZ/2\BZ.
\end{equation*}
We define the subgroup $K_n^{I,\circ}=K_n^{I}\cap \Ker\kappa$. We have the following:

\begin{proposition}[{\cite{PR08}}]\label{prop:parahoric}
\begin{altenumerate}
\item The groups $K_n^{I,\circ}$ are parahoric subgroups of $\U(V,\phi)$.
\item When $n=2m+1$ is odd, the subgroups $K_n^I\neq K_n^{I,\circ}$ are never parahoric, and the Kottwitz homomorphism induces an isomorphism $K_n^{I}/K_n^{I,\circ}\simeq \{\pm 1\}$.
\item When $n=2m$ is even, we have $K_n^{I}=K_n^{I,\circ}$ if and only if $n\in I$. Otherwise, we have $K_n^{I}/K_n^{I,\circ}\simeq \{\pm 1\}$.\qed
\item When $n=2m$ is even and $n-2\in I$, we have $K_n^{I,\circ}=K_n^{I\cup\{n\},\circ}$.
\end{altenumerate}
\end{proposition}

\subsection{The strengthened spin condition}\label{sec:spin}
The strengthened spin condition, introduced by Smithling \cite{S}, builds upon the spin condition proposed by Pappas and Rapoport \cite{PR}, which is used to construct a moduli functor for ramified unitary local models and RZ spaces. In this section, we provide a brief overview of the definition of the strengthened spin condition, comp. \cite{Luo24}.

We keep the notation as in the  previous subsection.
 Define the $2n$-dimensional $F$-vector space
\begin{equation*}
	\CV:=V\otimes_{F_0}F,
\end{equation*}
where $F$ acts on the right tensor factor. The $n$-th wedge power $\prescript{n}{}{\CV}:=\bigwedge^n_F \CV$ admits a canonical decomposition
\begin{equation}\label{eq:spin-subspace}
	\prescript{n}{}{\CV}=\bigoplus_{\substack{r+s=n\\\epsilon\in\{\pm 1\}}}\prescript{n}{}{\CV}^{r,s}_\epsilon
\end{equation}
which is described in \cite[\S 2.5]{S}. Let us briefly review it, with the same notation convention as in \cite{RSZ1}.
 The operator $\pi\otimes 1$ acts $F$-linearly on $\CV$ with eigenvalue $\pm \pi$; let\begin{equation*}
	\CV=\CV_\pi\oplus \CV_{-\pi}
\end{equation*}
be the corresponding eigenspace decomposition. For a partition $r+s=n$, define 
\begin{equation*}
	\prescript{n}{}{\CV}^{r,s}:=\bigwedge^r_F \CV_\pi\otimes_F \bigwedge^r_F \CV_{-\pi},
\end{equation*}
which is naturally a subspace of $\prescript{n}{}{\CV}$. Furthermore, the symmetric form $(\cdot,\cdot)$ splits after base change from $V$ to $\CV$, and therefore there is a decomposition
\begin{equation*}
	\prescript{n}{}{\CV}=\prescript{n}{}{\CV}_1\oplus \prescript{n}{}{\CV}_{-1}
\end{equation*}
as a $\SO((\cdot,\cdot))(F)$-representation. The subspaces $\prescript{n}{}{\CV}_{\pm}$ have the property that for any Lagrangian subspace $\CF\subset \CV$, the line $\bigwedge^n_F\CF\subset\prescript{n}{}{\CV}$ is contained in one of them, and in this way they distinguish the two connected components of the orthogonal Grassmannian $\OGr(n,\CV)$ over $\Spec F$.
The subspaces $\prescript{n}{}{\CV}_{\pm 1}$ are canonical up to labeling, and we will follow the labeling conventions in \cite{S} to which we refer the reader for details.
The summands in the decomposition \eqref{eq:spin-subspace} are then given by
\begin{equation*}
	\prescript{n}{}{\CV}^{r,s}_\epsilon:=\prescript{n}{}{\CV}^{r,s}\cap \prescript{n}{}{\CV}_\epsilon
\end{equation*}
as intersection in $\prescript{n}{}{\CV}$ for $\epsilon\in \{\pm 1\}$.

Given an $O_F$-lattice $\Lambda\subset V$,  define
\begin{equation*}
	\prescript{n}{}{\Lambda}:=\bigwedge^n_{O_F}(\Lambda\otimes_{O_{F_0}}O_F),
\end{equation*}
which is naturally a lattice in $\prescript{n}{}{\CV}$. For fixed $r,s$, and $\epsilon$, define
\begin{equation}\label{eq:spin-lattice-sum}
	\prescript{n}{}{\Lambda}^{r,s}_\epsilon:=\prescript{n}{}{\Lambda}\cap \prescript{n}{}{\CV}^{r,s}_{\epsilon}
\end{equation}
as intersection in $\prescript{n}{}{\CV}$. Then $\prescript{n}{}{\Lambda}^{r,s}_\epsilon$ is a direct summand of $\prescript{n}{}{\Lambda}$. For an $O_F$-scheme $S$, define
\begin{equation}\label{eq:spin-def}
	L^{r,s}_{\Lambda,\epsilon}(S):=\mathrm{im}\left[
	\prescript{n}{}{\Lambda}^{r,s}_\epsilon\otimes_{O_F}\CO_S\rightarrow \prescript{n}{}{\Lambda}\otimes_{O_F}\CO_S
	\right].
\end{equation}
Let $\Fil\subset \Lambda\otimes_{O_F}\CO_S$ be a $\CO_S$-direct summand of $\CO_S$-rank $n$.  We say that $\Fil$ satisfies the \emph{strengthened spin condition} if the line bundle
\begin{equation}
	\bigwedge^n_{\CO_S}\Fil\subset\prescript{n}{}{\Lambda}\otimes_{O_F}\CO_S
\end{equation}
is contained in $L^{n-1,1}_{\Lambda,-1}(S)$.

\subsection{Strict $O_{F_0}$-modules}
In this subsection, we briefly review the theory of strict $O_{F_0}$-modules. 
For more details, see \cite{M-Th,KRZ,LMZ}.
Assume $p\neq 2$ throughout, recall that $F_0/\BQ_p$ is an extension of $p$-adic field with a fixed uniformizer $\pi_0\in O_{F_0}$. 
Assume the residue field of $O_{F_0}$ is a finite field of order $q$.
Let $S$ be an $\Spf O_{F_0}$-scheme.
A \emph{strict $O_{F_0}$-module} $X$ over $S$ is a pair $(X,\iota)$ where $X$ is a $p$-divisible group over $S$ and $\iota: O_{F_0}\rightarrow \End(X)$ an action such that $O_{F_0}$ acts on $\Lie(X)$ via the structure morphism $O_{F_0}\rightarrow \CO_S$. 
A strict $O_{F_0}$-module is called \emph{formal} if the underlying $p$-divisible group is a formal group. 
By Ahsendorf-Cheng-Zink \cite{ACZ}, there is an equivalence of categories between the category of strict formal $O_{F_0}$-modules over $S$ and the category of nilpotent $O_{F_0}$-displays over $S$.
To any strict formal $O_{F_0}$-module, there is also an associated crystal $\BD_X$ on the category of $O_{F_0}$-pd-thickenings. 
We define the (covariant relative) de Rham homology $D(X):=\BD_X(S)$.
There is a short exact sequence of $\CO_S$-modules:
\begin{equation*}
	0\rightarrow \Fil(X)\rightarrow D(X)\rightarrow \Lie(X)\rightarrow 0;
\end{equation*}
where $\Fil(X)\subset D(X)$ is the Hodge filtration. The (relative) Grothendieck-Messing theory states that the deformations of $X$ along $O_{F_0}$-pd-thickenings are in bijection with liftings of the Hodge filtration.

We will restrict to the case when $X=(X,\iota)$ is \emph{biformal}, see \cite[Defn. 11.9]{M-Th} for the definition. 
For a biformal strict $O_{F_0}$-module $X$, we can define the \emph{(relative) dual} $X^\vee$ of $X$, and hence the \emph{(relative) polarization} and the \emph{(relative) height}.
It follows from the definition that there is a perfect pairing
\begin{equation*}
	D(X)\times D(X^\vee)\rightarrow \CO_S
\end{equation*}
such that $\Fil(X)\subset D(X)$ and $\Fil(X^\vee)$ are orthogonal complements to each other.

When $S=\Spec R$ is perfect, a nilpotent $O_{F_0}$-display is equivalent to a relative Dieudonne module $M(X)$ over $W_{O_{F_0}}(R)$ with the action of a $\sigma$-linear operator $\uF$ and a $\sigma^{-1}$-linear operator $\uV$ such that $\uF\uV=\uV\uF=\pi_0\id$.

\subsection{Hermitian $O_{F}$-modules}\label{sec:unitary-module}
\begin{definition}\label{def:unitary-module}
	Let $S$ be a formal scheme over $\Spf O_{\breve F}$.
	\begin{altenumerate}
		\item A \emph{hermitian $O_{F}$-module of type $t$ and dimension $n$ over $S$} is a triple
		$   (X,\iota_X,\lambda_X)$
		consisting of a strict biformal $O_{F_0}$-module $X$ of height $2n$ and dimension $n$ over $S$, a homomorphism
		\[
		\iota_X \colon O_F \to \End_S (X),
		\]
		and a \emph{relative} polarization
		\[
		\lambda_X \colon X \to X^\vee,
		\]
		subject to the following constraints:
		\begin{altitemize}
			\item 
			the Rosati involution on $\End_S^\circ(X)$ attached to $\lambda_X$ induces the nontrivial Galois automorphism on $O_F$; and
			\item 
			$\Ker \lambda_X \subset X[\iota_X(\pi)]$ has height $q^{t}$.
		\end{altitemize}
		\item A hermitian $O_F$-module is of signature $(1,n-1)$ if the $O_F$-action satisfies the \emph{strengthened spin condition}: if $n > 1$ is even and $t=n$, then the condition states that the operator $\iota_X(\pi)+\pi$ acts on $\Lie X$ with the image $\mathrm{im}(\iota_X(\pi)+\pi)$ a locally direct summand of $\Lie X$ of $\CO_S$-rank $1$\footnote{This is just a reformulation of the spin condition in \cite[\S 3.1]{RSZ1}}. 
In general, denote by $D(X)$ and $D(X^\vee)$ the respective de Rham homology of $X$ and $X^\vee$. Since $\ker\lambda_X$ is contained in $X[\iota(\pi)]$ and of rank $q^t$, there is a unique (necessarily $O_F$-linear) isogeny $\lambda^\vee$ such that the composite
		\begin{equation*}
			X\xrightarrow{\lambda}X^\vee\xrightarrow{\lambda^\vee}X
		\end{equation*}
		is $\iota(\pi)$, and the induced diagram
		\begin{equation*}
			D(X)\xrightarrow{\lambda_*}D(X^\vee)\xrightarrow{\lambda^\vee_*}D(X)
		\end{equation*}
		then extends periodically to a polarized chain of $O_F\otimes_{O_{F_0}}\CO_S$-modules of type $\Lambda_{[t]}$, comp.  \eqref{equ:lattice-chain}. By \cite[Th. 3.16]{RZ}, \'etale-locally on $S$ there exists an isomorphism of polarized chains
		\begin{equation*}\label{equ:ss-trivilization}
			[\cdots\xrightarrow{\lambda_*^\vee}D(X)\xrightarrow{\lambda_*}D(X^\vee)\xrightarrow{\lambda_*^\vee}\cdots]\xrightarrow{\sim}\Lambda_{[t]}\otimes_{O_{F_0}}\CO_S,
		\end{equation*}
		which in particular gives an isomorphism of $O_F\otimes_{O_{F_0}}\CO_S$-modules
		\begin{equation}\label{eq:lie-alg-trivilization}
			D(X)\overset{\sim}{\longrightarrow}\Lambda_{t}\otimes_{O_{F_0}}\CO_S.
		\end{equation}
		The strengthened spin condition we  impose is that
		\emph{
			upon identifying $\Fil(X)$ with a submodule of $\Lambda_{t}\otimes_{O_{F_0}}\CO_S$ via \eqref{eq:lie-alg-trivilization}, it satisfies the strengthened spin condition \eqref{eq:spin-def}, i.e., the line bundle $\bigwedge^n_{\CO_S}\Fil(X)$ is contained in $L^{n-1,1}_{\Lambda_{-\fkt},-1}(S)$.
		}
	\end{altenumerate}
\end{definition}

\begin{remarks}\label{rmk:unitry-module}
Let us make a few remarks on the definition.
\begin{altenumerate}
\item The Rosati condition on $\lambda_X$ is equivalent to requiring that $\lambda_X$ is $O_F$-linear, where $O_F$ acts on the dual $X^\vee$ via the rule
\begin{equation*}\label{dual action}
	\iota_{X^\vee}(a) = \iota_X(\ov a)^\vee.
\end{equation*}
\item For general $t$, we do not impose the strengthened spin condition on $X^\vee$ because it is automatic, cf. \cite[Prop. 2.4.3]{Luo24}.
\item When $n$ is even and $t=n$, the strengthened spin condition we impose here is equivalent to the combination of the Kottwitz condition, the wedge condition, and the spin condition in \cite[\S 6]{RSZ1}, see \cite[Rem. 6.1]{RSZ1}.
\item When $t=0$, the strengthened spin condition can be replaced by the following two conditions:
\begin{altitemize}
	\item 
	(\emph{Kottwitz condition})
	For the action of $O_F$ on $\Lie X$ induced by $\iota_X$, there is an equality of polynomials 
	\[
	\charac\bigl(\iota_X(\pi) \mid \Lie (X)\bigr) = (T - \pi)(T + \pi)^{n-1} \in \CO_S[T];
	\]
	\item 
	(\emph{Wedge condition}) $\bigwedge^2_{\CO_S} \bigl(\iota_X(\pi) + \pi \mid \Lie (X)\bigr) = 0$.
\end{altitemize}
\item In the remaining cases, the strengthened spin condition  implies the Kottwitz condition and the wedge condition, see \cite[Rem. 2.4.2]{Luo24}.
\end{altenumerate}
\end{remarks}

\section{Unitary RZ spaces }\label{sec:RZ-spaces}

\subsection{Framing objects}\label{sec:framing}
In this section  we consider hermitian $O_{F}$-modules of signature $(1,n-1)$ over $\Spec \BF$. This subsection continues the discussion in \cite{RSZ1,RSZ2}.

Let $(\BE,\iota_\BE,\lambda_\BE)$ be an isoclinic hermitian $O_{F}$-module of signature $(1,0)$. The deformation space is isomorphic to $\Spf O_{\breve{F}}$ and there is a unique lifting (\emph{the canonical lifting}) $\CE$ of the hermitian $O_{F}$-module. 
Define $\ov{\BE}$ to be the same $O_{F_0}$-module as $\BE$ but with $O_F$-action given by $\iota_{\ov{\BE}}:=\iota_\BE\circ (-)$, where $(-)$ is the Galois conjugation with respect to $F/F_0$, and $\lambda_{\ov{\BE}}:=\lambda_{\BE}$, and similarly define $\ov{\CE}$ and $\lambda_{\ov{\CE}}$.

Let $(\BX_n^{[t]},\iota_{\BX_n^{[t]}},\lambda_{\BX_n^{[t]}})$ be a hermitian $O_{F}$-module of signature $(1,n-1)$ and type $t$ over $\Spec \BF$ called the \emph{framing object} of dimension $n$. 
In this paper, we will require it to be isoclinic throughout. We define the space of \emph{special quasi-homomorphisms}
\begin{equation}\label{eq:special-hom}
\BV_n=\BV(\BX_n^{[t]}):=\Hom^\circ_{O_F}(\ov{\BE},\BX_n^{[t]}).
\end{equation}
Then $\BV_n$ is an $n$-dimensional $F$-vector space. It carries a natural nondegenerate $F/F_0$-hermitian form $h$: for $x,y\in \BV_n$, the composite
\begin{equation*}
	\ov{\BE}\xrightarrow{y}\BX_n
	\xrightarrow{\lambda_{\BX_n}}\BX_n^\vee\xrightarrow{x^\vee}\ov{\BE}^\vee\xrightarrow{\lambda^{-1}_{\ov{\BE}}}\ov{\BE}
\end{equation*}
lies in $\End^\circ_{O_F}(\ov{\BE})$ and, hence, identifies with an element $h(x,y)\in F$ via the isomorphism
\begin{equation*}
	\iota_{\ov{\BE}}:F\overset{\sim}{\longrightarrow}\End^\circ_{O_F}(\ov{\BE}).
\end{equation*}

We have the following result:
\begin{theorem}\label{thm:framing-obj}
	\begin{altenumerate}
		\item When $n$ is odd, there are two non-isomorphic framing objects $\BX^{[t]}_{n,\ep}$, where $\ep=\pm 1$. The Hasse invariant of  $\BV_n(\BX^{[t]}_{n,\ep})$ is $-\ep$.
		\item When $n$ is even and $0\leq t\leq n-2$, there are two non-isomorphic framing objects $\BX^{[t]}_{n,\ep}$, where $\ep=\pm 1$. The Hasse invariant of  $\BV_n(\BX^{[t]}_{n,\ep})$ is $-\ep$.
		\item When $n$ is even and $t=n$, up to isomorphism there is only one framing object $\BX_n^{[n]}$. Then  $\BV(\BX_n^{[n]})$ is non-split.
	\end{altenumerate}
\end{theorem}

In the sequel, we use the notation $\ep(\BX)=-\ep(\BV(\BX))$.
We will prove Theorem \ref{thm:framing-obj} in the end of this subsection.

Denote by $M=M(\BX)$ the relative covariant Dieudonn\'e module of the framing object $\BX$. It is endowed with its Frobenius operator $\uF$ and Verschiebung $\uV$ such that $\uF\circ\uV=\uV\circ\uF=\pi_0$. The polarization on $\BX$ translates to an alternating form $\langle\cdot,\cdot\rangle$ on $M$ satisfying
\begin{equation*}
	\langle \uF x,y\rangle=\langle x,\uV y\rangle^\sigma\quad\text{ for all }x,y\in M,
\end{equation*}
where $\sigma$ denotes the Frobenius operator on $W_{O_{F_0}}(\BF)=O_{\breve{F_0}}$.
The $O_F$-action on $X$ translates to an $O_F$-action on $M$ commuting with $\uF$ and $\uV$ and satisfying
\begin{equation*}
	\langle ax,y\rangle=\langle x,\ov{a}y\rangle
\end{equation*}
for all $x,y\in M$.

\begin{lemma}\label{lem:sspin-dieudonne module}
Let $\Pi:=\iota(\pi)$ be the induced action of $\pi$ on $M$.
\begin{altenumerate}
\item Assume that $n$ is even and $t=n$. Then the strengthened spin condition is equivalent to
\begin{equation*}
	\uV M\stackrel{1}{\subset}\uV M+\Pi M.
\end{equation*}
\item Assume $t\neq n$. Then  the strengthened spin condition is equivalent to
\begin{equation*}
	\uV M\stackrel{\leq 1}{\subseteq}\uV M+\Pi M.
\end{equation*}
\end{altenumerate}
\end{lemma}
\begin{proof}
Recall that we have the identification $\Lie(X)=M/\uV M$. The case (i) is proved in \cite[Prop. 3.10]{RSZ1}. For case (ii), by \cite[Prop. 2.4]{HLS2}, over  geometric points, the strengthened spin condition is equivalent to the Kottwitz condition plus the wedge condition; the assertion now follows from the definitions.
\end{proof}

Denote  by $N:=M\otimes_{O_{\breve{F}_0}}\breve{F}_0$ the relative rational Dieudonn\'e module of $\BX$. Then $\langle\cdot,\cdot\rangle$ extends to a nondegenerate alternating form on $N$. 
The classification of framing objects up to quasi-isogeny reduces to classifying such polarized isocrystals with $F$-action (in the relative sense).

Fix an element $\delta\in O_{\breve{F}_0}^\times$ such that $\sigma(\delta)=-\delta$. Then $N$ is an $n$-dimensional $\breve{F}$-hermitian space equipped with the hermitian form $h$ defined\footnote{The factor $\delta$  is necessary to descend the hermitian form to $C$. In \cite{RSZ1}, the operator $\tau$ is defined  in a different way, which explains why $\delta$  is  missing in \cite[\S 3.3]{RSZ1}.} 
by:
\begin{equation*}
	h(x,y):=\delta (\langle \Pi x,y\rangle+\pi\langle x,y\rangle).
\end{equation*}
We can use the relation 
\begin{equation*}
	\langle x,y\rangle=\frac{1}{2\delta}\tr_{\breve{F}/\breve{F}_0}(\pi^{-1}h(x,y))
\end{equation*}
to recover $\langle\cdot,\cdot\rangle$.

Since $N$ is supersingular, all slopes of the $\sigma$-linear operator
\begin{equation*}
	\tau:=\Pi\uV^{-1}:N\rightarrow N
\end{equation*}
are $0$. Hence, 
\begin{equation*}
	C:=N^{\sigma=1}
\end{equation*}
is an $F_0$-subspace of $N$ such that
\begin{equation}\label{eq:C-iso-to-V}
	C\otimes_{F_0}\breve{F}_0\xrightarrow{\sim} N;
\end{equation}
and in this way $\id_C\otimes\sigma$ identifies with $\tau$. Furthermore, $C$ is $F$-stable, and the restriction of $h$ to $C$ makes $C$ into a non-degenerate $F/F_0$-hermitian space of dimension $n$.

By choosing a generator of the relative Dieudonn\'e module of $\ov{\BE}$, we have an isomorphism
\begin{equation}\label{equ:BV-to-C}
	\BV(\BX_n)=\Hom_{O_F}^\circ(\ov{\BE},\BX_n)\overset{\sim}{\longrightarrow} C.
\end{equation}
We refer the reader to \cite[Lem. 3.9]{KR-U1} and \cite[Lem. 3.6]{Shi23} for its construction and proof. We fix this isomorphism once and for all, and use it throughout the paper.

Clearly, to classify $N$ up to isomorphism as a polarized isocrystal with $F$-action is to classify $C$ up to similarity as hermitian space.
It remains to construct those spaces. 
We begin by recalling the following proposition.
\begin{proposition}\label{prop:RSZ-frame}
	\begin{altenumerate}Let $\zeta\in O^\times_{F_0}\setminus N(O_F)$.
		\item When $n=1$ and $t=0$, then, up to quasi-isogeny, there are two framing objects,
		\begin{equation*}
			(\BX^{[0]}_{1,-1},\iota_{\BX^{[0]}_{1,-1}},\lambda_{\BX^{[0]}_{1,-1}})=(\ov{\BE},\iota_{\ov{\BE}},\lambda_{\ov{\BE}}),
			\quad\text{and}\quad
			(\BX^{[0]}_{1,1},\iota_{\BX^{[0]}_{1,1}},\lambda_{\BX^{[0]}_{1,1}})=(\ov{\BE},\iota_{\ov{\BE}},\zeta\lambda_{\ov{\BE}}).
		\end{equation*}
		\item When $n=2$ and $t=0$, then, up to quasi-isogeny, there exist two framing objects, 
		\begin{equation*}
			(\BX^{[0]}_{2,1},\iota_{\BX^{[0]}_{2,1}},\lambda_{\BX^{[0]}_{2,1}}),\quad\text{and}\quad(\BX^{[0]}_{2,-1},\iota_{\BX^{[0]}_{2,-1}},\lambda_{\BX^{[0]}_{2,-1}}).
		\end{equation*} 
		\item For $n$ even and $t=n$, up to quasi-isogeny, there exists a unique framing object $(\BX_n^{[n]},\iota_{\BX_n^{[n]}},\lambda_{\BX_n^{[n]}})$. In this case, the hermitian space $\BV(\BX_n^{[n]})$ is non-split.
	\end{altenumerate}
\end{proposition}
\begin{proof}
	(i)(iii) are proved in \cite[\S 3]{RSZ1}, and (ii) is proved in \cite[\S 8]{RSZ2}
\end{proof}

\begin{proof}[Proof of Theorem \ref{thm:framing-obj}]
We construct the required framing objects by induction.
The assertion is already established for $n=1$ and $2$ by Proposition \ref{prop:RSZ-frame}.
Now, assume that the assertion holds for $n-1$:
\begin{altitemize}
\item Suppose $n-1$ is even. Then we have constructed framing objects
\begin{equation*}
\BX_{n-1,\pm 1}^{[t]}\quad \text{for all }0\leq t\leq n-3,\quad\text{and}\quad\BX_{n-1,1}^{[n-1]}.
\end{equation*}
Define
\begin{equation*}
(\BX_{n,\ep}^{[t]},\iota_{\BX_{n,\ep}^{[t]}},\lambda_{\BX_{n,\ep}^{[t]}})
:=(\BX_{n-1,\ep}^{[t]}\times \ov{\BE},\iota_{\BX_{n-1,\ep}^{[t]}}\times \iota_{\ov{\BE}},\lambda_{\BX_{n-1,\ep}^{[t]}}\times \lambda_{\ov{\BE}}).
\end{equation*}
It is straightforward to verify they are the hermitian $O_F$-modules by Lemma \ref{lem:sspin-dieudonne module} and Proposition \ref{prop:RSZ-frame}(i).
Therefore, these form the desired framing objects.
To complete this step, it remains to  construct $\BX_{n,-1}^{[n-1]}$. It is
\begin{equation*}
(\BX_{n,-1}^{[n-1]},\iota_{\BX_{n,-1}^{[n-1]}},\lambda_{\BX_{n,-1}^{[n-1]}}):=(\BX_{n,1}^{[n-1]},\iota_{\BX_{n,1}^{[n-1]}},\delta\lambda_{\BX_{n,1}^{[n-1]}}).
\end{equation*}
		
\item Suppose $n-1$ is odd. In this case, we have constructed framing objects
\begin{equation*}
\BX_{n-1,\pm 1}^{[t]},\quad \text{for all }0\leq t\leq n-2.
\end{equation*}
Define 
\begin{equation*}
(\BX_{n,\ep}^{[t]},\iota_{\BX_{n,\ep}^{[t]}},\lambda_{\BX_{n,\ep}^{[t]}})
:=(\BX_{n-1,\ep}^{[t]}\times \ov{\BE},\iota_{\BX_{n-1,\ep}^{[t]}}\times \iota_{\ov{\BE}},\lambda_{\BX_{n-1,\ep}^{[t]}}\times \lambda_{\ov{\BE}}).
\end{equation*}
It is straightforward to verify that they are the hermitian $O_F$-modules by Lemma \ref{lem:sspin-dieudonne module} and Proposition \ref{prop:RSZ-frame}(i).
Therefore, these form the desired framing objects.
Finally, $\BX_{n,1}^{[n]}$ is already constructed in Proposition \ref{prop:RSZ-frame}(iii), which completes the proof.\qedhere
\end{altitemize}
\end{proof}

\subsection{RZ spaces}\label{sec:RZ-defn}

For $S$ a scheme over $\Spf O_{\breve{F}}$, let $
	\ov{S}:=\Spec \CO_S/\pi \CO_S$.
Let $(\BX_{n,\ep}^{[t]},\iota_{\BX_{n,\ep}^{[t]}},\lambda_{\BX_{n,\ep}^{[t]}})$ be the framing object constructed in Theorem \ref{thm:framing-obj}, it is a supersingular hermitian $O_{F}$-module of signature $(1,n-1)$ and type $t$ over $\Spec \BF$.

We define the \emph{unitary RZ space} $\CN_{n,\ep}^{[t]}$ over $\Spf O_{\breve F}$  as the formal scheme parametrizing isomorphism classes of quadruples $(X, \iota_X, \lambda_X, \rho_X)$, where
\begin{altitemize}
\item $(X,\iota_X,\lambda_X)$ is a hermitian $O_{F}$-module of signature $(1,n-1)$ and type $t$ over $S$;
\item  $\rho_X$, called the {\it framing} with framing object $(\BX_{n,\ep}^{[t]},\iota_{\BX_{n,\ep}^{[t]}},\lambda_{\BX_{n,\ep}^{[t]}})$, is an $O_F$-linear quasi-isogeny of height $0$
\begin{equation*}
	\rho_X: X\times_S \ov{S}\longrightarrow \BX_n^{[t]}\times_{\BF}\ov{S} ,
\end{equation*}
such that $\rho^*(\lambda_{\BX_n^{[t]}}\times_{\BF}\ov{S})=\lambda_X\times_S \ov{S}$. 
\end{altitemize}

Here an isomorphism between quadruples $(X,\iota_X,\lambda_X,\rho_X)\xrightarrow{\sim} (Y,\iota_Y,\lambda_Y,\rho_Y)$ is an $O_F$-linear isomorphism of hermitian $O_{F}$-modules $\alpha:X\xrightarrow{\sim}Y$ over $S$ such that we have $\rho_Y\circ(\alpha\times_S \ov{S})=\rho_X$ and such that $\alpha^*(\lambda_Y)=\lambda_X$.

The RZ spaces $\CN_{n,\ep}^{[t]}$ considered here are all flat, see Theorem \ref{thm:geo-RZ-canonical} for more properties.
\begin{remarks}\label{rmk:small-dim}
Let us make some remarks about the RZ spaces in small dimensions.
\begin{altenumerate}
\item When $n=2$ and $t=0$,  we only need to impose the Kottwitz condition to achieve flatness, see \cite[p. 596–7]{P}. Furthermore, the characteristic polynomial equals $(T-\pi)(T+\pi)=T^2-\pi_0$. Hence the RZ space $\CN_{2,\ep}^{[0]}$ can be defined over $\Spf O_{\breve{F}_0}$. When $\ep=1$, this model is isomorphic to $\Spf O_{\breve F_0}[[X, Y]]/(XY-\pi_0)$. When $\ep=-1$, this model is given by the formal Drinfeld halfplane attached to $F_0$. In either case, this model has  semi-stable reduction over $O_{\breve F_0}$, see \cite[Rem. 7.9 and \S 8]{RSZ2}. It is thus regular.  In \cite[\S 13]{RSZ2}   an Arithmetic Transfer conjecture for this model is formulated (and proved). Note that the  base change of this model  to $\Spf O_{\breve{F}}$ is not regular. Therefore the conjecture in \cite{RSZ2}  is genuinely different from the one discussed here, which pertains to our space over $O_{\breve F}$. 

\item When $n=2$ and $t=2$ (in which case $\ep=1$), we only need to  impose the Kottwitz condition and the spin condition defined in \cite[\S 7, esp. (7.10)]{PR}, to achieve flatness, see also \cite[Ex. 6.5]{RSZ2}.  The model $\CN_{2, \ep}^{[2]}$ is smooth. Again, it can be defined over $\Spf O_{\breve{F}_0}$, but we only consider the model over $O_{\breve{F}}$. 

\item When $n=3$ and $t=0$, we only need to impose the Kottwitz condition, see \cite[4.5, 4.15]{P} or \cite[\S 6]{PR}.

\item When $n=3$ and $t=2$, we need to impose the strengthened spin condition and the RZ space is smooth over $\Spf O_{\breve{F}}$, see \cite[\S 6]{PR}.
\end{altenumerate}
\end{remarks}

For $n\geq 1$, denote by $g\mapsto g^\dagger$ the Rosati involution on $\End^\circ_{O_F}(\BX_{n,\ep}^{[t]})$ induced by $\lambda_{\BX_{n,\ep}^{[t]}}$. Define
\begin{equation*}
	\U(\BX_{n,\ep}^{[t]}):=\U(\BX_{n,\ep}^{[t]},\iota_{\BX_{n,\ep}^{[t]}},\lambda_{\BX_{n,\ep}^{[t]}}):=\left\{
	g\in \End_{O_F}^{\circ}(\BX_{n,\ep}^{[t]})\mid gg^\dagger=\id_{\BX_{n,\ep}^{[t]}}
	\right\}.
\end{equation*}
By \eqref{eq:C-iso-to-V}, we see that $\U(\BX_{n,\ep}^{[t]})$ is a genuine unitary group, with Hasse invariant $-\ep$.
Each $g$ in the group $\U(\BX^{[t]}_{n,\ep})$, which is a quasi-isogeny of $\BX_{n,\ep}^{[t]}$ to itself of height $0$, and therefore, $\U(\BX_{n,\ep}^{[t]})$ acts naturally on $\CN_{n,\ep}^{[t]}$ on the left via the rule
\begin{equation*}
	g\cdot(X,\iota_X,\lambda_X,\rho_X)=(X,\iota_X,\lambda_X,g\circ\rho_X).
\end{equation*}

\begin{remark}
Using Proposition \ref{prop:RSZ-frame}, there exists an isomorphism
\begin{equation*}
	\nu:(\BX_{n,\ep}^{[t]},\iota_{\BX_{n,\ep}^{[t]}},\zeta\lambda_{\BX_{n,\ep}^{[t]}})\simeq (\BX_{n,\eta(\zeta)^n\ep}^{[t]},\iota_{\BX_{n,\eta(\zeta)^n\ep}^{[t]}},\lambda_{\BX_{n,\eta(\zeta)^n\ep}^{[t]}}).
\end{equation*}
Therefore, we can define the following isomorphism:
\begin{equation}\label{equ:odd-RZ-iso}
	\zeta_*: \CN_{n,\ep}^{[t]}\rightarrow \CN_{n,\eta(\zeta)^n\ep}^{[t]},\quad (X,\iota_X,\lambda_X,\rho_X)\mapsto (X,\iota_X,\zeta\lambda_X,\nu\circ \rho_X).
\end{equation}
Furthermore, the map
\begin{equation*}
\begin{aligned}
	\xymatrix{
		\nu_*:\U(\BX_{n,\ep}^{[t]})\ar[r]& \U(\BX_{n,\eta(\zeta)^n\ep}^{[t]})& g\ar@{|->}[r]&\nu\circ g\circ \nu^{-1}.
	}
\end{aligned}
\end{equation*}
is an isomorphism. 
From the definition, we see that $\zeta_*$ is compatible with the isomorphism $\nu_*$.

In particular, when $n$ is odd, there are two non-isomorphic framing objects, but the corresponding RZ spaces are isomorphic.
We cannot give an explicit description for $\zeta_*$, since the isomorphism $\nu$ is not explicit (it only exists by the equality of Hasse invariants).
\end{remark}

\subsection{Geometry of RZ space}\label{sec:geo-RZ}
In this subsection, we study some basic geometric structure of RZ spaces.

Recall that in \S \ref{sec:framing}, we defined $C:=N^{\tau=1}$. By restricting the hermitian form on $N$, it is a hermitian space over $F$. We have an isomorphism $C\simeq \BV$ and we can recover $N$ as $N\simeq C\otimes_{F}\breve{F}$.

When $t\neq n$, by computation on Dieudonn\'e modules, the geometric points of the RZ space are given as follows by $O_{\breve F}$-lattices (cf. \cite[Prop. 2.4]{RTW} and \cite[Prop. 3.5]{HLS2}),
\begin{align}
\CN_{n,\ep}^{[t]}(\BF)={}&\Big\{M\subset N\mid \Pi M^\vee\subseteq M\stackrel{t}{\subseteq}M^\vee,\quad
\pi_0 M \stackrel{n}{\subset} \uV M \stackrel{n}{\subset} M, \quad
 \uV M\,\stackrel{\leq 1}{\subseteq}(\uV M+\Pi M)\Big\}.\label{equ:RZ-geo-pts-1}\\
={}&\Big\{M\subset N\mid \Pi M^\vee\subseteq M\stackrel{t}{\subseteq}M^\vee,\quad\Pi M\stackrel{n}{\subset} \tau^{-1}(M)\stackrel{n}{\subset}\Pi^{-1}M, \quad M\,\stackrel{\leq 1}{\subseteq}(M+\tau(M))\Big\}.\label{equ:RZ-geo-pts}
\end{align}

When $t=n$, the geometric points of the RZ space are given as follows by $O_{\breve F}$-lattices (cf. \cite[Prop. 3.4]{Wu16}),
\begin{equation}\label{equ:RZ-geo-pts-pi-modular}
\CN_{n,\ep}^{[n]}(\BF)=\Big\{M\subset N\mid \pi M^\vee=M\stackrel{n}{\subset}M^\vee,\quad\Pi M\stackrel{n}{\subset} \tau^{-1}(M)\stackrel{n}{\subset}\Pi^{-1}M, \quad M\stackrel{ 1}{\subset}(M+\tau(M))\Big\}.
\end{equation}

Suppose $t\neq n$, and let $\Lambda\subset C$ be any vertex lattice of type $t$. The base change $\breve{\Lambda}:=\Lambda\otimes_{O_F}O_{\breve{F}}$ defines a  lattice in $N$ satisfying 
\begin{equation*}
\Pi\breve{\Lambda}^\vee\subseteq\breve{\Lambda}\stackrel{t}{\subseteq}\breve{\Lambda}^\vee,\quad\text{and}\quad\tau(\breve{\Lambda})=\breve{\Lambda}.
\end{equation*}
Such lattice defines a geometric point $\WT(\Lambda)\in\CN_{n,\ep}^{[t]}(\BF)$.
Conversely, for any lattice $M\in \CN_{n,\ep}^{[t]}(\BF)$ such that $\tau(M)=M$, the $\tau$-invariants $\Lambda:=M^{\tau=1}\subset C$ form  a vertex lattice of type $t$.
Therefore, there is a one-to-one correspondence between vertex lattices of type $t$ in $C$, and $\tau$-invariant lattices in $\CN_{n,\ep}^{[t]}(\BF)$. We call these points of $\CN_{n,\ep}^{[t]}(\BF)$ the \emph{worst points} of the RZ space, which explains the chosen notation  $\WT(\Lambda)$. 

The reason for the name ``worst point'' is the following:
recall that we have $\tau:=\Pi\uV^{-1}$. Therefore, $M=\tau(M)$ implies that $\Pi M=\uV M\subset M$. Denoting by $X$ the hermitian $O_F$-module corresponding to $M$, its  Hodge filtration equals
\begin{equation*}
\Bigl[\Fil(X)\subset D(X)\Bigr]=\Bigl[\uV M/\pi_0 M\subset M/\pi_0 M\Bigr]
=\Bigl[\Pi M/\pi_0 M\subset M/\pi_0 M\Bigr].
\end{equation*}
This defines the \emph{worst point} $*$ of the local model $\bM_n^{[t]}$ associated to the RZ space, comp. \cite{PR} (see also Definition \ref{defn:lm} for the definition of the local model).

Let $\Sing(\CN_{n,\ep}^{[t]})$ be the disjoint union of all worst points
\begin{equation*}\label{equ:def-Sing}
	\Sing(\CN_{n,\ep}^{[t]}):=\coprod_{\substack{\pi\Lambda^\vee\subseteq\Lambda\subseteq \Lambda^\vee\subset C \\ t(\Lambda)=t}}\WT(\Lambda).
\end{equation*}

When $t=n$, then, due to  the spin condition, there is no $\pi$-modular vertex lattice in $C$, see \cite[\S 3.3]{RSZ1}. This corresponds to the fact that the local model in this case does not have a worst point, see \cite[Rem. 5.3]{PR}.

We recall the following facts about the RZ space.
\begin{theorem}\label{thm:geo-RZ-canonical}
\begin{altenumerate}
\item The formal scheme  $\CN_{n,\ep}^{[t]}$ is flat, normal and Cohen-Macaulay.
\item The formal scheme  $\CN_{n,\ep}^{[t]}$ is smooth over $\Spf(O_{\breve F})$ if and only if $t=n$ or $t=n-1$.
\item In all cases, the formal scheme $\CN_{n,\ep}^{[t]}\setminus \Sing(\CN_{n,\ep}^{[t]})$ is semi-stable.
\end{altenumerate}
\end{theorem}
\begin{proof}
 For (i), we refer to  \cite{Luo24} (and the literature cited there). For part (ii), the ``if'' part follows from \cite{PR} and \cite{A}. The ``only if'' part follows from \cite{Luo24}, see also \cite{He-Pappas-Rapoport}. 
Part (iii) follows from \cite[Cor. 1.3.2]{HLS2}.
\end{proof}

We conclude this subsection with an analysis of $\pi_0(\CN_{n,\ep}^{[t]})$. 
While this computation can be approached through a detailed study of the basic locus (see \cite{RTW,Wu16,HLS2}), such an approach alone is insufficient for our purposes, since  we will later need to consider RZ spaces $\CN_{n,\ep}^{[r,s]}$.
Instead, we will use a group-theoretical approach and reduce the problem to a computation involving affine Deligne-Lusztig varieties (ADLV).

Recall that $(V,\phi)$ is a $F/F_0$-hermitian space and we have the associated alternating form
\begin{equation*}
	\langle x,y\rangle=\frac{1}{2}\tr_{F/F_0}(\pi^{-1}\cdot\phi(x,y)).
\end{equation*}
We have the unitary similitude group 
\begin{equation*}
\wt{G}=\GU(V,\phi):=\left\{
g\in\GL_F(V)\mid \langle gx,gy\rangle=c(g)\langle x,y\rangle,\, c(g)\in F_0^\times
\right\},
\end{equation*}
where $c\colon \wt{G}\to \BG_m$ is the similitude factor, cf.  \cite[\S 1.2]{PR}.
Recall from \cite[1.b.]{PR} the Kottwitz map  
\begin{equation*}
\kappa_{\wt{G}}: \wt{G}(\breve{F}_0)\to \pi_1(\wt{G})_I=\left\{\begin{array}{ll}
	\BZ	&	n=2m+1;\\
	\BZ\oplus \BZ/2	&	n=2m.
\end{array}\right.
\end{equation*}
Here the first factor is given by  the homomorphism  
\begin{equation*}
\begin{aligned}
	\xymatrix{
		\mathrm{ht}:\wt{G}(\breve{F}_0)\ar[r]&\BZ&g\ar@{|->}[r]&\val(c(g)).
	}
\end{aligned}
\end{equation*}
In particular, the action of  $\sigma$ on $\pi_1(\wt{G})_I$ is trivial. Recall that $\breve{F}_0$ is the completion of the maximal unramified extension of $F_0$, with $\sigma$ denoting the lifting of Frobenius. 

Now let $I$ be a non-empty subset of  $I=\{t_1,\ldots,t_k\}\subseteq \{2m,\ldots,2,0\}$ and consider the quasi-parahoric subgroup 
\begin{equation*}
P_I:=\mathrm{Stab}_{\wt{G}}(\Lambda_I)\subset \wt{G}(F_0).
\end{equation*}
We define subgroups $P^\circ_I:=P_I\cap \Ker\kappa_{\wt G}$. We have the following result:
\begin{proposition}[{\cite{PR}}]\label{prop:similitude-parahoric}
\begin{altenumerate}
	\item The groups $P^\circ_I$ are parahoric subgroups of $\wt{G}$.
	\item When $n=2m$ is even, then $P_I=P^\circ_I$ if and only if $n\in I$. Otherwise,  $P_I/P^\circ_I\simeq \{\pm 1\}$.
	\item When $n=2m+1$ is odd, then $P_I=P^\circ_I$ for any $I$.
	\item When $n=2m$ and $n-2 \in I$, then $P_I^\circ=P_{I\cup\{n\}}^\circ$.\qed
\end{altenumerate}
\end{proposition}

We refer the reader to \cite[\S 2.4]{PR} for the definition of the minuscule cocharacter $\wt\mu=\wt\mu_{1,n-1}$ and the admissible set $\mathrm{Adm}_{P_I^\circ}(\wt\mu)$. 
We denote by $\breve{P}_I$ and $\breve{P}_I^\circ$ the base change of $P_I$ and $P_I^\circ$ to $\wt{G}(\breve{F}_0)$, resp.
For any $\wt b\in \wt{G}(\breve{F}_0)$, we have the generalized ADLV:
\begin{equation*}
X_{P_I^\circ}(\wt{G},\wt\mu,\wt b)=\Big\{
g\in \wt{G}(\breve{F_0})/\breve{P}_I^\circ\mid g^{-1}b\sigma(g)\in \bigcup\nolimits_{w\in\mathrm{Adm}(\{\wt\mu\})}\breve{P}_I^\circ w\breve{P}_I^\circ
\Big\}.
\end{equation*}
Similarly we define $X_{P_I}(\wt{G},\wt\mu,\wt b)$.
This is a perfect subscheme of the partial (ramified) Witt vector flag variety attached to $\breve{P}_I^\circ$,  locally of finite type over $\BF$, in the sense of \cite{Zhu-17} or \cite{Bhatt-Scholze}. 

We define the relative unitary similitude RZ space\footnote{The notation $\wt{\CN}_{n,\ep}^{[t]}$ for the unitary similitude group is used exclusively in this subsection and will not cause confusion with the $\wt{\CN}_n^{[t]}$ appearing in later sections.}  $\wt{\CN}_{n,\ep}^{[t]}$ in the sense of \cite{RZ}, which parameterizes isomorphism classes of quadruples $(X,\iota_X,\lambda_X,\rho_X)$ similar to \S \ref{sec:RZ-defn}, with the following changes:
\begin{altitemize}
\item The polarization $\lambda_X$ is given up to a scalar in $O_{F_0}^\times$;
\item The framing map $\rho_X$ is any $O_F$-linear quasi-isogeny  which preserves the polarizations up to a scalar in $O_{F_0}^\times$.
\end{altitemize}

Note that   an isomorphism between two quadruples is now subjected to the condition that the pull-back of one polarization coincides with the other polarization up to an $O_{F_0}^\times$-scalar.

By \cite[Rem. 3.6]{RSZ1}, the height zero part of the relative unitary similitude RZ space $(\wt{\CN}_{n,\ep}^{[t]})_0$ is our RZ space $\CN_{n,\ep}^{[t]}$ (which is open and closed in $\wt{\CN}_{n,\ep}^{[t]}$).

\begin{proposition}\label{prop:ADLV-RZ}
Let  $[\wt b]\in B(\wt{G},\{\wt\mu\})$ be the $\sigma$-conjugacy class in $\wt{G}(\breve{F}_0)$ defined by the isocrystal $N$ of the framing object, comp. \cite[\S 1]{RZ}. Let   $\wt b\in [\wt b]$ be a representative.
Let $\BM:=M(\BX_{n, \ep}^{[t]})\subset N$ be the Dieudonn\'e module of the framing object $\BX_{n, \ep}^{[t]}$. Then the map
\begin{align}
	\Phi:X_{P_{[t]}}(\wt{G},\wt\mu,\wt b)\to{}&(\wt{\CN}_{n,\ep}^{[t]})_{\rm red}^{\rm perf},\notag\\
	g\mapsto{}&g\BM.\notag
\end{align}
defines an isomorphism  of perfect schemes between the ADLV and the perfection of the underlying reduced scheme of the relative unitary similitude RZ space. 
\end{proposition}
\begin{proof}
This follows directly from a lattice description of $\wt{\CN}_{n,\ep}^{[t]}(\BF)$ similar to \eqref{equ:RZ-geo-pts-1}, see \cite[Prop. 3.7]{Wu16} or \cite[\S 3.2]{Zhu-17} for more details.
\end{proof}

The key ingredient is now the following theorem of He-Zhou. Recall that there is an identification of the set of connected components of the affine partial  flag variety associated to the parahoric $P^\circ$,
\begin{equation}
\pi_0(\wt{G}(\breve{F}_0)/P^\circ)=\pi_1(\wt{G})_I .
\end{equation}In \cite[\S 6]{He-Zhou}, He-Zhou  define an element $c(\wt b, \wt\mu)\in \pi_1(\wt{G})_I$, well defined up to the action of the subgroup $\pi_1(\wt{G})_I^\sigma$, such that  the image of $X_{P^\circ}(\wt{G},\wt\mu,\wt b)$ equals  $c(\wt b, \wt\mu)+ (\pi_1(\wt{G})_I)^\sigma$.

\begin{theorem}[\cite{He-Zhou}, Thm. 0.1]\label{thm:He-Zhou}
The intersection of $X_{P_I^\circ}(\wt{G},\wt\mu,\wt b)$ with the connected component of $\wt{G}(\breve{F}_0)/P^\circ$ corresponding to an element in $c(\wt b, \wt\mu)+ (\pi_1(\wt{G})_I)^\sigma$ is connected. In particular, after the choice of an element in the coset $c(\wt b, \wt\mu)+ (\pi_1(\wt{G})_I)^\sigma$, there is an identification 
\begin{equation*}\pushQED{\qed} 
	\pi_0\left(X_{P^\circ}(\wt{G},\wt\mu,\wt b)\right)=(\pi_1(\wt{G})_I)^\sigma.\qedhere
	\popQED
\end{equation*}
\end{theorem}

The application of this general result in our specific context leads to the following result.
\begin{proposition}\label{prop:ADLV-comp}
For any non-empty subset $I\subseteq\{0,2,\cdots,2m\}$, the number of connected components of $X_{P_I}(\wt{G},\wt\mu,\wt b)$ is case by case:
\begin{itemize}
\item When $n=2m+1$ is odd, then  $\pi_0\left(X_{P_I}(\wt{G},\wt\mu,\wt b)\right)=\pi_1(\wt{G})_I=\BZ$, defined by $\mathrm{ht}(g)$;
\item When $n=2m$ is even and $n\in I$, then $\pi_0\left(X_{P_I}(\wt{G},\wt\mu,\wt b)\right)=\pi_1(\wt{G})_I=\BZ\oplus\BZ/2$, where the first factor is defined defined by $\mathrm{ht}(g)$; 
\item When $n=2m$ is even and $n\notin I$,  then $\pi_0\left(X_{P_I}(\wt{G},\wt\mu,\wt b)\right)=\BZ$, defined by $\mathrm{ht}(g)$.
\end{itemize}
Hence $X^0_{P_I}(\wt{G},\wt\mu,\wt b)$  has two connected components if and only if $n$ is even and $n\in I$, and is connected in all other cases. Here  $X^0_{P_I}(\wt{G},\wt\mu,\wt b)$ denotes  the height $0$ part of $X_{P_I}(\wt{G},\wt\mu,\wt b)$.
\end{proposition}
\begin{proof}
When $P_I=P_I^\circ$, the assertion follows from Theorem \ref{thm:He-Zhou}, which is the case unless $n=2m$ is even and $n\notin I$. Therefore, we only need to consider the latter case.
In this case, we have $\pi_1(\wt{G})_I=\BZ\oplus\BZ/2$, with trivial action by $\sigma$.
For any $g\in P_I$, we have $\mathrm{ht}(g)=0$ and $P_I^\circ$ is the subgroup of $P_I$ where the second factor of $\kappa_{\wt{G}}$ vanishes, comp. \cite[discussion below (1.9)]{PR}. 

On the other hand,  the affine flag variety $\wt{G}(\breve{F}_0)/P_I^\circ$ is isomorphic to the disjoint sum of two copies of $\wt{G}(\breve{F}_0)/P_I$, distinguished by the Kottwitz map, comp. \cite[\S 3.2]{PR} or \cite[\S 4]{PR08} (in the last paper,  the positive characteristic affine flag varieties is considered but the proof applies also to the Witt vector affine flag varieties). By Theorem \ref{thm:He-Zhou}, the intersection of the ADLV with each connected component of the affine flag varieties is non-empty, we conclude that $X_{P_I^\circ}(\wt{G},\wt\mu,\wt b)$ is the disjoint union of two copies of $X_{P_I}(\wt{G},\wt\mu,\wt b)$. The assertion  follows. 
\end{proof}
\begin{corollary}\label{thm:RZ-connected}
The RZ space $\CN_{n,\ep}^{[t]}$ has two connected components when $n=2m$ is even and $t=n$, and is connected in all other cases.
\end{corollary}
\begin{proof}
By \cite[Rem. 3.6]{RSZ1}, the RZ space $\CN_{n,\ep}^{[t]}$ is the height $0$ part of the unitary similitude RZ space $\wt{\CN}_{n,\ep}^{[t]}$. The assertion follows from Proposition \ref{prop:ADLV-RZ} and Proposition \ref{prop:ADLV-comp}.
Note that we have the following equality between the height of the framing map and the height of ADLV:
\begin{equation*}
	\frac{1}{n}\mathrm{ht}(\rho_X)=\mathrm{ht}(c(\Phi^{-1}(X))),
\end{equation*}
see for instance, \cite[Lem. 1.5]{V2}.
\end{proof}
\begin{remark}
The first assertion also follows from \cite{RSZ1}, and the second assertion follows from \cite{HLS2}. When $t=0$, connectedness also follows from \cite{RTW}.
\end{remark}

\begin{remark}\label{rmk:connected-components}
The decomposition
\begin{equation*}
	\CN_n^{[n]}=\CN_n^{[n],+}\amalg \CN_n^{[n],-}.
\end{equation*}
relates as follows  to the Kottwitz map. 
Let $\BM:=M(\BX_n^{[n]})$ be the relative Dieudonn\'e module of the framing object and recall that $N=\BM[1/\pi_0]$ is the rational Dieudonn\'e module. 
The space $N$ is a $\breve{F}/\breve{F}_0$-hermitian space of dimension $n$.
Recall from   \eqref{equ:RZ-geo-pts-pi-modular} the description of the set of geometric points of  $\CN_n^{[n]}$ in terms of  $O_{\breve{F}}$-lattices. Then 
a geometric point $M\in\CN_n^{[n]}(\BF)$ lies in $\CN_n^{[n],+}$ or $\CN_n^{[n],-}$ according as the $O_{\breve{F}}$-length of the module 
\begin{equation}\label{equ:pi-modular-parity}
	(M+\BM)/\BM
\end{equation}
is even or odd, cf.  \cite[\S 6]{RSZ2}.
The parity of this length may also be described as follows.
Let $g\in \U(N)$. Then the determinant $\det(g)$ is a norm one element  in $\breve{F}$, and hence lies in $O_{\breve{F}}$ with reduction mod $\pi$ equal to $\ep(g)=\pm 1$.
For any $(X,\iota,\lambda)$ corresponding to a $\pi$-modular lattice $M$ in $N$, we can find $g\in \U(N)$ such that $g\BM=M$. Then the length of \eqref{equ:pi-modular-parity} is even or odd according as $\ep(g)$ is $1$ or $-1$.
\end{remark} 

We record the following consequence of the above discussion. 
\begin{proposition}\label{prop:pi-modular-comp-iso}
Let $n\in I$. There are isomorphisms $\CN_n^{[n],+}\simeq \CN_n^{[n],-}$. 
More precisely, let $g\in \U(\BX_n^{[t]})$. 
For $\ep\in\{\pm\}$, the  automorphism $g:\CN_n^{[n]}\rightarrow \CN_n^{[n]}$ restricts to an automorphism
\begin{equation*}\pushQED{\qed} 
	g:\CN_n^{[n],\ep}\rightarrow \CN_n^{[n],\ep\ep(g)}.\qedhere
	\popQED
\end{equation*}
\end{proposition}

\subsection{Splitting model $\CN_{n,\ep}^{[t],\spt}$}\label{sec:splitting-model}
Let $S$ be a scheme over $\Spf O_{\breve F}$.  A \emph{splitting structure} on a hermitian $O_{F}$-module  $(X,\iota_X,\lambda_X)$ of dimension $n$ and type $t$ is a pair
of two locally $\CO_S$-direct summands of rank one,
\begin{equation*}
	\Fil^0(X)\subseteq\Fil(X)\subset M(X),\quad \Fil^0(X^\vee)\subseteq\Fil(X^\vee)\subset M(X^\vee),
\end{equation*}
subject to the following constraints:
\begin{altitemize}
\item  The morphisms between Hodge filtrations induced by the polarization carry the one additional filtration into the other:
\begin{equation*}
	\lambda_*(\Fil^0(X))\subseteq \Fil^0(X^\vee),\quad \lambda^\vee_*(\Fil^0(X^\vee))\subseteq \Fil^0(X).
\end{equation*}
\item 
(\emph{Kr\"amer condition}) if $n > 1$ is even and $t=n$, then the condition states that $\Fil^0(X)=(\iota(\pi)-\pi)\Fil(X)$.
In general, it requires
\begin{align*}
	(\iota(\pi)-\pi)\Fil(X)\subseteq \Fil^0(X), & \quad (\iota(\pi)+\pi)\Fil^0(X)=(0);\\
	(\iota(\pi)-\pi)\Fil(X^\vee)\subseteq \Fil^0(X^\vee), &\quad (\iota(\pi)+\pi)\Fil^0(X^\vee)=(0).
\end{align*}
\end{altitemize}

\begin{remark}
The strengthened spin condition is part of the definition of a hermitian $O_F$-module of signature $(1, n-1)$ and type $t$, cf. Definition \ref{def:unitary-module}. In fact, the existence of the filtrations $\Fil^0(X)$ and $\Fil^0(X^\vee)$ with the above conditions implies the strengthened spin condition, cf. \cite[Thm. 1.4.1]{HLS1}. 
\end{remark}

Fix a framing object $(\BX_{n,\ep}^{[t]},\iota_{\BX_{n,\ep}^{[t]}},\lambda_{\BX_{n,\ep}^{[t]}})$. We define the \emph{naive splitting model} $\CN_{n,\ep}^{[t],\nspt}$ over $\Spf O_{\breve F}$ parametrizing the collection of data 
\begin{equation*}
	(X, \iota, \lambda, \Fil^0(X),\Fil^0(X^\vee); \rho),
\end{equation*}
where $(X,\iota,\lambda,\Fil^0(X),\Fil^0(X^\vee))$ is a hermitian $O_F$-module of  of signature $(1,n-1)$ and type $t$ with splitting  structure, and where $\rho$ is a {\it framing} with the fixed framing object.
We define the \emph{splitting model} $\CN_{n,\ep}^{[t],\spt}$ over $\Spf O_{\breve F}$ as the flat closure of $\CN_{n,\ep}^{[t],\nspt}$.

\begin{remarks}\label{rmk:split-model}
\begin{altenumerate}
\item In the $\pi$-modular case, the splitting structure on a hermitian $O_F$-module is uniquely determined if  the spin condition is satisfied, see Definition \ref{def:unitary-module}(ii) and Remark \ref{rmk:unitry-module}(iii).
Therefore, the naive splitting model $\CN_n^{[n],\nspt}$  and the splitting model $\CN_n^{[n],\spt}$ are isomorphic to the RZ space $\CN_n^{[n]}$.
\item As one sees from the moduli description, the naive splitting model $\CN_{1,\ep}^{[0],\nspt}$ and the splitting model $\CN_{1,\ep}^{[0],\spt}$ are both isomorphic to the RZ space $\CN_{1,\ep}^{[0]}$.
\item In the remaining cases, the splitting structure is uniquely determined outside the worst points. In fact,  the splitting model $\CN_{n,\ep}^{[t],\spt}$ is the blow-up of the RZ space in the worst points, cf. \cite[Thm. 1.3.1]{HLS1}.
\end{altenumerate}
\end{remarks}

We summarize some geometric properties of the splitting model:
\begin{theorem}[\protect{\cite{HLS1}}]\label{thm:geo-RZ}
\begin{altenumerate}
\item All splitting models $\CN_{n,\ep}^{[t],\spt}$ are flat and semi-stable;
\item The splitting model $\CN_{n,\ep}^{[t],\spt}$ is smooth if and only if $t=n$, in which case $\CN_n^{[n],\spt}\simeq \CN_n^{[n]}$.
\end{altenumerate}
\end{theorem}
\begin{proof}
Part (i) follows from \cite[Thm. 1.3.1. (i) and Rem. 1.3.4.]{HLS1}. Part (ii) follows from Thm. 1.3.1.(ii) of loc.cit.
\end{proof}

When $t\neq n$, we denote by $\Exc:=\pi^{-1}(\Sing)$  the preimage of the worst points along the natural projection $\pi:\CN_{n,\ep}^{[t],\spt}\to\CN_{n,\ep}^{[t]}$.

\subsection{Rapoport--Zink spaces of deeper parahoric level}\label{sec:RZ-deeper}

Let $0\le t\le s\le n$ be even integers and let $\ep\in\{\pm 1\}$. Let $\BX_{n,\ep}^{[s]}$ (resp. $\BX_{n,\ep}^{[t]}$) be the framing \emph{hermitian $O_F$-modules of $\CN_{n,\ep}^{[s]}$ (resp. $\CN_{n,\ep}^{[t]}$}). Fix an $O_F$-linear isogeny  $\alpha: \BX_{\ep}^{[s]}\rightarrow \BX_{\ep}^{[t]}$ compatible with polarizations such that $\ker \alpha\subseteq \BX_{\ep}^{[s]}[\pi_0]$ and has degree $q^{(s-t)/2}$. 

Consider the functor sending a $\Spf \OFb$-scheme $S$ to the set of isomorphism classes of tuples $(X^{[s]},\iota^{[s]},  \lambda^{[s]},\rho^{[s]}, X^{[t]},\iota^{[t]},\lambda^{[t]},\rho^{[t]})$, where
\[
(X^{[i]},\iota^{[i]},\lambda^{[i]},\rho^{[i]})\in \CN_n^{[i]}(S),\quad i\in\{s,t\},
\]
such that $(\rho^{[t]})^{-1}\circ\alpha\circ \rho^{[s]}: X^{[s]} \times_S \bar S\rightarrow X^{[t]}\times_S\bar S$ lifts to an isogeny $\wit\alpha: X^{[s]}\rightarrow X^{[t]}$. Note that if $\wit\alpha$ exists then it is unique  and $\ker \alpha\subseteq X^{[s]}[\pi_0]$ and has degree $q^{(s-t)/2}$. This functor is represented by a formal scheme $\CN_{n,\ep}^{[s,t]}$ known as the (relative) \emph{unitary Rapoport--Zink space of parahoric level}. The Rapoport--Zink space $\CN_{n,\ep}^{[s,t]}$ is formally locally of finite type,  of relative dimension $n-1$. In general, it is not regular but it is always flat \cite{Luo24}. By definition there are natural projections 
\begin{equation*}\label{eq Nrs}
\begin{aligned}
\xymatrix{&\CN_{n, \varepsilon}^{[s, t]} \ar[rd]  \ar[ld] &\\ \CN_{n, \varepsilon}^{[s]}&&  \CN_{n, \varepsilon}^{[t]}.}
\end{aligned}
\end{equation*}

The isogeny $\alpha$ induces an identification of the rational (relative) Dieudonn\'e modules of $\BX_{n,\ep}^{[s]}$ and $\BX_{n,\ep}^{[t]}$ as hermitian spaces. Their common value will be denoted by $N$.
	
Recall the assumption that $s>t$. When $s\neq n$, by computation on Dieudonn\'e modules, the geometric points of the RZ space $\CN_{n,\ep}^{[s,t]}$ are given as follows by $O_{\breve{F}}$-lattices:
\begin{equation*}
	\CN_{n,\ep}^{[s,t]}(\BF)=\Big\{
	M_s\subseteq M_t\subset N\mid \pi M_i^\vee\subseteq M_i\stackrel{i}{\subseteq}M_i^\vee,\, \Pi M_i\stackrel{n}{\subset} \tau^{-1}(M_i)\stackrel{n}{\subset} \Pi^{-1}M,\, M_i\stackrel{\leq 1}{\subseteq}(M_i+\tau(M_i)),\, i=s,t
	\Big\}.
\end{equation*}
When $n=2m$ is even and $s=n$, we have a similar description for the geometric points $\CN_n^{[s,t]}$, except that the last condition for $s=n$ is replaced by 
\begin{equation*}
	M_n\stackrel{1}{\subset} (M_n+\tau(M_n)).
\end{equation*}

Similarly to the discussion in \S \ref{sec:geo-RZ}, we have:
\begin{proposition}
Let  $s>t$. The RZ space $\CN_{n,\ep}^{[s,t]}$ has two connected components when $n=2m$ is even and $s=n$, and is connected in all other cases.
\end{proposition}
\begin{proof}
We define $\wt{\CN}_n^{[s,t]}$ as the relative unitary similitude RZ space. Let $\BM^{[s]}\subset \BM^{[t]}$ be the relative Dieudonn\'e module of the framing objects, then the map
\begin{equation*}
\begin{aligned}
\xymatrix{
	\Phi:X_{P_{[s,t]}}(\wt{G},\wt{\mu},\wt{b})\ar[r]&(\wt{\CN}_{n,\ep}^{[s,t]})_{\rm red}&g\mapsto g(\BM^{[s]}\subset \BM^{[t]}),
	}
\end{aligned}
\end{equation*}
defines an isomorphism between the ADLV and the underlying reduced scheme of the relative unitary similitude RZ space as perfect schemes. The assertion  now follows from Proposition \ref{prop:ADLV-comp}.
\end{proof}

By \cite{He16}, the morphism $\CN_{n}^{[n,t]}\rightarrow \CN_{n}^{[n]}$ is surjective on  geometric points, and we define $\CN_n^{[n,t],\pm}$ as the preimage of $\CN_n^{[n],\pm}$.
\begin{proposition}\label{prop:pi-modular-comp-iso-two-ind}
Let $n=2m$ be an even integer and let $s=n$. There are isomorphisms $\CN_n^{[n,t],+}\simeq \CN_n^{[n,t],-}$. 
More precisely, let $g\in \U(\BX_n^{[n]})$. 
For $\ep\in\{\pm 1\}$, the  automorphism $g:\CN_n^{[n,t]}\rightarrow \CN_n^{[n,t]}$ restricts to an isomorphism
\begin{equation*}
g:\CN_n^{[n,t],\ep}\rightarrow \CN_n^{[n,t],\ep\ep(g)}.
\end{equation*}
\end{proposition}
\begin{proof}
The projection $\CN_n^{[n,t]}\rightarrow \CN_n^{[n]}$ is $\U(\BX_n^{[n]})$-equivariant, hence we can reduce to  Proposition \ref{prop:pi-modular-comp-iso}. 
\end{proof}

\begin{remark}
When $n=2m$ is even, we will also have occasion to consider the RZ space $\CN_{n}^{[n,n-2,t]}$ with three indices.
The definitions and properties of this space parallel those of $\CN_n^{[s,t]}$. In fact, by \cite[Prop. 9.12]{RSZ2} and Proposition \ref{prop:ADLV-comp}, the natural projection $\CN_n^{[n,n-2,t]}\to \CN_n^{[n-2,t]}$ is a trivial double cover. Given this straightforward relationship, we omit the detailed construction and properties here.
\end{remark}

 \section{Special cycles on RZ spaces}
\label{sec:cycle} 
 \subsection{Special cycles at vertex level}\label{sec:special-cycles}
Fix a framing object $\BX_{n,\ep}^{[t]}$ and the corresponding RZ space $\CN_{n, \ep}^{[t]}$. Recall the space of \emph{special quasi-homomorphisms} defined in \eqref{eq:special-hom}, 
\begin{equation*}\label{eq:special-quasihom}
	\BV_{n,\ep}=\BV(\BX_{n,\ep}^{[t]}):=\Hom^\circ_{O_F}(\ov{\BE},\BX_{n,\ep}^{[t]})\simeq C.
\end{equation*}
It is an $n$-dimensional $F/F_0$-hermitian space with Hasse invariant $-\ep$. See \eqref{eq:C-iso-to-V} for the definition of $C$ and the isomorphism.

 Just as in the unramified $F/F_0$ case \cite{LRZ2}, there are two types of special cycles on $\CN_{n, \varepsilon}^{[t]}$,  namely $\CZ(u)_{n,\ep}^{[t]}$ and $ \CY(u)_{n,\ep}^{[t]}$.
Recall that $t\equiv 0\mod 2$.

 \begin{definition}\label{def:speccy}
  Fix a vector $x\in \BV_{n,\ep}$.
  \begin{altenumerate}
  \item We define the $\CZ$-cycle $\CZ(x)^{[t]}_{n,\ep}\subseteq \CN_{n,\ep}^{[t]}$ to be the closed formal subscheme which represents the functor sending each scheme $S$ over $\Spf O_{\breve{F}}$  to the isomorphism classes of tuples $(X,\iota_X,\lambda_X,\rho)$ such that the quasi-homomorphism
  \begin{equation*}
  	\rho^{-1}\circ x\circ \rho_{\ov{\CE}}:\ov{\CE}_S\times_S \ov{S}\overset{\rho_{\ov{\CE}}}{\longrightarrow}\ov{\BE}\times_{\Spec\BF}\ov{S}\overset{x}{\longrightarrow}\BX^{[t]}_{n,\ep}\times_{\Spec\BF}\ov{S}\overset{\rho^{-1}}{\longrightarrow}X\times_S \ov{S}
  \end{equation*}
  extends to a homomorphism $\ov{\CE}_S\rightarrow X$ (this is a closed condition by \cite[Prop. 2.9]{RZ}).

  \item We define the $\CY$-cycle  $\CY(x)^{[t]}_{n,\ep}\subseteq \CN_{n,\ep}^{[t]}$ to be the closed formal subscheme which represents the functor sending each $S$ to the isomorphism classes of tuples $(X,\iota_X,\lambda_X,\rho)$ such that the quasi-homomorphism
    \begin{equation*}
  	\lambda\circ\rho^{-1}\circ x\circ \rho_{\ov{\CE}}:\ov{\CE}_S\times_S \ov{S}\overset{\rho_{\ov{\CE}}}{\longrightarrow}\ov{\BE}\times_{\Spec\BF}\ov{S}\overset{x}{\longrightarrow}\BX^{[t]}_{n,\ep}\times_{\Spec\BF}\ov{S}\overset{\rho^{-1}}{\longrightarrow}X\times_S \ov{S}\overset{\lambda}{\longrightarrow}X^\vee\times_S \ov{S}
  \end{equation*}
    extends to a homomorphism $\ov{\CE}_S\rightarrow X^\vee$.
  
  \smallskip 
  
Recall that by Remark \ref{rmk:split-model}, the splitting model $\CN_{n, \ep}^{[t],\spt}$ is the blow-up of $\CN_{n, \ep}^{[t]}$ over the worst points.
\item  We define the \emph{splitting $\CZ$-cycle}  $\CZ(x)^{[t],\spt}_{n,\ep}$ and the \emph{splitting $\CY$-cycle}  $\CY(x)^{[t],\spt}_{n,\ep}$ in $\CN^{[t],\spt}_{n,\ep}$ as the strict transforms of $\CZ(x)^{[t]}_{n,\ep}$ and $\CY(x)^{[t]}_{n,\ep}$, respectively. 
To be more precise, writing $\pi\colon\CN_{n,\ep}^{[t],\spt}\to \CN_{n,\ep}^{[t]}$  for the projection map, we define the splitting $\CZ$-cycle $\CZ(x)_{n,\ep}^{[t],\spt}$ as the closed formal subscheme of the pullback $\pi^{-1}(\CZ(x)_{n,\ep}^{[t]})\subset \CN_{n,\ep}^{[t],\spt}$ cut out by the quasi-coherent ideal of sections of $\CO_{\pi^{-1}(\CZ(x)_{n,\ep}^{[t]})}$ supported on $\Sing$. We  define $\CY(x)^{[t],\spt}_{n,\ep}$ in a similar way.

When $(\CN_{n,\ep}^{[t]})_{\rm red}$  is strictly larger than $\Sing$, an equivalent way of defining $\CZ(x)_{n,\ep}^{[t],\spt}$ is as follows. 
The morphism $\pi$ defines an  isomorphism $\CN_{n,\ep}^{[t],\spt}\setminus \Exc\to \CN_{n,\ep}^{[t]}\setminus \mathrm{Sing}$.
Therefore, we obtain a commutative diagram, in which the oblique arrow is a  locally closed immersion,
\begin{equation*}
	\begin{aligned}
		\xymatrix{
			&\CN_{n,\ep}^{[t],\spt}\ar[d]\\
			\CZ(x)^{[t]}_{n,\ep}\setminus\Sing\ar@{^(->}[r]\ar@{^(->}[ur]&\CN_{n,\ep}^{[t]}.
		}
	\end{aligned}
\end{equation*}
Then  the splitting $\CZ$-cycle $\CZ(x)_{n,\ep}^{[t],\spt}$ is the Zariski closure of the locally closed subscheme $\CZ(x)_{n,\ep}^{[t]}\setminus\Sing\hookrightarrow\CN_{n,\ep}^{[t],\spt}$. The same applies to $\CY(x)^{[t],\spt}_{n,\ep}$.
\end{altenumerate}
\end{definition}

The proof of the following theorem is given in \S \ref{sec:exc-iso-proof}.
\begin{theorem}\label{thm:ss-comparison}
	Let 
	 $(Y,\iota_Y,\lambda_Y)$ be a hermitian $O_F$-module of dimension $n-1$ and type $t$ over $S\in ({\rm Sch}/\Spf O_{\breve{F}})$.
	Let $\zeta\in O_{F_0}^\times$ be a unit and define
	\begin{equation*}
		(X,\iota_X,\lambda_X):=(Y\times\ov{\CE},\iota_Y\times\iota_{\ov{\CE}},\lambda_Y\times \zeta\lambda_{\ov{\CE}}).
	\end{equation*}
	Then the following assertions hold:
	\begin{altenumerate}
		\item $(X,\iota_X,\lambda_X)$ is a hermitian $O_F$-module of dimension $n$ and type $t$.
		\item If $(Y,\iota_Y,\lambda_Y)$ satisfies the strengthened spin condition, then so does $(X,\iota_X,\lambda_X)$.
		\item Suppose $t\neq n-1$. Then, if $(X,\iota_X,\lambda_X)$ satisfies the strengthened spin condition, then so does $(Y,\iota_Y,\lambda_Y)$.
	\end{altenumerate}
\end{theorem}

As a consequence, we deduce the following:
\begin{theorem}\label{thm: Z unit}
		Let $u\in \BV(\BX^{[t]}_{n,\ep})$ be a unit length vector, with corresponding special cycle $\CZ(u)_{n,\ep}^{[t]}\subset \CN_{n, \ep}^{[t]}$. Set $\ep(u):=\eta(h(u,u))$ and $\ep^\flat=\ep\ep(u)\eta((-1)^{n-1})$. Then:
		\begin{altitemize}
			\item When $n$ is even and $t=n$, $\CZ(u)^{[t]}_{n,\ep}$ is empty (note that $\ep=1$ here).
			\item In the remaining cases, we have an isomorphism
			\begin{equation*}
				\CZ(u)_{n,\ep}^{[t]}\simeq \CN_{n-1,\ep^\flat}^{[t]},
			\end{equation*}
			except when:
			\item $n$ is odd, $t=n-1$, and $\ep^\flat=-1$, in which case the RHS is not defined  and  the special cycle $\CZ(u)^{[t]}_{n,\ep}$ is the disjoint union of  points $\WT(\Lambda)$ in $\Sing(\CN_{n-1, \ep}^{[t]})$ (indexed by all  almost $\pi$-modular $\Lambda\subset \BV(\BX_{n,\ep}^{[t]})$  containing $u$).
		\end{altitemize}
\end{theorem}
\begin{proof}
By direct computation on Dieudonn\'e modules, we see that $\CZ(u)^{[t]}_{n,\ep}$ is empty when $n$ is even and $t=n$. In all other cases, $\CZ(u)^{[t]}_{n,\ep}$ is non-empty.

For any $(X,\iota,\lambda)\in\CZ(u)^{[t]}_{n,\ep}$, we define 
\begin{equation*}\begin{aligned}\xymatrix{
e:\ov\CE\ar[r]^-{u}&X\ar[r]^{\lambda}&X^\vee\ar[r]^{u^\vee}&\ov\CE^\vee\ar[r]^{\sim}&\ov\CE.
}\end{aligned}\end{equation*}
A standard computation shows that $e^2=e$. Define 
\begin{equation*}
	(X^\flat,\iota_{X^\flat},\lambda_{X^\flat}):=\left(
	(1-e)X,\iota_{(1-e)X},\lambda_{(1-e)X}
	\right).
\end{equation*}
It is a hermitian $O_F$-module of dimension $n-1$. Assume $t\neq n-1$ (we will consider the $t=n-1$ case in Corollary \ref{cor:pi-modular-iso}). 
By Theorem \ref{thm:ss-comparison}, $(X^\flat,\iota_{X^\flat},\lambda_{X^\flat})$ satisfies the strengthened spin condition. 
The framing object $(\BX^\flat,\iota_{\BX^\flat},\lambda_{\BX^\flat})$ is isomorphic to $(\BX_{n-1,\ep^\flat}^{[t]},\iota_{\BX_{n-1,\ep^\flat}^{[t]}},\lambda_{\BX_{n-1,\ep^\flat}^{[t]}})$, denote this isomorphism by $f$.  We have an isomorphism
\begin{equation*}\label{equ:exceptional-isomorphism}
\begin{aligned}\xymatrix{
\CZ(u)^{[t]}_{n,\ep}\ar[r]^-{\sim}&\CN_{n-1,\ep^\flat}^{[t]},&(X,\iota,\lambda,\rho)\ar@{|->}[r]& \Bigl((1-e)X, (1-e)\iota, (1-e)\lambda, f\circ(1-e)\eta, (1-e)\rho)\Bigr),
}\end{aligned}\end{equation*}
with the inverse given by
\begin{equation*}\begin{aligned}\xymatrix{
\CN_{n-1,\ep^\flat}^{[t]}\ar[r]^-{\sim}&\CZ(u)^{[t]}_{n,\ep},&(Y,\iota_Y,\lambda_Y,\rho_Y)\ar@{|->}[r]&(Y\times\CE,\iota_Y\times\iota_\CE,\lambda_Y\times\lambda_\CE,(\rho_Y\circ f^{-1})\times\zeta\rho_{\CE}),
}\end{aligned}\end{equation*}
where $\zeta=h(u,u)\in O_{F_0}^\times$.
\end{proof}
There is  one more exceptional isomorphism, as follows.
\begin{theorem}\label{thm:Y unit}
Let $n$ be even and $t=n$. Let $u\in \BV(\BX_{n}^{[n]})$ be a unit length vector (note that there is no need to mention the epsilon factor, as $\ep=1$). Set $\ep^\flat=\ep(u)\eta((-1)^{n-1})$.
Define $ \CN_{n-1, \ep^\flat}^{[n-2],\circ}$ by the following  fiber product diagram,
  \begin{equation*}
 \begin{aligned}
 \xymatrix{
 \CN_{n-1, \ep^\flat}^{[n-2],\circ}\ar@{}[rd]|{\square}\ar[d]\ar@{^(->}[r]&\CN_{n}^{[n-2,n]}\ar[d]\\
 \CN_{n-1, \ep^\flat}^{[n-2]}\ar@{^(->}[r]&\CN_{n}^{[n-2]} .&}
 \end{aligned}
 \end{equation*}
 Then the morphism $\CN_{n}^{[n-2,n]}\to \CN_{n}^{[n-2]}$ is a trivial double covering, cf. \cite[Prop. 6.4]{RSZ2}. Furthermore,  the composition $ \CN_{n-1,\ep^\flat}^{[n-2],\circ}\to \CN_{n}^{[n-2,n]}\to \CN_{n}^{[n]}$ factors through $\CY(u)^{[n]}_{n}$ and induces an isomorphism
 $$
 \CN_{n-1,\ep^\flat}^{[n-2],\circ}\simeq \CY(u)^{[n]}_{n}.
$$In particular, there is a natural morphism 
 $$
 \CY(u)^{[n]}_{n}\to  \CN_{n-1,\ep^\flat}^{[n-2]},
$$
which is a trivial double covering.
Furthermore, $\CY(u)^{[n]}_{n}=\CZ(\pi u)^{[n]}_{n}$. 
\end{theorem}
\begin{proof}

 This is proved in the paper of Yao \cite[Thm. 5.5]{Yao24}. In his paper, he only considers the situation where $\ep^\flat=1$, but the other case $\ep^\flat=-1$ follows from \eqref{equ:odd-RZ-iso}.
\end{proof}

In the splitting model, we have the following:
\begin{theorem}\label{thm: Z spt unit}
Let $u\in \BV(\BX^{[t]}_{n,\ep})$ be any non-zero vector, with corresponding special cycles $\CZ(u)_{n,\ep}^{[t], \spt}$ and  $\CY(u)_{n,\ep}^{[t], \spt}$ on $ \CN_{n, \ep}^{[t],\spt}$. 
\begin{altenumerate}
\item The special cycles $\CZ(u)^{[t],\spt}_{n,\ep}$ and $\CY(u)^{[t],\spt}_{n,\ep}$ are Cartier divisors.

\item Assume $u$ is a unit length vector, and set $\ep(u):=\eta(h(u,u))$ and $\ep^\flat=\ep\ep(u)\eta((-1)^{n-1})$. Then
\begin{altitemize}
\item When $n$ is even and $t=n$, $\CZ(u)^{[t], \spt}_{n,\ep}$ is empty.
\item In the remaining cases, we have  isomorphisms
\begin{equation*}
	\CZ(u)_{n,\ep}^{[t],\spt}\simeq \CN_{n-1, \varepsilon^\flat}^{[t],\spt},
\end{equation*}
except when
\item When $n$ is odd, $t=n-1$, and $\ep^\flat=-1$, then the RHS is not defined  and  $\CZ(u)^{[t], \spt}_{n,\ep}$ is  empty.
\end{altitemize}

\end{altenumerate}
\end{theorem}
\begin{proof}
\begin{altenumerate}
\item By \cite[Prop. 1.5.1]{HLS1}, the pull-backs of special cycles $\CZ(u)$ and $\CY(u)$ into the splitting RZ spaces are Cartier divisors. Therefore $\CZ(u)^{[t],\spt}_{n,\ep}$ and  $\CY(u)^{[t],\spt}_{n,\ep}$, as codimension $1$ closed subschemes in the regular formal scheme $\CN_{n, \ep}^{[t],\spt}$, are also Cartier divisors.
\item This follows from Theorem \ref{thm: Z unit}, and the definition of splitting $\CZ$-cycles.\qedhere
\end{altenumerate}
\end{proof}

 The divisors appearing in Theorem \ref{thm: Z spt unit}  are exceptional in that they are isomorphic to a splitting RZ space of lower dimension for a maximal parahoric level. The following conjecture would tell us that there are no further exceptional special divisors. 

\begin{conjecture}\label{conj:reg-exceptional}
Let $u\in \BV(\BX^{[t]}_{n,\ep})$ be any non-zero vector, with corresponding special cycle $\CZ(u)_{n,\ep}^{[t]}$. Assume that $\CZ(u)_{n,\ep}^{[t],\spt}$ is a non-empty regular scheme (hence $t\neq n$). Then $u$ is a unit-length vector and hence   $\CZ(u)_{n,\ep}^{[t], \spt}\subset \CN_{n, \ep}^{[t],\spt}$ is an exceptional special divisor.

\end{conjecture}

 \begin{remark}\label{rmk:char-regular-Y}
 To classify all cases when $\CY(u)_{n,\ep}^{[t], \spt}$ is non-empty regular seems more complicated.  Assume that $h(u,u)\in O_{F_0}$ and that $t\neq 0$.  Then $\CZ(u)_{n,\ep}^{[t], \spt}\subsetneq \CY(u)_{n,\ep}^{[t], \spt}$. This leads us to suspect  that if  $\CY(u)_{n,\ep}^{[t], \spt}$ is non-empty regular, then $t=n$ if $n$ is even and $t=n-1$ if $n$ is odd, and that $u$ is a unit-length vector. When $n$ is even, we conjecture that, if $\CY(u)_{n,\ep}^{[n],\spt}$ is a non-empty regular scheme, then $u$ is a unit-length vector and hence   $\CY(u)_{n,\ep}^{[n], \spt}\subset \CN_{n, \ep}^{[n],\spt}$ is an exceptional special divisor. When $n$ is odd, we conjecture that the divisor $\CY(u)_{n,\ep}^{[n-1], \spt}\subset \CN_{n, \ep}^{[n-1],\spt}$ is regular  if and only if $u$ has unit-length and $\ep^\flat=-1$. However, in the latter case we cannot relate this to an RZ space of dimension $n-1$. 
 \end{remark}

 The  special cycles $\CZ(u)_{n,\ep}^{[t]}\subset \CN_{n, \ep}^{[t]}$ and  $\CY(u)_{n,\ep}^{[t]}\subset \CN_{n, \ep}^{[t]}$ are not Cartier divisors, comp. \cite[Rem. 2.5.1]{BHKRY}. The following conjecture seems the best-possible replacement.

\begin{conjecture}
Let $u\in \BV(\BX^{[t]}_{n,\ep})$ be any non-zero vector. There exists a unique Cartier divisor $\fkZ(u)_{n,\ep}^{[t]}$ resp. $\fkY(u)_{n,\ep}^{[t]}$ on $\CN_{n, \ep}^{[t]}$ such that its restriction to $\CN_{n, \ep}^{[t]}\setminus \mathrm{Sing}$ equals $2\CZ(u)_{n,\ep}^{[t]}$ resp. $2\CY(u)_{n,\ep}^{[t]}$.
\end{conjecture}
The analogous conjecture for Shimura varieties holds for $t=0$, cf. \cite[Thm. 2.5.3]{BHKRY}.

\subsection{Embeddings of RZ spaces}\label{sec:emb-RZ}

The arithmetic transfer conjecture concerns the embedding of a hermitian lattice of rank $n$ into a hermitian lattice of dimension $n+1$. We make the following definition to keep track of these data.

Let $\BX=(\BX,\iota_\BX,\lambda_{\BX})$ be a framing object of dimension $n+1$. Recall from \S \ref{sec:framing} that $\BV(\BX)=\Hom^\circ_{O_F}(\ov{\BE},\BX)$ is the space of special quasi-homomorphisms. It is a non-degenerate hermitian space of dimension $n$. Also recall that  $\ep(\BX)$ is the negative of the Hasse invariant of $\BV(\BX)$, cf. Thm. \ref{thm:framing-obj}.

\begin{definition}\label{defn:aligned-module}
An \emph{aligned triple} of dimension $n$ and type $t$ is a triple $(\BY^{[t]},\BX^{[t]},u)=(\BY,\BX,u)$  consisting of   a type $t$ framing object $\BY$ of dimension $n$,  a type $t$ framing object $\BX$ of dimension $n+1$, and  a unit-length vector $u\in\BV(\BX)$ such that there is an isomorphism
\begin{equation*}
(\BY\times \ov{\BE},\iota_{\BY}\times\iota_{\BE},\lambda_{\BY}\times\zeta\cdot\lambda_{\ov{\BE}})\simeq (\BX,\iota_{\BX},\lambda_{\BX}),
\end{equation*}
identifying the inclusion map of the second factor on the LHS with $u$. Here $\zeta=h(u,u)\in O^\times_{F_0}$. 
\end{definition}
In particular, we have an isomorphism between hermitian spaces
\begin{equation*}
	\BV(\BY)\obot \langle u\rangle\simeq \BV(\BX).
\end{equation*}
Here $\langle u\rangle$ is the one-dimensional hermitian space spanned by $u$.
  Let $\ep(u):=\eta(\zeta)=\eta(h(u,u))$, then we have the relation,
\begin{equation*}\label{eprel}
	\ep(\BX)=\eta((-1)^{n})\ep(\BY) \ep(u).
\end{equation*}
Starting from an aligned triple of dimension $n+1$ and type $t$, we obtain a closed embedding of formal schemes,
\begin{equation}\label{equ:embd-RZ}
\begin{aligned}
\xymatrix{
\CN_{n}^{[t]}\ar@{^(->}[r]& \CN_{n+1}^{[t]} ,
}
\end{aligned} 
\end{equation} 
cf. Theorem \ref{thm:ss-comparison}. Note that here $\CN_{n}^{[t]}=\CN_{n, \ep^\flat}^{[t]}$ and $\CN_{n+1}^{[t]}=\CN_{n+1, \ep}^{[t]}$, where $\ep=\ep(\BX)$ and $\ep^\flat=\ep(\BY)$. 

For $t\leq r$, using the isogeny $\BY^{[r]}\to \BY^{[t]}$ from \S \ref{sec:RZ-deeper}, this isogeny extends in the obvious way to an isogeny $\BX^{[r]}\to \BX^{[t]}$. We then obtain  commutative diagrams in which the horizontal arrows are closed embeddings,
\begin{equation*}\begin{aligned}
\xymatrix{
\CN_{n}^{[r,t]}\ar@{^(->}[r]\ar@{}[rd]\ar[d]	&	\CN_{n+1}^{[r,t]}\ar[d]\\
\CN_{n}^{[r]}\ar@{^(->}[r]&\CN_{n+1}^{[r]},
}
  \end{aligned}\hspace{2cm}
  \begin{aligned}
\xymatrix{
\CN_{n}^{[r,t]}\ar@{^(->}[r]\ar@{}[rd]\ar[d]	&	\CN_{n+1}^{[r,t]}\ar[d]\\
\CN_{n}^{[t]}\ar@{^(->}[r]&\CN_{n+1}^{[t]}.
}
  \end{aligned}
\end{equation*}

\begin{proposition}\label{prop:deep-exc}
Let $0\leq t\leq r\leq n$. Fix an aligned triple $(\BY^{[t]}, \BX^{[t]}, u)$ of dimension $n+1$ and type $t$ and consider the isogeny $\BY^{[r]}\to \BY^{[t]}$ and its canonical extension $\BX^{[r]}\to \BX^{[t]}$. Then the corresponding commutative diagram is cartesian, 
\begin{equation*}
\begin{aligned}
\xymatrix{
\CN_{n}^{[r,t]}\ar@{^(->}[r]\ar@{}[rd]|{\square}\ar[d]	&	\CN_{n+1}^{[r,t]}\ar[d]\\
\CN_{n}^{[r]}\ar@{^(->}[r]&\CN_{n+1}^{[r]},
}
  \end{aligned}
\end{equation*}
\end{proposition}
In practice, we will only use the cases $r=n$ when $n$ is even and $r=n-1$ when $n$ is odd.
\begin{proof}
	Using Theorem \ref{thm: Z unit}, we identify $\CN_{n}^{[r]}\subset\CN_{n+1}^{[r]}$ with the special divisor $\CZ(u)_{n+1}^{[r]}\subset \CN_{n+1}^{[r]}$.
	Then, the pull-back along the projection $\CN_{n+1}^{[r,t]}\to \CN_{n+1}^{[r]}$ parametrizes all  pairs $(X^{[r]},X^{[t]})$ with a lifting $X^{[r]}\to X^{[t]}$ of $\BX^{[r]}\to \BX^{[t]}$ such that there exists a lifting $u:\CE\to X^{[r]}$ of $\ov\BE\to \BX^{[r]}$.
	By composition with $X^{[r]}\to X^{[t]}$, this lift also exists for $X^{[t]}$, hence using the standard splitting procedure we get splittings $X^{[r]}\simeq Y^{[r]}\times \CE$ and  $X^{[t]}\simeq Y^{[t]}\times \CE$ compatible with the isogeny $X^{[r]}\to X^{[t]}$. We obtain  the desired object  $Y^{[r]}\to Y^{[t]}$  of  $\CN_{n}^{[r,t]}$.
\end{proof}

\begin{remark}
\begin{altenumerate}
\item The statement fails if we replace the bottom arrow by the other injection $\CN_{n}^{[t]}\hookrightarrow\CN_{n+1}^{[t]}$. This is best illustrated by the lattice model, see \S \ref{sec: t n lattice} below for the notation used. In the lattice model we have  that $\BN_{n}^{[t]}\times_{\BN_{n+1}^{[t]}}\BN_{n+1}^{[r, t]}$ is given as
\begin{equation*}
\{(\Lambda^\flat, \Lambda, \Lambda_0)\in {\rm Vert}^{[t]}(W^\flat)\times {\rm Vert}^{[t]}(W)\times {\rm Vert}^{[r]}(W)\mid \Lambda_0\subset \Lambda=\Lambda^\flat\oplus\langle u\rangle\}
\end{equation*}
In general, $u$ does not lie in $\Lambda_0$ and hence $\Lambda_0$ is not of the form $\Lambda_0=\Lambda_0^\flat\oplus\langle u\rangle$, where $\Lambda_0^\flat$ is a vertex lattice of type $r$ with $\Lambda_0^\flat\subset \Lambda^\flat$. 

\item Note that when $n$ is even, the existence of the isogeny $\BY^{[r]}\to\BY^{[t]}$ imposes for $r=n$ the condition that $\ep^\flat=1$. In particular, if $\ep^\flat=-1$, there is no obvious candidate for the fiber product $\CN^{[t]}_{n, \ep^\flat}\times_{\CN^{[t]}_{n+1, \ep}}\CN_{n+1, \ep}^{[n, t]}$. In \S \ref{sec:(t,n)-naive}, we denote this space by $\wt{\CN}_n^{[t]}$.
\end{altenumerate}
\end{remark}

 \section{Comparison of  strengthened spin conditions}\label{sec:exc-iso-proof}
 In this section we  prove Theorem \ref{thm:ss-comparison}. Part (i) is straightforward. 

\subsection{From smaller space to larger space}
 
In this subsection we will  prove part (ii) of Theorem \ref{thm:ss-comparison}. Let $(V,\phi)$ be a hermitian space of dimension $n$. Recall the notation in \S \ref{sec:spin}; in particular, we have $\CV=V\otimes_{F_0}F$ and $\prescript{n}{}{\CV}^{r,s}_\epsilon= \prescript{n}{}{\CV}_\ep\cap  \prescript{n}{}{\CV}^{r,s}\subset \prescript{n}{}{\CV}=\bigwedge_F^{n}\CV$.
 \begin{lemma}\label{lem:spin-rep}
 	The subspace $\prescript{n}{}{\CV}_\epsilon\subset \prescript{n}{}{\CV}$ is spanned by the pure tensors.
 \end{lemma}
 \begin{proof}
 	 Denote by $\prescript{n}{}{\CV}^{\mathrm{pure}}_\ep\subset \prescript{n}{}{\CV}_\ep$ the subspace spanned by pure tensors.
 	In other words, this is the subspace spanned by $\bigwedge^{n}\CF$, where $\CF\subset \CV$ is an $n$-dimensional subspace such that $\bigwedge^{n}\CF\in \prescript{n}{}{\CV}_\epsilon$.
 	The subspace $\prescript{n}{}{\CV}^{\mathrm{pure}}_\ep$ is $\SO((\cdot,\cdot,))$-stable, and thus forms a sub-representation.
 	The assertion follows from the irreducibility of $\prescript{n}{}{\CV}_\epsilon$.
 \end{proof}
 
 Let $(V^\flat,\phi^\flat)$ be a hermitian space of dimension $n-1$ and let $\zeta\in O_{F_0}^\times$ be a unit.
 We choose the hermitian space $(V,\phi)$ with $V=V^\flat\obot F u$ and $\phi(u,u)=\zeta$. 
 Let $\CV=V\otimes_{F_0}F$, $\CV^\flat=V^\flat\otimes_{F_0}F$ and $\CV^\circ=Fu\otimes_{F_0}F$. 
 The symmetric forms on $\CV$, $\CV^\flat$ and $\CV^\circ$ split respectively. We have the following decomposition:
 \begin{equation*}
 	\prescript{n-1}{}{\CV^\flat}=\prescript{n-1}{}{\CV^\flat_{1}}\oplus \prescript{n-1}{}{\CV^\flat_{-1}},
 	\quad\text{and}\quad
 	\prescript{n}{}{\CV}=\prescript{n}{}{\CV_1}\oplus \prescript{n}{}{\CV_{-1}}.
 \end{equation*}
 
 Since $\CV=\CV^\flat\oplus \CV^\circ$, we have
 \begin{equation*}
 	\bigwedge^{n} \CV=\bigwedge^{n} \CV^\flat\oplus 
 	\Big(\bigwedge^{n-1}\CV^\flat\otimes \CV^\circ\Big)
 	\oplus 
 	\Big(\bigwedge^{n}\CV^\flat\oplus \bigwedge^2 \CV^\circ\Big).
 \end{equation*}

 Let $\Pi:=\pi\otimes 1$ and $\pi:=1\otimes \pi$ in $F\otimes_{F_0}F$. 
 For any $v\in V$, we have
 \begin{equation*}
 	\phi\Bigl( (\Pi-\pi)v,(\Pi-\pi)v\Bigr)=\phi \Bigl( (-\Pi-\pi)(\Pi-\pi)v,v \Bigr)=0 ,
 \end{equation*}
hence  $(\Pi-\pi)v$ is an isotropic vector.

 \begin{proposition}\label{prop:ori-red}
 	We have the following equality:
 	\begin{equation*}
 		\prescript{n}{}{\CV_{-1}}\bigcap \left(\prescript{n-1}{}{\CV}^\flat\otimes F(\Pi+\pi)u\right)
 		=\prescript{n-1}{}{\CV^\flat_{-1}}\otimes F(\Pi+\pi)u.
 	\end{equation*}
 \end{proposition}
 \begin{proof}
 	We have $\CV^\circ=\mathrm{Span}_F\Bigl((\Pi-\pi)u,(\Pi+\pi)u\Bigr)$.
 	We further decompose
\begin{multline}\label{equ:ss-cond-decomp}
\prescript{n-1}{}{\CV}^\flat\otimes \CV^\circ=\left(\prescript{n-1}{}{\CV^\flat_{1}}\otimes F(\Pi-\pi)u\right)\oplus \left(\prescript{n-1}{}{\CV^\flat_{-1}}\otimes F(\Pi-\pi)u\right)\\
\oplus \left(\prescript{n-1}{}{\CV^\flat_{1}}\otimes F(\Pi+\pi)u\right)\oplus \left(\prescript{n-1}{}{\CV^\flat_{-1}}\otimes F(\Pi+\pi)u\right).
\end{multline}
 	We claim that
 	\begin{equation*}
 		\prescript{n-1}{}{\CV^\flat_{1}}\otimes F\Bigl( (\Pi-\pi)u \Bigr)\subset \prescript{n}{}{\CV_{-1}},
 		\quad\text{and}\quad
 		\prescript{n-1}{}{\CV^\flat_{-1}}\otimes F\Bigl( (\Pi-\pi)u \Bigr)\subset \prescript{n}{}{\CV_{1}}.
 	\end{equation*}
 	We prove the first inclusion.
 	Let $\CF^\flat\subset \CV^\flat$ be any subspace of dimension $n-1$ such that $\bigwedge^{n-1}\CF^\flat\subset \prescript{n}{}{\CV}^\flat$.
 	Let $\CF=\CF^\flat\oplus F(\Pi-\pi)u$. 
 	By Lemma \ref{lem:spin-rep}, we only need to show that 
 	\begin{equation*}
 		\bigwedge^{n} \CF=\bigwedge^{n-1}\CF^\flat\otimes F(\Pi-\pi)u\subset \prescript{n}{}{\CV_{-1}}.
 	\end{equation*}
 	We define a morphism
 	\begin{equation*}
 		\iota:\BP(\prescript{n-1}{}{\CV_{-1}^{\flat}})\to \BP(\prescript{n}{}{\CV_{1}})\amalg \BP(\prescript{n}{}{\CV_{1}}).
 	\end{equation*}
 	By Lemma \ref{lem:spin-rep}, we only need to define $\iota$ for all $\bigwedge^{n-1}\CF^\flat\in \BP(\prescript{n-1}{}{\CV_{-1}^{\flat}})$, where $\CF^\flat$ is a total isotropic subspace. For such $\CF^\flat$, we define
 	\begin{equation*}
 		\iota\Bigl(\bigwedge^{n-1}\CF^\flat\Bigr):=\iota\Bigl(\bigwedge^{n}\CF\Bigr),\quad\text{where}\quad \CF:=\CF^\flat\oplus F(\Pi-\pi)u.
 	\end{equation*}
 	Since $\CF$ is totally isotropic, the map is well-defined. 
We claim that $\iota(\BP(\prescript{n-1}{}{\CV_{-1}^{\flat}}))\subset \BP(\prescript{n}{}{\CV_{-1}})$.
Since $\BP(\prescript{n-1}{}{\CV_{-1}^{\flat}})$ is connected, we must have either 
 	\begin{equation*}
 		\iota(\BP(\prescript{n-1}{}{\CV_{-1}^{\flat}}))\subset \BP(\prescript{n}{}{\CV_{1}}) \quad\text{or}\quad \iota(\BP(\prescript{n-1}{}{\CV_{-1}^{\flat}}))\subset \BP(\prescript{n}{}{\CV_{-1}}).
 	\end{equation*}
	Thus, it suffices to find a single geometric point of $\BP(\prescript{n-1}{}{\CV_{-1}^\flat})$ that maps to $\BP(\prescript{n}{}{\CV_{-1}})$. Such a point is constructed  in \cite[Lem. 3.1.1]{Luo24}.
 	
As a consequence of the claim, we see that the intersection of the subspace
\begin{equation*}
\prescript{n}{}{\CV_{-1}}\bigcap \left(\prescript{n-1}{}{\CV}^\flat\otimes F(\Pi+\pi)u\right)\subseteq \prescript{n}{}{\CV}
\end{equation*}
with the direct factor in \eqref{equ:ss-cond-decomp}: 
\begin{equation*}
\prescript{n-1}{}{\CV^\flat_{1}}\otimes F(\Pi-\pi)u\oplus\prescript{n-1}{}{\CV^\flat_{-1}}\otimes F(\Pi-\pi)u\oplus\prescript{n-1}{}{\CV^\flat_{1}}\otimes F(\Pi+\pi)u\subset \prescript{n}{}{\CV}
\end{equation*}
is zero.
Moreover, since we have
 	\begin{equation*}
 		\prescript{n}{}{\CV_{-1}}\bigcap \left(\prescript{n-1}{}{\CV}^\flat\otimes F(\Pi+\pi)u\right)\supseteq \prescript{n-1}{}{\CV^\flat_{-1}}\otimes F(\Pi+\pi)u,
 	\end{equation*}
 the equality in the assertion follows by dimension counting.
 \end{proof}

 \begin{lemma}\label{lem:sig-red}
 	For $\CV=\CV^\flat\oplus \CV^\circ$, we have
 	\begin{equation*}
 		\prescript{n}{}{\CV}^{r,s}\bigcap \left(\prescript{n-1}{}{\CV}^\flat\otimes F(\Pi+\pi)u\right)=\prescript{n-1}{}{\CV}^{\flat,r-1,s}\otimes F(\Pi+\pi)u.
 	\end{equation*}
 \end{lemma}
 \begin{proof}
 	We have
 	\begin{equation*}
 		\CV_{\pi}=\CV^\flat_{\pi}\oplus F(\Pi+\pi)u,
 		\quad\text{and}\quad
 		\CV_{-\pi}=\CV^\flat_{-\pi}\oplus F(\Pi-\pi)u.
 	\end{equation*}
 	This induces the decomposition of the eigenspaces,
\begin{multline*}
\prescript{n}{}{\CV}^{r,s}=\prescript{n-1}{}{\CV}^{\flat,r,s}\oplus\left(\prescript{n-1}{}{\CV}^{\flat,r-1,s}\otimes F(\Pi+\pi)u\right)\oplus \left(\prescript{n-1}{}{\CV}^{\flat,r,s-1}\otimes F(\Pi-\pi)u\right)\\ 
\oplus \left(\prescript{n-1}{}{\CV}^{\flat,r-1,s-1}\otimes F(\Pi-\pi)u\otimes F(\Pi+\pi)u\right).
\end{multline*}
 	Now the assertion follows.
 \end{proof}

 Let $\Lambda^\flat\subset V^\flat$ be a vertex lattice of type $t$ and let $\Lambda=\Lambda^\flat\oplus \langle u\rangle$. It is also a vertex lattice of type $t$. We have:
 \begin{equation*}
 	\bigwedge^{n}(\Lambda\otimes_{O_{F_0}}O_F)=\Bigl(\bigwedge^{n-1}\Lambda^\flat\Bigr)\otimes_{O_{F}}(O_F u\otimes_{O_{F_0}}O_F)\subset \bigwedge^{n-1}V^\flat\otimes F u.
 \end{equation*}
 
 \begin{proposition}\label{prop:ss-red}
 	The following equality of lattices in $L_{\Lambda,-1}^{n-1, 1}$ holds: 
 	\begin{equation*}
 		\left(\prescript{n-1}{}{\Lambda^\flat}\otimes O_F(\Pi+\pi)u\right)\cap \prescript{n}{}{\Lambda}^{n-1,1}_{-1}=\prescript{n-1}{}{\Lambda}^{\flat,n-2,1}_{-1}\otimes O_F(\Pi+\pi)u.
 	\end{equation*}
 \end{proposition}
 \begin{proof}
 	It is clear that we have
 	\begin{equation*}
 		\left(\prescript{n-1}{}{\Lambda}^\flat\otimes O_F(\Pi+\pi)u\right)\cap \prescript{n}{}{\Lambda}^{n-1,1}_{-1}\supseteq \prescript{n-1}{}{\Lambda}^{\flat,n-2,1}_{-1}\otimes O_F(\Pi+\pi)u.
 	\end{equation*}
 	We will show the converse inclusion. Note that
 	\begin{equation*}
 		\left(\prescript{n-1}{}{\Lambda}^\flat\otimes O_F(\Pi+\pi)u\right)\cap \prescript{n}{}{\Lambda}^{n-1,1}_{-1}\subset \prescript{n-1}{}{\Lambda}^\flat\otimes O_F(\Pi+\pi)u.
 	\end{equation*}
 	By the definition of the lattice \eqref{eq:spin-lattice-sum}, it suffices to show the inclusion
 	\begin{equation*}
 		\left(\prescript{n-1}{}{\Lambda}^\flat\otimes O_F(\Pi-\pi)u\right)\cap \prescript{n}{}{\Lambda}^{n-1,1}_{-1}\subset (\prescript{n-1}{}{\CV}^{\flat,n-2,1}_{-1})\otimes O_F(\Pi+\pi)u=\prescript{n-1}{}{\CV}^{\flat,n-2,1}_{-1}\otimes F(\Pi+\pi)u.
 	\end{equation*}
 	By Proposition \ref{prop:ori-red} and \ref{lem:sig-red}, we have
 	\begin{equation*}
 		\left(\prescript{n-1}{}{\CV^\flat}\otimes F(\Pi+\pi)u\right)\cap \prescript{n}{}{\CV}^{n-1,1}\cap \prescript{n}{}{\CV}_{-1}\subset \left(\prescript{n-1}{}{\CV}^{\flat,n-2,1}\cap \prescript{n-1}{}{\CV}^\flat_{-1}\right)\otimes F(\Pi+\pi)u.
 	\end{equation*}
 	The assertion now follows.
 \end{proof}
 
 Recall that for an $O_F$-lattice $\Lambda$ and an $O_F$-algebra $R$, we set $\Lambda_R:=\Lambda\otimes_{O_{F_0}}R$.
 \begin{corollary}\label{cor:ss-red}
 	For any $O_F$-algebra $R$, let $\CF^\flat\subset \Lambda^\flat_R$ be a totally isotropic subspace.
 	Let $\CF=\CF^\flat\oplus R(\Pi+\pi)u\subset \Lambda_R$.
 	If $\bigwedge^{n-1} \CF^\flat\in L^{n-2,1}_{\Lambda^\flat,-1}(R)$, then $\bigwedge^{n} \CF\in L^{n-1,1}_{\Lambda,-1}(R)$. In other words, if $\CF^\flat$ satisfies the strengthened spin condition, then so does $\CF$. 
 \end{corollary}
 \begin{proof}
 	By Proposition \ref{prop:ss-red}, we have an inclusion
 	\begin{equation}\label{equ:ss-comparison-inc}
 		\left(\prescript{n-1}{}{\Lambda}^{\flat,n-2,1}_{-1}\otimes R\right)\otimes R(\Pi+\pi)u\subset  
 		\left(
 		\left(\prescript{n-1}{}{\Lambda^\flat}\otimes R\right)\otimes R(\Pi+\pi)u\right)\cap \left(\prescript{n}{}{\Lambda}^{n-1,1}_{-1}\right)\otimes R.
 	\end{equation}
 	Since $\bigwedge^{n} \CF=\bigwedge^{n-1}\CF^\flat\otimes R(\Pi+\pi)u$, the assertion follows.
	\end{proof}

This corollary implies Theorem \ref{thm:ss-comparison}, (ii). Indeed, recall that, by definition, a hermitian $O_F$-module  $(Y,\iota_Y,\lambda_Y)$ of dimension $n-1$ and type $t$ satisfies the strengthened spin condition if after some \'etale extension and the choice of a trivialization
\begin{equation*}
			\Bigl[\cdots\xrightarrow{\lambda_*^\vee}D(Y)\xrightarrow{\lambda_*}D(Y^\vee)\xrightarrow{\lambda_*^\vee}\cdots\Bigr]\xrightarrow{\sim}\Lambda^\flat_{[t],\CO_S},
		\end{equation*}
the induced  filtration $\Fil(Y)$ satisfies the strengthened spin condition. In the case of Theorem \ref{thm:ss-comparison}(ii), we choose a trivialization for the de Rham homology of as above, then we have a trivialization
\begin{equation*}
\Bigl[\cdots\xrightarrow{\lambda_*^\vee}D(X)\xrightarrow{\lambda_*}D(X^\vee)\xrightarrow{\lambda_*^\vee}\cdots\Bigr]\xrightarrow{\sim}
\Bigl[\cdots \left(\Lambda^\flat_{t}\oplus\langle u\rangle\right)_{\CO_S}\to\left(\Lambda^{\flat,\vee}_{t}\oplus\langle u\rangle\right)_{\CO_S}\to\cdots\Bigr].
\end{equation*}
We can identify $\Fil(Y)$ and $\Fil(X)$ with $\CF^\flat$ and $\CF$ above.  Hence Theorem \ref{thm:ss-comparison}(ii) follows from Corollary \ref{cor:ss-red}.

\subsection{From larger space to smaller space}

 The inclusion \eqref{equ:ss-comparison-inc} can be strict in general. To study the converse, we need a more careful study of the strengthened spin condition. For this, it is convenient to introduce the local model and some auxiliary conditions.
 
  \begin{definition}\label{defn:lm}
 Let $\Lambda_t$ be a vertex lattice of type $t$. Denote by $\lambda_{t}:\Lambda_{t}\to\Lambda_{t}^\vee$ and $\lambda_{t}^\vee:\Lambda_{t}^\vee\to \pi^{-1}\Lambda_{t}$ the natural inclusions.
\begin{altenumerate}
\item The \emph{wedge local model} $\bM^{[t],\wedge}_n$  is a projective scheme over $\Spec  O_{F}$. It represents the moduli problem that sends each $ O_{F}$-algebra $R$ to the set of filtrations:
\begin{equation}\label{equ:lm-filtrations}
\begin{aligned}
\xymatrix{
\Lambda_{t,R}\ar[r]^-{\lambda_{t}}&\Lambda^\vee_{t,R}\ar[r]^-{\lambda^\vee_{t}}&\pi^{-1}\Lambda_{t,R}
\\
\CF_{\Lambda_t}\ar@{^(->}[u]\ar[r]&\CF_{\Lambda_t^\vee}\ar@{^(->}[u]\ar[r]&\ar@{^(->}[u]\CF_{\pi^{-1}\Lambda_t}
}
\end{aligned}
\end{equation}
such that:
\begin{altenumerate2}		
 \item For each lattice $\Lambda\in \{\Lambda_t, \Lambda_t^\vee, \pi^{-1}\Lambda_t\}$, the filtration  $\CF_\Lambda$ is an $ O_F \otimes_{ O_{F_0}} R$-submodule of $\Lambda_{R}$, and an $R$-direct summand of rank $n$;
 \item The natural arrow $\lambda_t:\Lambda_{t,R} \to \Lambda_{t,R}^\vee$ carries $\CF_{\Lambda_t}$
 into $\CF_{\Lambda_t^\vee}$, and the natural arrow $\lambda_t^\vee:\Lambda^\vee_{t,R} \to \pi^{-1}\Lambda_{t,R}$ carries $\CF_{\Lambda^\vee_t}$ into $\CF_{\pi^{-1}\Lambda_t}$.
 The isomorphism $\pi^{-1}\Lambda_{t,R}\overset{\pi}{\to}\Lambda_{t,R}$ identifies $\CF_{\pi^{-1}\Lambda_t}$ with $\CF_{\Lambda_t}$;
 \item The perfect $R$-bilinear pairing
 \begin{equation*}
 	\Lambda_{t,R} \times \Lambda^\vee_{t,R}
 	\xrightarrow{\langle-,-\rangle \otimes R} R
 \end{equation*}
 	identifies $\CF_{\Lambda^\vee_t}^\perp$ with $\CF_{\Lambda_t}$ inside $\Lambda_{t,R}$; and
 	\item For each lattice $\Lambda\in \{\Lambda_t, \Lambda_t^\vee, \pi^{-1}\Lambda_t\}$, the element $\pi \otimes 1 \in  O_F\otimes_{ O_{F_0}} R$ acting on $\CF_\Lambda$ satisfies the following signature conditions:
 	\begin{altitemize}
 	\item(Kottwitz condition) There is an equality of polynomials 
 	\begin{equation*}
 	\mathrm{char}(\Pi\mid\CF_\Lambda)=(T-\pi)(T+\pi)^{n-1};
 	\end{equation*}
 	\item(Wedge condition) We have
 	\begin{equation*}
 	\bigwedge^2(\Pi-\pi \mid \CF_\Lambda)=0;
 	\quad\text{and}\quad
 	\bigwedge^n(\Pi+\pi \mid \CF_\Lambda)=0.
 	\end{equation*}
 	\item(Spin condition) When $n$ is even and $t=n$, we further require that the operator $\Pi-\pi$ is nowhere vanishing in $\CF_\Lambda$.
 	\end{altitemize}
 			
 \end{altenumerate2}
 \item The \emph{canonical local model} $\bM_n^{[t]}$ is a flat projective scheme over $\Spec  O_{F}$, cf. \cite{Luo24}. It represents the moduli problem that sends each $ O_{F}$-algebra $R$ to the set of filtrations in \eqref{equ:lm-filtrations} satisfying the axioms (a)(b)(c) in (i), and the requirement that for all lattices $\Lambda\in \{\Lambda_t, \Lambda_t^\vee, \pi^{-1}\Lambda_t\}$, the element $\Pi \in  O_F\otimes_{ O_{F_0}} R$ acting on $\CF_\Lambda$ satisfies the strengthened spin condition.
 \item For $t\neq n$, we define the \emph{worst point} as the filtration $(\CF_\Lambda:=\Pi\Lambda\subset \Lambda_R)$. This defines a point $*\in \bM_n^{[t]}\in \bM_n^{[t]}(k)$ by \cite[Lem. 3.1.1]{Luo24}.
 	\end{altenumerate}
 \end{definition}

 The Kottwitz condition and the wedge condition are easier to handle than the strengthened spin condition, and it is not hard to check the following:
 
 \begin{lemma}\label{lem:KW-induction}
 	Let $\Lambda=\Lambda^\flat\oplus \langle u\rangle$ and let $\CF=\CF^\flat\oplus R(\Pi-\pi)u
 	\subset \Lambda_R$ be a filtration. Then $\CF^\flat$ satisfies the Kottwitz condition (resp. the wedge condition) if and only if $\CF$ satisfies the Kottwitz condition (resp. the wedge condition).\qed
 \end{lemma}

Next, we study the strengthened spin condition over the special fiber. The key input is the following computation:
\begin{theorem}\label{thm:ss-lattice}{\cite[Cor. 3.4.7]{Luo24}}
Let $V=F^n=\mathrm{Span}_F(\fke_1,\cdots,\fke_n)$ be the split hermitian space with basis such that $h(\fke_i,\fke_j)=\delta_{i,n+1-j}$. For any integer $\kappa$ such that $0\leq \kappa\leq n$, we let $\bLambda_\kappa$ be the standard integral lattice:
\begin{equation*}
\bLambda_\kappa:=\mathrm{Span}_{O_F}\left(\pi^{-1}\fke_1,\cdots,\pi^{-1}\fke_{\kappa},\fke_{\kappa+1},\cdots \fke_n\right)\subset V.
\end{equation*}
Consider $\bLambda_\kappa\otimes_{O_{F_0}}O_F\subset \CV:=V\otimes_{F_0}F$, it is spanned by the following basis:
\begin{equation*}
\pi^{-1}\fke_1\otimes 1,\cdots,\pi^{-1}\fke_{\kappa}\otimes 1,\fke_{\kappa+1}\otimes 1,\cdots ,\fke_n\otimes 1;\quad 
 \fke_1\otimes 1,\cdots,\fke_{\kappa}\otimes 1,\pi \fke_{\kappa+1}\otimes 1,\cdots,\pi \fke_n\otimes 1.
\end{equation*}
We denote them by order as $e_1,\cdots,e_{2n}$, hence $\bLambda_{\kappa,O_F}=\mathrm{Span}_{O_F}(e_1,\cdots,e_{2n})$.
Denote by $e_{[i,\wh{n+j}]}$ the vector
 	\begin{equation*}
 		e_{[i,\wh{n+j}]}:=e_{\{i,n+1,\cdots,\wh{n+j},\cdots,2n\}}:=e_i\wedge e_{n+1}\wedge\cdots\wedge\wh{e_{n+j}}\wedge\cdots \wedge e_{2n}.
 	\end{equation*}
For any $k$-algebra $R$, the standard lattice $L^{n-1,1}_{\bLambda_{\kappa},-1}(R)$ is generated by the following elements, where we set $i^\vee:=n+1-i$:
 	
 	\begin{altenumerate}
 		\item $e_{\{n+1,\cdots,2n\}}$;
 		\item $e_{[i,\wh{n+i^\vee}]}$ for $i\neq i^\vee$;
 		\item $e_{[i,\wh{n+j}]}-(-1)^{n+i+j}e_{[j^\vee,\wh{n+i^\vee}]}$ for $i<j^\vee\leq \kappa,i\neq j$;
 		\item $e_{[j^\vee,\wh{n+i^\vee}]}$ for $i\leq \kappa<j^\vee<n-\kappa+1$;
 		\item $e_{[i,\wh{n+j}]}+(-1)^{n+i+j}e_{[j^\vee,\wh{n+i^\vee}]}$ for $i\leq \kappa,j^\vee\leq n-\kappa+1,i\neq j$;
 		\item $e_{[i,\wh{n+j}]}-(-1)^{n+i+j}e_{[j^\vee,\wh{n+i^\vee}]}$ for $\kappa<i<j^\vee<n-\kappa+1,i\neq j$;
 		\item  $e_{[i,\wh{n+j}]}$ for $\kappa<i<n-k+1\leq j^\vee,i\neq j$;
 		\item $e_{[i,\wh{n+j}]}-(-1)^{n+i+j}e_{[j^\vee,\wh{n+i^\vee}]}$ for $n-\kappa+1\leq i<j^\vee,i\neq j$;
 		\item  $e_{[i,\wh{n+i}]}+(-1)^ne_{[i^\vee,\wh{n+i^\vee}]}$ for $i\leq \kappa$;
 		\item  $e_{[i,\wh{n+i}]}+(-1)^ne_{[i^\vee,\wh{n+i^\vee}]}$ for $\kappa< i\leq M$;
 		\item Let $w=\sum_{i=1}^M c_i e_{[i,\wh{n+i}]}\in W(\bLambda_\kappa)\otimes R$. Then $w$ lies in the image if and only if 
 		\begin{altenumerate2}
 			\item When $n=2m$, we have $\sum_{i=\kappa}^m (-1)^ic_i=0$;
 			\item When $n=2m+1$, we have $\sum_{i=\kappa}^m (-1)^ic_i+\frac{1}{2}(-1)^{m+1}c_{m+1}=0$.\\
 			All the $w$ of the form (a) and (b) generate a free submodule, a basis of which can be completed to a basis of $L_{-1}^{n-1,1}(\bLambda_\kappa)(R)$ by the elements (i)-(x).
 		\end{altenumerate2}
 	\end{altenumerate}
\qed
 
\end{theorem}

 \begin{theorem}\label{thm:ss-comp-converse-special fiber}
 	Suppose $t=2\fkt \neq n-1$. Let $(V^\flat,\phi^\flat)$ be a hermitian space of dimension $n$ and let $\Lambda_t^\flat\subset V^\flat$ be a vertex lattice of type $t$.
 	For any $k$-algebra $R$, let
 \begin{equation*}
 	\CF_{\Lambda^\flat_t}\subset \Lambda^\flat_{t,R}
 	\quad\text{and}\quad
 	\CF_{\Lambda^{\flat,\vee}_t}\subset \Lambda^{\flat,\vee}_{t,R}
 \end{equation*}
be $R$-submodules satisfying axioms (a)-(c) in Definition \ref{defn:lm}(i).
Denote by $\Lambda_t$ the vertex lattice $\Lambda_t^\flat\oplus\langle u\rangle$, where $u$ is a unit length vector.
Define
\begin{equation*}
\CF_{\Lambda_t}=\CF_{\Lambda_t^\flat}\oplus R(\Pi+\pi)u\subset \Lambda_{t,R},
\quad\text{and}\quad 
\CF_{\Lambda^\vee_t}=\CF_{\Lambda_t^{\flat,\vee}}\oplus R(\Pi+\pi)u\subset \Lambda^\vee_{t,R}.
\end{equation*}
Then $\CF_{\Lambda_t}$ and $\CF_{\Lambda^\vee_t}$ satisfies the axioms (a)-(c) in Definition \ref{defn:lm}(i).
Moreover, if $\bigwedge^{n} \CF_{\Lambda_t}\subset L^{n-1,1}_{\Lambda_t,-1}(R)$, then $\bigwedge^{n-1} \CF_{\Lambda_t^\flat}\subset L^{n-2,1}_{\Lambda_t^\flat,-1}(R)$.
In other words, if $\CF_{\Lambda_t}$ satisfies the strengthened spin condition, then so does  $\CF_{\Lambda_t^\flat}$.
 \end{theorem}

 \begin{proof}
The verification of axioms (a)-(c) for $\CF_{\Lambda_t}$ and $\CF_{\Lambda^\vee_t}$ follows standard arguments, so we focus on proving the strengthened spin condition.
By \cite[Prop. 2.4.3]{Luo24}, it suffices to verify the assertion for $\CF_{\Lambda^\vee_t}$.
By \cite[Thm. 3.16]{RZ}, after passing to some \'etale cover, we may assume that $\Lambda_t=\bLambda_{-\fkt}$ and $\Lambda_t^\vee=\bLambda_{\fkt}$, where  $\bLambda_{-\fkt} $ and $\bLambda_{\fkt}$ are standard lattices defined in Theorem \ref{thm:ss-lattice} with the following basis (see \cite[\S 3.1.1]{Luo24}):
 	\begin{align*}
 		\begin{split}\label{equ_chart:standard-basis}
 			\bLambda_{-\fkt,O_{\breve{F}}}\colon&
 			\fke_1\otimes 1,\cdots, \fke_{n-\fkt}\otimes 1,\pi \fke_{n+1-\fkt}\otimes 1,\cdots , \pi \fke_{n}\otimes 1;\\
 			&\pi \fke_1\otimes 1,\cdots,\pi \fke_{n-\fkt}\otimes 1,\pi_0 \fke_{n+1-\fkt}\otimes 1,\cdots,\pi_0 \fke_n\otimes 1.\\
 			\bLambda_{\fkt,O_{\breve{F}}}\colon  &
 			\pi^{-1}\fke_1\otimes 1,\cdots,\pi^{-1}\fke_{\fkt}\otimes 1,\fke_{\fkt+1}\otimes 1,\cdots ,\fke_{n}\otimes 1;\\ 
 			& \fke_1\otimes 1,\cdots,\fke_{\fkt}\otimes 1,\pi \fke_{\fkt+1}\otimes 1,\cdots,\pi \fke_{n}\otimes 1.
 		\end{split}
 	\end{align*}
	
We denote the basis of $\bLambda_{\fkt}\otimes_{O_{F_0}}O_F$ by order as $e_1,\cdots, e_{2n}$, this is also the ordered basis we chose for $\bLambda_{\kappa,O_F}$ in Theorem \ref{thm:ss-lattice}.

After  scaling, we may assume that $(u,u)=1$. Recall that the loop group acts on the local model, and its action on the quotient $\Lambda_t/\pi\Lambda^\vee_{t}$ factors through the orthogonal group $O(\Lambda_t/\pi\Lambda^\vee_{t})$. By Witt's theorem, there exists $g\in O(\Lambda_t/\pi\Lambda^\vee_{t})$ such that 
 	\begin{equation*}
 		gu\equiv u_0 \mod \pi\Lambda^\vee_{t}, \quad \text{where}\quad
		u_0=
		 \left\{
		\begin{array}{ll}
		\fke_{m+1} & n=2m+1; \\
		\frac{1}{\sqrt{2}}(\fke_{m}+\fke_{m+1}) & n=2m.
		\end{array}\right.
 	\end{equation*}
	
Therefore, without loss of generality, we may assume that $u=u_0+\pi \aleph$ for some gadget term $\aleph\in \Lambda$. 
 	Furthermore, for a $k$-algebra $R$, since $\pi_0=0$, we have the equality $\CF_{\Lambda_t^{\flat,\vee}}\oplus R\Pi u=\CF_{\Lambda_t^{\flat,\vee}}\oplus R\Pi u_0$.
Consequently, we may further assume $u=u_0$.
We choose a basis $\fke_1^\flat,\cdots,\fke_{n}^\flat$ for $V^\flat$ as follows:
\begin{itemize}
\item When $n=2m+1$ is odd, we set 
\begin{equation*}
	\fke_1^\flat=\fke_1,\cdots, \fke_m^\flat=\fke_m;\quad
	\fke_{m+1}^\flat=\fke_{m+2},\cdots, \fke_{2m}^\flat=\fke_{2m+1}.
\end{equation*}
\item When $n=2m$ is even, we set
\begin{equation*}
	\fke_1^\flat=\fke_1,\cdots, \fke_{m-1}^\flat=\fke_{m-1};\quad 
	\fke^\flat_{m}=\frac{1}{\sqrt{2}}(\fke_{m}-\fke_{m+1});\quad 
	\fke^\flat_{m+1}=\fke_{m+2},\cdots,\fke^\flat_{2m-1}=\fke_{2m}.
\end{equation*}
\end{itemize}
Then $\Lambda_t^{\flat,\vee}$ is spanned by the basis
\begin{equation*}
\Lambda_t^{\flat,\vee}=\mathrm{Span}_{O_F}(\pi^{-1}\fke_1^\flat,\cdots,\pi^{-1}\fke^\flat_\fkt,\fke^\flat_{\fkt+1},\cdots,\fke^\flat_{n}).
\end{equation*}
We now distinguish cases, according to the parity of $n$.

\bigskip

Suppose $n=2m+1$ is odd.
The lattice $\Lambda^{\flat,\vee}_{t,O_F}$ is generated by 
\begin{equation*}
 		e_1,\cdots,\wh{e_{m+1}},\ldots,e_{n},e_{n+1},\cdots,\wh{e_{n+m+1}},\ldots,e_{2n}.
 \end{equation*}
Since $t\neq n-1$, we can apply Theorem \ref{thm:ss-lattice} and find the basis of the standard lattice $L^{n-2,1}_{\Lambda_t^\flat,-1}(R)$ and $L_{\Lambda_t,-1}^{n-1,1}(R)$. The former is of the same form as the latter, except that $e_{[i,\wh{n+j}]}$ is defined by taking $(n-1)$-th wedge power, with vectors $e_{m+1}$ and $e_{n+m+1}$ being omitted.
Furthermore, since 
\begin{equation*}
\CF_{\Lambda_t^\vee}=\CF_{\Lambda^{\flat,\vee}_t}\oplus R(\Pi-\pi)u=\CF_{\Lambda^{\flat,\vee}_t}\oplus R e_{n+m+1},
\end{equation*}
we have the equality
 	\begin{equation*}
 		\bigwedge^{n}\CF_{\Lambda_t^\vee}=\bigwedge^{n-1}\CF_{\Lambda^{\flat,\vee}_t}\otimes Re_{n+m+1}.
 	\end{equation*}
By comparing the bases of $L^{n-1,1}_{\Lambda_t^\flat,-1}(R)$ and of $L^{n-2,1}_{\Lambda_t^{\flat,\vee},-1}(R)$, we directly conclude that $\bigwedge^{n}\CF_{\Lambda_t^\vee}\in L^{n-1,1}_{\Lambda_t^\flat,-1}(R)$ if and only if  $\CF_{\Lambda^{\flat,\vee}_t}\in L^{n-2,1}_{\Lambda^{\flat,\vee}_t,-1}(R)$.

\bigskip

Suppose $n=2m$ is even. The lattice $\Lambda^{\flat,\vee}_{t,O_F}$ is generated by 
\begin{align*}
	&e_1^\flat:=e_1,\ldots, e_{m-1}^\flat:=e_{m-1};\hspace{1.63cm} e_{m}^{\flat}:=e_{m}-e_{m+1};\hspace{2cm}
	e_{m+1}^\flat=e_{m+2},\ldots,e_{n-1}^\flat:=e_{n};\\
	&e_{n}^\flat:=e_{n+1},\ldots, e_{n+m-2}^\flat:=e_{n+m-1};\quad
	e_{n+m-1}^\flat:=e_{n+m}-e_{n+m+1};\quad
	e_{n+m}^\flat:=e_{n+m+2},\ldots,e_{2n-2}^\flat:=e_{2n}.
\end{align*}
Since 
\begin{equation*}
\CF_{\Lambda^{\vee}_t}=\CF_{\Lambda^{\flat,\vee}_t}\oplus R(\Pi-\pi)u=\CF_{\Lambda^{\flat,\vee}_t}\oplus R\Pi(\fke_{m}+\fke_{m+1})=\CF_{\Lambda^{\flat,\vee}_t}\oplus R(e_{n+m}+e_{n+m+1}),
\end{equation*}
we have the equality
\begin{equation}\label{equ:F-and-F-flat}
\bigwedge^{n+1}\CF_{\Lambda^{\vee}_t}=\bigwedge^{n}\CF_{\Lambda^{\flat,\vee}_t}\otimes(e_{n+m}+e_{n+m+1}).
\end{equation}
Let us write $\bigwedge^{n}\CF_{\Lambda^{\flat,\vee}_t}$  as a sum of pure tensors. For any subset $I\subset \{1,\cdots,2n\}$, let $e^\flat_I:=\bigwedge_{i\in I}e^\flat_i$. We consider the following cases: 
\begin{enumerate}
\item If the pure tensor has the form $e^\flat_I$, where $\{m,n+m-1\}\cap I=\emptyset$, then as in the even $n$ case, it is straightforward to verify that if $e^\flat_I\wedge (e_{n+m}+e_{n+m+1})\in L_{\Lambda^{\vee}_t,-1}^{n-1,1}$, then $e^\flat_I\in L_{\Lambda^{\flat,\vee}_t,-1}^{n-2,1}$.

\item If the pure tensor has the form $e^\flat_I\wedge e^\flat_{m}$, where $\{m,n+m-1\}\cap I=\emptyset$, the expansion becomes:
\begin{equation*}
e_I^\flat\wedge e_{m}^\flat\wedge u_0=e_I^\flat\wedge (e_{m}\wedge e_{n+m}-e_{m+1}\wedge e_{n+m+1})+e_I^\flat\wedge (e_{m}\wedge e_{n+m+1}-e_{m+1}\wedge e_{n+m}).
\end{equation*}
Consider the term $e_I^\flat\wedge (e_{m}\wedge e_{n+m}-e_{m+1}\wedge e_{n+m+1})$. 
Since $m^\vee:=n+1-m=m+1$, this term lies in $L^{n,1}_{\bLambda_{\fkt},-1}(R)$ if and only if
\begin{equation*}
	e_I^\flat\wedge (e_{m}\wedge e_{n+m}-e_{m+1}\wedge e_{n+m+1})=e_{[m,\wh{n+m^\vee}]}-e_{[m+1,\wh{n+(m+1)^\vee}]},
\end{equation*}
which is the case (ii) in the list of Theorem \ref{thm:ss-lattice}. In particular, we have $I\subset \{1,\cdots,n\}$, and thus $e^\flat_I\wedge e^\flat_{m+1}$ also belongs to case (ii).
\item If the pure tensor has the form $e^\flat_I\wedge e^\flat_{n+m-1}$, where $\{m,n+m-1\}\cap I=\emptyset$, we have
\begin{equation*}
e^\flat_I\wedge e^\flat_{n+m-1}\wedge u_0=2 e^\flat_I \wedge e_{n+m}\wedge e_{n+m+1}.
\end{equation*}
This term lies in $L^{n-1,1}_{\bLambda_{\fkt},-1}(R)$ if and only if
\begin{equation*}
	e^\flat_I \wedge e_{n+m}\wedge e_{n+m+1}=e_{\{n+1,\cdots,2n\}}
\end{equation*}
which is the case (i) of Theorem \ref{thm:ss-comp-converse-special fiber}: since $n+m$ and $n+m+1$ lie in $\{n+1,\cdots,2n\}$, and case (i) is the only case which allows more than one index in $\{n+1,\cdots,2n\}$ to appear.
As a consequence, we have $e^\flat_I\wedge e_{n+m-1}^\flat=e^\flat_{\{n,\cdots,2n-2\}}$ and is spanned by $L^{n-2,1}_{\Lambda^{\flat,\vee},-1}$.

\item If the pure tensor has the form $e^\flat_I\wedge e^\flat_{m}\wedge e^\flat_{n+m-1}$, where $\{m,n+m-1\}\cap I=\emptyset$. Then the generator of the space \eqref{equ:F-and-F-flat} will have a factor of the form $e_J\wedge e_{m}\wedge e_{n+m}\wedge e_{n+m+1}$ for some $J$, but none of the basis in Theorem \ref{thm:ss-lattice} contains a vector of such form.\qedhere
\end{enumerate}

 \end{proof}

 \begin{remark}
 	We can see from the proof of Theorem \ref{thm:ss-comp-converse-special fiber}(iii) that the reason why the exceptional isomorphism fails in the almost $\pi$-modular case is the following:
 	when $t=2m=n-1$, by \cite[Cor. 5.2.3]{Luo24}, the standard lattice $L^{n-2,1}_{\Lambda_t^\flat,-1}(R)$ is generated by the basis of the form (ii)-(xi) (and no $e_{m}$ and $e_{n+m-1}$ appear).
 	Therefore, if $\bigwedge^{n-1}\CF_{\Lambda^{\flat,\vee}_t}\subset L^{n-2,1}_{\Lambda^{\flat,\vee}_t,-1}(R)$, then $\bigwedge^{n}\CF_{\Lambda^{\vee}_t}\subset L^{n-1,1}_{\Lambda^{\vee}_t,-1}(R)$. 
 	But the converse is not true: for instance, if $\bigwedge^{n}\CF_{\Lambda^{\vee}_t}$ is spanned by $e_{\{n+1,\cdots,2n\}}$ (e.g. the worst point $*\in \bM_n^{[t]}$), then $\bigwedge^{n-1}\CF_{\Lambda^{\flat,\vee}_t}$ is spanned by $e_{\{n+1,\cdots,\wh{n+m-1},\cdots,2n\}}$, but this vector is not in $L^{n-2,1}_{\Lambda^{\flat,\vee}_t,-1}(R)$, see the proof of \cite[Cor. 5.2.3]{Luo24}.
 \end{remark}

The description of the basis of $L^{n-2,1}_{\Lambda^\flat_{t},-1}(R)$ is much more complicated when $\pi R\neq 0$.
In order to extend  Theorem \ref{thm:ss-comp-converse-special fiber} from the special fiber to the integral model,  and also study the case when $t=n-1$, we will use  properties of the wedge local models $\bM_{n}^{[t],\wedge}$ studied in \cite{PR,S-topflatI,S-topflatII}.

\begin{definition}
 Let $\Lambda_t^\flat\subset V^\flat$ be a vertex lattice of type $t$ and let $\Lambda_t=\Lambda_t^\flat\oplus O_F u$, such that $(u,u)=1$.
 We define $\bZ(u)^{[t],\wedge}_{n}\subset \bM_{n}^{[t],\wedge}$ as the closed subscheme of $\bM_{n}^{[t],\wedge}$ that sends each $O_F$-algebra $R$ to the set of all families $(\CF_{\Lambda_t}\subset \Lambda_{t,R},\CF_{\Lambda^\vee_t}\subset \Lambda^\vee_{t,R})$ in $\bM_{n}^{[t],\wedge}$ such that $\CF_{\Lambda_t}=\CF_{\Lambda^\flat_t}\oplus R(\Pi-\pi)u$. Similarly, we define\footnote{Despite the suggestive notation, the space $\bZ(u)$ is not a local model of the special cycle $\CZ(u)$! However, they do share the same first-order deformation space, even though their higher-order deformations may differ.} $\bZ(u)_{n}^{[t]}\subset \bM_{n}^{[t]}$.
\end{definition}

 Consider the map
 
 \begin{equation*}\begin{aligned}\xymatrix{
\bM_{n-1}^{[t],\wedge}\ar[r]& \bM_{n}^{[t],\wedge},&(\CF_{\Lambda^\flat})\ar@{|->}[r]&(\CF_{\Lambda^\flat}\oplus R(\Pi-\pi)u).
}\end{aligned}\end{equation*}
This map is well-defined by Lemma \ref{lem:KW-induction}, and it factors through $\bZ(u)_{n}^{[t],\wedge}\subset \bM_{n}^{[t],\wedge}$ by definition. We denote the resulting morphism by $\iota^\wedge:\bM_{n-1}^{[t],\wedge}\to \bZ_{n}^{[t],\wedge}$.
 	Similarly, by Theorem \ref{thm:ss-comp-converse-special fiber}, we have morphisms
 	\begin{equation*}
 		\begin{aligned}
 			\xymatrix{
 				\bM_{n-1}^{[t]}\ar@{^{(}->}[r]^-{{\iota}}&\bZ_{n}^{[t]}\subset \bM_{n}^{[t]}.
 			}
 		\end{aligned}
 	\end{equation*}

An immediate consequence of Lemma \ref{lem:KW-induction} is the following:
\begin{proposition}\label{prop:exc-iso-wedge}
Assume $t\neq n-1$, keep the notation as above. The inclusion
\begin{equation*}
\begin{aligned}
\xymatrix{
\iota^\wedge:\bM_{n-1}^{[t],\wedge}\ar@{^(->}[r] &\bZ_{n}^{[t],\wedge}
}
\end{aligned}
\end{equation*}
is an isomorphism.\qed
\end{proposition}

 The main result in the context of local models  is the following:
 \begin{theorem}\label{thm:exc-iso-loc}
 	Assume $t\neq n$, keep the notation as above.
 	\begin{altenumerate}
 		\item When $t\neq n-1$, there is an equality of closed subschemes of $\bM_{n}^{[t]}$: 
 		\begin{equation*}
 			\bM_{n-1}^{[t]}= \bZ_{n}^{[t]}
			 		\end{equation*}
 		\item When $n$ is odd and $t=n-1$, there is an equality of closed subschemes of $\bM_{n}^{[n-1]}$:
 		\begin{equation*}
 			\bZ_{n}^{[n-1]}= \bM^{[n-1]}_{n-1}\amalg \{*\},
 		\end{equation*}
 		where $*$ is the worst point of $\bM_{n}^{[n-1]}$.
 	\end{altenumerate}
 \end{theorem}
 \begin{proof} 
 For any $t\neq n$, by Proposition \ref{prop:exc-iso-wedge}, we have the inclusions
 	\begin{equation}\label{equ:inc-lmspecialcycle}
 		\bM_{n-1}^{[t]}\subseteq\bZ_{n}^{[t]}\subseteq \bZ_{n}^{[t],\wedge}\simeq \bM_{n-1}^{[t],\wedge}\subseteq \bM_{n}^{[t],\wedge}.
 	\end{equation}
 For part (i), when $t\neq n-1$, by Theorem \ref{thm:ss-comp-converse-special fiber}, we have an isomorphism over the special fiber,
\begin{equation*}\begin{aligned}\xymatrix{
\bM_{n-1,s}^{[t]}\ar[r]_-{\iota_s}^-\sim&\bZ_{n,s}^{[t]}\subset \bM_{n,s}^{[t]}.
}\end{aligned}\end{equation*}
On the other hand, \eqref{equ:inc-lmspecialcycle} induces  isomorphisms over the generic fiber,
\begin{equation*}\begin{aligned}\xymatrix{
\bM_{n-1,\eta}^{[t]}\ar@{^(->}[r]^-{\sim}&\bZ_{n,\eta}^{[t]}\ar@{^(->}[r]^-{\sim}&\bM_{n-1,\eta}^{[t],\wedge}.
}\end{aligned}\end{equation*}
 Since $\bM_{n-1}^{[t]}$ is flat, we conclude that $\iota$ is an isomorphism by \cite[Prop. 14.17]{GW}.

 	For part (ii), when $t=n-1$,  by \cite[Prop. 3.10]{RSZ1}, we have the identification of closed subschemes of $\bM_n^{[n-1],\wedge}$
 	\begin{equation}\label{equ:wedge-pi-modular-lm}
 		\bM_{n-1}^{[n-1],\wedge}= \bM_{n-1}^{[n-1]}\amalg \{*\}.
 	\end{equation}
 	Since the worst point is in the local model, i.e., $*\in \bM_{n}^{[n-1]}$, it also lies in the following intersection:
 	\begin{equation*}
 		*\in \bM_{n}^{[n-1]}\cap \bZ_{n}^{[n-1],\wedge}=\bZ_{n}^{[n-1]}.
 	\end{equation*}
Therefore, we have 
 	\begin{equation*}
 		\bM_{n-1}^{[n-1]}\amalg\{*\}\subseteq \bZ_{n}^{[n-1]}\subseteq \bZ_{n}^{[n-1],\wedge}=\bM_{n-1}^{[n-1],\wedge}= \bM_{n-1}^{[n-1]}\amalg \{*\}.
 	\end{equation*}
This proves (ii).
 \end{proof}
 
As a consequence, we deduce the following:
\begin{corollary}\label{cor:ss-red-converse}
Let $\zeta\in O_F^\times$ be a unit and let $S$ be a scheme over $\Spf O_{\breve{F}}$.
Let  $(Y,\iota_Y,\lambda_Y)$ be a hermitian $O_F$-module of dimension $n-1$ and type $t$ over $S$.
Define
\begin{equation*}
(X,\iota_X,\lambda_X):=(Y\times\ov{\CE},\iota_Y\times\iota_{\ov{\CE}},\lambda_Y\times \zeta\lambda_{\ov{\CE}}).
\end{equation*}
Suppose $t\neq n-1$. If $(X,\iota_X,\lambda_X)$ satisfies the strengthened spin condition, then so does $(Y,\iota_Y,\lambda_Y)$.
\end{corollary}
\begin{proof}
This follows from the definition of the strengthened spin condition and the local model result in Theorem \ref{thm:exc-iso-loc}(i). Note that by passing to some \'etale local extension, we may assume that $\zeta=1\in O_{F_0}^\times$.
\end{proof}

 We also deduce the remaining part of Theorem \ref{thm: Z unit}:
 \begin{corollary}\label{cor:pi-modular-iso}
 	Let $u\in \BV(\BX^{[t]}_{n,\ep})$ be a unit-length vector, with corresponding special cycle $\CZ(u)_{n,\ep}^{[t]}\subset \CN_{n, \ep}^{[t]}$. Set $\ep(u):=\eta(h(u,u))$ and $\ep^\flat=\ep\ep(u)\eta((-1)^{n-1})$. 
 	Suppose $n$ is odd and $t=n-1$.
 	\begin{altenumerate}
 		\item When $\ep^\flat=1$, there is the  exceptional isomorphism
 		\begin{equation*}
 			\CZ(u)_{n,\ep}^{[n-1]}\simeq \CN_{n-1,\ep^\flat}^{[n-1]} .
 		\end{equation*}
 		\item When $\ep^\flat=-1$, the space  $ \CN_{n-1,\ep^\flat}^{[n-1]}$ is not defined  and the special cycle $\CZ(u)^{[n-1]}_{n,\ep}$ is the disjoint union of  points in $\Sing(\CN_{n, \ep}^{[t]})$ (indexed by all almost $\pi$-modular lattices $\Lambda\subset \BV(\BX^{[t]}_{n,\ep})$ containing $u$).
 	\end{altenumerate}
 \end{corollary}
 \begin{proof}
 The worst point $*$ of $\bM_{n}^{[n-1]}$ is represented by the filtrations $(\Pi\Lambda\subset \Lambda\otimes_{O_{F_0}}\BF)$.
By \cite[Prop.  3.4]{HLS2}, this is the only closed point $(\CF_\Lambda\subset\Lambda\otimes_{O_{F_0}}R)$ in the special fiber of the local model satisfying $\Pi\CF_{\Lambda}=(0)$ for $\Lambda=\Lambda_t$ and $\Lambda_t^\vee$.

Let $(X,\iota_X,\lambda_X,\rho_X)\in \CN_{n+1,\ep}^{[n-1]}(\BF)$. It lies in $\Sing(\CN_{n}^{[n-1]})$ if and only if its Hodge filtration satisfies $\Pi D(X)=\Fil(X)\subset D(X)$. 
Equivalently, if we choose $M(X)\subset M(X)[\frac{1}{\pi_0}]$ as the almost $\pi$-modular lattice and use it to define the local model $\bM_{n}^{[n-1]}$, then $(X,\iota_X,\lambda_X,\rho_X)$ lies in $\Sing(\CN_{n}^{[t]})$ if and only if its Hodge filtration $\Fil(X)\subset D(X)\simeq M(X)\otimes_{O_{F_0}}\BF$ defines the worst point $*$ of $\bM_{n}^{[n-1]}$.

Recall that in \S \ref{sec:framing}, we define $N$ as the rational Dieudonn\'e module of the framing object $\BX:=\BX_{n,\ep}^{[n-1]}$, equipped with a hermitian form $\phi$ and a $\sigma$-linear operator $\tau:N\to N$.
Recall from \eqref{equ:RZ-geo-pts} that
the geometric points of the RZ space are given as follows by $O_{\breve F}$-lattices:
\begin{equation*}
\CN_{n,\ep}^{[n-1]}(\BF)=\Big\{M\subset N\mid M\stackrel{n-1}{\subset}M^\vee,\quad\Pi M\subset \tau^{-1}(M)\subset\Pi^{-1}M, \quad M\stackrel{\leq 1}{\subset}(M+\tau(M))\Big\}.
\end{equation*}
By the isometry $\BV(\BX)\otimes_{F_0}\breve{F}_0\simeq C\otimes_{F_0}\breve{F}_0\simeq N$, the unit-length element $u\in \BV(\BX)$ corresponds to a unit-length element in $N$, which we will still denote by $u$.
Under the identification \eqref{equ:RZ-geo-pts}, we have
\begin{equation}\label{equ:special-div-pts}
\CZ(u)_{n,\ep}^{[n-1]}(\BF)=\left\{
M\in \CN_{n,\ep}^{[n-1]}(\BF)\mid u\in M
\right\}=\left\{
M\in \CN_{n,\ep}^{[n-1]}(\BF)\mid M=M^\flat\oplus \langle u\rangle
\right\}.
\end{equation}

Now we prove  (i). By Theorem \ref{thm:ss-comp-converse-special fiber}, we have a closed embedding 
$\iota: \CN_{n-1,\ep^\flat}^{[n-1]}\subseteq \CZ(u)_{n,\ep}^{[n-1]}$. We will prove that this embedding is an isomorphism by proving that it induces a bijection on geometric points and is infinitesimally \'etale.

We first check the bijectivity. Denote by $N^\flat$  the rational Dieudonn\'e module of $\BX_{n-1}^{[n-1]}$, then we have $N=N^\flat\oplus\breve{F}u$ as hermitian spaces. 
This  identification is compatible with the hermitian form $h^\flat$,  the action $\Pi^\flat$, and the $\sigma$-linear operator $\tau^\flat$ in $N^\flat$. We will drop those flat symbols for simplicity.

Recall from \eqref{equ:RZ-geo-pts-pi-modular} that the space of geometric points of $\CN_{n-1}^{[n-1]}$ is given as follows by $O_{\breve F}$-lattices,
\begin{equation*}
\CN_{n-1}^{[n-1]}(\BF)=\Big\{M^\flat\subset N^\flat\mid M^\flat\stackrel{n-1}{\subset}M^{\flat,\vee},\quad\Pi M^\flat\subset \tau^{-1}(M^\flat)\subset\Pi^{-1}M^\flat, \quad M^\flat\stackrel{ 1}{\subset}(M^\flat+\tau(M^\flat))\Big\}.
\end{equation*}
One can rewrite \eqref{equ:special-div-pts} as 
\begin{equation*}
\CZ(u)_{n,\ep}^{[n-1]}(\BF)=\Big\{
M^\flat\subset N^\flat\mid M=M^\flat\oplus \langle u\rangle\in \CN_{n,\ep}^{[n-1]}(\BF)
\Big\} .
\end{equation*}
It is straightforward to verify that such a lattice $M^\flat$ satisfies
\begin{equation*}
M^\flat\stackrel{n-1}{\subset}M^{\flat,\vee},\quad\Pi M^\flat\subset \tau^{-1}(M^\flat)\subset\Pi^{-1}M^\flat, \quad M^\flat\stackrel{ \leq 1}{\subseteq}(M^\flat+\tau(M^\flat)).
\end{equation*}
In the last relation, we must  have $M^\flat\stackrel{1}{\subset}(M^\flat+\tau^\flat(M^\flat))$: otherwise, we have $M^\flat=\tau^\flat(M^\flat)$, and the $\tau^\flat$-invariants $M^{\flat,\tau}\subset N^{\flat,\tau}$ would be a $\pi$-modular vertex lattice in $C$.
This contradicts the fact that $C$ is non-split. 

Indeed, this computation extends naturally from $\BF$ to any algebraically closed field $\kappa$.
Therefore, the closed immersion $\iota: \CN_{n-1,\ep^\flat}^{[n-1]}\subset \CZ(u)_{n,\varepsilon}^{[n-1]}$ induces an equality on geometric points, i.e.,  $\CN_{n-1,\ep^\flat}^{[n-1]}(\kappa)=\CZ(u)_{n,\varepsilon}^{[n-1]}(\kappa)$ for any algebraically closed field $\kappa$ over $k$.

Moreover, by \cite[Lem. 3.3]{RSZ2}, these points are disjoint from $\Sing(\CN_{n}^{[n-1]})$. 
By Grothendieck-Messing theory and Theorem \ref{thm:exc-iso-loc}(ii), for each geometric point $x\in\CN_{n-1,\ep^\flat}^{[n-1]}(\BF)=\CZ(u)_{n,\ep}^{[n-1]}(\BF)$  the first order deformation theory of $x$ in $\CN_{n-1,\ep^\flat}^{[n-1]}$ equals the first order deformation theory  of $\CZ(u)_{n,\ep}^{[n-1]}$, hence $\iota$ is infinitesimally \'etale. 
Since the closed embedding $\iota:\CN_{n-1,\ep^\flat}^{[n-1]}\hookrightarrow \CZ(u)_{n,\ep}^{[n-1]}$ is both surjective and infinitesimally étale, it is an isomorphism.
Part (i) is proved.

For part (ii), we define $N^\flat$ as the orthogonal complement of $\breve{F}u\subset N$. Then $N^\flat$ inherits from $N$ the hermitian form $\phi^\flat$,  the action $\Pi^\flat$, and the $\sigma$-linear operator $\tau^\flat$. By assumption, the  hermitian space $C^\flat$ is split.
One can rewrite \eqref{equ:special-div-pts} as 
\begin{equation*}
\CZ(u)_{n,\ep}^{[n-1]}(\BF)=\Big\{
M^\flat\subset N^\flat\mid M=M^\flat\oplus \langle u\rangle\in \CN_{n,\ep}^{[n-1]}(\BF)
\Big\}.
\end{equation*}
The relation $M\stackrel{ \leq 1}{\subseteq}(M+\tau(M))$ implies
the relation $M^\flat\stackrel{ \leq 1}{\subseteq}(M^\flat+\tau(M^\flat))$.

Since $C^\flat$ is split, by \cite[Lem. 3.3]{RSZ1}, we deduce that  $M^\flat=(M^\flat+\tau(M^\flat))$, i.e., $M^\flat=\tau(M^\flat)$. Hence $M=\tau(M)$. 
This implies that $\CZ(u)^{[n-1]}_{n,\ep}(\BF)=\Sing(\CN_{n, \ep}^{[n-1]})(\BF)$, and hence is in bijection with the set of  almost $\pi$-modular lattices $\Lambda\subset C$ containing $u$, see \S \ref{sec:geo-RZ}.

Recall that  $\tau:=\Pi \uV^{-1}$. Therefore, $M=\tau(M)$ implies that $\Pi M=\uV M\subset M$. Equivalently, the Hodge filtration
\begin{equation*}
\Bigl[\Fil(X)\subset D(X)\Bigr]=
\Bigl[
\uV M/\pi_0 M\subset M/\pi_0 M
\Bigr]=
\Bigl[
\Pi M/\pi_0 M\subset M/\pi_0 M
\Bigr],
\end{equation*}
defines the worst point $*$ of the local model $\bM_{n}^{[n-1]}$.
By Theorem \ref{thm:exc-iso-loc}(ii), the worst point $*\in \bZ(u)_{n,\ep}^{[n-1]}$ has trivial  deformation for any first order infinitesimal thickening. Therefore, by Grothendieck-Messing, each point in $\CZ(u)^{[n-1]}_{n,\ep}(\BF)$ has trivial deformation theory. This proves that the special cycle $\CZ(u)^{[n-1]}_{n,\ep}$ is a disjoint union of discrete geometric points.
 \end{proof}

 \section{AT conjecture of type $(n, t)$}
 \label{sec:type-(n,t)}

 Let $n=2m$ be even. Let $0\leq t\leq n$ be even. We fix an aligned triple of framing objects $(\BY, \BX, u)=(\BY_{\ep^\flat}^{[t]}, \BX_\ep^{[t]}, u)$ of dimension $n+1$ and type $t$, with corresponding embedding of RZ spaces $\CN^{[t]}_{n, \ep^\flat}=\CN^{[t]}_n\hookrightarrow \CN^{[t]}_{n+1}=\CN^{[t]}_{n+1, \ep}$, cf. \eqref{equ:embd-RZ}.
 We also fix an isogeny $\BY^{[n]}\to \BY^{[t]}$ as in \S \ref{sec:RZ-deeper} and extend this in the obvious way to an isogeny $\BX^{[n]}\to \BX^{[t]}$. Note that the existence of this isogeny forces $\ep^\flat=1$ (otherwise $\BY^{[n]}$ is not defined).  
 This defines embeddings of RZ spaces $\CN^{[n]}_n\hookrightarrow \CN^{[n]}_{n+1}$ and  $\CN^{[n,t]}_n\hookrightarrow \CN^{[n,t]}_{n+1}$, cf. \S \ref{sec:emb-RZ}.  
 We will formulate AT conjectures for cycles on  $\CN_n^{[n]} \times \CN_{n+1}^{[t],\spt}$, by using the exceptional special divisor $\CN_n^{[n]}$ on $\CN_{n+1}^{[n]}$. We also use the notation $W_1=\BV(\BX)$ and $W_1^\flat=\BV(\BY)$ so that $u\in W_1$ and $W_1^\flat=\langle u\rangle^\perp$. We denote by $W_0$ the hermitian space of dimension $n+1$ with opposite Hasse invariant of $W_1$, and fix a vector $u_0\in W_0$ of the same length as $u$ and set $W_0^\flat  =\langle u_0\rangle^\perp$. Then $W_0^\flat $ has the opposite Hasse invariant of $W_1^\flat$. We therefore depart from the conventions in \S \ref{sec:thesetting} and \S \ref{sec:AFL}, where $W_0$ and $W_0^\flat$ denoted split spaces and $W_1$ and $W_1^\flat$ non-split spaces.

\subsection{The naive version}\label{sec:(n,t)-naive}

We will use  correspondences to obtain cycles on the product space $ \CN_{n}^{[n]} \times \CN_{n+1}^{[t]}$. Here the adjective \emph{naive} refers to the fact that this product space is not always regular. For simplicity we use a product of two correspondences one of which is trivial, i.e., the identity.   There are two ways to do so. The first one is to use the correspondence $\CN^{[n,t]}_n $ on the smaller space 
 \begin{equation}\label{pismall}
 \begin{aligned}
 \xymatrix{&&\CN^{[n,t]}_n \times \CN_{n+1}^{[t]}  \ar[rd]   \ar[ld] &
 \\ \CN_n^{[t]}\ar@{^(->}[r]& \CN_n^{[t]} \times \CN_{n+1}^{[t]} &&  \CN_{n}^{[n]} \times \CN_{n+1}^{[t]} .
 }
 \end{aligned}
 \end{equation}

 The second one  is to use the correspondence $\CN^{[n,t]}_{n+1}$  on the bigger space,
 \begin{equation}\label{pi1pi2711}
 \begin{aligned}
 \xymatrix{&&\CN^{[n]}_n \times \CN_{n+1}^{[n,t]}  \ar[rd]   \ar[ld] &
 \\ \CN_n^{[n]}\ar@{^(->}[r]& \CN_n^{[n]} \times \CN_{n+1}^{[n]} &&  \CN_{n}^{[n]} \times \CN_{n+1}^{[t]} .
 }
 \end{aligned}
 \end{equation}
The first one leads  us (by taking fiber products) to the following diagram  
 \begin{equation*}
 \begin{aligned}
 \xymatrix{&\CN^{[n,t]}_n \ar[rd]   \ar[ld]_{\pi_1} \ar[rrd]^{\pi_2} &&
 \\ \CN_n^{[n]} &&  \CN_{n}^{[t]}\ar@{^(->}[r]&\CN_{n+1}^{[t]}.
 }
 \end{aligned}
 \end{equation*}
 Indeed, we complete the left oblique arrow in \eqref{pismall} to a fiber square. Then the new vertex can be identified with $\CN^{[n,t]}_n$ (Proposition \ref{prop:deep-exc}), compatibly with its projection to $\CN^{[n]}_n$. The composition of the map to $\CN^{[n,t]}_n \times \CN_{n+1}^{[t]}$ with the projection to the right factor yields $\pi_2$. 
 
 The second one leads us to define $\wh{\CN}_n^{[t]}$ by the cartesian square in the following diagram,
  \begin{equation*}
 \begin{aligned}
 \xymatrix{
 &\wh{\CN}_n^{[t]}\ar@{}[d]|{\square}\ar[dl]_{\pi_1}\ar@{^(->}[r]\ar[rrd]^{\pi_2}
 &\CN_{n+1}^{[n,t]}\ar[rd]\ar[ld]|!{[l];[rd]}\hole&\\
 \CN_n^{[n]}\ar@{^(->}[r]&\CN_{n+1}^{[n]}&&\CN_{n+1}^{[t]}.}
 \end{aligned}
 \end{equation*}

\begin{lemma}
The morphism $(\pi_1,\pi_2): \wh{\CN}_n^{[t]}\to \CN_n^{[n]}\times \CN_{n+1}^{[t]}$ is a closed immersion.
\end{lemma}
\begin{proof}
Indeed, $\wh{\CN}_n^{[t]}$ is the closed formal subscheme of $\CN_n^{[n]}\times \CN_{n+1}^{[t]}$ parameterizing pairs $(Y^{[n]}, X^{[t]})\in \CN_n^{[n]}\times \CN_{n+1}^{[t]}(S)$ such that the quasi-isogeny $\rho_{X^{[t]}}^{-1}\circ\alpha\circ(\rho_{Y^{[n]}}\times\rho_{\ov{\CE}})$ lifts to an isogeny $Y^{[n]}\times\ov{\CE}\to X^{[t]}$ over $S$, 
\begin{equation*}
\begin{aligned}
\xymatrix{
Y^{[n]}_{\ov{S}}\times \ov{\CE}_{\ov{S}}\ar[d]_{\rho_{Y^{[n]}}\times \rho_{\ov{\CE}}}\ar@{-->}[r]&X^{[t]}_{\ov{S}}\ar[d]^{\rho_{X^{[t]}}}\\
\BY^{[n]}_{\ov{S}}\times \ov{\BE}_{\ov{S}}\ar[r]^-{\alpha}&\BX^{[t]}_{\ov{S}}
}
\end{aligned}
\end{equation*}
Here $\ov{S}$ is the special fiber of $S$.
\end{proof}
  
 There is a natural morphism which is a closed embedding,
 $$\CN^{[n,t]}_n\to  \wh{\CN}_n^{[t]}
 $$
 
\begin{theorem}\label{th:wh}
There  is an equality of closed formal subschemes of $\CN_n^{[n]}\times\CN_{n+1}^{[t]}$
$$\wh{\CN}^{[t]}_n={\CN}^{[n, t]}_n.$$  
\end{theorem}
\begin{proof}
Indeed, we have a cartesian diagram 
 \begin{equation*}
 \begin{aligned}
  \xymatrix{\wh{\CN}^{[t]}_n \ar[r] \ar[d] \ar@{}[rd]|{\square}  & \CN^{[n,t]}_{n+1}\ar[d] 
  \\ \CZ(u)^{[ n]} \ar[r] &\CN_{n+1}^{[n]}.}
  \end{aligned}
\end{equation*}
By Theorem \ref{thm:ss-comparison}, we may identify $\CZ(u)^{[n]}$ with $\CN^{[n]}_{n,\ep^\flat}=\CN_n^{[n]}$.  Now we apply  Proposition \ref{prop:deep-exc}. 
\end{proof}
\begin{corollary}
$\wh{\CN}^{[t]}_n$ is flat. \qed
\end{corollary}
This set-up leads one to consider the intersection of $\wh{\CN}^{[t]}_n$ with its translate under an automorphism of $\CN_n^{[n]}\times\CN_{n+1}^{[t]}$. However, the product space $\CN_n^{[n]}\times\CN_{n+1}^{[t]}$ is not regular in general and, therefore, it seems impossible to deduce a finite intersection number in this way. In general, we bypass this problem by passing to $\CN_n^{[n]}\times\CN_{n+1}^{[t], \spt}$. However, there is one case, in which this product space is regular, and we consider this case in the next subsection.

 \subsection{The exotic case $t =n$}
 
In this subsection, we consider the case $t=n$. Here we have regularity without passing to the splitting model. Namely, in the case $t=n$, the formal scheme  $ \CN_{n+1}^{[n]}$ is formally smooth (\emph{exotic smoothness}). In this case, we  have the AT conjecture in \cite{RSZ1}, which we recall briefly. 

 The RZ spaces $\CN^{[n]}_{n}$ and $ \CN^{[n]}_{n+1} $ are both smooth (exotic smoothness).
 There is a natural   closed immersion\footnote{Note that the meaning of   $\CN_{n+1}^{[n]}$ in \cite{RSZ1} is different: due to the spin condition imposed in \cite[\S 3]{RSZ1}, the space in loc.~cit.   is an open subscheme of our RZ space; it is, however, large enough to contain the image of $\CN^{[n]}_{n}$.} $\CN^{[n]}_{n}\incl \CN^{[n]}_{n+1} $, cf. \cite[Lem. 4.2]{RSZ1}, comp. Lemma \ref{thm: Z unit}. Note that here $\ep(\BX_{}^{[n]})$ is uniquely determined, since $\CN^{[n]}_{n}$ is only defined for  $\varepsilon (\BY^{[n]})=1$. Consider   its graph
\begin{equation*}
 \xymatrix{\Delta:\CN^{[n]}_n \ar@{^(->}[r]& \CN_n^{[n]} \times \CN_{n+1}^{[n]}}.
\end{equation*}
We define the arithmetic intersection number
\begin{equation*}
  \label{eq:Int1nevent}
  \left\langle \CN^{[n]}_n , g\CN^{[n]}_n \right\rangle_{\CN_n^{[n]} \times \CN_{n+1}^{[n]} }:=\chi(\CN_n^{[n]} \times \CN_{n+1}^{[n]}, \CN^{[n]}_n\cap^\BL g\CN^{[n]}_n ),
\end{equation*}
 which is a finite number when $g\in G_{W_1}(F_0)$ is regular semisimple, cf. \cite[Remark 4.5]{RSZ1}. Recall that   $W_1=\BV(\BX)$ and $W_1^\flat=\BV(\BY)$.

  Let $\Lambda^\flat_{0}$ be a vertex lattice of type $n$ in $W^\flat_0$. Denote its stabilizer by $K_n^{[n]}$.   We normalize the Haar measure such that $
 \vol(K_n^{[n]})=1.
 $
  Fix a special vector $u_0$ of unit norm in $W_0$, and let  $\Lambda_0=\Lambda^\flat_{0}\obot \langle u_0\rangle \in \Ver^n(W_0)$. 
   Denote by $K_{n+1}^{[n]}$ the stabilizer of $\Lambda_0$.
 \begin{conjecture}[\cite{RSZ1}, Conj. 5.6]\label{conj (n even,t=n)}
 Let $n=2m$ be even.
\begin{altenumerate}
\item\label{item:conjrodd1} There exists $\varphi'\in C_c^\infty(G')$ with transfer $( {\bf 1}_{K_n^{[n]}\times K^{[n]}_{n+1}},0)\in C_c^\infty(G_{W_0})\times C_c^\infty(G_{W_1})$ such that, if $\gamma\in G'(F_0)_\rs$ is matched with  $g\in G_{W_1}(F_0)_\rs$, then
 \begin{equation*}
  \left\langle \CN^{[n]}_n , g\CN^{[n]}_n \right\rangle_{\CN_n^{[n]} \times \CN_{n+1}^{[n]} }  \cdot\log q=- \del\big(\gamma,  \varphi' \big).
\end{equation*}
\item\label{item:conjrodd2} For any $\varphi'\in C_c^\infty(G')$ transferring to  $( {\bf 1}_{K_n^{[n]}\times K^{[n]}_{n+1}},0)\in C_c^\infty(G_{W_0})\times C_c^\infty(G_{W_1})$, there exists $\fp'_\corr\in C_c^\infty(G')$  such that, if $\gamma\in G'(F_0)_\rs$ is matched with  $g\in G_{W_1}(F_0)_\rs$, then
\begin{equation*}\pushQED{\qed}
  \left\langle \CN^{[n]}_n , g\CN^{[n]}_n \right\rangle_{\CN_n^{[n]} \times \CN_{n+1}^{[n]} } \cdot\log q= -\del\big(\gamma,  \varphi' \big) - \Orb\big(\gamma,  \fp'_\corr \big).
\end{equation*}
\end{altenumerate}
\end{conjecture}

\subsection{Lattice models} \label{sec:(n,t)-lat}
Now let us return to the case of general $t$. To get the test functions corresponding to the intersection problems, we now follow \cite[\S9.1]{LRZ2}.

  Let $\Lambda^\flat_{0}$ be the lattice of type $n$ in $W^\flat_0$. Note that since $W^\flat_0$ is defined  to be the hermitian space opposite to $W^\flat_1=\BV(\BY)$, which has Hasse invariant $-1$, the hermitian space $W^\flat_0$ is split and hence does contain $\pi$-modular lattices. Fix a special vector $u_0$ of unit norm in $W_0$, and let  $\Lambda_0=\Lambda^\flat_{0}\obot \langle u_0\rangle \in \Ver^n(W_0)$.  Denote $K_{n+1}^{[n]}$ (resp. $K_{n}^{[n]}$)  the stabilizer of  $\Lambda_0$ (resp.  $\Lambda_0^\flat$).  
  We also fix a lattice   $\Lambda\in \Ver^t(W_0)$ such  that   $\Lambda_0\subset\Lambda$. Then the unit normed vector $u_0$ belongs to $\Lambda$ and hence $\pair{u_0}$ is a direct summand of $\Lambda$ with its orthogonal complement  denoted by $\Lambda^\flat$. Denote by  $K_{n+1}^{[t]}$ (resp $K_{n}^{[t]}$) the stabilizer of $\Lambda$ (resp. $\Lambda^\flat$). Denote by  $K_{n+1}^{[n,t]}$ (resp $K_{n}^{[n,t]}$) the stabilizer of the chain $\Lambda_0\subset \Lambda$ (resp. $\Lambda^\flat_0\subset \Lambda^\flat$).

We have the lattice models for the spaces defined earlier: for a hermitian space $W$ of dimension $n$, let $\BN^{[t]}(W)$ be the space of vertex lattices of type $t$  in  $W$, and let ${\BN}^{[s,t]}(W)$ be the space of pairs of vertex lattices of type $s$, resp. $t$, which are included one in the other. In our situation, we have two hermitian spaces $W_0^\flat$ of dimension $n$ and $W_0=W_0^\flat\oplus \langle u_0 \rangle$ of dimension $n+1$, and let $\BN^{[s, t]}_n=\BN^{[s, t]}(W_0^\flat)$ and $\BN^{[s, t]}_{n+1}=\BN^{[s, t]}(W_0)$.  Besides ${\BN}^{[t, n]}_n$, we consider  
the cartesian product
  \begin{equation*}
 \begin{aligned}
 \xymatrix{
 \wh{\BN}_n^{[t]}\ar@{}[rd]|{\square}\ar[d]_{\pi_1}\ar@{^(->}[r]
 &\BN_{n+1}^{[n,t]}\ar[d]\\
 \BN_n^{[n]}\ar@{^(->}[r]&\BN_{n+1}^{[n]}.}
 \end{aligned}
 \end{equation*}

  Again, there are two ways to obtain cycles on the product space $ \BN_{n}^{[n]} \times \BN_{n+1}^{[t]}$, using correspondences. The first one is to use a correspondence $\BN^{[n,t]}_n $ on the smaller space
 \begin{equation*}
 \begin{aligned}
 \xymatrix{&&\BN^{[n,t]}_n \times \BN_{n+1}^{[t]}  \ar[rd]   \ar[ld] &
 \\ \BN_n^{[t]}\ar@{^(->}[r]& \BN_n^{[t]} \times \BN_{n+1}^{[t]} &&  \BN_{n}^{[n]} \times \BN_{n+1}^{[t]} .
 }
 \end{aligned}
 \end{equation*}
 The second one  is to use a correspondence $\BN^{[n,t]}_{n+1}$  on the bigger space
 \begin{equation*}
 \begin{aligned}
 \xymatrix{&&\BN^{[n]}_n \times \BN_{n+1}^{[n,t]}  \ar[rd]   \ar[ld] &
 \\ \BN_n^{[n]}\ar@{^(->}[r]& \BN_n^{[n]} \times \BN_{n+1}^{[n]} &&  \BN_{n}^{[n]} \times \BN_{n+1}^{[t]} .
 }
 \end{aligned}
 \end{equation*}
 
 We obtain embeddings 
  \begin{equation*}
 \begin{aligned}
 \xymatrix{\BN^{[n,t]}_n \ar@{^(->}[r]& \BN_{n}^{[n]} \times \BN_{n+1}^{[t]}
 }
 \end{aligned}
 \end{equation*}
 and
   \begin{equation*}
 \begin{aligned}
 \xymatrix{\wh\BN^{[t]}_n \ar@{^(->}[r]& \BN_{n}^{[n]} \times \BN_{n+1}^{[t]}
.}
 \end{aligned}
 \end{equation*}
 We thus have two intersection problems in the ambient space $\BN_{n}^{[n]} \times \BN_{n+1}^{[t]}$, namely 
 $$ \#(\BN^{[n,t]}_n \cap g \BN^{[n,t]}_n), 
 \quad  \#(\wh\BN^{[t]}_n \cap g \wh\BN^{[t]}),
 $$
 where $g$ is regular semisimple in $\U(W_0^\flat)\times\U(W_0)$.

For the first one, we have the Hecke correspondence $\BT_{} $   consisting of the triples  $(\Lambda^\flat,\Lambda^\flat_{0},\Lambda'^\flat_{0}  )\in \BN_n^{[n]}\times \BN_n^{[t]}\times \BN_n^{[t]}$ such that  $ \Lambda^\flat\subset \Lambda^\flat_{0}\cap  \Lambda'^\flat_{0}$. In other words,  $\BT_{}$  is the composition of 
  the obvious correspondence with its transpose
   \begin{equation*}
   \label{Hklat}
\begin{aligned}
\xymatrix{&&\BT_{} \ar[rd]  \ar[ld] &
 \\ &\BN_n^{[t,n]} \ar[rd]  \ar[ld] & &\BN_n^{[n,t]} \ar[rd]  \ar[ld] &
\\ \BN_n^{[t]} &&  \BN_n^{[n]}&&  \BN_n^{[t]}}
\end{aligned}
\end{equation*}
Associated to the correspondence is the bi-$K_n^{[t]}$-invariant Hecke function  (cf. \cite[\S4.1]{LRZ1},
   \begin{equation}\label{eq: phi(t,n)}
\varphi^{[t,n]}_{n}:=\vol(K^{[n]}_n)^{-1} {\bf 1}_{K^{[t]}_nK^{[n]}_n}\ast {\bf 1}_{K^{[n]}_n K^{[t]}_n}.
\end{equation}
Due to the volume factor, the function is independent of the choice of Haar measure used to define the convolution. 
We form the cartesian product $\BN_{n}^{[n,t]}(g)$,
\begin{equation*}\label{cartdi}
\begin{aligned}\xymatrix{ \BN_{n}^{[t,n]}(g)\ar[d]\ar[r] & \BT_{ }\times \Delta_{  \BN^{[t]}_{n+1}} \ar[d]\\
\BN^{[t]}_n \times \BN^{[t]}_n\,\,  \ar[r]^-{(\id, g)}& \,\,(\BN^{[t]}_n\times \BN^{[t]}_{n+1})\times (\BN^{[t]}_n\times \BN^{[t]}_{n+1}) .
}
\end{aligned}
\end{equation*}

Similarly to \cite[Lem. 9.1.3]{LRZ2}, we have an interpretation of orbital integrals in terms of lattice counting,
$$
\Orb(g,\vol(K^{[t]}_n)^{-2}\varphi^{[t,n]}_{n}\otimes{\bf 1}_{ K_{n+1}^{[t]} } )= \#\BN_{n}^{[n,t]}(g)= \#(\BN^{[n,t]}_n \cap g \BN^{[n,t]}_n).
$$
Note that the orbital integral on the product of unitary groups depends on the choice of a Haar measure; but the factor $\vol(K^{[t]}_n)^{-2}$ in our formula above makes the orbital integral independent of such a choice.

Analogously, we define the bi-$K_{n+1}^{[n]}$-invariant Hecke function 
   \begin{equation}\label{eq: phi n+1 (n,t)}
\varphi^{[n,t]}_{n+1}:=\vol(K^{[t]}_{n+1})^{-1} {\bf 1}_{K^{[n]}_{n+1}K^{[t]}_{n+1}}\ast {\bf 1}_{K^{[t]}_{n+1}K^{[n]}_{n+1}} ,
\end{equation}
where  in the definition of convolution we normalize the Haar measure such that $\vol(K^{[n]}_{n+1})=1$. Then we have an interpretation of the second intersection problem,
 $$
\Orb(g,\vol(K_n^{[n]})^{-2} {\bf 1}_{K_n^{[n]}}\otimes \varphi^{[n,t]}_{n+1})= \#(\wh\BN^{[t]}_n \cap g \wh\BN^{[t]}_n).
$$
 
\begin{lemma}\label{lem: lat (n,t)}  
Let $0\leq t\leq n$ be even.
 \begin{altenumerate}
 \item  We have $\wh{\BN}^{[t]}_n={\BN}^{[t, n]}_n$  (viewed as subsets in $ \BN_{n}^{[n]} \times \BN_{n+1}^{[t]}$).
\item We have $K_{n+1}^{[n]} K_{n+1}^{[t]}=K_n^{[n]} K_{n+1}^{[t]}$.
\item  We have
  $$
  \vol( K_n^{[n,t]})^{-2}{\bf 1}_{ K_n^{[n]}} \otimes{\bf 1}_{ K_{n+1}^{[t]} }\sim \vol(K^{[t]}_n)^{-2}\varphi^{[t,n]}_{n}\otimes{\bf 1}_{ K_{n+1}^{[t]} }\sim \vol(K_n^{[n]})^{-2} {\bf 1}_{K_n^{[n]}}\otimes \varphi^{[n,t]}_{n+1}.
  $$
  {\rm Here $\varphi_1\sim\varphi_2$ means that  $\varphi_1$ and $\varphi_2$ have identical regular semi-simple orbital integrals.}
\end{altenumerate}
\end{lemma}

\begin{proof}
Part (i). Write the cartesian product explicitly:
$$
\wh{\BN}^{[t]}_n=\{(\Lambda^\flat, \Lambda_1, \Lambda)\in {\rm Vert}^n(W^\flat)\times {\rm Vert}^t(W)\times{\rm Vert}^n(W)\mid \Lambda=\Lambda^\flat\oplus\langle u\rangle\subset\Lambda_1\}. 
$$
But then $\Lambda_1=\Lambda^\flat_1\oplus\langle u\rangle$ with $\Lambda^\flat\subset\Lambda_1^\flat$. Hence
$$
\wh{\BN}^{[t]}_n={\BN}^{[t, n]}_n
$$
and part (i) is proved.

Part (ii). Clearly we have $K_{n+1}^{[n]} K_{n+1}^{[t]}\supset K_n^{[n]} K_{n+1}^{[t]}$. It suffices to show both have the same number of right $K_{n+1}^{[t]}$-cosets. We have natural bijections
$$
K_{n+1}^{[n]} K_{n+1}^{[t]}/K_{n+1}^{[t]}\simeq K_{n+1}^{[n]}/K_{n+1}^{[n,t]},
$$
and
$$
K_n^{[n]} K_{n+1}^{[t]}/K_{n+1}^{[t]}\simeq K_{n}^{[n]}/K_{n}^{[n,t]}
$$
where we note that $K_{n}^{[n]}\cap K_{n+1}^{[t]}=K_{n}^{[n,t]}$ under our choice of lattices used to define these compact open subgroups. Now note that $ K_{n+1}^{[n]}/K_{n+1}^{[n,t]}$ is bijective to the set of  lattices in $W_0$ of type $t$ containing a fixed lattice $\Lambda$ of type $n$, which is bijective to the set of isotropic subspaces of dimension $\frac{n-t}{2}$ in $\Lambda^\vee/\Lambda$ with the induced symplectic pairing. Similarly  $ K_{n}^{[n]}/K_{n}^{[n,t]}$ is  bijective to the set of isotropic subspaces of dimension $\frac{n-t}{2}$ in $\Lambda_0^\vee/\Lambda_0$. But we have  $\Lambda^\vee/\Lambda\simeq \Lambda_0^\vee/\Lambda_0$ since  $\Lambda_0=\Lambda^\flat_{0}\obot \langle u_0\rangle$ with $u_0$  a unit norm vector.

Part (iii). We follow the proof of \cite[Lem. 9.1.2]{LRZ2}. We recall from \cite[(9.1.4)]{LRZ2}
\begin{align}\label{eq: eM}
   \Orb(g, f)=    \Orb(g, e_{\Delta(M)}\ast f \ast e_{\Delta(M)}), \quad e_M=\vol(M)^{-1} {\bf 1}_{\Delta(M)},
\end{align}
for any compact open subgroup $M$ of $\U(W_0^\flat)$. We apply this to  $f={\bf 1}_{ K_n^{[n]}} \otimes{\bf 1}_{ K_{n+1}^{[t]} }$ and $M=K_n^{[t]}$. By the bi-$K_{n}^{[t]}$-invariance of ${\bf 1}_{ K_{n+1}^{[t]} }$ we obtain
\begin{align*}
&e_{\Delta( K_n^{[t]})}\ast({\bf 1}_{ K_n^{[n]}} \otimes{\bf 1}_{ K_{n+1}^{[t]} })\ast e_{\Delta( K_n^{[t]})}
\\
=&(e_{ K_n^{[t]}}\ast{\bf 1}_{ K_n^{[n]}}  \ast e_{ K_n^{[t]}}) \otimes{\bf 1}_{ K_{n+1}^{[t]} }
\\=& (\vol( K_n^{[n]})^{-1}  e_{ K_n^{[t]}}\ast{\bf 1}_{ K_n^{[n]}} \ast{\bf 1}_{ K_n^{[n]}}  \ast e_{ K_n^{[t]}})\otimes{\bf 1}_{ K_{n+1}^{[t]} }
\\=&\vol( K_n^{[n]})^{-1} \vol( K_n^{[t]})^{-2} \vol( K_n^{[n,t]})^{2} {\bf 1}_{K_n^{[t]} K_n^{[n]}} \ast{\bf 1}_{ K_n^{[n]}K_n^{[t]}}  \otimes{\bf 1}_{ K_{n+1}^{[t]}}
\\=& \vol( K_n^{[n,t]})^{2}  \vol( K_n^{[t]})^{-2} \varphi_{n}^{[t,n]}\otimes{\bf 1}_{ K_{n+1}^{[t]}}.
\end{align*}
This proves that $\vol( K_n^{[n,t]})^{-2}{\bf 1}_{ K_n^{[n]}} \otimes{\bf 1}_{ K_{n+1}^{[t]} }\sim  \vol( K_n^{[t]})^{-2} \varphi^{[t,n]}_{n}\otimes{\bf 1}_{ K_{n+1}^{[t]} }$. 

Next we apply \eqref{eq: eM} to $f={\bf 1}_{ K_n^{[n]}} \otimes{\bf 1}_{ K_{n+1}^{[t]} }$ and $M=K_n^{[n]}$:
\begin{align*}
&e_{ \Delta(K_n^{[n]})}\ast({\bf 1}_{ K_n^{[n]}} \otimes {\bf 1}_{ K_{n+1}^{[t]} })\ast e_{ \Delta(K_n^{[n]})}
\\
=&{\bf 1}_{ K_n^{[n]}}  \otimes (e_{ K_n^{[n]}}\ast {\bf 1}_{ K_{n+1}^{[t]} }\ast e_{ K_n^{[n]}})
\\
=&{\bf 1}_{ K_n^{[n]}}  \otimes (\vol( K_n^{[n]})^{-2} \vol(K_n^{[n,t]})^2 \vol( K_{n+1}^{[t]})^{-1}  {\bf 1}_{ K_n^{[n]}K_{n+1}^{[t]} } \ast {\bf 1}_{ K_{n+1}^{[t]}K_n^{[n]} } ) .
\end{align*}
Here, for $\phi \in C_c^\infty(\U(W^\flat)), \varphi\in  C_c^\infty(\U(W))$, the convolution $\phi \ast  \varphi$ is defined as the function on $\U(W)$ given by
$$
(\phi \ast  \varphi)(g)=\int_{h\in \U(W^\flat)}\phi(h)\varphi(h^{-1}g)\,dh.
$$

By part (ii) we rewrite it as
\begin{align*}
&e_{ \Delta(K_n^{[n]})}\ast({\bf 1}_{ K_n^{[n]}} \otimes {\bf 1}_{ K_{n+1}^{[t]} })\ast e_{ \Delta(K_n^{[n]})}
\\
=&{\bf 1}_{ K_n^{[n]}}  \otimes (\vol( K_n^{[n]})^{-2} \vol(K_n^{[n,t]})^2  \vol( K_{n+1}^{[t]})^{-1} {\bf 1}_{ K_{n+1}^{[n]}K_{n+1}^{[t]} } \ast {\bf 1}_{ K_{n+1}^{[t]}K_{n+1}^{[n]} } )
\\=&{\bf 1}_{ K_n^{[n]}}  \otimes (\vol( K_n^{[n]})^{-2} \vol(K_n^{[n,t]})^2 \varphi^{[n,t]}_{n+1}).
\end{align*}
This proves that $
  \vol( K_n^{[n,t]})^{-2}{\bf 1}_{ K_n^{[n]}} \otimes{\bf 1}_{ K_{n+1}^{[t]} }\sim \vol(K_n^{[n]})^{-2} {\bf 1}_{K_n^{[n]}}\otimes \varphi^{[n,t]}_{n+1}.  $
\end{proof}

 \subsection{Intersection numbers on the splitting model for $0\leq t\leq n$}

Let $\wh\CN^{[t], \spt}_{n}$ be the flat closure of  the base change of $\wh\CN^{[t]}_{n}$ along the morphism $\CN_n^{[n]}\times\CN_{n+1}^{[t],\spt}\to \CN_n^{[n]}\times\CN_{n+1}^{[t]}$. Then $\wh\CN^{[t], \spt}_{n}$ is a closed formal subscheme of $\CN_n^{[n]}\times\CN_{n+1}^{[t],\spt}$, flat over $\Spf O_{\breve{F}}$ of relative dimension $n-1$.  We have the commutative diagram 
  \begin{equation}\label{splcy}
 \begin{aligned}
  \xymatrix{
\wh\CN^{[t], \spt}_{n}\ar[r] \ar[d]  &\CN_n^{[n]}\times\CN_{n+1}^{[t],\spt} \ar[d]   \\
\wh\CN^{[t]}_{n}  \ar[r] & \CN^{[n]}_{n}\times\CN^{[t]}_{n+1}.
}
\end{aligned}
\end{equation}

Since $\CN_{n+1}^{[t],\spt}$ is regular and $ \CN_n^{[n]} $ is formally smooth over $\Spf\OFb$, we know that the product $\CN_n^{[n]} \times \CN_{n+1}^{[t],\spt}$ is regular. Hence it makes sense to define arithmetic intersection numbers of closed formal subschemes of the ambient space $\CN_n^{[n]} \times \CN_{n+1}^{[t],\spt}$,
\begin{equation*}
  \label{eq:Int1nevent}
  \left\langle \wh\CN^{[t],\spt}_n , g\wh\CN^{[t],\spt}_n \right\rangle_{\CN_n^{[n]} \times \CN_{n+1}^{[t],\spt}}:=\chi(\CN_n^{[n]} \times \CN_{n+1}^{[t],\spt}, \wh\CN^{[t],\spt}_n\cap^\BL g\wh\CN^{[t],\spt}_n ),
\end{equation*}
 where $g\in \U(W^\flat_1)(F_0)\times \U(W_1)(F_0)$. The arithmetic intersection number is finite as long as $ \wh\CN^{[t],\spt}_n\cap g \wh\CN^{[t],\spt}_n$ is a proper scheme over $\Spf\OFb$, which is the case if $g$ is regular semisimple by the standard argument, cf. \cite[proof of Lem. 6.1]{M-Th}.

\subsection{The AT conjecture $0\leq t\leq n$}  
 The considerations on lattice models lead us (by following the heuristic principles of \cite[\S 9.1]{LRZ2}) to state the following conjecture. 
 \begin{conjecture}\label{conj (n even,t<n)}
 Let $n=2m$ be even, and $0\leq t\leq n$ even.
There exists $\varphi'\in C_c^\infty(G')$ with transfer $(\vol( K_n^{[n,t]})^{-2}  {\bf 1}_{K_n^{[n]}\times K^{[t]}_{n+1}},0)\in C_c^\infty(G_{W_0})\times C_c^\infty(G_{W_1})$ such that, if $\gamma\in G'(F_0)_\rs$ is matched with  $g\in G_{W_1}(F_0)_\rs$, then
 \begin{equation*}
  \left\langle \wh\CN^{[t],\spt}_n , g\wh\CN^{[t],\spt}_n \right\rangle_{\CN_n^{[n]} \times \CN_{n+1}^{[t],\spt}} \cdot\log q=- \del\big(\gamma,  \varphi' \big).
\end{equation*}

\end{conjecture}

\begin{remark}
By Lemma \ref{lem: lat (n,t)}  part (iii), one could replace the function $\vol( K_n^{[n,t]})^{-2}  {\bf 1}_{K_n^{[n]}\times K^{[t]}_{n+1}}$ by either of the other two.
\end{remark}

\subsection{Relation of two versions when $t=n$}

The map $\CN_{n+1}^{[n],\spt}\rightarrow \CN_{n+1}^{[n]}$ is a blow-up morphism and hence the two horizontal maps in  \eqref{splcy} are genuinely different.

\begin{proposition}\label{prop: (n even,t=n)}
Conjecture \ref{conj (n even,t=n)} is equivalent to Conjecture \ref{conj (n even,t<n)} in the case $t=n$.
\end{proposition}
\begin{proof}
The embedding $\CN^{[n]}_n\hookrightarrow \CN^{[n]}_{n+1}$ factors through $ \CN^{[n]}_{n+1}\setminus \Sing(\CN^{[n]}_{n+1})$, comp. \cite[Lem. 4.2]{RSZ1}. Hence the support of 
the intersection $ \wh\CN^{[n]}_n \cap g\wh\CN^{[n]}_n$ is away from the worst points. Hence we get
$$
\left\langle \wh\CN^{[n]}_n , g\wh\CN^{[n]}_n \right\rangle_{\CN_n^{[n]} \times \CN_{n+1}^{[n]}}=\left\langle \wh\CN^{[n]}_n , g\wh\CN^{[n],\spt}_n \right\rangle_{\CN_n^{[n],\spt} \times \CN_{n+1}^{[n],\spt}} ,
$$
since the map $\CN_{n+1}^{[n],\spt}\rightarrow \CN_{n+1}^{[n]}$ is an isomorphism away from the worst points. 
\end{proof}

\begin{remark}
We cannot make a graph version AT conjecture in this case, since the space $\wh{\CN}_n^{[t],\spt} \times \CN_{n+1}^{[t],\spt}$ is not regular. In general, the diagonal embedding of regular (even semi-stable) schemes is not locally complete intersection.
\end{remark}

  \section{AT conjecture of type $(n-1, t)$}\label{sec:type-(n-1,t)}

 Let $n=2m+1$ be odd. Let $0\leq t\leq  n+1$ be even.  There are two versions now: one for an exceptional special $\CZ$-divisor on $\CN_{n+1, \ep}^{[n-1]}$, and one for an exceptional special $\CY$-divisor on $\CN_{n+1, \ep}^{[n+1]}$. Both  lead to an intersection product on $\CN_n^{[n-1]}\times \CN_{n+1}^{[t], \spt}$. 
 
 When $t\leq n-1$, we fix an aligned triple of framing objects $(\BY, \BX, u)=(\BY_{\ep^\flat}^{[t]}, \BX_\ep^{[t]}, u)$ of dimension $n+1$ and type $t$, with corresponding embedding of RZ spaces $\CN^{[t]}_{n, \ep^\flat}=\CN^{[t]}_n\hookrightarrow \CN^{[t]}_{n+1}=\CN^{[t]}_{n+1, \ep}$, cf. \S \ref{sec:emb-RZ}. We also fix an isogeny $\BY^{[n-1]}\to \BY^{[t]}$ as in \S \ref{sec:RZ-deeper}, and extend this in the obvious way to an isogeny $\BX^{[n-1]}\to \BX^{[t]}$.   This defines embeddings of RZ spaces $\CN^{[n-1]}_n\hookrightarrow \CN^{[n-1]}_{n+1}$ and  $\CN^{[n-1,t]}_n\hookrightarrow \CN^{[n-1,t]}_{n+1}$, cf. \S \ref{sec:emb-RZ}. We will formulate an AT conjecture for cycles on  $\CN_n^{[n-1]} \times \CN_{n+1}^{[t],\spt}$, by using the exceptional special divisor $\CN_n^{[n-1]}$ on $\CN_{n+1}^{[n-1]}$. 
 
 We again use the notation $W_1=\BV(\BX)$ and $W_1^\flat=\BV(\BY)$, so that $u\in W_1$ and $W_1^\flat=\langle u\rangle^\perp$. We denote by $W_0$ the hermitian space of dimension $n+1$ with opposite Hasse invariant of $W_1$, and fix a vector $u_0\in W_0$ of the same length as $u$ and set $W_0^\flat  =\langle u_0\rangle^\perp$. Then $W_0^\flat $ has the opposite Hasse invariant of $W_1^\flat$.

  \subsection{The naive version via $\CZ$-divisors}

 Let  $0\leq t\leq n-1$. 
Similarly to \S\ref{sec:type-(n,t)}, there are two ways to use correspondences to obtain natural cycles on the ambient space  $\CN_{n}^{[n-1]} \times \CN_{n+1}^{[t]}$.  The first one is to use a correspondence $\CN^{[n-1,t]}_n $ on the smaller space
 \begin{equation*}
 \begin{aligned}
 \xymatrix{&&\CN^{[n-1,t]}_n \times \CN_{n+1}^{[t]}  \ar[rd]   \ar[ld] &
 \\ \CN_n^{[t]}\ar@{^(->}[r]& \CN_n^{[t]} \times \CN_{n+1}^{[t]} &&  \CN_{n}^{[n-1]} \times \CN_{n+1}^{[t]} .
 }
 \end{aligned}
 \end{equation*}
 The second one  is to use a correspondence $\CN^{[n-1,t]}_{n+1}$  on the bigger space
 \begin{equation}\label{eq:type (n-1,t) big}
 \begin{aligned}
 \xymatrix{&&\CN^{[n-1]}_n \times \CN_{n+1}^{[n-1,t]}  \ar[rd]   \ar[ld] &
 \\ \CN_n^{[n-1]}\ar@{^(->}[r]& \CN_n^{[n-1]} \times \CN_{n+1}^{[n-1]} &&  \CN_{n}^{[n-1]} \times \CN_{n+1}^{[t]} .
 }
 \end{aligned}
 \end{equation}

  The first leads to the small correspondence,
 \begin{equation*}
 \begin{aligned}
 \xymatrix{&\CN^{[n-1,t]}_n \ar[rd]   \ar[ld]_{\pi_1}  \ar[rrd]^{\pi_2}&&
 \\ \CN_n^{[n-1]} &&  \CN_{n}^{[t]}\ar@{^(->}[r]&\CN_{n+1}^{[t]}.
 }
 \end{aligned}
 \end{equation*}
 The second leads to the big correspondence,

  \begin{equation}\label{pi1pi281}
 \begin{aligned}
 \xymatrix{
 &\wh{\CN}_n^{[t]}\ar@{}[d]|{\square}\ar[dl]_{\pi_1}\ar@{^(->}[r]\ar[rrd]^{\pi_2}
 &\CN_{n+1}^{[n-1,t]}\ar[rd]\ar[ld]|!{[l];[rd]}\hole&\\
 \CN_n^{[n-1]}\ar@{^(->}[r]&\CN_{n+1}^{[n-1]}&&\CN_{n+1}^{[t]}.}
 \end{aligned}
 \end{equation}
\begin{lemma}
 The morphism $(\pi_1, \pi_2)\colon \wh{\CN}_n^{[t]}\to \CN_n^{[n-1]}\times\CN_{n+1}^{[t]}$ is a closed embedding. 
 \end{lemma}
 \begin{proof}
Indeed, $\wh{\CN}_n^{[t]}$ is the closed formal subscheme of $\CN_n^{[n-1]}\times \CN_{n+1}^{[t]}$ parameterizing pairs $(Y^{[n-1]}, X^{[t]})$ such that the quasi-isogeny $\rho_{X^{[t]}}^{-1}\circ\alpha\circ(\rho_Y^{[n-1]}\times\rho_{\ov{\CE}})$
lifts to an isogeny $Y^{[n-1]}\times\ov{\CE}\to X^{[t]}$. Here $\alpha:\BY^{[n-1]}\times \ov{\BE}\to \BX^{[t]}$ is the isogeny between framing objects.
\end{proof}
 
 There is a natural morphism which is a closed embedding,
 $$\CN^{[n-1,t]}_n\hookrightarrow  \wh{\CN}_n^{[t]}.
 $$
 
\begin{theorem} \label{th: wt N=N(t,n-1)}
Let $0\leq t\leq n-1$.
There  is an equality of closed formal subschemes of $\CN_n^{[n-1]}\times\CN_{n+1}^{[t]}$,
$$\wh{\CN}^{[t]}_n={\CN}^{[t, n-1]}_n.$$  
\end{theorem}
\begin{proof}
Indeed, we have a cartesian diagram 
 \begin{equation*}
 \begin{aligned}
  \xymatrix{\wh{\CN}^{[t]}_n \ar[r] \ar[d] \ar@{}[rd]|{\square}  & \CN^{[t,n-1]}_{n+1}\ar[d] 
  \\ \CZ(u)^{[ n-1]} \ar[r] &\CN_{n+1}^{[n-1]}.}
  \end{aligned}
\end{equation*}
By Theorem \ref{thm: Z unit}, we may identify $\CZ(u)^{[n-1]}$ with $\CN_n^{[n-1]}$. Now we apply  Proposition \ref{prop:deep-exc}. 
\end{proof}
\begin{corollary}
The space $\wh{\CN}^{[t]}_n$ is flat.\qed
\end{corollary}

   \subsection{Lattice models for the $\CZ$-divisors}
 \label{ss: Lat(n-1,t) Z}
 
  We continue to assume $t\leq n-1$. Let $\Lambda^\flat_{0}$ be the lattice of type $n-1$ in $W^\flat_0$. Fix  a special vector $u_0\in W_0$ of the same unit norm as $u$. Then   $\Lambda_0=\Lambda^\flat_{0}\obot \langle u_0 \rangle \in \Ver^{n-1}(W_0)$.  Denote by  $K_{n+1}^{[n]}$ (resp. $K_n^{[n-1]}$) the stabilizer of $\Lambda_0$ (resp.  $\Lambda^\flat_{0}$).
 We also fix a lattice   $\Lambda\in \Ver^t(W_0)$ such  that   $\Lambda_0\subset\Lambda$. Then $u_0\in \Lambda$ and hence $\Lambda= \Lambda^\flat\obot\pair{u_0}$ for a lattice $\Lambda^\flat \in \Ver^{t}(W_0)$. Denote by  $K_{n+1}^{[t]}$ (resp. $K_{n}^{[t]}$)  the stabilizer of $\Lambda$ (resp. $\Lambda^\flat$).

We have the lattice models for the RZ spaces defined earlier, similarly to \S\ref{sec:(n,t)-lat}.
  There are again two ways to define intersection problems: the first one uses the correspondence $\BN_n^{[n-1, t]}$ 
\begin{equation*}
 \begin{aligned}
 \xymatrix{&&\BN^{[n-1,t]}_n \times \BN_{n+1}^{[t]}  \ar[rd]   \ar[ld] &
 \\ \BN_n^{[t]}\ar@{^(->}[r]& \BN_n^{[t]} \times \BN_{n+1}^{[t]} &&  \BN_{n}^{[n-1]} \times \BN_{n+1}^{[t]} .
 }
 \end{aligned}
 \end{equation*}
 The second one  is to use the correspondence $\BN^{[n-1,t]}_{n+1}$
  \begin{equation*}
 \begin{aligned}
 \xymatrix{&&\BN^{[n-1]}_n \times \BN_{n+1}^{[n-1,t]}  \ar[rd]   \ar[ld] &
 \\ \BN_n^{[n-1]}\ar@{^(->}[r]& \BN_n^{[n-1]} \times \BN_{n+1}^{[n-1]} &&  \BN_{n}^{[n-1]} \times \BN_{n+1}^{[t]} .
 }
 \end{aligned}
 \end{equation*}
   The first one leads to ${\BN}^{[t, n-1]}_n$, considered as a subset of $ \BN_{n}^{[n]} \times \BN_{n+1}^{[t]}$; the second leads to  
the cartesian product
  \begin{equation*}
 \begin{aligned}
 \xymatrix{
 \wh{\BN}_n^{[t]}\ar@{}[rd]|{\square}\ar[d]_{\pi_1}\ar@{^(->}[r]
 &\BN_{n+1}^{[n-1,t]}\ar[d]\\
 \BN_n^{[n-1]}\ar@{^(->}[r]&\BN_{n+1}^{[n-1]}.}
 \end{aligned}
 \end{equation*}

  The argument in \S\ref{sec:(n,t)-lat} applies verbatim and we only record the test functions and results and omit the details of the proof.
 We introduce the
bi-$K_n^{[t]}$-invariant Hecke function,
   \begin{equation}\label{phih}
\varphi^{[t,n-1]}_{n}:=\vol(K^{[n-1]}_n)^{-1} {\bf 1}_{K^{[t]}_nK^{[n-1]}_n}\ast {\bf 1}_{K^{[n-1]}_n K^{[t]}_n} .
\end{equation}
Then we have 
$$
\Orb(g,\vol( K_n^{[t]})^{-2} \varphi^{[t,n-1]}_{n}\otimes{\bf 1}_{ K_{n+1}^{[t]} } )= \#(\BN^{[n-1,t]}_n \cap g \BN^{[n-1,t]}_n).
$$
We define the bi-$K_{n+1}^{[n-1]}$-invariant Hecke function 
   \begin{equation*}
\varphi^{[n-1,t]}_{n+1}:=\vol(K^{[t]}_{n+1})^{-1} {\bf 1}_{K^{[n-1]}_{n+1}K^{[t]}_{n+1}}\ast {\bf 1}_{K^{[t]}_{n+1}K^{[n-1]}_{n+1}}
\end{equation*}
and we obtain
 $$
\Orb(g,\vol(K_n^{[n-1]})^{-2} {\bf 1}_{K_n^{[n-1]}}\otimes \varphi^{[n-1,t]}_{n+1})= \#(\wh\BN^{[t]}_n \cap g \wh\BN^{[t]}_n).
$$

\begin{lemma}
\label{lem: lat (n-1,t)}  
Let $0\leq t\leq n-1$ be even.
 \begin{altenumerate}
 \item $\wh{\BN}^{[t]}_n={\BN}^{[t, n-1]}_n$ (viewed as subsets in $ \BN_{n}^{[n]} \times \BN_{n+1}^{[t]}$).
 \item
 $K_{n}^{[n-1]} K_{n+1}^{[t]}=K_{n+1}^{[n-1]} K_{n+1}^{[t]}$.

 \item
 $
  \vol( K_n^{[n-1,t]})^{-2}{\bf 1}_{ K_n^{[n-1]}} \otimes{\bf 1}_{ K_{n+1}^{[t]} }\sim  \vol( K_n^{[t]})^{-2}  \varphi^{[t,n-1]}_{n}\otimes{\bf 1}_{ K_{n+1}^{[t]} }\sim \vol(K_n^{[n-1]})^{-2} {\bf 1}_{K_n^{[n-1]}}\otimes \varphi^{[n-1,t]}_{n+1}.
  $
 \end{altenumerate}
 
 \end{lemma}
 \begin{proof}
 The proof of Lemma
\ref{lem: lat (n,t)} still applies and we only sketch the proof of part (iii). 
 By the bi-$ K_n^{[t]}$-invariance of $ {\bf 1}_{ K_{n+1}^{[t]} }$  and  \eqref{eq: eM} we have 
\begin{align*}
 {\bf 1}_{ K_n^{[n-1]}} \otimes {\bf 1}_{ K_{n+1}^{[t]} }
 \sim & ( e_{ K_n^{[t]}}\ast {\bf 1}_{ K_n^{[n-1]}}\ast  e_{ K_n^{[t]}}) \otimes {\bf 1}_{ K_{n+1}^{[t]} } \\ 
 =& ( \vol( K_n^{[t]})^{-2}  \vol( K_n^{[t,n-1]})^2  \varphi_n^{[t,n-1]} \otimes {\bf 1}_{ K_{n+1}^{[t]} } .
 \end{align*}
 Similarly, by the bi-$ K_n^{[n-1]}$-invariance of $ {\bf 1}_{ K_{n}^{[n-1]} }$  and  \eqref{eq: eM} we have 
\begin{align*}
 {\bf 1}_{ K_n^{[n-1]}} \otimes {\bf 1}_{ K_{n+1}^{[t]} }
 \sim &  {\bf 1}_{ K_n^{[n-1]}} \otimes( e_{ K_n^{[n-1]}}\ast {\bf 1}_{ K_{n+1}^{[t]} }\ast  e_{ K_n^{[n-1]}}) \\ 
 =& ( \vol( K_n^{[n-1]})^{-2}  \vol( K_n^{[t,n-1]})^2   {\bf 1}_{ K_n^{[n-1]}}  \otimes  \varphi^{[n-1,t]}_{n+1}.\qedhere
 \end{align*}
 \end{proof}

   \subsection{Intersection numbers on the splitting model for $\CZ$-divisors}

We continue to assume $0\leq t\leq n-1$. Let $\wh\CN^{[t], \spt}_{n}={\CN}^{[t, n-1], \spt}_n$ be the flat closure of  the base change of $\wh\CN^{[t]}_{n}={\CN}^{[t, n-1]}_n$ along the morphism $\CN_n^{[n-1]}\times\CN_{n+1,\ep}^{[t],\spt}\to \CN_n^{[n-1]}\times\CN_{n+1,\ep}^{[t]}$. Then $\wh\CN^{[t], \spt}_{n}={\CN}^{[t, n-1], \spt}_n$ is a closed formal subscheme of $\CN_n^{[n-1]}\times\CN_{n+1,\ep}^{[t],\spt}$, flat over $\Spf O_{\breve{F}}$ of relative dimension $n-1$.  We have the commutative diagram
  \begin{equation*}
 \begin{aligned}
  \xymatrix{
\wh\CN^{[t], \spt}_{n} \ar@{^(->}[r] \ar[d]  &\CN_n^{[n-1]}\times\CN_{n+1}^{[t],\spt} \ar[d]   \\
\wh\CN^{[t]}_{n}  \ar@{^(->}[r] & \CN^{[n-1]}_{n}\times\CN^{[t]}_{n+1} . 
}
\end{aligned}
\end{equation*}
Now the product $\CN_{n}^{[n-1]} \times \CN_{n+1,\ep}^{[t],\spt}$ is regular since $ \CN_{n}^{[n-1]}$ is formally smooth over $\Spf\OFb$ and $  \CN_{n+1,\ep}^{[t],\spt}$ is regular. 
We form the arithmetic intersection numbers 
\begin{equation}
  \label{eq:Int183}
  \left\langle\wh\CN^{[t], \spt}_{n}, g\wh\CN^{[t], \spt}_{n} \right\rangle_{ \CN_{n}^{[n-1]} \times \CN_{n+1,\ep}^{[t],\spt} }:=\chi(  \CN_{n}^{[n-1]} \times \CN_{n+1,\ep}^{[t],\spt} , \wh\CN^{[t], \spt}_{n} \cap^\BL g\wh\CN^{[t], \spt}_{n}),
\end{equation}
for regular semisimple $g\in G_{W_1}(F_0)$. 

\subsection{The AT conjecture via $\CZ$-divisors}
We now come to the AT conjecture. 
 \begin{conjecture}\label{conj (n odd,t<n) Z}
 Let $n=2m+1$ be odd, and let $t$ be even with $0\leq t\leq n-1$.
 There exists $\varphi'\in C_c^\infty(G')$ with transfer $(\vol( K_n^{[n-1,t]})^{-2} {\bf 1}_{K_n^{[n-1]}\times K^{[t]}_{n+1}},0)\in C_c^\infty(G_{W_0})\times C_c^\infty(G_{W_{1}})$ such that, if $\gamma\in G'(F_0)_\rs$ is matched with  $g\in G_{W_1}(F_0)_\rs$, then
 \begin{equation*}
  \left\langle\wh\CN^{[t], \spt}_{n}, g\wh\CN^{[t], \spt}_{n} \right\rangle_{ \CN_{n}^{[n-1]} \times \CN_{n+1,\ep}^{[t],\spt} }\cdot\log q=- \del\big(\gamma,  \varphi' \big).
\end{equation*}
\end{conjecture}

\begin{remark}
By Lemma \ref{lem: lat (n-1,t)}  part (iii), one could replace the function $\vol( K_n^{[n-1,t]})^{-2}  {\bf 1}_{K_n^{[n-1]}\times K^{[t]}_{n+1}}$ by either of the other two.
\end{remark}

\subsection{The exotic case $t=n+1$} 
\label{ss: exotic t=n+1}
When $t=n+1$, we do not have an aligned triple of type $t+1$. Still, we may formally extrapolate the previous definitions to the case $t=n+1$. For this we fix an aligned triple $(\BY, \BX, u)$ of dimension $n+1$ and type $n-1$. This defines an embedding of RZ spaces $\CN_{n, \ep^\flat}^{[n-1]}\hookrightarrow\CN_{n+1, \ep}^{[n-1]}$. We can form the analogue of the ``big diagram'' \eqref{pi1pi281}. Note, however,  that the space $\CN_{n+1}^{[n+1]}$ is only defined when $\ep(\BX_{n+1}^{[n+1]})=1$, which we assume in this subsection. Define $ \CN_n^{[n-1],\circ}$ by the  cartesian product 
  \begin{equation}\label{eq: N circ}
 \begin{aligned}
 \xymatrix{
 \CN_n^{[n-1],\circ}\ar@{}[rd]|{\square}\ar[d]\ar@{^(->}[r]&\CN_{n+1}^{[n-1,n+1]}\ar[d]&\\
 \CN_n^{[n-1]}\ar@{^(->}[r]&\CN_{n+1}^{[n-1]},&}
 \end{aligned}
 \end{equation}
 comp. Theorem \ref{thm:Y unit}. Then  $ \CN_n^{[n-1],\circ}$ is the disjoint union of two copies of $ \CN_n^{[n-1]}$, cf. \cite[Prop. 6.4]{RSZ2}. It is the analogue of $\wh\CN^{[t]}_{n} $ in this context. Note that  $\CN_{n+1}^{[n+1]}=\CN_{n+1}^{[n+1], \spt}$ (recall that  $\CN_{n+1}^{[n+1]}$ does not contain worst points, cf. \S \ref{sec:geo-RZ}). Hence $ \CN_n^{[n-1],\circ}$ is equal to the splitting cycle in this context. This leads us to consider the intersection number which is the analogue of \eqref{eq:Int183} in this context, 
 \begin{equation}
  \left\langle\CN_n^{[n-1],\circ} , g\CN_n^{[n-1],\circ} \right\rangle_{ \CN_{n}^{[n-1]} \times \CN_{n+1}^{[n+1]} }= \chi(  \CN_{n}^{[n-1]} \times \CN_{n+1}^{[n+1]} , \CN_n^{[n-1],\circ} \cap^\BL g\CN_n^{[n-1],\circ}).
  \end{equation}
 But this intersection number coincides precisely with the one occurring in \cite[\S 12]{RSZ2}, and  the analogue of Conjecture \ref{conj (n odd,t<n) Z}  is identical  to the conjecture in  \cite[\S 12]{RSZ2}. We will encounter the conjecture of \cite{RSZ2} again in the context of $\CY$-divisors (when $t= n+1$ (Remark \ref{rmk:MA-Y-comp})). A closely related conjecture arises  in the context of $\CZ$-divisors (ATC of type $(n-1, n+1 )$ in the sense of \S \ref{ss: case (n odd, n+1)}, comp. \S\ref{sec:MA-Z-comp}.

\subsection{The naive version via the $\CY$-divisor}\label{sec:naive-Y-cycle}
In \eqref{eq:type (n-1,t) big} we used the graph of the embedding exhibiting an exceptional $\CZ$-divisor as an RZ space. We can replace it by the graph of an embedding of an exceptional $\CY$-divisor, namely $ \CN_n^{[n-1],\circ}\incl \CN_{n+1}^{[n+1]}$. We fix an  aligned triple $(\BY, \BX, u)$ of dimension $n+1$ and type $n-1$. We also assume $\ep=1$ so that $\CN_{n+1}^{[n+1]}$ is defined. We have the following correspondence

\begin{equation}\label{eq:type (n-1,t) big Y}
 \begin{aligned}
 \xymatrix{&&\CN^{[n-1],\circ}_n \times \CN_{n+1}^{[n+1,t]}  \ar[rd]   \ar[ld] &
 \\ \CN_n^{[n-1],\circ}\ar@{^(->}[r]& \CN_n^{[n-1],\circ} \times \CN_{n+1}^{[n+1]} &&  \CN_{n}^{[n-1],\circ} \times \CN_{n+1}^{[t]} .
 }
 \end{aligned}
 \end{equation}
Then we are led to define   $ \wt\CM^{[t]}_n$ as the fiber product
 \begin{equation}\label{eq:def Mt}
  \begin{aligned}
  \xymatrix{
   \wt\CM^{[t]}_n \ar@{^(->}[r] \ar@{^(->}[r] \ar[d] \ar@{}[rd]|{\square}  &\CN^{[n+1,t]}_{n+1} \ar[d] 
  \\ \CN^{[n-1],\circ}_n \ar@{^(->}[r] & \CN^{[n+1]}_{n+1}.}
  \end{aligned}
\end{equation}

\begin{conjecture}\label{conjMflat}
$\wt\CM^{[t]}_n$ is flat. 
\end{conjecture}

\begin{lemma}\label{lem:naive-Y-cycle-emb}
The natural map, given by the left vertical arrow in \eqref{eq:def Mt} and the upper horizontal map in \eqref{eq:def Mt} composed with the projection map $\CN^{[n+1,t]}_{n+1}\to \CN^{[n+1]}_{n+1}$,
$$
   \wt\CM^{[t]}_n\rightarrow \CN^{[n-1],\circ}_n\times \CN^{[t]}_{n+1}
$$ 
 is a closed immersion.
 \end{lemma}
 \begin{proof}
By \eqref{eq:type (n-1,t) big Y}, the space $\CN_n^{[n-1],\circ}$ is the closed sublocus of $(Y^{[n-1]},X^{[n+1]})\in \CN_n^{[n-1]}\times \CN_{n+1}^{[n+1]}$ where the quasi-isogeny $(\rho_{Y^{[n-1]}}\times\rho_{\ov{\CE}})^{-1}\circ\alpha_{n+1}\circ\rho_{X^{[n+1]}}$ over the special fiber lifts to an isogeny $X^{[n+1]}\to Y^{[n-1]}\times \ov{\CE}$:
\begin{equation*}
\begin{aligned}
\xymatrix{
X^{[n+1]}\ar[d]_{\rho_{X^{[n+1]}}}\ar@{-->}[r]&Y^{[n-1]}\times \ov{\CE}\ar[d]^{\rho_{Y^{[n-1]}}\times \rho_{\ov{\CE}}}\\
\BX^{[n+1]}\ar[r]^-{\alpha_{n+1}}&\BY^{[n-1]}\times \ov{\BE} .
}
\end{aligned}
\end{equation*}

By \eqref{eq:def Mt}, the space $\wt{\CM}_n^{[t]}$ is the closed sublocus of $(Y^{[n-1]},X^{[n+1]},X^{[t]})\in \CN_n^{[n-1]}\times \CN_{n+1}^{[n+1]}\times \CN_{n+1}^{[t]}$  where the quasi-isogenies $(\rho_{Y^{[n-1]}}\times\rho_{\ov{\CE}})^{-1}\circ\alpha_{n+1}\circ\rho_{X^{[n+1]}}$ and $\rho^{-1}_{X^{[t]}}\circ\alpha_{n-1}\circ\alpha_{n+1}\circ\rho_{X^{[n+1]}}$ over the special fiber lift to isogenies $X^{[n+1]}\to Y^{[n-1]}\times\ov{\CE}$ and $X^{[n+1]}\to X^{[t]}$:
\begin{equation*}
\begin{aligned}
\xymatrix{
X^{[n+1]}\ar[d]_{\rho_{X^{[n+1]}}}\ar@{-->}[r]&Y^{[n-1]}\times \ov{\CE}\ar[d]^{\rho_{Y^{[n-1]}}\times \rho_{\ov{\CE}}}\ar@{-->}[r] &X^{[t]}\ar[d]^{\rho_{X^{[t]}}}\\
\BX^{[n+1]}\ar[r]^-{\alpha_{n+1}}&\BY^{[n-1]}\times \ov{\BE}\ar[r]^-{\alpha_{n-1}}&\BX^{[t]}
}
\end{aligned}
\end{equation*}
 By description, we have closed embedddings
\begin{equation*}
    \wt{\CM}_n^{[t]}\hookrightarrow \CN_n^{[n-1],\circ}\times\CN_{n+1}^{[t]}\hookrightarrow \CN_n^{[n-1]}\times \CN_{n+1}^{[n+1]}\times \CN_{n+1}^{[t]}.
\end{equation*}
 \end{proof}
 
Note that there is a non-trivial involution $\sigma$ acting on $\CN^{[n-1],\circ}_n$, induced by an involution $\wt\sigma$ on $\CN_{n+1}^{[n-1, n+1]}$.  Indeed, $\CN_{n+1}^{[n-1, n+1]}$ is the parameter space of tuples $(X, \iota,\lambda, \rho, X', \iota', \lambda',\phi)$ where $\phi: X'\to X$ lifts the given quasi-isogeny $\BX'\to \BX$. However, as shown in \cite[Thm. 9.3]{RSZ2}, given $(X, \iota,\lambda, \rho)\in \CN_{n+1}^{[n-1]}$, there are exactly two ways to complete it into an object of $\CN_{n+1}^{[n-1, n+1]}$. The involution $\wt\sigma$ by definition interchanges these two possibilities. The involution $\sigma$ commutes with the action of $\U({W_1^\flat})(F_0)$. Using the  involution $(\sigma,1)$ on the product $\CN^{[n-1],\circ}_n\times \CN^{[t]}_{n+1}$, we obtain another cycle $(\sigma,1)   \wt\CM^{[t]}_n$.

 There is a closely related construction. Since $\CN^{[n-1],\circ}_n$ itself may be viewed as built from $\CN^{[n-1]}_n$ via a correspondence \eqref{eq: N circ}, we may  interpret the above as a composition of two correspondences. More precisely, we consider the $\CZ$-divisor embedding $ \CN_n^{[n-1]}\to \CN_{n+1}^{[n-1]}$, and its graph $ \CN_n^{[n-1]}\to \CN_n^{[n-1]}\times \CN_{n+1}^{[n-1]}$. We then apply the composition of the following two correspondences:
 {\small
 \begin{equation}\label{eq:type (n-1,t) big Y alt}
 \begin{aligned}
 \xymatrix{&\CN^{[n-1]}_n \times \CN_{n+1}^{[n-1,n+1]}  \ar[rd]   \ar[ld] &&\CN^{[n-1]}_n \times \CN_{n+1}^{[n+1,t]}  \ar[rd]   \ar[ld] &
 \\ 
 \CN_n^{[n-1]} \times \CN_{n+1}^{[n-1]}  && \CN_n^{[n-1]} \times \CN_{n+1}^{[n+1]} &&  \CN_{n}^{[n-1]} \times \CN_{n+1}^{[t]} .
 }
 \end{aligned}
 \end{equation}
 }
 The resulting cycle is again $\wt\CM^{[t]}_n$ mapping naturally to  $\CN_{n}^{[n-1]} \times \CN_{n+1}^{[t]} $ (not necessarily by an embedding).

 These two versions are related in the following way. Let us denote by the same symbol  the cycle  in $\CN_{n}^{[n-1]} \times \CN_{n+1}^{[t]} $ arising by push-forward of the cycle $ \wt\CM^{[t]}_n$ along the \'etale double covering 
 $$
 \pi: \CN_{n}^{[n-1],\circ} \times \CN_{n+1}^{[t]} \to  \CN_{n}^{[n-1]} \times \CN_{n+1}^{[t]}.
 $$ 
 Note that the morphism $ \pi$ is compatible with the action of $G_{W_1}(F_0)$.
It follows that we can recover the pull-back cycle on $ \CN_{n}^{[n-1],\circ} \times \CN_{n+1}^{[t]}$ with the help of the involution,
 $$
 \pi^{-1}(  \wt\CM^{[t]}_n)= \wt\CM^{[t]}_n\amalg (\sigma,1) \wt\CM^{[t]}_n.
 $$

Moreover, since $\CN_{n}^{[n-1],\circ}\to \CN_{n}^{[n-1]}  $ is a trivial double covering, we may write $\CN_{n}^{[n-1],\circ}=\CN_{n}^{[n-1],+}\coprod \CN_{n}^{[n-1],-}$ as a disjoint union, where each of $\CN_{n}^{[n-1],\pm}$ maps isomorphically to $\CN_{n}^{[n-1]}$. Via the map $\wt\CM^{[t]}_n\to \CN_n^{[n-1],\circ}$, we have an induced decomposition 
$$
\wt\CM^{[t]}_n=\wt\CM^{[t],+}_n\amalg \wt\CM^{[t],-}_n.
$$
Then the natural maps $\wt\CM^{[t],\pm}_n\to  \CN_{n}^{[n-1],\circ} \times \CN_{n+1}^{[t]}$ and $ \wt\CM^{[t],\pm}_n\to  \CN_{n}^{[n-1]} \times \CN_{n+1}^{[t]}$ are both closed immersions. We have on $\CN_{n}^{[n-1],\circ}\times \CN_{n+1}^{[t]}  $,
$$ 
 (\sigma,1)\wt\CM^{[t],\pm}_n=\wt\CM^{[t],\mp}_n
$$ and
$$
\pi^{-1}  (\wt\CM^{[t],\pm}_n)=\wt\CM^{[t],\pm}_n\amalg  (\sigma,1)\wt\CM^{[t],\pm}_n.
$$ 
 Note, however, that the names given to each summand of $\wt\CM^{[t]}_n$ is not canonical. It will turn out that for intersection numbers this non-canonicality plays no role. 
 \subsection{Lattice model for  the $\CY$-divisor}\label{ss: Lat(n-1,t) Y}
We continue from \S\ref{ss: Lat(n-1,t) Z} to
let $\Lambda^\flat_{0}$ be a vertex lattice of type $n-1$ in $W^\flat_0$ with stabilizer $K_n^{[n-1]}$. Let $K_n^{[n-1],\circ}$ be the index two subgroup of  $K_n^{[n-1]}$, being the kernel of the map $\det\!\!\!\mod \pi_0: K_n^{[n-1]}\to \mu_2(k)=\{\pm 1\}$. Fix a special vector $u_0$ of the same unit norm as $u$. Then   $\Lambda_0=\Lambda^\flat_{0}\obot \langle u_0 \rangle \in \Ver^{n-1}(W_0)$. We assume that $W_0$ is the split hermitian space.  There are exactly two vertex lattices of type $n+1$  contained in  $ \Lambda_0$, see \cite[\S 9]{RSZ2} or Lemma \ref{lem: lat (n-1,t) Y} below; we fix one of them, called $\Lambda$. The choice of $\Lambda$ will play no role, see Remark \ref{rem: ind Lbd} below.
We fix a lattice   $\Lambda^{[t]}\in \Ver^{t}(W_0)$ such  that   $\Lambda^{[t]}\supset\Lambda$.   Denote by  $K_{n+1}^{[n+1]}$ (resp. $K_{n+1}^{[t]}$) the stabilizer of $\Lambda$ (resp. $\Lambda^{[t]}$), and by $K_{n+1}^{[n-1]}$ the stabilizer of $\Lambda_0$. Let $K_{n+1}^{[n-1,n+1]}=K_{n+1}^{[n-1]}\cap K_{n+1}^{[n+1]}$ be the 
stabilizer of the lattice chain $\Lambda\subset\Lambda_0 $.

We now consider the lattice models of the RZ spaces. 
We have the following cartesian diagrams: one analogous to \eqref{eq: N circ},
  \begin{equation}\label{}
 \begin{aligned}
 \xymatrix{\BN^{[n-1],\circ}_n\ar@{}[rd]|{\square}\ar[r] \ar[d]&  \BN_{n+1}^{[n-1,n+1]} \ar[d]\\
 \BN_n^{[n-1]}\ar[r]& \BN_{n+1}^{[n-1]} }
 \end{aligned}
 \end{equation}
 and the other one analogous to \eqref{eq:def Mt}
   \begin{equation}\label{}
 \begin{aligned}
 \xymatrix{\wt\BM^{[t]}_n\ar@{}[rd]|{\square}\ar[r] \ar[d]&  \BN_{n+1}^{[n+1,t]} \ar[d]\\
 \BN_n^{[n-1],\circ}\ar[r]& \BN_{n+1}^{[n+1]} }
 \end{aligned}
 \end{equation}
 More explicitly we have
 $$
 \wt\BM^{[t]}_n=\{(\Lambda^\flat,\Lambda,\Lambda^{[t]}) \in {\rm Vert}^{n-1}(W_0^\flat)\times {\rm Vert}^{n+1}(W_0) \times {\rm Vert}^{t}(W_0)\mid \Lambda^\flat \obot \langle u_0 \rangle\supset \Lambda\subset \Lambda^{[t]} \}.
 $$
\begin{remark}\label{notRZ9}In general the set
$ \wt\BM^{[t]}_n$ is not an RZ space, in the sense that the action of $H=G_{W^\flat}$ is not transitive.  To see this, we first note that the action on the set of pairs $(\Lambda^\flat,\Lambda) \in {\rm Vert}^{n-1}(W_0^\flat)\times {\rm Vert}^{n+1}(W_0)$ such that $ \Lambda^\flat \obot \langle u_0 \rangle\supset \Lambda$ is  transitive. Fix such a pair $(\Lambda^\flat,\Lambda)$. Its stabilizer is $K_n^{[n-1],\circ}$,  a subgroup of index two of the stabilizer $K_n^{[n-1]}$ of $\Lambda^\flat$. Then the set of $\Lambda^{[t]}\in  {\rm Vert}^{t}(W_0)$ such that $ \Lambda\subset \Lambda^{[t]} $ is bijective to the set of isotropic subspaces $\BL$ of dimension $\frac{n+1-t}{2}$ in $\BW=\Lambda^\vee/\Lambda$.  Let $\ov u_0$ denote the reduction of $u_0$ in $\BW$.  By Witt's  theorem there are  several orbits under $K_n^{[n-1],\circ}$, characterized as follows: (1) $\ov u_0\in \BL$ (this case does not arise if $t=n+1$), (2) $\ov u_0\notin \BL$ and $\ov u_0\perp \BL$, (3)  $\ov u_0\notin \BL$ and $\ov u_0$ is not perpendicular to $\BL$.
\end{remark}
  
Note that when $t=n+1$ we have $\wt\BM^{[n+1]}_n=\BN^{[n-1],\circ}_n$. We have the lattice model of \eqref{eq:type (n-1,t) big Y}
 \begin{equation*}
 \begin{aligned}
 \xymatrix{&&\BN^{[n-1],\circ}_n \times \BN_{n+1}^{[n+1,t]}  \ar[rd]   \ar[ld] &
 \\ \BN_n^{[n-1],\circ}\ar@{^(->}[r]& \BN_n^{[n-1],\circ} \times \BN_{n+1}^{[n+1]} &&  \BN_{n}^{[n-1],\circ} \times \BN_{n+1}^{[t]} 
 }
 \end{aligned}
 \end{equation*}

 We introduce the bi-$K_{n+1}^{[n+1]}$-invariant Hecke function  
   \begin{equation}\label{varphis}
\varphi^{[n+1,t]}_{n+1}:=\vol(K^{[t]}_{n+1})^{-1} {\bf 1}_{K^{[n+1]}_{n+1}K^{[t]}_{n+1}}\ast {\bf 1}_{K^{[t]}_{n+1}K^{[n+1]}_{n+1}}.
\end{equation}
Then we have
 $$
\Orb(g,\vol(K_n^{[n-1],\circ})^{-2} {\bf 1}_{K_n^{[n-1],\circ}}\otimes \varphi^{[n+1,t]}_{n+1})=\#(\wt\BM^{[t]}_n \cap g \wt\BM^{[t]}_n)_{ \BN_{n}^{[n-1],\circ} \times \BN_{n+1}^{[t] }}.
$$
There is an involution $\sigma: \BN_{n}^{[n-1],\circ}\to \BN_{n}^{[n-1],\circ}$ over $  \BN_{n}^{[n-1]}$ and we can form intersection numbers of two different cycles (interchanged by the involution):
 $$
\Orb(g,\vol(K_n^{[n-1],\circ})^{-2} {\bf 1}_{K_n^{[n-1],\circ} h}\otimes \varphi^{[n+1,t]}_{n+1})=\#((\sigma,1)\wt\BM^{[t]}_n \cap g \wt\BM^{[t]}_n)_{ \BN_{n}^{[n-1],\circ} \times \BN_{n+1}^{[t] }}
$$
where $h$ is any element in $K_n^{[n-1]}\setminus K_n^{[n-1],\circ}$ (see also Remark \ref{rem: ind Lbd}). We may push-forward the cycle $\wt\BM^{[t]}_n$ along the \'etale double covering map $\pi: \BN_{n}^{[n-1],\circ}\to \BN_{n}^{[n-1]}$ down to $\BN_{n}^{[n-1]} \times \BN_{n+1}^{[t]}$, still denoted by $\wt\BM^{[t]}_n$. Then the projection formula shows that the intersection number is then the sum of the two above, hence 
 $$
\Orb(g,\vol(K_n^{[n-1],\circ})^{-2} {\bf 1}_{K_n^{[n-1]}}\otimes \varphi^{[n+1,t]}_{n+1})=\#(\wt\BM^{[t]}_n \cap g \wt\BM^{[t]}_n)_{ \BN_{n}^{[n-1]} \times \BN_{n+1}^{[t] }}.
$$

There is an alternative interpretation of the last one, namely  
via the lattice model of
 \eqref{eq:type (n-1,t) big Y alt},  involving a composition of two correspondences on the larger space:
  \begin{equation*} \begin{aligned}
 \xymatrix{&  \BN_n^{[n-1]}\times \BN_{n+1}^{[n-1,n+1]}\ar[dr] \ar[ld]_{} && \ar[dl]  \BN_n^{[n-1]}\times \BN_{n+1}^{[n+1,t]} \ar[dr]& \\
  \BN_n^{[n-1]}\times  \BN_{n+1}^{[n-1]} && \BN_n^{[n-1]}\times  \BN_{n+1}^{[n+1]}&&  \BN_n^{[n-1]}\times \BN_{n+1}^{[t]}  }
 \end{aligned}
 \end{equation*}
Correspondingly  we define the Hecke function
    \begin{equation*}
\varphi^{[n-1,n+1,t]}_{n+1}:=\vol(K^{[t]}_{n+1})^{-1}\vol(K^{[n+1]}_{n+1})^{-2} {\bf 1}_{K^{[n-1]}_{n+1}K^{[n+1]}_{n+1}}\ast {\bf 1}_{K^{[n+1]}_{n+1}K^{[t]}_{n+1}}\ast {\bf 1}_{K^{[t]}_{n+1}K^{[n+1]}_{n+1}}\ast {\bf 1}_{K^{[n+1]}_{n+1}K^{[n-1]}_{n+1}}
\end{equation*}
and we have
 $$
\Orb(g,\vol(K_n^{[n-1]})^{-2} {\bf 1}_{K_n^{[n-1]}}\otimes \varphi^{[n-1,n+1,t]}_{n+1})=\#(\wt\BM^{[t]}_n \cap g \wt\BM^{[t]}_n)_{ \BN_{n}^{[n-1]} \times \BN_{n+1}^{[t] }}.
$$
As the following lemma part (iii) shows, the two  interpretations are equivalent.

\begin{lemma}
\label{lem: lat (n-1,t) Y} 
 \begin{altenumerate}
 \item We have $K_{n+1}^{[n-1]}/K_{n+1}^{[n-1,n+1]}\simeq \mu_2(k)$ and 
the composition map  
$$
\xymatrix{K_n^{[n-1]} \ar[r]&K_{n+1}^{[n-1]}\ar[r]& K_{n+1}^{[n-1]}/K_{n+1}^{[n-1,n+1]}\simeq \mu_2(k)&  }$$
is surjective with kernel $K_n^{[n-1],\circ}$.

\item
We have
$K_n^{[n-1]} K_{n+1}^{[n+1]}=K_{n+1}^{[n-1]} K_{n+1}^{[n+1]}.$

\item We have 
$$
\vol(K_n^{[n-1],\circ})^{-2} {\bf 1}_{K_n^{[n-1]}}\otimes \varphi^{[n+1,t]}_{n+1}\sim \vol(K_n^{[n-1]})^{-2} {\bf 1}_{K_n^{[n-1]}}\otimes \varphi^{[n-1,n+1,t]}_{n+1}.
$$
\end{altenumerate}
\end{lemma}
\begin{proof}

(i)
The hermitian form induces a non-degenerate split quadratic form on the $2$-dimensional vector space $\Lambda^{\vee}_0/\Lambda_0$ over the residue field $k=O_{F}/(\pi)$. 
Then the reduction induces a surjective homomorphism $K^{[n-1]}_{n+1}\to \mathrm{O}(\Lambda^{\vee}_0/\Lambda_0)=\mathrm{O}(2)(k)$. The two $\pi$-modular lattices $\Lambda^\pm$ correspond to the two isotropic lines and hence their stabilizers are the same and can be identified as $\SO(2)(k)\simeq k^\times$, the kernel of $\det: \mathrm{O}(2)(k)\to\mu_2= \{\pm 1\}$. We summarize these facts in the following commutative diagram
$$
\xymatrix{ & K_{n}^{[n-1]}\ar[d]&&\\ K_{n+1}^{[n-1,n+1]} \ar[d]\ar[r]&K_{n+1}^{[n-1]}\ar[d]\ar[r]& K_{n+1}^{[n-1]}/K_{n+1}^{[n-1,n+1]}\ar[d]& \\
\SO(2)\ar[r]&\mathrm{O}(2)\ar[r]& \mu_2 .& }
$$
Note that  the image of the vector $u$ in  $\Lambda^{\vee}_0/\Lambda_0$ is anisotropic.
The image of $K_{n}^{[n-1]}$ in $\mathrm{O}(2)$, being the stabilizer of this vector, is therefore isomorphic to $\mathrm{O}(1)\simeq\mu_2$, which therefore maps onto the quotient $\mathrm{O}(2)/\SO(2)\simeq \mu_2$. 

(ii)
By (i), $K_{n+1}^{[n-1]}\cap K_{n+1}^{[n+1], +}=K_{n+1}^{[n+1,n-1]}$ is an index two subgroup of $K_{n+1}^{[n-1]}$. There are exactly two cosets in $K_{n+1}^{[n-1]} K_{n+1}^{[n+1], +}/K_{n+1}^{[n+1], +}\simeq K_{n+1}^{[n-1]}/ K_{n+1}^{[n+1,n-1]} $ and since by (i), $K_{n}^{[n-1]}$ maps onto this last quotient, the assertion follows.

(iii)
By (ii) we have $ {\bf 1}_{K^{[n-1]}_{n+1}K^{[n+1]}_{n+1}}= {\bf 1}_{K^{[n-1]}_{n}K^{[n+1]}_{n+1}}$ and hence
    \begin{align*}
&\varphi^{[n-1,n+1,t]}_{n+1}
\\=&\vol(K^{[t]}_{n+1})^{-1}\vol(K^{[n+1]}_{n+1})^{-2} {\bf 1}_{K^{[n-1]}_{n}K^{[n+1]}_{n+1}}\ast {\bf 1}_{K^{[n+1]}_{n+1}K^{[t]}_{n+1}}\ast {\bf 1}_{K^{[t]}_{n+1}K^{[n+1]}_{n+1}}\ast {\bf 1}_{K^{[n+1]}_{n+1}K^{[n-1]}_{n}}\\
=&\vol(K^{[t]}_{n+1})^{-1}\vol(K^{[n+1]}_{n+1})^{-2} \vol(K^{[n-1],\circ}_{n})^{-2}{\bf 1}_{K^{[n-1]}_{n}} \ast{\bf 1}_{K^{[n+1]}_{n+1}}\ast {\bf 1}_{K^{[n+1]}_{n+1}K^{[t]}_{n+1}}\ast {\bf 1}_{K^{[t]}_{n+1}K^{[n+1]}_{n+1}}\ast {\bf 1}_{K^{[n+1]}_{n+1}}\ast {\bf 1}_{K^{[n-1]}_{n}}\\
=&\vol(K^{[n-1],\circ}_{n})^{-2}{\bf 1}_{K^{[n-1]}_{n}} \ast  \varphi^{[n+1,t]}_{n+1} \ast {\bf 1}_{K^{[n-1]}_{n}}\\
=&\vol(K^{[n-1],\circ}_{n})^{-2} \vol(K^{[n-1]}_{n})^{2}(e_{K^{[n-1]}_{n}} \ast  \varphi^{[n+1,t]}_{n+1} \ast e_{K^{[n-1]}_{n}}) .
\end{align*}
By the bi-$ K_n^{[n-1]}$-invariance of $ {\bf 1}_{ K_{n}^{[n-1]} }$  and  \eqref{eq: eM} we have 
\begin{align*}
 {\bf 1}_{K_n^{[n-1]}}\otimes \varphi^{[n+1,t]}_{n+1}
\sim&{\bf 1}_{K_n^{[n-1]}}\otimes (e_{ K_n^{[n-1]}}\ast  \varphi^{[n+1,t]}_{n+1}\ast e_{ K_n^{[n-1]}}) .
\end{align*}
It follows that $\vol(K^{[n-1],\circ}_{n})^{-2}{\bf 1}_{K_n^{[n-1]}}\otimes \varphi^{[n+1,t]}_{n+1}\sim  \vol(K^{[n-1]}_{n})^{-2} {\bf 1}_{K_n^{[n-1]}}\otimes \varphi^{[n-1,n+1,t]}_{n+1}$.
\end{proof}
 
 \begin{remark}
 \label{rem: ind Lbd}
 We comment on the (independence of the) choice of the type $n+1$ lattice $\Lambda$ contained in $\Lambda_0$. Let $\Lambda^\pm$ be the two choices and add the superscript $\pm$ for the various groups in the lemma above. Then 
for any element $h$ in $K_n^{[n-1]}\setminus K_n^{[n-1],\circ}$ we have $\Lambda^-=h\Lambda^+$ and 
$K_{n+1}^{[n+1], +}=h K_{n+1}^{[n+1], -}h^{-1}, K^{[t],-}_{n+1}=h K^{[t],+}_{n+1}h^{-1} $. Then in the function   $\varphi^{[n+1,t]}_{n+1}$ we have
$$ {\bf 1}_{K^{[n+1],-}_{n+1}K^{[t],-}_{n+1}}\ast {\bf 1}_{K^{[t],-}_{n+1}K^{[n+1],-}_{n+1}}={\bf 1}_{hK^{[n+1],+}_{n+1}K^{[t],+}_{n+1}h^{-1}}\ast {\bf 1}_{hK^{[t],+}_{n+1}K^{[n+1],+}_{n+1}h^{-1}}
$$ 
and hence $\varphi^{[n+1,t],-}_{n+1}= \varphi^{[n+1,t],+}_{n+1}\circ \Ad(h)$ where $\Ad(h) $ denotes the conjugation by $h$. Then we have
$$
{\bf1}_{K_n^{[n-1],\circ}}\otimes \varphi^{[n+1,t],-}_{n+1}\sim ({\bf1}_{K_n^{[n-1],\circ}}\circ  \Ad(h^{-1}))\otimes \varphi^{[n+1,t],+}_{n+1}.
$$Since $ K_n^{[n-1],\circ}$ is a normal subgroup of $K_n^{[n-1]}$, we have ${\bf1}_{K_n^{[n-1],\circ}}\circ  \Ad(h^{-1})={\bf1}_{K_n^{[n-1],\circ}}$.

 \end{remark}

  \subsection{Intersection numbers on the splitting model for $\CY$-divisors} \label{sec:n-1,t Y-cycle}
  Recall that we defined  $\wt{\CM}_n^{[t]}$ in \eqref{eq:def Mt} and that we have a closed immersion 
  $$
   \wt\CM^{[t]}_n\hookrightarrow \CN^{[n-1],\circ}_n\times \CN^{[t]}_{n+1}.
$$
 Let $\wt\CM^{[t], \spt}_{n}$ be the flat closure of  the base change of $\wt\CM^{[t]}_{n}$ along the morphism $\CN_n^{[n-1],\circ}\times\CN_{n+1}^{[t],\spt}\to \CN_n^{[n-1],\circ}\times\CN_{n+1}^{[t]}$. Then $\wt\CM^{[t], \spt}_{n}$ is a closed formal subscheme of $\CN_n^{[n-1],\circ}\times\CN_{n+1}^{[t],\spt}$, flat over $\Spf O_{\breve{F}}$ of relative dimension $n-1$.  If  Conjecture \ref{conjMflat} on the flatness of $ \wt\CM^{[t]}_n$ holds, then $\wt\CM^{[t], \spt}_{n}$ coincides with $\wt\CM^{[t]}_{n}$ outside the worst points. We have the commutative diagram
  \begin{equation*}
 \begin{aligned}
  \xymatrix{
\wt\CM^{[t], \spt}_{n} \ar@{^(->}[r] \ar[d]  &\CN_n^{[n-1],\circ}\times\CN_{n+1}^{[t],\spt} \ar[d]   \\
\wt\CM^{[t]}_{n}  \ar@{^(->}[r] & \CN^{[n-1],\circ}_{n}\times\CN^{[t]}_{n+1}  .
}
\end{aligned}
\end{equation*}
Now the product  $\CN_n^{[n-1],\circ}\times\CN_{n+1}^{[t],\spt}$ is regular since $ \CN_{n}^{[n-1],\circ}$ is formally smooth over $\Spf\OFb$ and $  \CN_{n+1}^{[t],\spt}$ is regular. 
 We define the arithmetic intersection numbers 
\begin{equation*}
  \left\langle\wt\CM^{[t],\spt}_{n} , g\wt\CM^{[t],\spt}_{n} \right\rangle_{ \CN_{n}^{[n-1],\circ} \times \CN_{n+1}^{[t],\spt} }:=\chi(  \CN_{n}^{[n-1]} \times \CN_{n+1}^{[t],\spt} , \wt\CM^{[t],\spt}_{n}  \cap^\BL g\wt\CM^{[t],\spt}_{n}  ),
\end{equation*}
for regular semisimple $g\in G_{W_1}(F_0)$.

\subsection{Refinement in terms of $\wt\CM^{[t],\pm}_{n}$}
\label{ss: refine}
 We have a decomposition
\begin{equation*}
\wt{\CM}_n^{[t]}=\wt{\CM}_n^{[t],+}\amalg \wt{\CM}_n^{[t],-}
\end{equation*}
and, correspondingly, the splitting version
\begin{equation*}
\wt{\CM}_n^{[t],\spt}=\wt{\CM}_n^{[t],+,\spt}\amalg \wt{\CM}_n^{[t],-,\spt} .
\end{equation*}
 We define the arithmetic intersection numbers 
\begin{equation}
  \label{eq:Int1}
  \left\langle\wt\CM^{[t],\pm,\spt}_{n} , g\wt\CM^{[t],\pm,\spt}_{n} \right\rangle_{ \CN_{n}^{[n-1]} \times \CN_{n+1}^{[t],\spt} }:=\chi(  \CN_{n}^{[n-1]} \times \CN_{n+1}^{[t],\spt} , \wt\CM^{[t],\pm,\spt}_{n}  \cap^\BL g\wt\CM^{[t],\pm,\spt}_{n}  ),
\end{equation}
for regular semisimple $g\in G_{W_1}(F_0)$.

\begin{lemma}\label{lem (n odd,t) Y}
There are the following equalities of  intersection numbers:  
  \begin{altenumerate}
  \item
$$\left\langle\wt\CM^{[t],+,\spt}_{n} , g\wt\CM^{[t],+,\spt}_{n} \right\rangle_{ \CN_{n}^{[n-1]} \times \CN_{n+1}^{[t],\spt} }=\left\langle\wt\CM^{[t],-,\spt}_{n} , g\wt\CM^{[t],-,\spt}_{n} \right\rangle_{ \CN_{n}^{[n-1]} \times \CN_{n+1}^{[t],\spt} }.
$$
\item 
$$\left\langle\wt\CM^{[t],+,\spt}_{n} , g\wt\CM^{[t],-,\spt}_{n} \right\rangle_{ \CN_{n}^{[n-1]} \times \CN_{n+1}^{[t],\spt} }=\left\langle\wt\CM^{[t],-,\spt}_{n} , g\wt\CM^{[t],+,\spt}_{n} \right\rangle_{ \CN_{n}^{[n-1]} \times \CN_{n+1}^{[t],\spt} }.
$$
\item  $$
  \left\langle\wt\CM^{[t],\spt}_{n} , g\wt\CM^{[t],\spt}_{n} \right\rangle_{ \CN_{n}^{[n-1],\circ} \times \CN_{n+1}^{[t],\spt}}=2\left\langle\wt\CM^{[t],+,\spt}_{n} , g\wt\CM^{[t],+,\spt}_{n} \right\rangle_{ \CN_{n}^{[n-1]} \times \CN_{n+1}^{[t],\spt} }.
 $$
 \item
 $$
  \left\langle(\sigma,1)\wt\CM^{[t],\spt}_{n} , g\wt\CM^{[t],\spt}_{n} \right\rangle_{ \CN_{n}^{[n-1],\circ} \times \CN_{n+1}^{[t],\spt}}=2\left\langle\wt\CM^{[t],+,\spt}_{n} , g\wt\CM^{[t],-,\spt}_{n} \right\rangle_{ \CN_{n}^{[n-1]} \times \CN_{n+1}^{[t],\spt} }.
 $$
\end{altenumerate}
 \emph{Here $\wt\CM^{[t],+,\spt}_{n}$ means both the cycle on $ \CN_{n}^{[n-1],\circ} \times \CN_{n+1}^{[t],\spt}$ and on $ \CN_{n}^{[n-1]} \times \CN_{n+1}^{[t],\spt}$.}
\end{lemma}
\begin{proof}
Note that $\wt\CM^{[t],+,\spt}_{n}=\pi_\ast \wt\CM^{[t],+,\spt}_{n}$ along the \'etale double covering map $\pi:  \CN_{n}^{[n-1],\circ} \times \CN_{n+1}^{[t],\spt} \to  \CN_{n}^{[n-1]} \times \CN_{n+1}^{[t],\spt} $. By the projection formula and the fact that $\pi$ commutes with the action of $g\in G_{W_1}(F_0)$, we have
\begin{align*}
&\left\langle\wt\CM^{[t],+,\spt}_{n} , g\wt\CM^{[t],+,\spt}_{n} \right\rangle_{ \CN_{n}^{[n-1]} \times \CN_{n+1}^{[t],\spt} }\\
=&\left\langle\wt\CM^{[t],+,\spt}_{n} ,\pi^\ast( g\wt \CM^{[t],+,\spt}_{n}) \right\rangle_{ \CN_{n}^{[n-1],\circ} \times \CN_{n+1}^{[t],\spt} }\\
=&\left\langle\wt\CM^{[t],+,\spt}_{n} , g \pi^\ast \wt \CM^{[t],+,\spt}_{n} \right\rangle_{ \CN_{n}^{[n-1],\circ} \times \CN_{n+1}^{[t],\spt} }\\
=&\left\langle\wt\CM^{[t],+,\spt}_{n} , g \wt \CM^{[t],+,\spt}_{n} \right\rangle_{ \CN_{n}^{[n-1],\circ} \times \CN_{n+1}^{[t],\spt} }+\left\langle\wt\CM^{[t],+,\spt}_{n} , g (\sigma,1) \wt \CM^{[t],+,\spt}_{n} \right\rangle_{ \CN_{n}^{[n-1],\circ} \times \CN_{n+1}^{[t],\spt} }.
\end{align*}

Similarly, we have 
\begin{align*}
&\left\langle\wt\CM^{[t],-,\spt}_{n} , g\wt\CM^{[t],-,\spt}_{n} \right\rangle_{ \CN_{n}^{[n-1]} \times \CN_{n+1}^{[t],\spt} }\\
=&\left\langle\wt\CM^{[t],-,\spt}_{n} , g \wt \CM^{[t],-,\spt}_{n} \right\rangle_{ \CN_{n}^{[n-1],\circ} \times \CN_{n+1}^{[t],\spt} }+\left\langle\wt\CM^{[t],-,\spt}_{n} , g (\sigma,1) \wt \CM^{[t],-,\spt}_{n} \right\rangle_{ \CN_{n}^{[n-1],\circ} \times \CN_{n+1}^{[t],\spt} }.
\end{align*}

Finally we have 
\begin{align*}&\left\langle\wt\CM^{[t],+,\spt}_{n} , g \wt \CM^{[t],+,\spt}_{n} \right\rangle_{ \CN_{n}^{[n-1],\circ} \times \CN_{n+1}^{[t],\spt} }\\
=&\left\langle (\sigma,1)\wt\CM^{[t],-,\spt}_{n} , g  (\sigma,1) \wt \CM^{[t],-,\spt}_{n} \right\rangle_{ \CN_{n}^{[n-1],\circ} \times \CN_{n+1}^{[t],\spt} }\\
=&\left\langle\wt\CM^{[t],-,\spt}_{n} , g \wt \CM^{[t],-,\spt}_{n} \right\rangle_{ \CN_{n}^{[n-1],\circ} \times \CN_{n+1}^{[t],\spt} } ,
\end{align*}
since $ (\sigma,1)g  (\sigma,1)=g  (\sigma,1)^2=g$. Similarly, 
$$
\left\langle\wt\CM^{[t],+,\spt}_{n} , g (\sigma,1) \wt \CM^{[t],-,\spt}_{n} \right\rangle_{ \CN_{n}^{[n-1],\circ} \times \CN_{n+1}^{[t],\spt} }=\left\langle\wt\CM^{[t],-,\spt}_{n} , g (\sigma,1) \wt \CM^{[t],+,\spt}_{n} \right\rangle_{ \CN_{n}^{[n-1],\circ} \times \CN_{n+1}^{[t],\spt} }.
$$
This proves (i) and (ii). The remaining equalities are proved similarly.
\end{proof}

\begin{remark}
The identities (i) and (ii) in Lemma \ref{lem (n odd,t) Y} show that for the intersection numbers, the labeling of the two summands of $\wt\CM^{[t]}_{n} $ (which is not canonical) is unimportant: if  the meaning of $+$ and $-$ are switched, the intersection numbers are unchanged.
\end{remark}

\subsection{AT conjecture: the case of the $\CY$-divisor}
We now come to the AT conjecture.  

 \begin{conjecture}\label{conj (n odd,t) Y}
 Let $n=2m+1$ be odd, and let $t$ be even with $0\leq t\leq n+1$. Recall from \eqref{varphis} the function $\varphi^{[n+1,t]}_{n+1}\in C_c^\infty(G_{W_{1}})$. 
 
  \begin{altenumerate}
  \item There exists $\varphi'\in C_c^\infty(G')$ with transfer $( \vol(K_n^{[n-1],\circ})^{-2} {\bf 1}_{K_n^{[n-1]}}\otimes \varphi^{[n+1,t]}_{n+1},0)\in C_c^\infty(G_{W_0})\times C_c^\infty(G_{W_{1}})$ such that, if $\gamma\in G'(F_0)_\rs$ is matched with  $g\in G_{W_1}(F_0)_\rs$, then
 \begin{equation*}
  \left\langle\wt\CM^{[t],\spt}_{n} , g\wt\CM^{[t],\spt}_{n} \right\rangle_{ \CN_{n}^{[n-1]} \times \CN_{n+1}^{[t],\spt} }\cdot\log q=- \del\big(\gamma,  \varphi' \big).
\end{equation*}
 \item 
There exists $\varphi'\in C_c^\infty(G')$ with transfer $(\vol(K_n^{[n-1],\circ})^{-2} {\bf 1}_{K_n^{[n-1],\circ}}\otimes \varphi^{[n+1,t]}_{n+1},0)\in C_c^\infty(G_{W_0})\times C_c^\infty(G_{W_{1}})$ such that, if $\gamma\in G'(F_0)_\rs$ is matched with  $g\in G_{W_1}(F_0)_\rs$, then
 \begin{equation*}
2 \left\langle\wt\CM^{[t],+,\spt}_{n} , g\wt\CM^{[t],+,\spt}_{n} \right\rangle_{ \CN_{n}^{[n-1]} \times \CN_{n+1}^{[t],\spt} }\cdot\log q=- \del\big(\gamma,  \varphi' \big).
\end{equation*}
Similarly there exists $\varphi'\in C_c^\infty(G')$ with transfer $(\vol(K_n^{[n-1],\circ})^{-2} {\bf 1}_{K_n^{[n-1],\circ}h}\otimes \varphi^{[n+1,t]}_{n+1},0)\in C_c^\infty(G_{W_0})\times C_c^\infty(G_{W_{1}})$ such that, if $\gamma\in G'(F_0)_\rs$ is matched with  $g\in G_{W_1}(F_0)_\rs$, then
 \begin{equation*}
2\left\langle\wt\CM^{[t],+,\spt}_{n} , g\wt\CM^{[t],-,\spt}_{n} \right\rangle_{ \CN_{n}^{[n-1]} \times \CN_{n+1}^{[t],\spt} }\cdot\log q=- \del\big(\gamma,  \varphi' \big).
\end{equation*}
Here $h\in \U(W_0)(F_0)$ is an element in $K_{n}^{[n-1]}\setminus K_{n}^{[n-1],\circ}$. 
\end{altenumerate}
\end{conjecture}

\begin{remark}
In (i), by Lemma \ref{lem: lat (n-1,t) Y}  part (iii), one could replace the function $\vol(K_n^{[n-1],\circ})^{-2} {\bf 1}_{K_n^{[n-1]}}\otimes \varphi^{[n+1,t]}_{n+1}$ by $ \vol(K_n^{[n-1]})^{-2} {\bf 1}_{K_n^{[n-1]}}\otimes \varphi^{[n-1,n+1,t]}_{n+1}$. 
\end{remark}

\begin{remark}\label{rmk:MA-Y-comp}
When $t=n+1$, we have identifications  $\wt\CM^{[t],-,\spt}_n= \wt\CM^{[t],+,\spt}_n = \CN^{[n-1]}_n$ and an identification
\begin{equation*}
	\CN_n^{[n-1]}\times \CN_{n+1}^{[n+1],\spt}=\CN_n^{[n-1]}\times \CN_{n+1}^{[n+1]}.
\end{equation*}
Therefore, thanks to Lemma \ref{lem (n odd,t) Y} and \cite[Rem. 12.5]{RSZ2},  Conjecture \ref{conj (n odd,t) Y}(i) is equivalent to \cite[Conj. 12.4]{RSZ2}. On the other hand, (ii) and (iii) are refinements of  (i). Compare also  \S\ref{ss: exotic t=n+1}.
\end{remark}

\begin{remark}\label{rmk:(n-1,t)-Y-smaller}
For $t\leq n-1$, we can also consider the small correspondence for $\CY$-cycles (see also \S \ref{ss: naive (n odd, n+1)}). Let $$ \CN^{[n-1,t],\circ}_n:=\CN^{[n-1,t]}_n\times_{\CN^{[n-1]}_n}\CN_n^{[n-1],\circ}.$$ 
Then we may consider the following correspondence:
 \begin{equation*}
	\begin{aligned}
		\xymatrix{&&\CN^{[n-1,t],\circ}_n \times \CN_{n+1}^{[t]}  \ar[rd]   \ar[ld] &
			\\ \CN_n^{[t]}\ar@{^(->}[r]& \CN_n^{[t]} \times \CN_{n+1}^{[t]} &&  \CN_{n}^{[n-1],\circ} \times \CN_{n+1}^{[t]} .
		}
	\end{aligned}
\end{equation*}
Then $\CN^{[n-1,t],\circ}_n$ is a closed formal subscheme of $\CN_n^{[n-1],\circ}\times \CN^{[t]}_{n+1}$.
Recall from Lemma \ref{lem:naive-Y-cycle-emb} that $\wt{\CM}_n^{[t]}$ is also a  closed formal subscheme of $\CN_n^{[n-1],\circ}\times \CN^{[t]}_{n+1}$.
We claim that there is a natural closed embedding 
\begin{equation}\label{relNcirctilde}
	\CN^{[n-1,t],\circ}_n\hookrightarrow 
	\wt\CM^{[t]}_n.
\end{equation}
Indeed,  we have a cartesian diagram by Theorem \ref{th: wt N=N(t,n-1)},
\begin{equation*}
	\begin{aligned}
		\xymatrix{
			\CN^{[n-1,t]}_n \ar@{^(->}[r] \ar@{^(->}[r] \ar[d] \ar@{}[rd]|{\square}  &\CN^{[n-1,t]}_{n+1} \ar[d] 
			\\ \CN^{[n-1]}_n \ar@{^(->}[r] & \CN^{[n-1]}_{n+1} ,}
	\end{aligned}
\end{equation*}
and the base change $\CN^{[n-1,t]}_n\times_{\CN^{[n-1]}_n} \CN_n^{[n-1],\circ}$ is isomorphic to the fiber product 
\begin{equation*}
	\begin{aligned}
		\xymatrix{
			\CN^{[n-1,t],\circ}_n \ar@{^(->}[r] \ar[d] \ar@{}[rd]|{\square}  &\CN^{[n+1,n-1,t]}_{n+1} \ar[d]   \\
			\CN^{[n-1,t]}_n \ar@{^(->}[r]   &\CN^{[n-1,t]}_{n+1}  .
		}
	\end{aligned}
\end{equation*}
Hence we obtain a natural morphism  $\CN^{[n-1,t]}_n\times_{\CN^{[n-1]}_n} \CN_n^{[n-1],\circ}\simeq  \CN^{[n-1,t],\circ}_n  \to \CN^{[n+1,t]}_{n+1}$. We also have a natural morphism $\CN^{[n-1,t]}_n\times_{\CN^{[n-1]}_n} \CN_n^{[n-1],\circ}\to \CN_n^{[n-1],\circ}$, and it is easy to verify that the resulting diagram 
\begin{equation*}
	\begin{aligned}
		\xymatrix{
			\CN^{[n-1,t],\circ}_n \ar@{^(->}[r] \ar[d] \ar@{}[rd]|{\square}  & \CN^{[n+1,t]}_{n+1}\ar[d]   \\
			\CN^{[n-1],\circ}_n \ar@{^(->}[r]   &\CN^{[n+1]}_{n+1}  
		}
	\end{aligned}
\end{equation*}
is cartesian. Therefore  by the definition \eqref{eq:def Mt} of $\wt\CM^{[t]}_n$, we obtain the map \eqref{relNcirctilde}.

We do not expect \eqref{relNcirctilde} to be an isomorphism. This is based on the computation on the corresponding lattice models. Indeed, using \eqref{eq: N circ}, we have
\begin{equation*}
\BN_n^{[n-1],\circ}=\{
(\Lambda^\flat_{n-1},\Lambda_{n+1})\in \mathrm{Vert}^{n-1}(W^\flat)\times \mathrm{Vert}^{n+1}(W)\mid \Lambda_{n+1}\subset \Lambda_{n-1}^\flat\oplus\langle u\rangle
\}.
\end{equation*}
Using \eqref{eq:def Mt}, we have
$$
\wt{\BM}_n^{[t]}=\{
(\Lambda_t,\Lambda_{n-1}^\flat,\Lambda_{n+1})\in \mathrm{Vert}^t(W)\times\mathrm{Vert}^{n-1}(W^\flat)\times \mathrm{Vert}^{n+1}(W)\mid \Lambda_{n+1}\subset\Lambda_{n-1}^\flat\oplus\langle u\rangle,\Lambda_{n+1}\subset\Lambda_t
\}.
$$
On the other hand, we have (see also \eqref{eq:lattice-t n-1 circ})
\begin{equation*}
\BN_n^{[n-1,t],\circ}=\{
(\Lambda^\flat_t,\Lambda_{n-1}^\flat,\Lambda_{n+1}\in \mathrm{Vert}^t(W^\flat)\times \mathrm{Vert}^{n-1}(W^\flat)\times \mathrm{Vert}^{n+1}(W)\mid \Lambda_{n-1}^\flat\subset \Lambda_t^\flat,\Lambda_{n+1}\subset \Lambda_{n-1}^\flat\oplus\langle u\rangle
\}.
\end{equation*}
One immediately sees that $\BN_n^{[n-1,t],\circ}\subset \wt{\BM}_n^{[t]}$ is the proper subset consisting of triples $(\Lambda_t,\Lambda_{n-1}^\flat,\Lambda_{n+1})$ such that $\Lambda_{n+1}\subset \Lambda_{n-1}^\flat \oplus\langle u\rangle\subset \Lambda_t$.

Therefore, we do not expect that the AT conjecture for the smaller correspondence is equivalent to the AT conjecture for the larger correspondence in Conjecture \ref{conj (n odd,t) Y}. On the other hand, the AT conjecture for the smaller correspondence for the $\CY$-cycle is equivalent to Conjecture \ref{conj (n odd,t<n) Z}.
This is straightforward from the following cartesian diagram:
\begin{equation*}
\begin{aligned}
\xymatrix{
\CN_n^{[n-1,t],\circ}\ar@{^(->}[r]\ar[d]^-{p}\ar@{}[rd]|{\square}&\CN_n^{[n-1],\circ}\times \CN_{n+1}^{[t]}\ar[d]\\
\CN_n^{[n-1,t]}\ar@{^(->}[r]&\CN_n^{[n-1]}\times \CN_{n+1}^{[t]} .
}
\end{aligned}
\end{equation*}
Indeed, using the projection formula, we have
\begin{align*}
\Bigl\langle \CN_n^{[n-1,t],\circ}, g\CN_n^{[n-1,t],\circ}\Bigr\rangle_{\CN_n^{[n-1],\circ}\times \CN_{n+1}^{[t]}}={}&\Bigl\langle \CN_n^{[n-1,t],\circ}, p^*(g\CN_n^{[n-1,t]})\Bigr\rangle_{\CN_n^{[n-1],\circ}\times \CN_{n+1}^{[t]}},\\
={}&\Bigl\langle p_*(\CN_n^{[n-1,t],\circ}), g\CN_n^{[n-1,t]}\Bigr\rangle_{\CN_n^{[n-1]}\times \CN_{n+1}^{[t]}},\\
={}&2\Bigl\langle \CN_n^{[n-1,t]}, g\CN_n^{[n-1,t]}\Bigr\rangle_{\CN_n^{[n-1]}\times \CN_{n+1}^{[t]}}.
\end{align*}
On the analytic side, we replace the transfer function  $(\vol( K_n^{[n-1,t]})^{-2}{\bf 1}_{K_n^{[n-1]}\times K^{[t]}_{n+1}},0)\in C_c^\infty(G_{W_0})\times C_c^\infty(G_{W_{1}})$ by $(\vol( K_n^{[n-1,t],\circ})^{-2}{\bf 1}_{K_n^{[n-1],\circ}\times K^{[t]}_{n+1}},0)$. One can use  \eqref{eq: eM} to show that  the orbit integral of the latter function is twice that of the former.

\end{remark}

  \section{AT conjecture of type $(t, n)$}\label{sec:type-(t,n)}

 Let $n=2m$ be even.  
Let $0\leq t\leq n$ be even. We use the exceptional special divisor $\CZ(u)^{[t], \spt}$ on $\CN_{n+1}^{[t],\spt}$,  for a unit length vector $u\in \BV_{n+1}$. This can be identified with $  \CN^{[t], \spt}_{n}$, cf. Theorem \ref{thm: Z unit}. The intersection product takes place on $\CN_{n}^{[t],\spt} \times \CN_{n+1}^{[n]}$.

 We fix an aligned triple of framing objects $(\BY, \BX, u)=(\BY_{\ep^\flat}^{[t]}, \BX_\ep^{[t]}, u)$ of dimension $n+1$ and type $t$, with corresponding embedding of RZ spaces $\CN^{[t]}_{n, \ep^\flat}=\CN^{[t]}_n\hookrightarrow \CN^{[t]}_{n+1}=\CN^{[t]}_{n+1, \ep}$, cf. \S \ref{sec:emb-RZ}. When $t=n$, we impose that $\ep^\flat=1$, since otherwise $\CN^{[t]}_n$ is not defined.  As in \S \ref{sec:type-(n,t)}, we have a small correspondence and a big correspondence. For the small correspondence, we  fix an isogeny $\BY^{[n]}\to \BY^{[t]}$ as in \S \ref{sec:RZ-deeper}, and extend this in the obvious way to an isogeny $\BX^{[n]}\to \BX^{[t]}$.   Again, this is only possible when $\ep^\flat=1$. This defines additional embeddings of RZ spaces  $\CN^{[n]}_n\hookrightarrow \CN^{[n]}_{n+1}$ and  $\CN^{[n,t]}_n\hookrightarrow \CN^{[n,t]}_{n+1}$, cf. \S \ref{sec:emb-RZ}. For the big correspondence, we fix an isogeny $\BX^{[n]}\to \BX^{[t]}$, as in \S \ref{sec:RZ-deeper}, which allows us to write $\CN^{[n, t]}_{n+1}$. 
 
 We again use the notation $W_1=\BV(\BX)$ and $W_1^\flat=\BV(\BY)$, so that $u\in W_1$ and $W_1^\flat=\langle u\rangle^\perp$. We denote by $W_0$ the hermitian space of dimension $n+1$ with opposite Hasse invariant of $W_1$, and fix a vector $u_0\in W_0$ of the same length as $u$ and set $W_0^\flat  =\langle u_0\rangle^\perp$. Then $W_0^\flat $ has the opposite Hasse invariant of $W_1^\flat$.

\subsection{The naive version}\label{sec:(t,n)-naive}

Similar to \S \ref{sec:type-(n,t)}, there are two ways to use correspondences to obtain natural cycles on the ambient space  $\CN_{n,\ep}^{[t]} \times \CN_{n+1}^{[n]}$. The first one is to use a correspondence $\CN^{[n,t]}_n $ on the smaller space
 \begin{equation*}
 \begin{aligned}
 \xymatrix{&&\CN^{[n,t]}_{n} \times \CN_{n+1}^{[n]}  \ar[rd]   \ar[ld] &
 \\ \CN_{n}^{[n]}\ar@{^(->}[r]& \CN_n^{[n]} \times \CN_{n+1}^{[n]} &&  \CN_{n}^{[t]} \times \CN_{n+1}^{[n]} .
 }
 \end{aligned}
 \end{equation*}
  Recall that here $\ep^\flat=1$.
 The second one  is to use a correspondence $\CN^{[n,t]}_{n+1}$  on the bigger space
 \begin{equation*}
 \begin{aligned}
 \xymatrix{&&\CN_{n}^{[t]} \times \CN_{n+1}^{[n,t]}  \ar[rd]   \ar[ld] &
 \\ \CN_{n}^{[t]}\ar@{^(->}[r]& \CN_{n}^{[t]} \times \CN_{n+1}^{[t]} &&  \CN_{n}^{[t]} \times \CN_{n+1}^{[n]} .
 }
 \end{aligned}
 \end{equation*}
Here is the resulting small correspondence (for which we assume $\ep^\flat=1$):
 \begin{equation*}
 \begin{aligned}
 \xymatrix{&\CN^{[n,t]}_n \ar[rd]   \ar[ld]_{\pi_1} &&
 \\ \CN_n^{[t]} &&  \CN_{n}^{[n]}\ar@{^(->}[r]&\CN_{n+1}^{[n]}.
 }
 \end{aligned}
 \end{equation*}
 Here is the resulting big correspondence:

  \begin{equation*}
 \begin{aligned}
 \xymatrix{
 &\wt{\CN}_n^{[t]}\ar@{}[d]|{\square}\ar[dl]_{\pi_1}\ar@{^(->}[r]\ar[rrd]^{\pi_2}
 &\CN_{n+1}^{[n,t]}\ar[rd]\ar[ld]|!{[l];[rd]}\hole&\\
 \CN_n^{[t]}\ar@{^(->}[r]&\CN_{n+1}^{[t]}&&\CN_{n+1}^{[n]}.}
 \end{aligned}
 \end{equation*}
 There is a natural morphism 
 $$\iota\colon\CN^{[n,t]}_n\to  \wt{\CN}_n^{[t]}.
 $$
 
 \begin{lemma}\label{lem:(t,n)-naive-emb}
 Both morphisms
 $$
\iota\colon  \CN^{[n,t]}_n\to  \wt{\CN}_n^{[t]} , \quad (\pi_1, \pi_2)\colon \wt{\CN}_n^{[t]}\to \CN_n^{[t]}\times\CN_{n+1}^{[n]}
 $$
 are closed embeddings.
 \end{lemma}
 \begin{proof}
We first show that $(\pi_1,\pi_2)$ is a closed embedding. Indeed, $\wt{\CN}_n^{[t]}$ is the closed formal subscheme of $\CN_n^{[t]}\times \CN_{n+1}^{[n]}$ parameterizing pairs $(Y^{[t]}, X^{[n]})$ such that the quasi-isogeny 
$(\rho_{Y^{[t]}}\times \rho_{\ov{\CE}})^{-1}\circ\alpha \circ \rho_{X^{[t]}}$ over the special fiber lifts to an isogeny $X^{[n]}\to Y^{[t]}\times \ov{\CE}$, where $\alpha:\BX^{[n]}\to \BY^{[t]}\times \ov{\BE}$ is the isogeny between the framing objects.
Since $\CN_n^{[n,t]}$ is also a closed subscheme of $\CN_n^{[t]}\times \CN_{n+1}^{[n]}$, we conclude that $\iota$ is a closed embedding.
 \end{proof}

 \begin{conjecture}\label{conj-wtNfl}
 The space $\wt{\CN}_n^{[t]}$ is flat. 
  \end{conjecture}

 \begin{remark}\label{rem: t=n even wtN}
 Note that when $t=n$ (and hence $\ep^\flat=1$), the cartesian diagram defining $ \wt{\CN}_n^{[t]}$  shows that the morphism $\CN^{[n]}_n=\CN^{[n,t]}_n\to  \wt{\CN}_n^{[t]}=\wt{\CN}_n^{[n]}$ is an isomorphism. In general we expect $\CN^{[n,t]}_n$ and $  \wt{\CN}_n^{[t]}$ to be non-isomorphic. 
 \end{remark}
 
 \begin{proposition}
 There is a cartesian diagram, where in the bottom line appear  special cycles on $\CN_{n+1}^{[n]}$, 
  \begin{equation}
  \begin{aligned}
  \xymatrix{\CN^{[n,t]}_n \ar@{^(->}[r] \ar[d] \ar@{}[rd]|{\square}  & \wt{\CN}^{[t]}_n\ar[d] 
  \\ \CZ(u)^{[n]} \ar@{^(->}[r] &\CY(u)^{[n]}.}
  \end{aligned}
\end{equation}

\end{proposition}
\begin{proof}
Let $R$ be an algebra over $\Spf O_{\breve{F}}$ and let $(X^{[n]},Y^{[t]})\in \wt{\CN}_n^{[t]}(R)$ be any $R$-point. By definition, we have an isogeny 
$X^{[n]}\to Y^{[t]}\times \ov{\CE}$ lifting the quasi-isogeny between the framing objects. The composition 
$$\ov{\CE}\longrightarrow Y^{[t]}\times\ov{\CE}\longrightarrow Y^{[t],\vee}\times\ov{\CE}\longrightarrow X^{[n],\vee}$$
defines a lifting of the quasi-isogeny $u:\ov{\BE}\to \BX^{[n],\vee}$. Therefore, the composition $\wt{\CN}_n^{[t]}\hookrightarrow \CN_{n+1}^{[n,t]}\rightarrow \CN_{n+1}^{[n]}$ factors through $\CY(u)^{[n]}$.
Consider the following commutative diagram,
\begin{equation*}
\begin{aligned}
\xymatrix{
\CN_n^{[n,t]}\ar@{^(->}[r]\ar@{_(->}[d]&\wt{\CN}_n^{[t]}\ar@{_(->}[d]&\\
\CN_n^{[n,t]}\ar@{^(->}[r]\ar[d]&\CY(u)^{[n]}\times_{\CN_{n+1}^{[n]}}\CN_{n+1}^{[n,t]}\ar[d]\ar@{^(->}[r]\ar@{}[rd]|{\square}&\CN_{n+1}^{[n,t]}\ar[d]\\
\CZ(u)^{[n]}\ar@{^(->}[r]&\CY(u)^{[n]}\ar@{^(->}[r]&\CN_{n+1}^{[n]}.
}
\end{aligned}
\end{equation*}
The bottom left square is cartesian since the bottom rectangle is cartesian. The top left square is cartesian since $\CN_n^{[n,t]}\subset \wt{\CN}_n^{[t]}\subset \CY(u)^{[n]}\times_{\CN_{n+1}^{[n]}}\CN_{n+1}^{[n,t]}\subset \CN_{n+1}^{[n,t]}$ as closed formal subschemes. Therefore, the left rectangle is cartesian.
\end{proof}

 \subsection{Lattice models}\label{sec: t n lattice}
This subsection will again be modeled on  \S\ref{sec:(n,t)-lat}.
  Let $\Lambda^\flat_{\ep}$ be a vertex lattice of type $t$ in $W^\flat_\ep$ and denote its stabilizer by $K_{n}^{[t]}$.  Fix a special vector $u_0$ of unit norm such that the lattice $\Lambda_0=\Lambda^\flat_{\ep}\obot \langle u_0 \rangle \in \Ver^{t}(W_0)$ (here $W_0$ is the split hermitian space). We fix a lattice   $\Lambda\in \Ver^{n}(W_0)$ such  that   $\Lambda \subset\Lambda_0$.  Denote by  $K_{n+1}^{[n]}$ (resp. $K_{n+1}^{[t]}$)  the stabilizer of $\Lambda$ (resp. $\Lambda_0$).

   We have the lattice model of RZ spaces: $\BN_n^{[t,n]}$ and
 $$
 \wt{\BN}_n^{[t]}=\{(\Lambda^\flat, \Lambda)\in {\rm Vert}^t(W^\flat)\times {\rm Vert}^n(W)\mid \Lambda\subset \Lambda^\flat\oplus\langle u_0\rangle\} .
 $$
 We have a disjoint union induced by the  two possibilities that either $u_0\in\Lambda$ or $u_0\notin\Lambda$ :
  $$
 \wt{\BN}_n^{[t]}= \wt{\BN}_n^{[t],+}\coprod \wt{\BN}_n^{[t],-}.
  $$
 
 We first assume  $\ep^\flat=+1$.  In the case when $u_0\in\Lambda$, we have $\Lambda=\Lambda_1^\flat\obot\langle u_0\rangle$, with $\Lambda_1^\flat\subset\Lambda^\flat$. Therefore for this part of $\wt{\BN}_n^{[t]}$, we obtain an identification
$$
\wt{\BN}_n^{[t],+}=\BN_n^{[t,n]} .
$$
 We then look at $\wt{\BN}_n^{[t],-}$.
 If $u_0\notin\Lambda$, then $\Lambda+\langle u_0\rangle$ is a vertex lattice of type $n-2$, and is of the form $\Lambda+\langle u_0\rangle=\Lambda_0^\flat\obot\langle u_0\rangle$, for a unique $\Lambda_0^\flat\in{\rm Vert}^{n-2}(W^\flat)$ with $\Lambda_0^\flat\subset \Lambda^\flat$. Hence we obtain a map from this part of $\wt{\BN}_n^{[t]}$ to $\BN_n^{[n-2, t]}$. Hence this part of $\wt{\BN}_n^{[t]}$ appears in a commutative diagram
  \begin{equation}\label{eq: wtNt -}
 \begin{aligned}
  \xymatrix{\wt{\BN}^{[t], -}_n \ar[r] \ar[d] \ar@{}[rd]  & \BN^{[n-2,n]}_{n+1}\ar[d] 
  \\ \BN_n^{[n-2, t]} \ar[r] &\BN_{n+1}^{[n-2]}.}
  \end{aligned}
\end{equation}
However, we caution the reader that the diagram is not cartesian! To see this, we note that the right downward arrow has fibers of the form $\BP^1(k)$, parametrizing all isotropic lines in the three-dimensional quadratic space $\BW:=(\Lambda_0^\flat\oplus\langle u_0\rangle)/\pi(\Lambda_0^\flat\oplus\langle u_0\rangle)^\vee $, while the image of the fiber over a point $(\Lambda_0^\flat,\Lambda_{\ep}^{\flat}) \in  \BN_n^{[n-2, t]}$ under the top map omits exactly two of the $q+1$ points in the fiber of the right downward map (the two  isotropic lines orthogonal to the image of $u_0$ in $\BW$).

 We now assume $\ep^\flat=-1$. Then $W^\flat$ is not split. The argument above shows that $\wt{\BN}_n^{[t], +}$ is empty and we have 
$$
\wt{\BN}_n^{[t]}=\wt{\BN}_n^{[t], -}.
$$
Now the diagram \eqref{eq: wtNt -} is cartesian! The difference is that now the orthogonal complement of the image of $u_0$ in $\BW$ is a non-split 2-dimensional quadratic space.

Similar to \S\ref{sec:(n,t)-lat},
 we have the
bi-$K_n^{[n]}$-invariant Hecke function, cf. \eqref{eq: phi(t,n)}
   \begin{equation}\label{phiz}
\varphi^{[n,t]}_{n}:=\vol(K^{[t]}_n)^{-1} {\bf 1}_{K^{[n]}_nK^{[t]}_n}\ast {\bf 1}_{K^{[t]}_n K^{[n]}_n} .
\end{equation}
 Then  we have 
 $$
\Orb(g,\vol(K^{[n]}_n)^{-2}\varphi^{[n,t]}_{n}\otimes{\bf 1}_{ K_{n+1}^{[n]} } )= \#(\BN^{[n,t]}_n \cap g \BN^{[n,t]}_n).
$$

Similarly we have the bi-$K_{n+1}^{[t]}$-invariant Hecke function, cf.\eqref{eq: phi n+1 (n,t)}
   \begin{equation} \label{phiz big}
\varphi^{[t,n]}_{n+1}:=\vol(K^{[n]}_{n+1})^{-1} {\bf 1}_{K^{[t]}_{n+1}K^{[n]}_{n+1}}\ast {\bf 1}_{K^{[n]}_{n+1}K^{[t]}_{n+1}},
\end{equation}
and we have
 $$
\Orb(g,\vol(K_n^{[t]})^{-2} {\bf 1}_{K_n^{[t]}}\otimes \varphi^{[t,n]}_{n+1})= \#(\wt\BN^{[t]}_n \cap g \wt\BN^{[t]}_n).
$$

\begin{lemma}
\label{lem: lat (t,n)} 
 \begin{altenumerate}
 
 \item When $0\leq t\leq n-2$, we have $$ \# K_{n}^{[t]} \bs K_{n+1}^{[t]}/K_{n+1}^{[t,n]}=\begin{cases} 2,& \ep^\flat=+1\\
 1,& \ep^\flat=-1
 \end{cases}
 $$
 \item We have 
 $$
 \vol(K^{[n]}_n)^{-2}\varphi^{[n,t]}_{n}\otimes {\bf 1}_{K^{[n]}_{n+1}}  \sim  \vol(K_n^{[n,t]})^{-2}  {\bf 1}_{ K_{n}^{[t]}  \times K^{[n]}_{n+1} }
 $$ when $\ep^\flat=+1$, and 
 $$
 \vol(K_n^{[t]})^{-2} {\bf 1}_{K_n^{[t]}}\otimes \varphi^{[t,n]}_{n+1} \sim  \vol(K_n^{[n,t]})^{-2}  {\bf 1}_{ K_{n}^{[t]}  \times K^{[n]}_{n+1} }
 $$ 
  when $\ep^\flat=-1$ (this case only makes sense when $t\leq n-2$).

\end{altenumerate}
\end{lemma}

\begin{proof}
(i) Denote by $\BW=\Lambda_0/\pi\Lambda_0^\vee$ the vector space over $k$ of dimension $(n+1)-t\geq 3$ with  the induced quadratic form.  Then $\pi\Lambda^\vee\subset \BW$ is an isotropic subspace of dimension $(n-t)/2$, and hence defines a parabolic subgroup $P\subset O(\BW)$.
The reduction of the vector $u_0$ is anisotropic and its orthogonal complement in $\BW$ is denoted by  $\BW^\flat$. 
Then we have a natural bijection:
$$K_{n}^{[t]} \bs K_{n+1}^{[t]}/K_{n+1}^{[t,n]}\simeq 
O(\BW^\flat)\bs O(\BW)/P\simeq O(\BW)\bs [ \big(O(\BW)/ O(\BW^\flat)\big)\times \big(O(\BW)/P\big)]
$$
where in the last quotient $O(\BW)$ acts diagonally. Now by Witt's theorem (cf. the argument in the proof of Lemma 6.1.1 in \cite{LRZ2}), there are two (resp. one) orbits if $\BW^\flat$ is split (resp. non-split), corresponding to $\ep^\flat=+1$ (resp. $\ep^\flat=-1$).

(ii) The proof is similar to that of Lemma \ref{lem: lat (n,t)}  part (iii). We sketch the proof for the case $\ep=+1$; the other case is similar using part (i) $\ep=-1$.
We use \eqref{eq: eM}, and $ {\bf 1}_{K_{n}^{[n]} } \ast {\bf 1}_{ K_{n}^{[t]} }=\vol(K_{n}^{[n,t]}) {\bf 1}_{K_{n}^{[n]} K_{n}^{[t]} }$. Then 
\begin{align*}
{\bf 1}_{ K_{n}^{[t]}}  \otimes {\bf 1}_{K^{[n]}_{n+1} }\sim &  (e_{ K_{n}^{[n]}} \ast {\bf 1}_{ K_{n}^{[t]}} \ast e_{ K_{n}^{[n]}}) \otimes {\bf 1}_{K^{[n]}_{n+1} }\\
=& \vol(K_{n}^{[n]})^{-2} \vol(K_{n}^{[t]})^{-1}\vol(K_{n}^{[n,t]})^2    ( {\bf 1}_{K_{n}^{[n]} K_{n}^{[t]} }  \ast {\bf 1}_{K_{n}^{[t]} K_{n}^{[n]} })  \otimes {\bf 1}_{K^{[n]}_{n+1} }\\
=&\vol(K_{n}^{[n]})^{-2} \vol(K_{n}^{[n,t]})^2  \varphi^{[n,t]}_{n}\otimes {\bf 1}_{K^{[n]}_{n+1}}.\qedhere
\end{align*}
\end{proof}
 \begin{remark}
 \label{rem: ind Lbd (t,n)}We comment that
the product $K_{n+1}^{[t]} K^{[n]}_{n+1} K_{n+1}^{[t]}$ depends only on $\Lambda_0=\Lambda^\flat_{\ep}\obot \langle u_0 \rangle \in \Ver^{t}(W_0)$ but not on the choice of $\Lambda\in \Ver^{n}(W_0)$ such  that   $\Lambda \subset\Lambda_0$.
 \end{remark}
 
 \begin{remark}
 \label{rem: no RZ} Let $0\leq t\leq n-2$. When $\ep^\flat=+1$, part (i)  shows that $K_{n}^{[t]}$ acts on $K_{n+1}^{[t]}/K_{n+1}^{[t,n]}$ with exactly two orbits. The first one is $K_{n}^{[t]}K_{n+1}^{[t,n]}$ with the stabilizer being the subgroup $K_n^{[n,t]}$ of $K_n^{[t]}$. For the second one, we choose any representative $h\in K_{n+1}^{[t]}\setminus K_{n}^{[t]}K_{n+1}^{[t,n]}$. Then the stabilizer is the subgroup $$K_n^{[t],-}\colon =K_n^{[t]}\cap h K_{n+1}^{[n,t]} h^{-1}$$ of $K_n^{[t]}$,
 which is independent of the choice of $h$.
  Then in the disjoint union  $\wt{\BN}_n^{[t]}= \wt{\BN}_n^{[t],+}\coprod \wt{\BN}_n^{[t],-}.
$, we may naturally identify $ \wt{\BN}_n^{[t],+}$ with $U(W_\ep^\flat)/K_n^{[n,t]}$, and  identify $ \wt{\BN}_n^{[t],-}$ with $U(W_\ep^\flat)/K_n^{[t],-}$.
 Similarly, when  $\ep^\flat=-1$,  we may naturally identify $ \wt{\BN}_n^{[t]}= \wt{\BN}_n^{[t],-}$ with $U(W_\ep^\flat)/(K_n^{[t]}\cap K_{n+1}^{[n,t]})$. 

 \end{remark}

\subsection{Intersection numbers on the splitting model for the smaller correspondence }
Assume that $\ep^\flat=+1$, so that $\CN_n^{[n]}$ is defined.

Let $ \wh\CN_{n}^{[t],\spt}$ be the flat closure of  the base change of $ \CN^{[t,n]}_{n}$ along the morphism $\CN_n^{[t],\spt}\times\CN_{n+1}^{[n]}\to  \CN^{[t]}_{n}\times\CN^{[n]}_{n+1} $. Then $\wh\CN_{n}^{[t],\spt}$ is a closed formal subscheme of $\CN_n^{[t],\spt}\times\CN_{n+1}^{[n]}$, flat over $\Spf O_{\breve{F}}$ of relative dimension $n-1$.  We have the commutative diagram
  \begin{equation*}
 \begin{aligned}
  \xymatrix{
  \wh\CN_{n}^{[t],\spt} \ar@{^(->}[r] \ar[d]  &\CN_n^{[t],\spt}\times\CN_{n+1}^{[n]} \ar[d]   \\
\CN^{[t,n]}_{n}  \ar@{^(->}[r] & \CN^{[t]}_{n}\times\CN^{[n]}_{n+1} 
}
\end{aligned}
\end{equation*}
Now the product $\CN_n^{[t],\spt}\times\CN_{n+1}^{[n]}$ is regular since $ \CN_{n+1}^{[n]}$ is formally smooth over $\Spf\OFb$ and $  \CN_n^{[t],\spt}$ is regular. 
We form the arithmetic intersection numbers 
\begin{equation*}
  \left\langle\wh\CN_{n}^{[t],\spt} , g\wh\CN_{n}^{[t],\spt} \right\rangle_{\CN_n^{[t],\spt} \times\CN_{n+1}^{[n]}}:=\chi( \wh\CN_{n}^{[t],\spt} \times \CN_{n+1}^{[n]} , \wh\CN_{n}^{[t],\spt}\cap^\BL g\wh\CN_{n}^{[t],\spt} ),
\end{equation*}
for regular semisimple $g$.

\subsection{The AT conjecture for the smaller correspondence}
We now come to the AT conjecture. 

 \begin{conjecture}\label{conj (t,n even) sm}
 Let $n=2m$ be even, $0\leq t\leq n$ even. Assume that $\ep^\flat=1$.
There exists $\varphi'\in C_c^\infty(G')$ with transfer $(  \vol(K_n^{[n,t]})^{-2}  {\bf 1}_{ K_{n}^{[t]}  \times K^{[n]}_{n+1}  },0)\in C_c^\infty(G_{W_0})\times C_c^\infty(G_{W_{1}})$ such that, if $\gamma\in G'(F_0)_\rs$ is matched with  $g\in G_{W_1}(F_0)_\rs$, then
 \begin{equation*}
  \left\langle \wh\CN_{n}^{[t],\spt}   , g \wh\CN_{n}^{[t],\spt}  \right\rangle_{ \CN_{n}^{[t],\spt} \times \CN_{n+1}^{[n]} }\cdot\log q=- \del\big(\gamma,  \varphi' \big).
\end{equation*}
\end{conjecture}

\begin{remark}\label{rem: even t=n}
When $t=n$,  since  $\CN_{n}^{[n],\spt}$ is identical to $\CN_{n}^{[n]}$, Conjecture \ref{conj (t,n even) sm} is identical with Conjecture \ref{conj (n even,t=n)} in \cite{RSZ1}.
\end{remark}

\subsection{Intersection numbers on the splitting model for the larger  correspondence}

Let $\wt\CN^{[t], \spt}_{n}$ be the flat closure of  the base change of $\wt\CN^{[t]}_{n}$ along the morphism $\CN_n^{[t],\spt}\times\CN_{n+1}^{[n]}\to \CN_n^{[t]}\times\CN_{n+1}^{[n]}$. Then $\wt\CN^{[t], \spt}_{n}$ is a closed formal subscheme of $\CN_{n+1}^{[t],\spt}\times \CN_{n+1}^{[n]}$, flat over $\Spf O_{\breve{F}}$ of relative dimension $n-1$. If  Conjecture \ref{conj-wtNfl}  on the flatness of $ \wt\CN^{[t]}_n$ holds, then $\wt\CN^{[t], \spt}_{n}$ coincides with $\wt\CN^{[t]}_{n}$ outside the worst points. We have the commutative diagram
  \begin{equation}\label{splcy95}
 \begin{aligned}
  \xymatrix{
  \wt\CN^{[t], \spt}_{n} \ar@{^(->}[r] \ar[d]  &\CN_n^{[t],\spt}\times\CN_{n+1}^{[n]} \ar[d]   \\
\wt\CN^{[t]}_{n}  \ar@{^(->}[r] & \CN^{[t]}_{n}\times\CN^{[n]}_{n+1} 
}
\end{aligned}
\end{equation}

By  Lemma \ref{lem:(t,n)-naive-emb}, the horizontal maps  are closed embeddings. 
We define the arithmetic intersection numbers 
\begin{equation*}
  \left\langle\wt\CN^{[t],\spt}_{n,\ep} , g\wt\CN^{[t],\spt}_{n,\ep} \right\rangle_{ \CN_{n,\ep}^{[t],\spt} \times \CN_{n+1}^{[n]} }:=\chi(\CN_{n,\ep}^{[t],\spt} \times \CN_{n+1}^{[n]} , \wt\CN^{[t],\spt}_{n,\ep} \cap^\BL g\wt\CN^{[t],\spt}_{n,\ep}  ),
\end{equation*}
for regular semisimple $g\in G_{W_1}(F_0)$.

\subsection{The AT conjecture for  the larger correspondence}

We now come to the AT conjecture.

 \begin{conjecture}\label{conj (t,n even)}
 Let $n=2m$ be even, $0\leq t\leq n$ even. Recall from \eqref{phiz big} the function $\varphi^{[t,n]}_{n+1} \in C_c^\infty(G_{W_{0}})$. 
 
There exists $\varphi'\in C_c^\infty(G')$ with   transfer $(\vol(K_n^{[t]})^{-2} {\bf 1}_{K_n^{[t]}}\otimes \varphi^{[t,n]}_{n+1} ,0)\in C_c^\infty(G_{W_0})\times C_c^\infty(G_{W_{1}})$ such that, if $\gamma\in G'(F_0)_\rs$ is matched with  $g\in G_{W_1}(F_0)_\rs$, then
 \begin{equation*}
  \left\langle\wt\CN^{[t],\spt}_{n} , g\wt\CN^{[t],\spt}_{n} \right\rangle_{ \CN_{n}^{[t],\spt} \times \CN_{n+1}^{[n]} }\cdot\log q=- \del\big(\gamma,  \varphi' \big).
\end{equation*}
\end{conjecture}

\begin{remark}\label{rem:t=n alt}
When $t=n$,  we have $ \wt{\CN}_n^{[t]}=\wt{\CN}_n^{[n]}\simeq  {\CN}_n^{[n]}$.  Since  $\CN_{n}^{[n],\spt}$ is identical to $\CN_{n}^{[n]}$ by Remark \ref{rem: t=n even wtN}, Conjecture \ref{conj (t,n even)} is identical to Conjecture \ref{conj (n even,t=n)}, i.e.,  \cite[Conj. 5.6]{RSZ1}.
\end{remark}

\begin{remark}\label{rem:t=n alt}For each $t\leq n-2$, there are two cases depending on $\ep^\flat=\pm 1$. When $\ep^\flat=-1$ we could replace the test function by the simpler $ \vol(K_n^{[n,t]})^{-2}  {\bf 1}_{ K_{n}^{[t]}  \times K^{[n]}_{n+1} }$, by Lemma  \ref{lem: lat (t,n)}  part (ii).
\end{remark}

    \section{AT conjecture of type $(t, n+1)$}\label{sec:type-(t,n+1)}

 Let $n=2m+1$ be odd. Let $0\leq t\leq n-1$ be even. 
 We use the exceptional special divisor $\CZ(u)^{[t], \spt}$ on $\CN_{n+1}^{[t],\spt}$,  for a unit length vector $u\in \BV_{n+1}$. This can be identified with $  \CN^{[t], \spt}_{n}$, cf. Theorem  \ref{thm: Z unit}. The intersection product takes place on $\CN_{n}^{[t],\spt} \times \CN_{n+1}^{[n+1]}$.

 We fix an aligned triple of framing objects $(\BY, \BX, u)=(\BY_{\ep^\flat}^{[t]}, \BX_\ep^{[t]}, u)$ of dimension $n+1$ and type $t$, with corresponding embedding of RZ spaces $\CN^{[t]}_{n, \ep^\flat}=\CN^{[t]}_n\hookrightarrow \CN^{[t]}_{n+1}=\CN^{[t]}_{n+1, \ep}$, cf. \S \ref{sec:emb-RZ}. We also fix an isogeny $\BY^{[n-1]}\to \BY^{[t]}$ as in \S \ref{sec:RZ-deeper}, and extend this in the obvious way to an isogeny $\BX^{[n-1]}\to \BX^{[t]}$.    This defines embeddings of RZ spaces $\CN^{[t]}_n\hookrightarrow \CN^{[t]}_{n+1}$ and $\CN^{[n-1]}_n\hookrightarrow \CN^{[n-1]}_{n+1}$ and  $\CN^{[n-1,t]}_n\hookrightarrow \CN^{[n-1,t]}_{n+1}$, cf. \S \ref{sec:emb-RZ}. 
 
We also fix an isogeny $\BX^{[n+1]}\to \BX^{[n-1]}$. This is only possible when $\ep=1$, which we assume throughout this section. It allows us to consider the RZ space $\CN^{[n+1]}_{n+1}$, and $\CN^{[n-1, n+1]}_{n+1}$ with its natural morphisms to $\CN^{[n+1]}_{n+1}$ and $\CN^{[n-1]}_{n+1}$.

 We again use the notation $W_1=\BV(\BX)$ and $W_1^\flat=\BV(\BY)$, so that $u\in W_1$ and $W_1^\flat=\langle u\rangle^\perp$. We denote by $W_0$ the hermitian space of dimension $n+1$ with opposite Hasse invariant of $W_1$, and fix a vector $u_0\in W_0$ of the same length as $u$ and set $W_0^\flat  =\langle u_0\rangle^\perp$. Then $W_0^\flat $ has the opposite Hasse invariant of $W_1^\flat$. Note that $W_0$ is the split hermitian space.

  \subsection{The naive version}\label{ss: naive (n odd, n+1)}
  
  Similar to \S\ref{sec:type-(n,t)}, there are two ways to use correspondences to obtain natural cycles on the ambient space  $\CN_{n}^{[t]} \times \CN_{n+1}^{[n+1]}$. The first one is to use a correspondence $\CN^{[n-1,t]}_n $ on the smaller space
 \begin{equation*}
 \begin{aligned}
 \xymatrix{&&\CN^{[n-1,t]}_{n} \times \CN_{n+1}^{[n-1,n+1]}  \ar[rd]   \ar[ld] &
 \\ \CN_{n}^{[n-1]}\ar@{^(->}[r]& \CN_n^{[n-1]} \times \CN_{n+1}^{[n-1]} &&  \CN_{n}^{[t]} \times \CN_{n+1}^{[n+1]} .
 }
 \end{aligned}
 \end{equation*}
  Note that, since there is no (natural) embedding from $ \CN_{n}^{[n-1]}$ to $ \CN_{n+1}^{[n+1]}$,  we are forced to use  an almost trivial but nevertheless essential  correspondence on the larger space.
 
 The second one  is to use a correspondence $\CN^{[n,t]}_{n+1}$  on the bigger space
 \begin{equation*}
 \begin{aligned}
 \xymatrix{&&\CN_{n}^{[t]} \times \CN_{n+1}^{[n+1,t]}  \ar[rd]   \ar[ld] &
 \\ \CN_{n}^{[t]}\ar@{^(->}[r]& \CN_{n}^{[t]} \times \CN_{n+1}^{[t]} &&  \CN_{n}^{[t]} \times \CN_{n+1}^{[n+1]} .
 }
 \end{aligned}
 \end{equation*}

 We define the spaces $\CN_{n}^{[n-1,t],\circ}$ and $\CN_{n}^{[n-1],\circ}$ by the following cartesian diagram (see also Theorem \ref{thm:Y unit} and Remark \ref{rmk:(n-1,t)-Y-smaller}),
 \begin{equation}\label{eq:N n-1 t circ def}
 \begin{aligned}
 \xymatrix{
 \CN_{n}^{[n-1,t],\circ} \ar@{}[rd]|{\square}\ar[d]\ar[r]
 &\CN_{n}^{[n-1],\circ}\ar[d]\ar@{^(->}[r] \ar@{}[rd]|{\square}&\CN_{n+1}^{[n+1,n-1]}\ar[d]  \\
 \CN_n^{[n-1,t]}\ar[r]&\CN_{n}^{[n-1]}\ar@{^(->}[r]& \CN_{n+1}^{[n-1]} .}
 \end{aligned}
 \end{equation}
 
 We can factorize the outer cartesian diagram
 \begin{equation}\label{eq:N n-1 t circ def out}
 \begin{aligned}
 \xymatrix{
 \CN_{n}^{[n-1,t],\circ} \ar@{}[rd]|{\square}\ar[d]\ar[r]
 &\CN_{n+1}^{[n+1,n-1]}\ar[d]  \\
 \CN_n^{[n-1,t]}\ar[r]& \CN_{n+1}^{[n-1]} }
 \end{aligned}
 \end{equation}
  in an alternative way,
   \begin{equation}\label{biggerdi}
 \begin{aligned}
 \xymatrix{
 \CN_{n}^{[n-1,t],\circ} \ar@{}[rd]|{\square}\ar[d]\ar@{^(->}[r]
 &\CN_{n+1}^{[n+1,n-1,t]}\ar[d]\ar[r] \ar@{}[rd]|{\square}&\CN_{n+1}^{[n+1,n-1]}\ar[d]  \\
 \CN_n^{[n-1,t]}\ar@{^(->}[r]&\CN_{n+1}^{[n-1,t]}\ar[r]& \CN_{n+1}^{[n-1]}. }
 \end{aligned}
 \end{equation}
  To see these cartesian diagrams,
we construct all of these spaces from the following skeleton, 
 \begin{equation*}\label{eq:(t,n)-skeleton}
\xymatrix{  &&  && \\
& &&  \CN_{n+1}^{[n+1,n-1]}\ar[ddd]&& \\
\\
 && \CN_{n+1}^{[n-1,t]}\ar[dr]^{} &&\\
&\CN_{n}^{[n-1]} \ar[rr]_{}&& \CN_{n+1}^{[n-1]} .&& 
}
\end{equation*}
We arrive at the following cube by taking successive cartesian products
 \begin{equation}\label{eq:(t,n)-cubic}
 \begin{aligned}
\xymatrix{ \CN_{n}^{[n-1,t],\circ} \ar[ddd]\ar[dr]_{} \ar@{^(->}[rr]_{} && \CN_{n+1}^{[n+1,n-1,t]}  \ar[ddd]|!{[d]}\hole \ar[dr]_{}&& \\
&    \CN_{n}^{[n-1],\circ} \ar[ddd] \ar@{^(->}[rr]_{} &&  \CN_{n+1}^{[n+1,n-1]}\ar[ddd]&& \\
\\
  \CN_{n}^{[n-1,t]} \ar[dr]_{} \ar@{^(->}[rr]_{}|!{[r]}\hole && \CN_{n+1}^{[n-1,t]}\ar[dr]^{} &&\\
&\CN_{n}^{[n-1]} \ar@{^(->}[rr]_{}&& \CN_{n+1}^{[n-1]} .&& 
}
\end{aligned}
\end{equation}
In particular, all of the faces are cartesian, and all of the vertical morphisms are \'etale of degree two.   Note that the bottom cartesian square follows from Proposition \ref{prop:deep-exc} and the right vertical face is obvious. Then the two factorizations \eqref{eq:N n-1 t circ def} and  \eqref{biggerdi} come from the front-and-left faces and the right-and-back faces respectively.

From  the big correspondence, we have the following fiber product diagram  
  \begin{equation}\label{pi1pi2}
 \begin{aligned}
 \xymatrix{
 &\wh{\CN}_n^{[t]}\ar@{}[d]|{\square}\ar[dl]_{\pi_1}\ar@{^(->}[r]\ar[rrd]^{\pi_2}
 &\CN_{n+1}^{[n+1,t]}\ar[rd]\ar[ld]|!{[l];[rd]}\hole&\\
 \CN_n^{[t]}\ar@{^(->}[r]&\CN_{n+1}^{[t]}&&\CN_{n+1}^{[n+1]}.}
 \end{aligned}
 \end{equation}

\begin{lemma}
The proper morphism
\[
\xymatrix{ (\pi_1,\pi_2):&
\wh\CN^{[t]}_{n} \ar[r]& \CN_{n}^{[t]} \times \CN_{n+1}^{[n+1]}
}
\]
is a closed embedding of formal schemes.
\end{lemma}
\begin{proof}
Indeed, $\wh{\CN}_n^{[t]}$ is the closed formal subscheme of $\CN_n^{[t]}\times \CN_{n+1}^{[n+1]}$ parameterizing pairs $(Y^{[t]}, X^{[n+1]})$ such that the quasi-isogeny $X^{[n+1]}\to Y^{[t]}\times\ov{\CE}$
lifts to an isogeny.
\end{proof}

  Note that the cartesian square in the back face of \eqref{eq:(t,n)-cubic} can be enlarged into a commutative diagram
  \begin{equation}\label{equ:three-squares}
 \begin{aligned}
 \xymatrix{
 \CN_{n}^{[n-1,t],\circ} \ar@{}[rd]|{\square}\ar[d]\ar@{^(->}[r]
 &\CN_{n+1}^{[t,n-1,n+1]}\ar[d]\ar[r]&\CN_{n+1}^{[t,n+1]}\ar[dd]  \\
 \CN_n^{[t,n-1]}\ar[d]\ar@{^(->}[r]&\CN_{n+1}^{[t,n-1]}& \\
 \CN_n^{[t]}\ar[rr] &&  \CN_{n+1}^{[t]} .}
 \end{aligned}
 \end{equation}
 Therefore, from the outer square there is an induced morphism
 $$
 \CN_{n}^{[n-1,t],\circ}\to \wh{\CN}_n^{[t]},
  $$
which is a closed embedding.   
 \begin{theorem}\label{sm=big10}
 The natural map 
 $$
 \CN_n^{[n-1,t],\circ}\to \wh{\CN}_n^{[t]}
 $$
 is an isomorphism. In other words, the big correspondence and the small correspondence coincide as formal subschemes of $\CN_{n}^{[t]} \times \CN_{n+1}^{[n+1]}$. 
 \end{theorem}
 When $t=n-1$ this holds trivially and both spaces are isomorphic to $ \CN_n^{[n-1],\circ}$.
 The proof of Theorem \ref{sm=big10} is given in \S \ref{s:pfofthmBS}. In particular, we obtain from the structure of $\CN_n^{[n-1, t]}$ the following statement.
 \begin{corollary}
 The formal scheme $\wh{\CN}_n^{[t]}$ is flat. \qed
 \end{corollary}

 \subsection{Lattice models}\label{sec:lattice (n odd, n+1)}

 Let $\Lambda^\flat$ be a vertex lattice of type $t$ in $W^\flat_0$ and denote its stabilizer by $K_{n}^{[t]}$.  Fix a special vector $u_0$ of unit norm such that the lattice $\Lambda_0=\Lambda^\flat\obot \langle u_0 \rangle \in \Ver^{t}(W_0)$. We fix a lattice   $\Lambda\in \Ver^{n+1}(W_0)$ such  that   $\Lambda \subset\Lambda_0$ and  denote by  $K_{n+1}^{[n+1]}$  the stabilizer of $\Lambda$ (recall that  $W_0$ is the split hermitian space).

We have the lattice models $ \BN_{n}^{[n-1,t],\circ}$ and $ \wh{ \BN}_{n}^{[n-1,t]}$ of RZ spaces.
Similar to \S\ref{sec:(n,t)-lat},
 we have a
bi-$K_n^{[n-1]}$-invariant Hecke function, cf. \eqref{eq: phi(t,n)},
   \begin{equation}\label{phie}
\varphi^{[n-1,t]}_{n}:=\vol(K^{[t]}_n)^{-1} {\bf 1}_{K^{[n-1]}_nK^{[t]}_n}\ast {\bf 1}_{K^{[t]}_n K^{[n-1]}_n},
\end{equation}
and a
bi-$K_{n+1}^{[n-1]}$-invariant function 
   \begin{equation}\label{phif}
\varphi^{[n-1,n+1]}_{n+1} =\vol(K^{[n+1]}_{n+1})^{-1} {\bf 1}_{ K_{n+1}^{[n-1]}  K_{n+1}^{[n+1]}} \ast {\bf 1}_{K_{n+1}^{[n+1]}  K_{n+1}^{[n-1]}}.
\end{equation}
We have 
 $$
\Orb(g,\vol(K^{[n-1]}_n)^{-2}\varphi^{[n-1,t]}_{n}\otimes \varphi^{[n-1,n+1]}_{n+1} )= \#(\BN^{[n-1,t],\circ}_n \cap g \BN^{[n-1,t],\circ}_n).
$$

Similarly we have the bi-$K_{n+1}^{[t]}$-invariant Hecke function, cf.\eqref{eq: phi n+1 (n,t)}
   \begin{equation}\label{phig}
\varphi^{[t,n+1]}_{n+1}:=\vol(K^{[n+1]}_{n+1})^{-1} {\bf 1}_{K^{[t]}_{n+1}K^{[n+1]}_{n+1}}\ast {\bf 1}_{K^{[n+1]}_{n+1}K^{[t]}_{n+1}}.
\end{equation}
Then we have
 $$
\Orb(g,\vol(K_n^{[t]})^{-2} {\bf 1}_{K_n^{[t]}}\otimes \varphi^{[t,n+1]}_{n+1})= \#(\wt\BN^{[t]}_n \cap g \wt\BN^{[t]}_n).
$$

\begin{lemma}\label{lem: lat (t,n+1)} Let $t\leq n-1$.
 \begin{altenumerate}
  \item
 We have
 $$
 \BN_{n}^{[n-1,t],\circ}\simeq \wh \BN_{n}^{[n-1,t]}.
 $$
 Both are ``finite \'etale  double" coverings of  $ \BN_{n}^{[n-1,t]}$ (namely, every fiber consists of two elements).
 
 \item We have 
 $$
\# K_{n}^{[t]}\bs K_{n+1}^{[t]}/ K_{n+1}^{[t,n+1]}=1,
$$
and $$
\# K_{n}^{[t],\circ}\bs K_{n+1}^{[t]}/ K_{n+1}^{[t,n+1]}=2.
$$
\item We have 
\begin{align*}
\vol(K_n^{[n-1,t],\circ})^{-2} {\bf 1}_{K_{n}^{[t]}}\otimes{\bf 1}_{ K_{n+1}^{[n+1]} }\sim  \vol(K_n^{[t]})^{-2} {\bf 1}_{K_{n}^{[t]}}\otimes\varphi^{[t,n+1]}_{n+1} \sim \vol(K^{[n-1]}_n)^{-2}\varphi^{[n-1,t]}_{n}\otimes \varphi^{[n-1,n+1]}_{n+1} .
\end{align*}

\end{altenumerate}
\end{lemma}

\begin{proof}
(i)
For the big correspondence we have
 \begin{equation}
 \begin{aligned}
 \wh{\BN}_n^{[t]}&=\{(\Lambda^\flat, \Lambda', \Lambda)\in  {\rm Vert}^t(W^\flat)\times{\rm Vert}^t(W)\times {\rm Vert}^{n+1}(W)\mid \Lambda\subset \Lambda'=\Lambda^\flat\oplus \langle u \rangle \}=\\
&= \{(\Lambda^\flat,  \Lambda)\in  {\rm Vert}^t(W^\flat)\times{\rm Vert}^{n+1}(W)\mid \Lambda\subset\Lambda^\flat\oplus \langle u \rangle \} .
 \end{aligned}
 \end{equation}
For the small correspondence we have
 \begin{equation}\label{eq:lattice-t n-1 circ}
 \begin{aligned}
 {\BN}_n^{[t, n-1], \circ}=\{(\Lambda^\flat,& \Lambda_1, \Lambda',\Lambda_1, \Lambda)\in  {\rm Vert}^t(W^\flat)\times  {\rm Vert}^{n-1}(W^\flat)\times{\rm Vert}^t(W)\times  {\rm Vert}^{n-1}(W)\times{\rm Vert}^{n+1}(W)\mid\\
 &\Lambda'=\Lambda^\flat\oplus\langle u \rangle, \Lambda_1=\Lambda_1^\flat\oplus \langle u \rangle,   \Lambda_1^\flat\subset \Lambda^\flat,\Lambda\subset \Lambda_1 \}\\
= \{(\Lambda^\flat,& \Lambda^\flat_1,  \Lambda)\in  {\rm Vert}^t(W^\flat)\times  {\rm Vert}^{n-1}(W^\flat)\times{\rm Vert}^{n+1}(W)\mid \Lambda_1^\flat\subset \Lambda^\flat, \Lambda\subset \Lambda_1^\flat \oplus\langle u \rangle\}.
 \end{aligned}
 \end{equation}
 We get a bijection because given $(\Lambda^\flat,  \Lambda)\in  \wh{\BN}_n^{[t]}$, we can reconstruct uniquely the missing entry $\Lambda_1^\flat$ by the chain of inclusions $\Lambda\subset^1 \Lambda+\langle u \rangle\subset \Lambda^\flat\oplus \langle u \rangle$, which implies that $\Lambda_1^\flat$ is the unique solution of $ \Lambda+\langle u \rangle= \Lambda_1^\flat\oplus \langle u \rangle$ (we use that $u$ is a vector of unit length).

(ii) The proof is similar to that of Lemma \ref{lem: lat (t,n)}   part (ii) as an application of Witt's theorem and we omit the details.

(iii)
The proof is similar to that of Lemma \ref{lem: lat (n,t)}  part (iii). By the bi-$ K_{n}^{[t]}$-invariance of ${\bf 1}_{ K_{n}^{[t]}}$ we obtain
\begin{align*}
{\bf 1}_{ K_{n}^{[t]}}  \otimes {\bf 1}_{K^{[n+1]}_{n+1} }\sim &  {\bf 1}_{ K_{n}^{[t]}} \otimes (e_{K_{n}^{[t]} } \ast  {\bf 1}_{K^{[n+1]}_{n+1} } \ast e_{ K_{n}^{[t]}}) .
\end{align*}
Now note 
$$e_{K_{n}^{[t]} } \ast  {\bf 1}_{K^{[n+1]}_{n+1} }= \vol(K_{n}^{[t]})^{-1}  {\bf 1}_{K^{[t]}_{n} }\ast  {\bf 1}_{K^{[n+1]}_{n+1} }=\vol(K_{n}^{[t]})^{-1}  \vol(K_{n}^{[t]}\cap  K_{n+1}^{[t,n+1]})  {\bf 1}_{K^{[t]}_{n} K^{[n+1]}_{n+1} } ,
$$ where $K_{n}^{[t]}\cap  K_{n+1}^{[n+1]}=K_{n}^{[n-1,t],\circ}$. By part (ii), we have $K^{[t]}_{n} K^{[n+1]}_{n+1}=K_{n}^{[t]} K_{n+1}^{[t,n+1]} K^{[n+1]}_{n+1}=K_{n+1}^{[t]} K^{[n+1]}_{n+1}$. Hence
\begin{align*}
{\bf 1}_{ K_{n}^{[t]}}  \otimes {\bf 1}_{K^{[n+1]}_{n+1} }
\sim &{\bf 1}_{ K_{n}^{[t]}} \otimes (\vol(K_{n}^{[n-1,t],\circ})^{2}  \vol(K_{n}^{[t]})^{-2} \vol(K^{[n+1]}_{n+1})^{-1}{\bf 1}_{K_{n+1}^{[t]} K^{[n+1]}_{n+1} } \ast {\bf 1}_{K^{[n+1]}_{n+1} K_{n+1}^{[t]}})\\
= & \vol(K_{n}^{[n-1,t],\circ})^{2}  \vol(K_{n}^{[t]})^{-2}  {\bf 1}_{ K_{n}^{[t]}} \otimes  \varphi^{[t,n+1]}_{n+1}.
\end{align*}

Similarly, by the bi-$K_{n}^{[n-1],\circ}$-invariance of ${\bf 1}_{K^{[n+1]}_{n+1} }$ we have
\begin{align*}
{\bf 1}_{ K_{n}^{[t]}}  \otimes {\bf 1}_{K^{[n+1]}_{n+1} }
\sim &  (e_{ K_{n}^{[n-1],\circ}} \ast {\bf 1}_{ K_{n}^{[t]}}\ast e_{ K_{n}^{[n-1],\circ}}) \otimes   {\bf 1}_{K^{[n+1]}_{n+1} } \\
= & \vol( K_{n}^{[n-1],\circ})^{-2} ({\bf 1}_{ K_{n}^{[n-1],\circ}} \ast {\bf 1}_{ K_{n}^{[t]}}\ast {\bf 1}_{ K_{n}^{[n-1],\circ}}) \otimes   {\bf 1}_{K^{[n+1]}_{n+1} } \\
= &\vol( K_{n}^{[n-1],\circ})^{-2} \vol(K_{n}^{[t]})^{-1}\vol(K_{n}^{[n-1,t],\circ} )^2 ({\bf 1}_{ K_{n}^{[n-1],\circ} K_{n}^{[t]}}\ast {\bf 1}_{  K_{n}^{[t]} K_{n}^{[n-1],\circ}})\otimes  {\bf 1}_{K^{[n+1]}_{n+1} }\\
= &\vol( K_{n}^{[n-1],\circ})^{-2} \vol(K_{n}^{[n-1,t],\circ} )^2 \varphi^{[n-1,t]}_{n}\otimes  {\bf 1}_{K^{[n+1]}_{n+1} } ,
\end{align*}
where we have used $K_{n}^{[n-1],\circ} K_{n}^{[t]}=K_{n}^{[n-1]} K_{n}^{[t]}$. Next by the bi-$K_n^{[n-1]}$-invariance of $\varphi^{[n-1,t]}_{n}$,  we continue to obtain
\begin{align*}
{\bf 1}_{ K_{n}^{[t]}}  \otimes {\bf 1}_{K^{[n+1]}_{n+1} }\sim &\vol( K_{n}^{[n-1],\circ})^{-2} \vol(K_{n}^{[n-1,t],\circ} )^2 \varphi^{[n-1,t]}_{n}\otimes  (e_{K_n^{[n-1]}}\ast {\bf 1}_{K^{[n+1]}_{n+1} }\ast e_{K_n^{[n-1]}})\\
= & \vol(K_{n}^{[n-1,t],\circ} )^2 \vol( K_{n}^{[n-1]})^{-2}  \varphi^{[n-1,t]}_{n}\otimes   \varphi^{[n-1,n+1]}_{n+1}.\qedhere
\end{align*}
\end{proof}

 \subsection{The exotic case $t=n-1$}\label{ss: case (n odd, n+1)}
 In this subsection,  we consider the case $t=n-1$, in which we have regularity of $\CN_{n}^{[t]} \times \CN_{n+1}^{[n+1]}$ without passing to the splitting model. Namely, in the case $t=n-1$, the formal scheme  $ \CN_{n}^{[n-1]}$ is formally smooth (\emph{exotic smoothness}). In this case, we  have the AT conjecture in \cite{RSZ2}, which we recall briefly.
 The RZ spaces $\CN^{[n-1]}_{n}$ and $ \CN^{[n+1]}_{n+1} $ are both smooth (exotic smoothness), and hence it makes sense to consider the intersection number 
  \begin{equation}
  \left\langle\CN_n^{[n-1],\circ} , g\CN_n^{[n-1],\circ} \right\rangle_{ \CN_{n}^{[n-1]} \times \CN_{n+1}^{[n+1]} }=\chi(  \CN_{n}^{[n-1]} \times \CN_{n+1}^{[n+1]} , \CN_n^{[n-1],\circ} \cap^\BL g\CN_n^{[n-1],\circ} ) .
  \end{equation}
 But this intersection number coincides precisely with the one occurring in \cite[\S 12]{RSZ2}, and  the analogue of Conjecture \ref{conj (n odd,t<n) Z} is identical  with the conjecture in  \cite[\S 12]{RSZ2}.
 Note that in this case we also obtain the identical conjecture with  that in \S\ref{ss: exotic t=n+1}.

\subsection{Intersection numbers on the splitting model}
Let $\wh{\CN}_n^{[t], \spt}=\CN^{[n-1,t], \circ,\spt}_{n}$ be the flat closure of  the base change of $\wh{\CN}_n^{[t]}= \CN_{n}^{[n-1,t],\circ}$ along the morphism $\CN_n^{[t],\spt}\times\CN_{n+1}^{[n+1]}\to \CN_n^{[t]}\times\CN_{n+1}^{[n+1]}$. Then $ \wh{\CN}_n^{[t], \spt}$ is a closed formal subscheme of $\CN_n^{[t],\spt}\times\CN_{n+1}^{[n+1]}$, flat over $\Spf O_{\breve{F}}$ of relative dimension $n-1$.  We have the commutative diagram

 \[
\xymatrix{
 \wh{\CN}_n^{[t], \spt} \ar[d]  \ar@{^(->}[r]& \CN_n^{[t],\spt} \times \CN_{n+1}^{[n+1]}\ar[d]\\
  \CN_{n}^{[n-1,t],\circ}   \ar@{^(->}[r]& \CN_n^{[t]} \times \CN_{n+1}^{[n+1]}.
 }
 \]
We define the intersection numbers
 $$
  \left\langle   \wh{\CN}_n^{[t], \spt}  , g   \wh{\CN}_n^{[t], \spt}  \right\rangle_{ \CN_{n}^{[t],\spt} \times \CN_{n+1}^{[n+1]} }=\chi(  \CN_{n}^{[t], \spt} \times \CN_{n+1}^{[n+1]} , \wh{\CN}_n^{[t], \spt}  \cap^\BL g\wh{\CN}_n^{[t], \spt}  ) .
  $$

 \subsection{The AT conjecture}
We now come to the AT conjecture.  Let $\Lambda^\flat$ (resp. $\Lambda^{'\flat}$) be a vertex lattice of type $t$ (resp. type $n-1$) in $W^\flat_0$. Denote by $K_{n}^{[t]}$ (resp.  $K_{n}^{[n-1]}$) the stabilizer of  $\Lambda^\flat$ (resp. $\Lambda^{'\flat}$).  Fix a special vector $u_0$ of unit norm, and let $\Lambda_0=\Lambda^\flat\obot \langle u_0 \rangle$ and $\Lambda_0'= \Lambda^{'\flat}\obot \langle u_0 \rangle $; they are vertex lattices of type $t$ and $n-1$ respectively. We fix a lattice   $\Lambda\in \Ver^{n+1}(W_0)$ such  that   $\Lambda \subset\Lambda'_0$ (recall that $W_0$ is the split hermitian space).  Denote by  $K_{n+1}^{[n+1]}$ (resp.  $K_{n+1}^{[n-1]}$)  the stabilizer of $\Lambda$ (resp. $\Lambda'_0$).
 
  \begin{conjecture}\label{conj (t,n odd) sm}
 Let $n=2m+1$ be odd, and let $0\leq t\leq n-1$ even.
There exists $\varphi'\in C_c^\infty(G')$ with transfer $(
\vol(K_n^{[n-1,t],\circ})^{-2} {\bf 1}_{K_{n}^{[t]}}\otimes{\bf 1}_{ K_{n+1}^{[n+1]} } ,0)\in C_c^\infty(G_{W_0})\times C_c^\infty(G_{W_{1}})$ such that, if $\gamma\in G'(F_0)_\rs$ is matched with  $g\in G_{W_1}(F_0)_\rs$, then
 \begin{equation*}
  \left\langle   \wh{\CN}_n^{[t], \spt} , g   \wh{\CN}_n^{[t], \spt}  \right\rangle_{ \CN_{n}^{[t],\spt} \times \CN_{n+1}^{[n+1]} }\cdot\log q=- \del\big(\gamma,  \varphi' \big).
\end{equation*}
\end{conjecture}

\begin{remark}
By Lemma \ref{lem: lat (t,n+1)}  part (iii), one could replace the function $
\vol(K_n^{[n-1,t],\circ})^{-2} {\bf 1}_{K_{n}^{[t]}}\otimes{\bf 1}_{ K_{n+1}^{[n+1]} } $ by either of the other two.
\end{remark}

\subsection{Comparison when $t=n-1$}\label{sec:MA-Z-comp}
Conjecture \ref{conj (t,n odd) sm} for $t=n-1$ is {\em not} identical to \cite[Conj. 12.4]{RSZ2} which concerns the intersection number $\langle   \wh{\CN}_n^{[n-1]} , g   \wh{\CN}_n^{[n-1]}  \big\rangle_{ \CN_{n}^{[n-1]} \times \CN_{n+1}^{[n+1]} }$, cf. also Remark \ref{rmk:MA-Y-comp}.  We expect the difference between these conjectures  to be  given on the analytic side by an orbital integral function. 
To be more precise, consider the following commutative diagram,
 \[
\xymatrix{
	\wh{\CN}_n^{[n-1],\spt}\ar@{}[dr] \ar[d]  \ar@{^(->}[r]& \CN_n^{[n-1],\spt} \times \CN_{n+1}^{[n+1]}\ar[d]\\
	\wh{\CN}_{n}^{[n-1]}   \ar@{^(->}[r]& \CN_n^{[n-1]} \times \CN_{n+1}^{[n+1]} .
}
\]

The vertical arrows are  isomorphisms away from the worst points of $\CN_{n}^{[n-1]} $. Each automorphism  $g\in \U(\BY)$ induces a permutation of these worst points.
Let $\Lambda\in C^\flat$ be a type $n-1$ vertex lattice, with corresponding exceptional divisor $\Exc_\Lambda\subset \CN_n^{[n-1],\spt}$. Then $\mathrm{WT}(\Lambda)\in \CN_n^{[n-1]}\cap g\CN_n^{[n-1]}$ if and only if $g\in \mathrm{Stab}_{\U(\BY)}(\Lambda)$. In this case, the automorphism induces an automorphism of the exceptional divisor $g:\Exc_\Lambda\to\Exc_\Lambda$. 

\begin{conjecture}\label{conj: t=n-1 corr}
Let $n=2m+1$ be odd.
There exists $\varphi_{\mathrm{corr}}\in C_c^\infty(G')$ such that, if $\gamma\in G'(F_0)_\rs$ is matched with  $g\in G_{W_1}(F_0)_\rs$, then
 \begin{equation*}
  \Big(\big\langle   \wh{\CN}_n^{[n-1], \spt} , g   \wh{\CN}_n^{[n-1], \spt}  \big\rangle_{ \CN_{n}^{[n-1],\spt} \times \CN_{n+1}^{[n+1]} }-\big\langle   \wh{\CN}_n^{[n-1]} , g   \wh{\CN}_n^{[n-1]}  \big\rangle_{ \CN_{n}^{[n-1]} \times \CN_{n+1}^{[n+1]} }\Big)\cdot\log q=- \del\big(\gamma,  \varphi_{\mathrm{corr}} \big).
\end{equation*}
\end{conjecture}
In other words, the additional contribution to the intersection number in the splitting model (compared to the intersection occurring in \cite[Conj. 12.4]{RSZ2}) comes from the automorphisms of the exceptional divisors and their formal neighbourhoods, and we expect this will contribute an error term expressed by the orbit integral on the RHS.

\section{Proof of Theorem \ref{sm=big10} }\label{s:pfofthmBS}
In this section, we prove  Theorem \ref{sm=big10}. 

\subsection{Local model for $\CN_n^{[s,t]}$}
Let $F/F_0$ be a ramified quadratic extension with uniformizers $\pi^2=\pi_0$.
Let $(V,\phi)$ be a hermitian space of dimension $n$ over $F$. Let $\Lambda_s\subset \Lambda_t$ be vertex lattices of type $s$ and $t$, resp., where $s$ and $t$ are even numbers.
We have a natural lattice chain:
\begin{equation*}
    \Lambda_{s}\subset\Lambda_{t}\subset\Lambda^\vee_{t}\subset\Lambda_{s}^\vee\subset \pi^{-1}\Lambda_s.
\end{equation*}
We can further complete it into a polarized lattice chain $\Lambda_{[s,t]}$, see \S \ref{sec:parahoric}.

\begin{definition}\label{def:two-ind-lm}
The local model $\bM_{n}^{[s,t]}$ is a projective scheme over $\Spec O_{F}$. It represents the moduli problem that sends each $O_{F}$-algebra $R$ to the set of filtrations
\begin{equation}\label{equ:LM-Mst}
\begin{aligned}
\xymatrix{
\Lambda_{s,R}\ar[r]^{\lambda_{s}}&
\Lambda_{t,R}\ar[r]^{\lambda_{t}}&
\Lambda^\vee_{t,R}\ar[r]^{\lambda^\vee_{t}}&
\Lambda^\vee_{s,R}\ar[r]^{\lambda^\vee_{s}}&
\pi^{-1}\Lambda_{s,R}\\
\CF_{\Lambda_s}\ar@^{^(->}[u]\ar[r]&
\CF_{\Lambda_t}\ar@^{^(->}[u]\ar[r]&
\CF_{\Lambda_t^\vee}\ar@^{^(->}[u]\ar[r]&
\CF_{\Lambda_s^\vee}\ar@^{^(->}[u]\ar[r]&
\CF_{\pi^{-1}\Lambda_s}\ar@^{^(->}[u].&
}
\end{aligned}
\end{equation}
such that the following axioms are satisfied:
\begin{enumerate}[label=(\alph*)\,]
\item For all lattices $\Lambda$ occurring in \eqref{equ:LM-Mst}, $\CF_\Lambda$ is an $O_F\otimes_{O_{F_0}}R$-submodule of $\Lambda_{R}$, and an $R$-direct summand of rank $n$;
\item Any arrow $\lambda:\Lambda\to \Lambda'$ in \eqref{equ:LM-Mst} carries $\CF_\Lambda$ into $\CF_{\Lambda'}$. 
The isomorphism $\pi^{-1}\Lambda_{s,R}\overset{\pi}{\to}\Lambda_{s,R}$ identifies $\CF_{\pi^{-1}\Lambda_s}$ with $\CF_{\Lambda_s}$;
\item For $i=s$ and $t$, the perfect $R$-bilinear pairing
\begin{equation*}
\Lambda_{i,R} \times \Lambda^\vee_{i,R}
\xrightarrow{\langle-,-\rangle \otimes R} R
\end{equation*}
identifies $\CF_{\Lambda_i}^\perp$ with $\CF_{\Lambda_i^\vee}$ inside $\Lambda^\vee_{i,R}$; and
\item For all lattices $\Lambda$ occurring in \eqref{equ:LM-Mst},  $\CF_{\Lambda}$ satisfies the strengthened spin condition, see \S \ref{sec:spin}.
\end{enumerate}
\end{definition}

By \cite{Luo24}, the local model $\bM_n^{[s,t]}$ is flat. From the definition, we have natural projections
\begin{equation*}
\begin{aligned}
\xymatrix{
&\bM_n^{[s,t]}\ar[dl]_{p_s}\ar[dr]^{p_t}&\\
\bM_n^{[s]}&&\bM_n^{[t]},
}
\end{aligned}
\end{equation*}
which are isomorphisms over the generic fiber. When $s=t$, the projection $p_t:\bM_n^{[s,t]}\to \bM_n^{[t]}$ is an isomorphism.

When $n=2m$ is even, we will also consider the local model $\bM_{n}^{[n,n-2,t]}$ with three indices, this relates to the RZ space $\CN_{n}^{[n,n-2,t]}$.
By \cite[Prop. 9.12]{RSZ2}, the projection $\bM_{n}^{[n,n-2,t]}\to \bM_{n}^{[n-2,t]}$
is an isomorphism. But the corresponding map between RZ spaces is a trivial double cover, see \S \ref{sec:RZ-deeper}.

\subsection{Auxiliary space of the $\CY$-cycle}
In this subsection, we first recall the auxiliary spaces constructed in \cite{RSZ2}, then relate them with $\CY$-cycles. A similar  construction occurs in \cite{Yao24}.

Suppose from now on that $n=2m+1$ is an odd number and that $V$ is a split hermitian space of dimension $n+1$ over $F$.
Let $\Lambda_{n+1}\subset V$ be a $\pi$-modular lattice, i.e., a vertex lattice such that
\begin{equation*}
\pi\Lambda_{n+1}^\vee=\Lambda_{n+1}\stackrel{n+1}{\subset}\Lambda_{n+1}^\vee.
\end{equation*}

Let $u\in\Lambda^\vee_{n+1}\subset V$ be a unit-length vector. We have the orthogonal decomposition of the hermitian space $V=V^\flat\obot Fu$.

\begin{lemma}\label{lem:pi-modular-submmand}
The vector $\pi u\in \Lambda_{n+1}$ is primitive.
\end{lemma}
\begin{proof}
We cannot have $\pi^{-1} u\in \Lambda_{n+1}^\vee$: otherwise, we would have $u\in \pi\Lambda_{n+1}^\vee=\Lambda_{n+1}$, but then $(\pi^{-1}u,u)=\pi^{-1}$, contradicting the definition of the dual lattice. 
\end{proof}

Define  $\Lambda_{n-1}:=\Lambda_{n+1}+\langle u\rangle$, which is a vertex lattice of type $n-1$. 
The lattice chain $\Lambda_{n+1}\subset\Lambda_{n-1}$ defines the local model $\bM_{n+1}^{[n+1,n-1]}$. 
By \cite[Prop. 9.12]{RSZ2}, the natural projection $p_{n-1}:\bM_{n+1}^{[n+1,n-1]}\to \bM_{n+1}^{[n-1]}$ is an isomorphism.

The submodule $\langle u\rangle\subset \Lambda_{n-1}$ is a direct summand with orthogonal decomposition $\Lambda_{n-1}=\Lambda_{n-1}^\flat\obot\langle u\rangle$, where $\Lambda_{n-1}^\flat\subset V^{\flat}$ is a vertex lattice of type $n-1$. 
We define the local model $\bM_n^{[n-1]}$ using $\Lambda_{n-1}^\flat$.
Let $\iota:\bM_n^{[n-1]}\to \bM_{n+1}^{[n+1]}$ be the composition of the following maps
\begin{equation}\label{equ:define-y-cycle}
\begin{aligned}
\xymatrix{
\iota:\bM_n^{[n-1,t]}\ar@{^(->}[r]&\bM_{n+1}^{[n-1,t]}\simeq \bM_{n+1}^{[n+1,n-1,t]}\ar[r]&\bM_{n+1}^{[n+1,t]} ,
}
\end{aligned}
\end{equation}
where the closed immersion $\bM_n^{[n-1]}\hookrightarrow\bM_{n+1}^{[n-1]}$ is defined by sending 
\begin{equation*}
(\CF_{\Lambda_{n-1}^\flat},\CF_{\Lambda_{n-1}^{\flat,\vee}})\mapsto (\CF_{\Lambda_{n-1}},\CF_{\Lambda_{n-1}^\vee})=\Bigl(\CF_{\Lambda_{n-1}^\flat}\oplus R(\Pi-\pi) u,\CF_{\Lambda_{n-1}^{\flat,\vee}}\oplus R(\Pi-\pi) u\Bigr) ,
\end{equation*}
and where the identification in the middle is via the isomorphism $p_{n-1}$. 

\begin{proposition}[{\cite[Prop. 12.1]{RSZ2}}]\label{prop:hecke-imbedding}
The composition $\iota$ is a closed embedding.
\end{proposition}
\begin{proof}
By descent, it suffices to verify the statement after base change along an unramified extension.
This allows us to assume that $u$ has length $-1$ and $V^\flat$ is split. The assertion now follows from \cite[Prop. 12.1]{RSZ2}.
\end{proof}
Let $(\CF_{\Lambda_{n-1}^\flat},\CF_{\Lambda_{n-1}^{\flat,\vee}})$ be an $R$-point of $\bM_n^{[n-1]}$, we denote by $\bigl(\iota(\CF_{\Lambda_{n-1}^\flat}),\iota(\CF_{\Lambda_{n-1}^{\flat,\vee}})\bigr)\in \bM_{n+1}^{[n+1]}(R)$ its image under $\iota$.

\begin{definition}
Define the closed subscheme
\begin{equation*}
\bY^{[n+1]}(u)=\bZ^{[n+1]}(\pi u)\subset \bM^{[n+1]}_{n+1}
\end{equation*}
as the closed subscheme of $\bM^{[n+1]}_{n+1}$ which parametrizes filtrations $(\CF_\Lambda\subset \Lambda_R)\in \bM_{n+1}^{[n+1]}(R)$ that satisfy $R(\Pi-\pi) u \subset \CF_{\Lambda^\vee_{n+1}}$.
\end{definition}

\begin{lemma}\label{lem:Y-cycle}
The closed embedding $\iota$ factors through $\bY^{[n+1]}(u)$, 
\begin{equation*}
\begin{aligned}
\xymatrix{
\bM_n^{[n-1]}\ar@{^(->}[rr]^-{\iota}\ar@{-->}[dr]&&\bM_{n+1}^{[n+1]}\\
&\bY^{[n+1]}(u)\ar@{^(->}[ur]&
}
\end{aligned}
\end{equation*}
We also denote by $\iota:\bM_{n}^{[n-1]}\hookrightarrow \bY(u)$ the resulting map. 
\end{lemma}
\begin{proof}
Let $(\CF_{\Lambda_{n-1}^\flat},\CF_{\Lambda_{n-1}^{\flat,\vee}})$ be an $R$-point of $\bM_n^{[n-1]}$. By Proposition \ref{prop:hecke-imbedding}, it is a closed subfunctor of $\bM_{n+1}^{[n+1]}$ characterized by the subset of filtrations $\bM_{n+1}^{[n+1,n-1]}$ of the form:
\begin{equation}\label{equ:Mnn-1filtration}
\begin{aligned}
\xymatrix{
\Lambda_{n+1,R}\ar[r]^{\lambda_{n+1}}&
\Lambda_{n-1,R}\ar[r]^{\lambda_{n-1}}&
\Lambda_{n-1,R}^\vee\ar[r]^{\lambda_{n-1}^\vee}&
\Lambda^\vee_{n+1,R}
\\
\iota(\CF_{\Lambda_{n-1}^\flat})\ar@{^(->}[u]\ar[r]&
\CF_{\Lambda^\flat_{n-1}}\oplus R(\Pi-\pi)u\ar@{^(->}[u]\ar[r]&
\CF_{\Lambda_{n-1}^{\flat,\vee}}\oplus R(\Pi-\pi) u\ar@{^(->}[u]\ar[r]&
\iota(\CF_{\Lambda_{n-1}^{\flat,\vee}})\ar@{^(->}[u] .
}
\end{aligned}
\end{equation}
By Lemma \ref{lem:pi-modular-submmand}, the sublattice $\langle u\rangle \subset \Lambda_{n+1}^\vee$ is an $O_F$-direct summand. 
Hence, the transition map restricts to an isomorphism $\lambda_{n-1}^{\vee}:\langle u\rangle\otimes_{O_{F_0}}R\to \langle u\rangle\otimes_{O_{F_0}}R$ under which
$$\lambda_{n-1}^\vee(R(\Pi-\pi)u)=R(\Pi-\pi)u\subset \Lambda_{n+1,R}^\vee.$$
Therefore, the filtration in \eqref{equ:Mnn-1filtration} satisfies $R(\Pi-\pi)u\subset \iota(\CF_{\Lambda_{n-1}^{\flat,\vee}})$ and thus defines a point in $\bY^{[n+1]}(u)$.
\end{proof}

\begin{theorem}\label{thm:Yao-isom}
The induced map $\iota:\bM_n^{[n-1]}\hookrightarrow \bY^{[n+1]}(u)$ is an isomorphism.
\end{theorem}
\begin{proof}
In \cite[Thm. 5.5]{Yao24}, Yao proves that the closed immersion $\CN_n^{[n-1]}\hookrightarrow \CY^{[n+1]}(u)$ inside the RZ space $\CN_{n+1}^{[n+1]}$ is an isomorphism. The proof of Theorem \ref{thm:Yao-isom} follows the same strategy, and we briefly sketch the main ideas below.

First, in \cite[Lem. 5.11]{Yao24}, Yao shows that every point of $\CY^{[n+1]}(u)(\BF)$ is smooth by computing its tangent space using the local model (see footnote \ref{footnote:GM-of-special-cycle}); his argument implies that the special fiber of $\bY^{[n+1]}(u)$ is smooth.

Next, it is straightforward to verify that the closed immersion $\iota$ induces an isomorphism on the generic fiber, since both spaces are isomorphic to the Grassmannian $\Gr(1, F^{n-1})$. 

Finally, we claim that $\iota$ induces a bijection $\iota:\bM_n^{[n-1]}(\BF)\overset{\sim}{\to} \bY^{[n+1]}(u)(\BF)$ between the geometric points of the special fiber.
To be more precise, given any filtration $(\CF_{\Lambda_{n+1}},\CF_{\Lambda_{n+1}^{\vee}})\in \bY^{[n+1]}(u)(\BF)$, consider the intersection (cf. the proof of \cite[Prop. 12.1]{RSZ2}):
\begin{equation*}
\CF_{\Lambda_{n-1}^{\flat,\vee}}:=\CF_{\Lambda^\vee_{n+1}}\cap \lambda_{n-1}^{\vee}(\Lambda^{\flat,\vee}_{n-1,\BF})\subseteq \lambda^\vee_{n-1}(\Lambda^{\flat,\vee}_{n-1,\BF})\simeq \Lambda^{\flat,\vee}_{n-1,\BF}.
\end{equation*}
The last isomorphism is due to the fact that $\Lambda^{\flat,\vee}_{n-1}\subset \Lambda^{\vee}_{n-1}\subset \Lambda^\vee_{n+1}$ presents $\Lambda_{n-1}^{\flat,\vee}$ as a direct summand of $\Lambda^\vee_{n+1}$.
We aim to show that $\CF_{\Lambda_{n-1}^{\flat,\vee}}$ together with its isotropic complement $\CF_{\Lambda_{n-1}^{\flat}}$ defines a point in $\bM_n^{[n-1]}(\BF)$. This boils down to verifying the following conditions:
\begin{enumerate}
\item $\CF_{\Lambda_{n-1}^{\flat,\vee}}$ is orthogonal to itself with respect to the symmetric form induced from $\Lambda_{n+1,\BF}$ on $\Lambda_{n-1,\BF}^{\flat,\vee}$, this follows directly from the construction.
\item $\CF_{\Lambda_{n-1}^{\flat,\vee}}\subset \Lambda_{n-1}^{\flat,\vee}$ is a direct summand of rank $n-1$.
\item $\Pi(\CF_{\Lambda_{n-1}^{\flat,\vee}})$ is locally free of rank $\leq 1$, and hence satisfies the strengthened spin condition.
\end{enumerate}
Both (2) and (3) essentially follow from \cite[Prop. 5.14]{Yao24}, where Yao shows that the map $\CN_n^{[n-1]}(\BF)\hookrightarrow \CY^{[n+1]}(u)(\BF)$ is a bijection. He reduces the problem to a question about Dieudonn\'e lattices and carries out an explicit computation after fixing a basis using \cite[Lem. 5.10]{Yao24}.
Since for a Dieudonn\'e lattice $M$, the quotient $\uV M/\pi_0M\subset M/\pi_0M$ defines a geometric point in the local model after trivialization, his computation applies directly to the local model.
\end{proof}

\subsection{Auxiliary space of $\wh{\CN}_n^{[t]}$}
In this subsection, we construct and study the local model of the space $\wh{\CN}_n^{[t]}$ defined in \S \ref{ss: naive (n odd, n+1)}.
Let us recall the notations and assumptions.
We denote by $V=W_0$ the split hermitian space of dimension $n+1=2m+2$. Let $u\in V$ be a unit length vector.
We have the orthogonal decomposition $V=V^\flat\obot Fu$.
Let $t$ be an even integer with $0\leq t\leq n-1$.

Recall the lattice model $\wh{\BN}_n^{[t]}$ defined in \S \ref{sec:lattice (n odd, n+1)}.
Let $(\Lambda_{t}^\flat,\Lambda_{n+1})\in\wh{\BN}_n^{[t]}$.
By definition, we have $\Lambda_{n+1}\stackrel{1}{\subset}\Lambda_{n-1}:=\Lambda_{n+1}+\langle u\rangle$. 
 We have orthogonal decompositions
\begin{equation}\label{equ:orthogonal-decompo-of-lattice-chain}
\Lambda_{n-1}=\Lambda_{n-1}^\flat\obot\langle u\rangle, \quad\text{and}\quad \Lambda_{t}=\Lambda_{t}^\flat\obot\langle u\rangle .
\end{equation}
Here $\Lambda_{n-1}^\flat$ and $\Lambda_{t}^\flat$ are vertex lattices in $V^\flat$.

We define $\wh{\bM}_n^{[t]}$ as the closed subscheme of $\bM_{n+1}^{[n+1,t]}$ given by following condition:
\begin{equation*}
\CF_{\Lambda_t}=\CF_{\Lambda^\flat_{t}}\oplus R(\Pi-\pi)u,\quad\text{where}\quad [\CF_{\Lambda^\flat_{t}}\subset\Lambda^\vee_{t,R}]\in \bM_{n}^{[t]} .
\end{equation*}
In other words, it is characterized by the following cartesian diagram:
\begin{equation*}
\begin{aligned}
\xymatrix{
\wh{\bM}_n^{[t]}\ar[r]\ar[d]\ar@{}[rd]|{\square}	&\bM_{n+1}^{[n+1,t]}\ar[d]\\
\bM_n^{[t]}\cong \bZ^{[t]}(u)\ar@{^(->}[r]		& \bM_{n+1}^{[t]}.
}
\end{aligned}
\end{equation*}

On the other hand, the chain of lattices 
\begin{equation*}
    \Lambda_{n-1}^\flat\subset\Lambda_{t}^\flat\subset\Lambda_{t}^{\flat \vee}\subset\Lambda_{n-1}^{\flat \vee},
\end{equation*}
defines the local model $\bM_{n}^{[n-1,t]}$ as in Definition \ref{def:two-ind-lm}. 
Similar to \eqref{equ:define-y-cycle}, we define the map $\iota:\bM_n^{[n-1,t]}\to \bM_{n+1}^{[n+1,t]}$ as the composition:
\begin{equation*}
\begin{aligned}
\xymatrix{
\iota:\bM_n^{[n-1,t]}\ar@{^(->}[r]&\bM_{n+1}^{[n-1,t]}\simeq \bM_{n+1}^{[n+1,n-1,t]}\ar[r]&\bM_{n+1}^{[n+1,t]}.
}
\end{aligned}
\end{equation*}

\begin{lemma}
The map $\iota$ is a closed embedding, and it factors through  $\wh{\bM}_n^{[t]}\subset \bM_{n+1}^{[n+1,t]}$, which we will still denote by $\iota$.
\end{lemma}
\begin{proof}
The embedding $\bM_n^{[n-1,t]}\hookrightarrow \bM_{n+1}^{[n-1,t]}$ sends $(\CF_{\Lambda})_\Lambda$ to $\bigl(\CF_{\Lambda}\oplus R(\Pi-\pi)\bigr)_\Lambda$, and the composition $\bM_{n+1}^{[n-1,t]}\simeq \bM_{n+1}^{[n+1,n-1,t]}\to \bM_{n+1}^{[n+1,t]}$ does not change the filtrations $\CF_{\Lambda_t}$ and $\CF_{\Lambda_t^\vee}$. The assertion then follows from Proposition \ref{prop:hecke-imbedding} and Lemma \ref{lem:Y-cycle}.
\end{proof}

\begin{proposition}\label{prop:local-model}
The closed immersion $\iota:\bM_n^{[n-1,t]}\hookrightarrow \wh{\bM}_n^{[t]}$ is an isomorphism.
\end{proposition}
\begin{proof}
Since both of them are closed subschemes of $\bM_{n+1}^{[n+1,t]}$, we only need to show that any $R$-point of $\wh{\bM}_n^{[t]}$ lies in $\bM_n^{[n-1,t]}$.
Let $R$ be an $O_F$-algebra, any $R$-point of $\wh{\bM}_n^{[t]}$ represents a filtration of the form:
\begin{equation*}
\begin{aligned}
	\xymatrix{
		\Lambda_{n+1,R}\ar[r]^{\lambda_{n+1}}&
		\Lambda_{t,R}\ar[r]^{\lambda_{t}}&
		\Lambda^\vee_{t,R}\ar[r]^{\lambda^\vee_{t}}&
		\Lambda^\vee_{n+1,R}
		\\
		\CF_{\Lambda_{n+1}}\ar@{^(->}[u]\ar[r]&
		\CF_{\Lambda^\flat_{t}}\oplus R(\Pi-\pi)u\ar@{^(->}[u]\ar[r]&
		\CF_{\Lambda^{\flat,\vee}_{t}}\oplus R(\Pi-\pi)u\ar@{^(->}[u]\ar[r]&
		\CF_{\Lambda_{n+1}^\vee}\ar@{^(->}[u].
	}
\end{aligned}
\end{equation*}
Since $\lambda_t^\vee(R(\Pi-\pi)u)=R(\Pi-\pi)u\subset \CF_{\Lambda_{n+1}^\vee}$, we see that $(\CF_{\Lambda_{n+1}},\CF_{\Lambda_{n+1}^\vee})$ defines a point in $\bY^{[n+1]}(u)(R)\subset \bM_{n+1}^{[n+1]}(R)$ under the projection $\bM_{n+1}^{[n+1,t]}\to \bM_{n+1}^{[n+1]}$.
By Theorem \ref{thm:Yao-isom}, we have $(\CF_{\Lambda_{n+1}},\CF_{\Lambda^\vee_{n+1}})=\bigl(\iota(\CF_{\Lambda_{n-1}^\flat}),\iota(\CF_{\Lambda_{n-1}^{\flat,\vee}})\bigr)$ for some $(\CF_{\Lambda_{n-1}^\flat},\CF_{\Lambda_{n-1}^{\flat,\vee}})\in \bM_{n}^{[n-1]}(R)$. 

By the proof of \cite[Prop. 12.1]{RSZ2}, we further have
\[
\CF_{\Lambda_{n-1}^{\flat,\vee}}=\CF_{\Lambda_{n+1}^\vee}\cap \lambda_{n-1}(\Lambda_{n-1,R}^{\flat,\vee}),
\]
where $\lambda_{n-1}:\Lambda^\vee_{n-1,R}\to \Lambda^\vee_{n+1,R}$ is the natural transition map. 
Since 
\[
\lambda_t^\vee(\CF_{\Lambda^{\flat,\vee}_{t}}\oplus R(\Pi-\pi)u)\subseteq\CF_{\Lambda_{n+1}^\vee}\quad\text{and}\quad 
\CF_{\Lambda_t^{\flat,\vee}}\subset \Lambda_{t,R}^{\flat,\vee} ,
\] we conclude that the transition map $\lambda^{\flat,\vee}_{t}:\Lambda_{t}^{\flat,\vee}\to \Lambda_{n-1}^{\flat,\vee}$ carries $\CF_{\Lambda_t^{\flat,\vee}}$ to $\CF_{\Lambda_{n-1}^{\flat,\vee}}$. 
By duality, the transition map 
$\lambda^\flat_{n-1}:\Lambda_{n-1}^{\flat}\to \Lambda_t^\flat$ carries $\CF_{\Lambda_{n-1}^\flat}$ to $\CF_{\Lambda_t^\flat}$, hence the filtrations $(\CF_{\Lambda_{n-1}^\flat},\CF_{\Lambda_t^\flat},\CF_{\Lambda_t^{\flat,\vee}},\CF_{\Lambda_{n-1}^{\flat,\vee}})$ define an $R$-point of $\bM_n^{[n-1,t]}$. 
\end{proof}

\subsection{Proof of Theorem \ref{sm=big10}}
The proof of Theorem \ref{sm=big10} proceeds along the same lines as the proof of Corollary \ref{cor:pi-modular-iso}: we will show that the map $ \CN_n^{[n-1,t],\circ}\to \wh{\CN}_n^{[t]}$ induces a bijection on geometric points; then,  Proposition \ref{prop:local-model} shows that each point in $ \CN_n^{[n-1,t],\circ}(\BF)=\wh{\CN}_n^{[t]}(\BF)$ has the same first order deformation theory in either formal scheme, hence the map is also infinitesimally \'etale, and the theorem follows.

Recall from \S \ref{sec:RZ-deeper} that  $N=M(\BX)[\frac{1}{\pi_0}]$ is the common rational Dieudonn\'e module of the framing objects $\BX^{[n+1]}$ and $\BX^{[t]}$ of $\CN_{n+1}^{[n+1,t]}$. It is equipped with a hermitian form $h$ and a $\sigma$-linear operator $\tau: N\to N$.
By the isometry $\BV(\BX^{[n+1]})\otimes_{F_0}\breve{F}_0\simeq C\otimes_{F_0}\breve{F}_0\simeq N$, the unit length vector $u\in \BV(\BX^{[n+1]})$ corresponds to a unit length element in $N$, which we will still denote by $u$. We have a decomposition of the hermitian space $N=N^\flat\oplus Fu$, where $N^\flat$ is the rational Dieudonn\'e module of the framing object $\BY^{[n-1]}$ of $\CN_{n+1}^{[t,n-1]}$.

First of all, by computation on Dieudonn\'e modules, the geometric points $\CN_{n+1}^{[n+1,n-1,t]}(\BF)$ are given as follows by $O_{\breve{F}}$-lattices
\begin{equation*}
M_{n+1}\subset M_{n-1}\subset M_{t}\subset N
\end{equation*}
such that
\begin{itemize}
\item We have $\Pi M_i\subset\tau^{-1}(M_i)\subset\Pi^{-1}M_i$ for $i=n+1$, $n-1$ and $t$;
\item We have relations
\begin{equation*}
M_{t}\stackrel{\leq 1}{\subset}(M_t+\tau(M_t)),\quad\text{and}\quad M_{n-1}\stackrel{\leq 1}{\subset}(M_{n-1}+\tau(M_{n-1})),\quad\text{and}\quad
M_{n+1}\stackrel{ 1}{\subset}(M_{n+1}+\tau(M_{n+1}));
\end{equation*}
\item We have chains
\begin{equation*}
M_{n+1}\,\stackrel{1}{\subset}\,M_{n-1}\stackrel{\frac{n-t-1}{2}}{\subset}M_t\,\stackrel{t}{\subset}\, M_t^\vee\stackrel{\frac{n-t-1}{2}}{\subset}\,M_{n-1}^\vee\stackrel{1}{\subset}\,M_{n+1}^\vee=\Pi^{-1}M_{n+1}.
\end{equation*}
\end{itemize}
By definition, the vector $u\in N$ satisfies $\tau(u)=u$. By the definition of $\CN_n^{[n-1,t],\circ}$ (see \eqref{equ:three-squares}), we can write 
\begin{equation*}
	\CN_{n}^{[n-1,t],\circ}(\BF)=\left\{
	(M_{n+1},M_{n-1},M_t)\in \CN_{n+1}^{[t,n-1,n+1]}(\BF)\mid M_{t}=M_t^\flat\oplus\langle u\rangle, M_{n-1}=M_{n-1}^\flat\oplus\langle u\rangle
	\right\}.
\end{equation*}
The following lemma is straigthforward and ensures that $M^\flat$ are still Dieudonn\'e modules:
\begin{lemma}\label{lem:die-red}
Let $M\subset N$ be a vertex lattice such that $M=M^\flat\oplus\langle u\rangle$.
Then
\begin{enumerate}
\item The relation $\Pi M\subset\tau^{-1}(M)\subset \Pi^{-1}M$ is equivalent to $\Pi M^\flat\subset\tau^{-1}(M^\flat)\subset \Pi^{-1}M^\flat$.
\item The relation $M\subset M+\tau(M)$ is equivalent to $M^\flat\subset M^\flat+\tau(M^\flat)$. Moreover, we have 
\begin{equation*}
    \big(M+\tau(M)\big)/M\simeq \big(M^\flat+\tau(M^\flat)\big)/M^\flat.
\end{equation*}
\end{enumerate}
In particular, the submodule $M\subset M+\tau(M)$ has the same colength as the submodule $M^\flat\subset M^\flat+\tau(M^\flat)$.
\qed
\end{lemma}

Next, by computation on Dieudonn\'e modules, the geometric points $\wh{\CN}_{n}^{[t]}(\BF)$
are given as follows by $O_{\breve{F}}$-lattices (see again \eqref{equ:three-squares})
\begin{equation*}
M_{n+1}\subset M_{t}\subset N
\end{equation*}
such that
\begin{itemize}
\item We have $\Pi M_t\subset\tau^{-1}(M_t)\subset\Pi^{-1}M_t$ and  $\Pi M_{n+1}\subset\tau^{-1}(M_{n+1})\subset\Pi^{-1}M_{n+1};$
\item We have relations
\begin{equation*}
M_{t}\stackrel{\leq 1}{\subset}(M_t+\tau(M_t)),\quad\text{and}\quad  M_{n+1}\stackrel{ 1}{\subset}(M_{n+1}+\tau(M_{n+1}))\quad\text{and}\quad
M_{t}=M_{t}^\flat\oplus\langle u\rangle;  
\end{equation*}
\item We have decomposition 
\begin{equation*}
M_{t}=M_{t}^\flat\oplus\langle u\rangle,
\end{equation*}
where $M_t^\flat\subset N^\flat$ is a lattice which defines a geometric point of $\CN_n^{[t]}(\BF)$;
\item We have chains
\begin{equation*}
M_{n+1}\stackrel{\frac{n-t+1}{2}}{\subset}M_t\,\stackrel{t}{\subset}\, M_t^\vee\stackrel{\frac{n-t+1}{2}}{\subset} M_{n+1}^\vee=\Pi^{-1}M_{n+1}
\end{equation*}
\end{itemize}

The induced map $\iota:\CN_n^{[n-1,t],\circ}(\BF)\to \wh{\CN}_n^{[t]}(\BF)$ forgets the lattice $M_{n-1}$. We show that $\iota$ is a bijection.

For any $(M_{n+1},M_t)\in \wh{\CN}_n^{[t]}(\BF)$, since $M_t=M_t^\flat\oplus \langle u\rangle$, we also have a decomposition $M_t^\vee=M_t^{\flat,\vee}\oplus\langle u\rangle$. By Lemma \ref{lem:pi-modular-submmand}, $u\in M_{n+1}^\vee$ is a primitive vector.
Therefore, we have a decomposition
\begin{equation*}
    M_{n+1}=M_{n-1}^\flat\oplus\langle \pi u\rangle\quad\text{and}\quad M^\vee_{n+1}=\Pi^{-1}M_{n-1}^\flat\oplus\langle u\rangle,
\end{equation*}
where $M_{n-1}^\flat\subset N^\flat$ is a lattice. 
The inclusion $M_{n+1}\subset M_t$ has colength $\frac{n-t+1}{2}$. Hence, the inclusion $M_{n-1}^{\flat}\subset M_t^\flat$ is a inclusion of colength $\frac{n-t-1}{2}$.
Therefore, $M_{n-1}^\flat\subset N^\flat$ is a vertex lattice of type $n-1$. 
Define $M_{n-1}:=M_{n-1}^\flat\oplus\langle u\rangle$. By Lemma \ref{lem:die-red}, the triple $(M_{n+1},M_{n-1},M_t)$ defines a point in $\CN_n^{[n-1,t],\circ}(\BF)$.

Conversely, for any $(M_{n+1},M_{n-1},M_t)$ in $\CN_n^{[n-1,t],\circ}(\BF)$, it is straightforward to verify that $(M_{n+1},M_t)$ defines a point in $\wh{\CN}_n^{[t]}(\BF)$. Moreover, from the construction we see that this defines a bijection.\qed

\section{The proof of Conjecture \ref{conj (n odd,t) Y} for type $(n-1,t)=(0,0)$ }
\label{sec:(0,0)-case}

In this section, we prove the Conjecture \ref{conj (n odd,t) Y} for type $(n-1,t)=(0,0)$. The proof proceeds as follows. In \S \ref{sec:inhomogeneous}, we first reduce Conjecture \ref{conj (n odd,t) Y} to the inhomogeneous setting, allowing us to apply the germ expansion results from \cite{Mih-AT} and \cite{RSZ2}; see Theorem \ref{thm:germ-expan}. Then, in \S \ref{sec:iwahori-exceptional-isom}, we recall the exceptional isomorphisms: one between $\CN_{2,1}^{[0]}$ and the Iwahori level Lubin–Tate moduli space $\CM_{\Gamma_0(\pi_0)}$, and another between $\CN_2^{[2]}$ and the (hyperspecial level) Lubin–Tate moduli space $\CM$, as constructed in \cite{RSZ1} and \cite{RSZ2}. These isomorphisms allow us to relate the cycles $\wt{\CM}_1^{[0],\pm}$ on $\CN_{2,1}^{[0]}$ to canonical and quasi-canonical lifts on $\CM_{\Gamma_0(\pi_0)}$, as studied in \S \ref{sec:describe cycles}.
We then compute the intersection multiplicity by reducing to calculations involving canonical lifts in \S \ref{sec:intersection numbers}, and compare the outcome with the corresponding analytic computation in \S \ref{sec:proof of 00}.

\subsection{Inhomogeneous setting}\label{sec:inhomogeneous}
We reduce Conjecture \ref{conj (n odd,t) Y} to the inhomogeneous setting.
First, since the bottom row of the diagram \eqref{eq: N circ} is $H(F_0)$-equivariant, the embedding $\CN_n^{[n-1],\circ}\hookrightarrow \CN_{n+1}^{[n-1,n+1]}$ is $H(F_0)$-equivariant. As the projection $\CN_{n+1}^{[n-1,n+1]}\to \CN_{n+1}^{[n+1]}$ is also $H(F_0)$-equivariant, the induced embedding $\wt{\CM}_{n}^{[t]}\hookrightarrow \CN_{n+1}^{[n+1,t]}$ is $H(F_0)$-equivariant. 

Since the group action on the splitting models is defined via changes in the framing object (just as for RZ spaces), it follows that the intersection number in Conjecture \ref{conj (n odd,t) Y} remains unchanged if we replace $g = (g_1, g_2) \in G_{W_1}(F_0)$ with $g’ = (1, g_1^{-1} g_2)$.
The same reduction applies on the analytic side via a change of variables. Thus, without loss of generality, we may assume that 
$$G'(F_0)_{\rs}\ni \gamma=(1,\gamma)\longleftrightarrow g = (1, g) \in G_{W_i}(F_0)_{\rs}$$ in Conjecture \ref{conj (n odd,t) Y}.

For $f\in C_c^{\infty}(G_W)$, recall from \cite[\S 4]{LRZ2} the definition of the function $f^\sharp\in C_c^{\infty}(\U(W))$ given by
\begin{equation}\label{equ:define-inhomo-transfer}
f^{\sharp}(g)=\int_{\U(W^\flat)}f(h,hg)dh,\quad g\in \U(W)(F_0).
\end{equation}

One also has $\varphi^\sharp$ for $\varphi\in C_c^{\infty}(G')$, whose precise definition is not needed here, see \cite[(4.2.16)]{LRZ2}. We recall the following result:
\begin{proposition}\label{prop:homo-to-inhomo}
\begin{altenumerate}
\item The function $\varphi'\in C_c^\infty(G')$ is a transfer of $(f_1,f_2)\in C_c^{\infty}(G_{W_0})\times C_c^{\infty}(G_{W_1})$ if and only if $\varphi^{\prime,\sharp}$ is a transfer of $(f_0^\sharp,f_1^\sharp)\in C_c^{\infty}(\U(W_0))\times C_c^{\infty}(\U(W_1))$.
\item Let $\varphi'\in C_c^{\infty}(G')$ and $\gamma\in \U(W_1)(F_0)_{\rs}$, then
$$
\partial\!\Orb((1,\gamma),\varphi')=2\partial\!\Orb(\gamma,\varphi^{\prime,\sharp})
$$
\end{altenumerate}
\end{proposition}
\begin{proof}
(i) is \cite[Cor. 4.2.5]{LRZ2}, (ii) is \cite[Cor. 4.2.4]{LRZ2}, see also \cite[\S 5]{RSZ1}.
\end{proof}

With this, we may now state the inhomogeneous version of Conjecture \ref{conj (n odd,t) Y}. 

 \begin{corollary}\label{cor:inhomo-version}
 Let $n$ be odd, and let $t=n-1$. 
 Conjecture \ref{conj (n odd,t) Y}, (ii) is equivalent to the following statement:
 
$\bullet$  there exists $\varphi'\in C_c^\infty(S_n(F_0))$ with transfer $(\vol(K_n^{[n-1],\circ})^{-1} \varphi^{[n+1,t]}_{n+1},0)\in C_c^\infty(\U(W_0))\times C_c^\infty(\U(W_1))$ such that, if $\gamma\in S_n(F_0)_\rs$ is matched with  $g\in \U(W_1)(F_0)_\rs$, then
 \begin{equation*}
\left\langle\wt\CM^{[t],+,\spt}_{n} , g\wt\CM^{[t],+,\spt}_{n} \right\rangle_{ \CN_{n}^{[n-1]} \times \CN_{n+1}^{[t],\spt} }\cdot\log q=- \del\big(\gamma,  \varphi' \big).
\end{equation*}
Similarly, Conjecture \ref{conj (n odd,t) Y}, (iii) is equivalent to the following statement: 

$\bullet$ there exists $\varphi'\in C_c^\infty(S_n)$ with transfer $(\vol(K_n^{[n-1],\circ})^{-1} h_0\cdot\varphi^{[n+1,t]}_{n+1},0)\in C_c^\infty(\U(W_0))\times C_c^\infty(\U(W_1))$ such that, if $\gamma\in G'(F_0)_\rs$ is matched with  $g\in \U(W_1)(F_0)_\rs$, then
 \begin{equation*}
\left\langle\wt\CM^{[t],+,\spt}_{n} , g\wt\CM^{[t],-,\spt}_{n} \right\rangle_{ \CN_{n}^{[n-1]} \times \CN_{n+1}^{[t],\spt} }\cdot\log q=- \del\big(\gamma,  \varphi' \big).
\end{equation*}
Here $h_0\in \U(W_0)(F_0)$ is an element in $K_{n}^{[n-1]}\setminus K_{n}^{[n-1],\circ}$, and $h_0\cdot\varphi_{n+1}^{[n+1,t]}(g):=\varphi_{n+1}^{[n+1,t]}(h_0g).$
\end{corollary}
\begin{proof}
By Proposition \ref{prop:homo-to-inhomo}, we are reduced to computing 
$$
\Bigl(\vol(K_n^{[n-1],\circ})^{-2} {\bf 1}_{K_n^{[n-1],\circ}h_0}\otimes \varphi^{[n+1,t]}_{n+1}\Bigr)^\sharp,\quad \text{for some }h_0\in K_n^{[n-1]}.
$$
First, since $K_{n}^{[n-1],\circ}\subset K_{n+1}^{[n-1,n+1]}\subset K_{n+1}^{[n+1]}$, for any $k\in K_n^{[n-1],\circ}$ and $x\in \U(W_0)(F_0)$, we have 
\begin{align*}
\varphi_{n+1}^{[n+1,n-1]}(kx)={}&{\bf 1}_{K^{[n+1]}_{n+1}K^{[n-1]}_{n+1}}\ast {\bf 1}_{K^{[n-1]}_{n+1}K^{[n+1]}_{n+1}}(kx)\\
={}&\int_{\U(W_0)}{\bf 1}_{K_{n+1}^{[n+1]}K_{n+1}^{[n-1]}}(h){\bf 1}_{K_{n+1}^{[n-1]}K_{n+1}^{[n+1]}}(h^{-1}kx)dh\\
={}&\int_{\U(W_0)}{\bf 1}_{K_{n+1}^{[n+1]}K_{n+1}^{[n-1]}}(kh){\bf 1}_{K_{n+1}^{[n-1]}K_{n+1}^{[n+1]}}(h^{-1}x)dh\\
={}&\int_{\U(W_0)}{\bf 1}_{K_{n+1}^{[n+1]}K_{n+1}^{[n-1]}}(h){\bf 1}_{K_{n+1}^{[n-1]}K_{n+1}^{[n+1]}}(h^{-1}x)dh\\
={}&\varphi_{n+1}^{[n+1,n-1]}(x),
\end{align*}
i.e., the Schwartz function $\varphi^{[n+1,n-1]}_{n+1}$ is left $K_n^{[n-1],\circ}$-invariant.
Therefore, for any $h_0\in K_n^{[n-1]}$,  we have
\begin{align*}
\Bigl(\vol(K_n^{[n-1],\circ})^{-2} {\bf 1}_{K_n^{[n-1],\circ}h_0}\otimes \varphi^{[n+1,t]}_{n+1}\Bigr)^{\sharp}(g)
={}&\vol(K_n^{[n-1],\circ})^{-2}\int_{\U(W^\flat)} {\bf 1}_{K_n^{[n-1],\circ}h_0}(h) \varphi^{[n+1,t]}_{n+1}(hg) dh\\
={}&\vol(K_n^{[n-1],\circ})^{-2}\int_{\U(W^\flat)} {\bf 1}_{K_n^{[n-1],\circ}}(h) \varphi^{[n+1,t]}_{n+1}(hh_0g) dh\\
={}&\vol(K_n^{[n-1],\circ})^{-1}\varphi^{[n+1,t]}_{n+1}(h_0g).
\end{align*}
The assertion then follows from \cite[Cor. 4.2.4 and Cor. 4.2.5]{LRZ2}, see also \cite[\S 5]{RSZ1}.
\end{proof}
For the rest of the section, we take $n=1$ and $t=0$.

\subsection{Iwahori level Lubin-Tate moduli space and exceptional isomorphism}\label{sec:iwahori-exceptional-isom}
To describe the cycles and compute the intersection numbers on the geometric side, we recall some constructions in \cite[\S 8 and \S 9]{RSZ2}.
\begin{definition}
\begin{altenumerate}
\item The  Lubin-Tate moduli space $\CM$ of formal $O_{F_0}$-modules of dimension $1$ and relative height $2$ is the formal scheme representing the functor over $\Spf O_{\breve{F}_0}$ that associates to each $\Spf O_{\breve{F}_0}$-scheme $S$ the set of isomorphism classes of pairs $(Y,\rho_Y)$, where $Y$ is a formal $O_{F_0}$-module of dimension $1$ and relative height $2$ over $S$ and $\rho_{Y}:Y\times_S\ov{S}\to \BE\times_{\Spec \BF}\ov{S}$ is an $O_{F_0}$-linear quasi-isogeny of height 0.

\item The Iwahori level Lubin-Tate moduli space $\CM_{\Gamma(\pi_0)}$ is the formal scheme  representing the functor over $\Spf O_{\breve{F}_0}$ that associates to each $\Spf O_{\breve{F}_0}$-scheme $S$ the set of isomorphism classes of quadruples 
$$
(Y,Y',\phi:Y\to Y',\rho_Y),
$$
where $Y$ and $Y'$ are formal $O_{F_0}$-modules of dimension $1$ and relative height $2$   over $S$, $\phi$ is an $O_{F_0}$-linear isogeny of degree $q$ and $\rho_{Y}:Y\times_S\ov{S}\to \BE\times_{\Spec \BF}\ov{S}$ is an $O_{F_0}$-linear quasi-isogeny of height $0$. 
\end{altenumerate}
\end{definition}
From $(Y,Y',\phi:Y\to Y',\rho_Y)$, we deduce the following composition of quasi-isogenies 
\begin{equation*}\label{equ:framing for Y'}
\rho_{Y'}:Y'\times_S \ov{S}\overset{\phi^{-1}}{\to} Y\times_S \ov{S}\overset{\rho_Y}{\to}\BE\times_{\Spec \BF}\ov{S}\overset{\iota_\BE(\pi)}{\to}\BE\times_{\Spec \BF}\ov{S},
\end{equation*}
which is of height zero. Then the pullbacks $\rho_Y^*(\lambda_{\BE})$ and $\rho_{Y'}^*(\lambda_{\BE})$ lift to principal polarizations $\lambda_Y$ of $Y$ and $\lambda_{Y'}$ of $Y'$ since the same holds for the universal object over the  Lubin-Tate moduli space $\CM$.
We denote by $\phi':Y'\to Y$  the unique isogeny such that $\phi\circ\phi'=\iota(\pi)$ and $ \phi'\circ\phi=\iota(\pi).$

Consider the following framing object for $\CN_{2,1}^{[0]}=\CN_2^{[0]}$ (cf. \cite[\S 8 and Ex.  9.4]{RSZ2}):
$$
(\BX_2,\iota_{\BX_2},\lambda_{\BX_2}):=\Bigl(
\BE\times\BE,
\begin{pmatrix}
&\iota_{\BE}\\
\iota_{\BE}&
\end{pmatrix},-2(\lambda_{\BE}\times\lambda_{\BE})
\Bigr).
$$
We have an  isomorphism $(\CM_{\Gamma_0(\pi_0)})_{O_{\breve{F}}}\to \CN_2^{[0]}$ given by 
\begin{equation}\label{equ:exceptional-isom-1}
(Y,Y',\phi,\rho_Y)\mapsto \Bigl(Y\times Y',\begin{pmatrix}
&\phi'\\
\phi&
\end{pmatrix},-2(\lambda_Y\times \lambda_{Y'}),\rho_Y\times\rho_{Y'}\Bigr) ,
\end{equation}
comp. \cite[Prop. 8.2]{RSZ2}.
We consider another framing object
$$
(\wt{\BX}_2,\iota_{\wt{\BX}_2},\lambda_{\wt{\BX}_2}):=(\BE\times\ov{\BE},\iota_{\BE}\times \iota_{\ov{\BE}},\lambda_{\BE}\times\lambda_{\ov{\BE}}).
$$
By \cite[Ex. 9.4]{RSZ2}, there is an $O_F$-linear isomorphism 
$$
\psi_0=
\begin{pmatrix}
1&1\\ 1&-1
\end{pmatrix}: 
(\BX_2,\iota_{\BX_2},\lambda_{\BX_2})\to (\wt{\BX}_2,\iota_{\wt{\BX}_2},\lambda_{\wt{\BX}_2}).
$$
Combining with \eqref{equ:exceptional-isom-1}, we have an isomorphism 
\begin{equation}\label{equ:exceptional-isomorphism}
(\CM_{\Gamma_0(\pi_0)})_{\Spf O_{\breve{F}}}\to \CN_2^{[0]},\quad (Y,Y',\phi,\rho_Y)\mapsto \Bigl(Y\times Y',\begin{pmatrix}
&\phi'\\
\phi
\end{pmatrix}, -2(\lambda_Y\times\lambda_{Y'}),\psi_0\circ(\rho_Y\times\rho_{Y'})\Bigr).
\end{equation}
For the remainder of this section, we fix $\wt{\BX}_2$ as the framing object for $\CN_2^{[0]}$.

Next, we compare  the group actions along the isomorphism \eqref{equ:exceptional-isomorphism}.
Let $D$ be the unique quaternion algebra over $F_0$ and let $F\hookrightarrow D$ be the fixed embedding defined via $\iota_{\BE}$.
Recall from \cite[\S 15.1]{RSZ2} that for $(\wt{\BX}_2,\iota_{\wt{\BX}_2},\lambda_{\wt{\BX}_2})$, we have
$$
\End_{O_{F_0}}(\wt{\BX}_2)=\End_{O_{F_0}}(\BE^2)=M_2(O_D),\,\text{with $O_F$-action }\pi\mapsto \begin{pmatrix}
\pi&\\
&-\pi
\end{pmatrix},
$$
hence  
$$\End^{\circ}_{O_F}(\wt{\BX}_2)=\left\{
\begin{pmatrix}
a&b\\ c&d 
\end{pmatrix}\Biggl|\,\,
a,d\in F, b,c\in D^-
\right\},$$
where $ D=F\oplus D^-$ is the eigenspace decomposition under the conjugation action of $\pi$. 
The Rosati involution is given by 
$$x\mapsto x^\dagger:=\lambda_{\wt{\BX}_2}^{-1}\circ x^\vee\circ\lambda_{\wt{\BX}_2},\quad
\begin{pmatrix}
a&b\\ c&d
\end{pmatrix}\mapsto \begin{pmatrix}
\ov a&\ov c\\ \ov b&\ov d
\end{pmatrix},
$$
where $x\mapsto \ov{x}$ is the main involution of $D$. The unitary group $ \U(W_1)$ may be explicitly presented as 
\begin{equation}\label{equ:non-split unitary rank 2}
\U(W_1)(F_0)=\Biggl\{
\begin{pmatrix}
1&\\
&\alpha
\end{pmatrix}
\begin{pmatrix}
a&b\\
b&a
\end{pmatrix}\in M_2(D)\,\Biggl|
\begin{array}{l}
a\in F,\quad b\in D^-,\\
Na+Nb=1,\quad \alpha\in F^1
\end{array}
\Bigr.\Biggr\}.
\end{equation}

Let $D^1$ be the group of elements of norm $1$. If we fix a basis element $\zeta\in D^{-}\cap D^1$, then we may also express these presentations in terms of special embeddings into $M_2(F)$:
$$
g=\begin{pmatrix}
1&\\
&\alpha
\end{pmatrix}
\begin{pmatrix}
a&b\\
b&a
\end{pmatrix}\in \U(W_1)(F_0)\subset M_2(D)\quad\text{identifies with}\quad
\begin{pmatrix}
a&b\zeta^{-1}\\
\ov{\alpha b \zeta}& \ov{\alpha a}
\end{pmatrix}\in M_2(F).
$$
Hence, we have $\det g=a\ov{\alpha a}-(b\zeta^{-1})\cdot \ov{\alpha b \zeta}=\ov{\alpha}(Na+Nb)=\ov{\alpha}$.
In particular, $g$ lies in $\SU(W_1)(F_0)$ if and only if $\alpha=1$.

Next, we compare the actions of $D^1$ on $(\CM_{\Gamma_0(\pi_0)})_{\Spf O_{\breve{F}}}$ with the action of $\SU(W_1)(F_0)\subset \U(W_1)(F_0)$ on $\CN_2^{[0]}$.
Any $x\in D^\times=\End(\BE)^\times$ with norm a unit in $O_{F_0}$ induces an action on the Iwahori level Lubin-Tate moduli space by 
$$
(Y,Y',\phi,\rho_Y)\mapsto (Y,Y',\phi,\iota_{\BE}(x)\circ \rho_Y).
$$
It acts on $\rho_{Y'}$ via conjugation: $\iota_{\BE}(\pi)\iota_{\BE}(x)\iota_{\BE}(\pi)^{-1}=\iota_{\BE}(\pi x\pi^{-1})$.
By \eqref{equ:exceptional-isom-1}, this induces an action on $\CN_2^{[0]}$ by
$$
x\mapsto \begin{pmatrix}
1&1\\ 1&-1
\end{pmatrix}\begin{pmatrix}
x&\\ & \pi^{-1}x\pi
\end{pmatrix}\begin{pmatrix}
1&1\\ 1&-1
\end{pmatrix}^{-1}=\begin{pmatrix}
\frac{x+\pi x \pi^{-1}}{2}& \frac{x-\pi x\pi^{-1}}{2}\\
\frac{x-\pi x \pi^{-1}}{2}& \frac{x+\pi x\pi^{-1}}{2}
\end{pmatrix}\in \End_{O_{F_0}}(\wt{\BX}_2).
$$
Write $x=a+b\in F\oplus D^{-}$, then $\pi x\pi^{-1}=a-b$. Hence under the isomorphism \eqref{equ:exceptional-isomorphism}, the element $x=a+b\in \End_{O_{F_0}}(\BE)$ corresponds to $\begin{pmatrix}
a&b\\ b&a
\end{pmatrix}\in M_2(D)=\End_{O_{F_0}}(\wt{\BX}_2)$. The resulting identification $D^{1}\cong \SU(W_1)(F_0)$ is then given by 
$$
x=a+b\in D^1\longleftrightarrow \begin{pmatrix}
a&b\\ b&a
\end{pmatrix}\ni \SU(W_1)(F_0).
$$

Next, we recall the orbit matching in the inhomogeneous setting as described in \cite[§15]{RSZ2}.
On the symmetric space $S(F_0)=S_2(F_0)$, we write an element as
$$
\gamma=\begin{pmatrix}
a&b\\ c&d
\end{pmatrix}\in S(F_0).
$$
Then $\gamma$ is regular semisimple if and only if $bc\neq 0$, in which case we may write $\gamma$ as (see \cite[(15.1)]{RSZ2})
$$
\gamma=\gamma(a,b):=\left(\begin{array}{cc}
1&\\ & -b/\ov{b}
\end{array}\right)\left(\begin{array}{cc}
a&b\\
-(1-Na)/b& \ov{a}
\end{array}\right)\in S_0(F_0)_{\rs}
$$
for $a\in F\setminus F^1$ and $b\in F^\times$. 
By \cite[\S 15.2]{RSZ2},
$$
\gamma(a,b)\in S(F_0)_{\rs}\quad\text{matches}\quad \begin{pmatrix}
1&\\
&\alpha
\end{pmatrix}
\begin{pmatrix}
a'&b'\\
b'&a'
\end{pmatrix}\in\U(W_1)(F_0)_{\rs}
$$
if and only if $a=a'$ and $-b/\ov{b}=\det \gamma(a,b)=\det g=\ov{\alpha}$.

\subsection{Description of cycles}\label{sec:describe cycles}
We describe the cycles $\wt{\CM}_0^{[0]}$ and $\wt{\CM}_0^{[0],\pm}$ under the exceptional isomorphism \eqref{equ:exceptional-isomorphism}.
Recall from \cite[(12.1)]{RSZ2} that there is an embedding $\CN_1^{[0]}\hookrightarrow \CN_2^{[0]}$, identifying $\CN_1^{[0]}$ with the special cycle $\CZ(u)$, where $u=\begin{pmatrix}
0\\1
\end{pmatrix}\in \Hom(\ov{\BE},\wt{\BX}_2)$. 

By \cite[Ex. 12.2]{RSZ2}, there is also a closed embedding  $\Spf O_{\breve{F}}\hookrightarrow (\CM_{\Gamma_0(\pi_0)})_{O_{\breve{F}}}$ given by sending the canonical lift $(\CE,\rho_{\CE})$ to $ (\CE,\CE,\iota(\pi),\rho_{\CE})$. 
This identifies with the embedding $\CN_1^{[0]}\hookrightarrow \CN_2^{[0]}$ under the isomorphism \eqref{equ:exceptional-isomorphism}.

By \cite[Ex. 12.2]{RSZ2}, the cartesian diagram \eqref{eq: N circ}, after applying the exceptional isomorphism \eqref{equ:exceptional-isomorphism}, becomes
\begin{equation*}
\begin{aligned}
\xymatrix{
\CN_0^{[0],\circ}\ar@{}[rd]|{\square}\ar@{^(->}[r]\ar[d]&\CN_2^{[0,2]}\ar[d]\\
\CN_1^{[0]}\ar@{^(->}[r]&\CN_2^{[0]}
}\end{aligned}
\quad\longleftrightarrow \quad\begin{aligned}
\xymatrix{
\Spf O_{\breve{F}}\amalg \Spf O_{\breve{F}}\ar@{}[rd]|{\square}\ar@{^(->}[r]\ar[d]&(\CM_{\Gamma_0(\pi_0)})_{O_{\breve{F}}}\amalg (\CM_{\Gamma_0(\pi_0)})_{O_{\breve{F}}}\ar[d]\\
\Spf O_{\breve{F}}\ar@{^(->}[r]&(\CM_{\Gamma_0(\pi_0)})_{O_{\breve{F}}},
}
\end{aligned}
\end{equation*}
where the upper horizontal arrow respects the disjoint sum decomposition.

Next, the composition of maps  $$\CN_1^{[0],\circ}\hookrightarrow \CN_2^{[0,2]}\to \CN_2^{[2]}$$
becomes in terms of the exceptional isomorphism 
$$
\Spf O_{\breve{F}}\amalg \Spf O_{\breve{F}}\hookrightarrow (\CM_{\Gamma_0(\pi_0)})_{O_{\breve{F}}}\amalg (\CM_{\Gamma_0(\pi_0)})_{O_{\breve{F}}}\overset{\varphi}{\to} \CM_{O_{\breve{F}}}\amalg \CM_{O_{\breve{F}}},
$$
where $\varphi$ is the disjoint sum of two morphisms $p_1, p_2: (\CM_{\Gamma_0(\pi_0)})_{O_{\breve{F}}}\to \CM_{O_{\breve{F}}}$. Here   $p_1$ maps $(Y,Y',\phi,\rho_Y)$ to $(Y,\rho_Y)$ and $p_2$ maps $(Y,Y',\phi,\rho_Y)$ to $(Y',\rho_{Y'})$,  comp. \cite[diagram on page 1126]{RSZ2}. In particular, the composition is the standard canonical lift to the Lubin-Tate moduli space on each summand, see \cite[Ex. 12.2]{RSZ2}.

Under the exceptional isomorphism, the diagram \eqref{eq:def Mt} now becomes
\begin{equation}\label{eqCM0}
\begin{aligned}
\xymatrix{
\wt{\CM}_1^{[0]}\ar@{}[rd]|{\square}\ar@{^(->}[r]\ar[d]&(\CM_{\Gamma_0(\pi_0)})_{O_{\breve{F}}}\amalg (\CM_{\Gamma_0(\pi_0)})_{O_{\breve{F}}}\ar[d]^-{\varphi}\\
\Spf O_{\breve{F}}\amalg \Spf O_{\breve{F}}\ar@{^(->}[r]&\CM_{O_{\breve{F}}}\amalg \CM_{O_{\breve{F}}}.
}
\end{aligned}
\end{equation}
In particular, we have the decomposition $\wt{\CM}_1^{[0]}=\wt{\CM}_1^{[0],+}\amalg \wt{\CM}_1^{[0],-}$, where $\wt{\CM}_1^{[0],+}=p_1^*\Spf O_{\breve{F}}$ and $\wt{\CM}_1^{[0],-}=p_2^*\Spf O_{\breve{F}}$, such that 
\begin{equation}\label{equ:cycle in the iwahori lubin tate}
\wt{\CM}_1^{[0],+}\amalg \wt{\CM}_1^{[0],-}=p_1^*\Spf O_{\breve{F}}\amalg p_2^*\Spf O_{\breve{F}}\subset (\CM_{\Gamma_0(\pi_0)})_{O_{\breve{F}}}\amalg (\CM_{\Gamma_0(\pi_0)})_{O_{\breve{F}}}.
\end{equation}

By Miracle flatness \cite[\href{https://stacks.math.columbia.edu/tag/00R4}{00R4}]{stacksproject}, the projection maps  $p_i: \CM_{\Gamma_0(\pi_0)}\to \CM$ are finite flat for $ i\in\{1,2\}$. Therefore, the composition $\wt{\CM}_1^{[0]}\to \CN_1^{[0],\circ}\to \Spf O_{\breve{F}}$ is flat, hence Conjecture \ref{conjMflat} holds in this case.

To give a more precise description of $\wt{\CM}_1^{[0]}$, we recall the quasi-canonical lifting divisors of $\CM_{\Gamma_0(\pi_0)}$, introduced in \cite[\S 3.2]{Shi23a}.
For any $j\geq 0$, let $(\CE_{j},\rho_j)$ be the quasi-canonical divisor on $\CM$ of level $j$ over $\CW_j:=\Spf W_j$, see \cite[Def. 3.1]{Shi23a}. 

\begin{proposition}[\protect{\cite[Prop. 3.6]{Shi23a}}]\label{prop:quasi-canonical-iwahori}
For any $j\geq 0$, let 
$$m(j):=\left\{\begin{array}{ll}
j-1&\text{if $j\geq 1$};\\
0&\text{if $j=0$}.
\end{array}\right.$$
For all $j\geq 0$, there exist quasi-canonical lifts $\CE_{j}$ of level $j$ with a morphism $\beta_j:\CE_{m(j)}\to \CE_j$ such that the following diagrams commute:
\begin{equation}\label{equ:quasi-canonical lift in iwahori}
\begin{aligned}
\xymatrix{
\CE_{0}\otimes \BF\ar[r]^{\beta_0\otimes \BF}\ar[d]^{\rho_0}&\CE_0\otimes \BF\ar[d]^{\rho_0'}\\
\BE\ar[r]^{\iota(\pi)}&\BE
}\end{aligned}\quad\text{and}\quad \begin{aligned}
\xymatrix{
\CE_{m(j)}\otimes \BF\ar[r]^{\beta_j\otimes \BF}\ar[d]^{\rho_{m(j)}}&\CE_j\otimes \BF\ar[d]^{\rho_j}\\
\BE\ar[r]^{\iota(\pi)}&\BE,
}
\end{aligned}
\end{equation}
where $j\geq 1$ in the second diagram. These define Weil divisors $\CY_{j,+}$ isomorphic to $\Spf W_j$ of $(\CM_{\Gamma_0(\pi_0)})_{O_{\breve{F}}}\cong \CN_2^{[0]}$. The morphism 
$$
\CY_{j,+}\to (\CM_{\Gamma_0(\pi_0)})_{O_{\breve{F}}}\overset{p_2}{\to}\CM_{O_{\breve{F}}}
$$
induces an isomorphism from $\CY_{j,+}$ to its image $\CW_j$.
\end{proposition}

By taking the dual isogenies, for all $j\geq 0$, we obtain morphisms $\beta_j':\CE_j\to \CE_{m(j)}$ with commuting diagrams similar to \eqref{equ:quasi-canonical lift in iwahori}, see \cite[after (3.13)]{Shi23a}. This defines Weil divisors $\CY_{j,-}\hookrightarrow (\CM_{\Gamma_0(\pi_0)})_{\Spf O_{\breve{F}}}$ such that the morphism
$$
\CY_{j,-}\to (\CM_{\Gamma_0(\pi_0)})_{O_{\breve{F}}}\overset{p_1}{\to}\CM_{O_{\breve{F}}}
$$
induces an isomorphism from $\CY_{j,-}$ to its image $\CW_j$. Note that for $j=0$ we have $\CY_{0,+}=\CY_{0,-}$; we set $\CY_{0}=\CY_{0,+}=\CY_{0,-}$.

By \cite[Lem. 3.8]{Shi23a}, we have the following relations as Weil divisors:
 
\begin{equation}\label{equ:pull-back}
\wt{\CM}_1^{[0],+}= \CY_{1,+}+\CY_{0},
\quad\text{and}\quad
\wt{\CM}_1^{[0],-}=\CY_{1,-}+\CY_{0}.
\end{equation}

\begin{lemma}\label{lem:compute-Hecke}
There are the following equalities as divisor classes in $\CM_{\Spf O_{\breve{F}}}$,
\begin{equation*}
\begin{aligned}
p_{2*}p_1^*\CW_0&=\CW_0+\CW_1,\quad p_{1*}p_1^*\CW_0&=(q+1)\CW_0,\\
p_{1*}p_2^*\CW_0&=\CW_0+\CW_1,\quad p_{2*}p_2^*\CW_0&=(q+1)\CW_0.
\end{aligned}
\end{equation*}
\end{lemma}
\begin{proof}
Without loss of generality, we prove the identities in the first line. By Proposition \ref{prop:quasi-canonical-iwahori}, the restriction 
$p_2:\CY_{i,+}\to \CW_i$ is an isomorphism. 
Hence $p_{2*}\CY_{1,+}=\CW_1$. Therefore, $p_{2*}p_1^*\CW_0=p_{2*}(\CY_0+\CY_{1,+})=\CW_0+\CW_1$.
On the other hand,  $p_1$ is a finite flat map of degree $q+1$ (comp.  \cite[Lem. 2.6]{Shi23a}), hence $p_{1*}p_1^*\CW_0=(q+1)\CW_0$.
\end{proof}

\subsection{Intersection numbers}\label{sec:intersection numbers}

From now on, we identify $\CY_{j,\pm}$ with its image in $\CN_{2}^{[0]}$ under the isomorphism $\CN_{2}^{[0]}\cong (\CM_{\Gamma_0(\pi_0)})_{O_{\breve{F}}}$. 
The relation to the splitting model is as follows. Let $\CZ_{j,\pm}\subset \CN_{2}^{[0],\spt}$ be the strict transform of $\CY_{j,\pm}$ under the blow up $\CN_{2}^{[0],\spt}\to \CN_{2,1}^{[0]}$.  We write again $\CZ_0=\CZ_{0, +}=\CZ_{0, -}$.

\begin{proposition}[\protect{\cite[Prop. 3.10]{Shi23a}}]\label{prop:Shi-Iwahori-canonical-lift}
For any integer $j\geq 0$,
$\CZ_{j,\pm}$ is a Cartier divisor of $\CN_{2}^{[0],\spt}$. The blow-up map $\CN_{2}^{[0],\spt}\to \CN_{2}^{[0]}$ induces an isomorphism 
$
\CZ_{j,\pm}\cong \CY_{j,\pm}.
$
In particular, $\CZ_{j,\pm}\cong\CW_j$.\qed
\end{proposition}
We can now compute the intersection numbers in Conjecture \ref{conj (n odd,t) Y}, (ii).

\begin{proposition}\label{prop:compute-hecke-1}
Let 
$
g=\begin{pmatrix}
1&\\&\alpha
\end{pmatrix}
\begin{pmatrix}
a'&b'\\ b'&a'
\end{pmatrix}\in \U(W_1)(F_0)_{\rs},
$
expressed in the presentation \eqref{equ:non-split unitary rank 2}. Then

\begin{equation*}
\left\langle\wt{\CM}_1^{[0],+,\spt} , g\wt{\CM}_1^{[0],-,\spt} \right\rangle_{ \CN_{1}^{[0]} \times \CN_{2}^{[0],\spt} =}\left\{\begin{array}{ll}
(q+1)(v(b')+1)&\text{if }\alpha\equiv 1\mod \pi,\\
v(b')+2&\text{if }\alpha\equiv -1\mod \pi,
\end{array}\right.
\end{equation*}
Here $v$ denotes the natural extension of the normalized valuation $v_0$ on $F_0$  to $F$, i.e., $v(b')=v_0(b'\bar b')$.
\end{proposition}
\begin{proof}
Since the reduced locus of $\CN_2^{[0]}\cong (\CM_{\Gamma_0(\pi_0)})_{O_{\breve{F}}}$ is a single point, by the same reasoning as in the proof of \cite[Prop. 8.10]{RSZ1}, the problem reduces to computing the length of the intersection, which is an Artinian scheme,
$$
\mathrm{length}(\wt{\CM}_1^{[0],+,\spt}\cap g\wt{\CM}_1^{[0],-.\spt}).
$$
By the description of cycles \eqref{equ:pull-back} in  $\CN_2^{[0]}$ and the definition of the splitting cycles in \S \ref{sec:n-1,t Y-cycle}, we have 
$$
\wt{\CM}_1^{[0],+,\spt}= \CZ_{1,+}+\CZ_{0},
\quad\text{and}\quad
\wt{\CM}_1^{[0],-,\spt}=\CZ_{1,-}+\CZ_{0}.
$$
By Proposition \ref{prop:Shi-Iwahori-canonical-lift}, the restrictions of the forgetful map $\CN_2^{[0],\spt}\to \CN_2^{[0]}$ to quasi-canonical divisors are isomorphisms. Therefore, we obtain  an isomorphism
$$
\wt{\CM}_1^{[0],+,\spt}\cap g\wt{\CM}_1^{[0],-.\spt}\cong \wt{\CM}_1^{[0],+}\cap g\wt{\CM}_1^{[0],-}.
$$
Hence the intersection number equals $\mathrm{length}(\wt{\CM}_1^{[0],+}\cap g\wt{\CM}_1^{[0],-})$.

Write $g_0=\begin{pmatrix}
1&\\ &\alpha
\end{pmatrix}$ and $g_1=\begin{pmatrix}
a'&b'\\b'&a'
\end{pmatrix}$, so that $g=g_1g_2$.  We have 
$$
\mathrm{length}(\wt{\CM}_1^{[0],+}\cap g\wt{\CM}_1^{[0],-})=\mathrm{length}(g_0^{-1}\wt{\CM}_1^{[0],+}\cap g_1\wt{\CM}_1^{[0],-}).
$$
We begin by considering the action of $g_0^{-1}$ on $\wt{\CM}_1^{[0],+}$. Recall from \eqref{eq:def Mt} that $\wt{\CM}_1^{[0]}=\wt{\CM}_1^{[0],+}\amalg \wt{\CM}_1^{[0],-}$ is defined via the base change of the embedding $\CN_1^{[0],\circ}\hookrightarrow \CN_2^{[2]}$. We claim that $g_0^{-1}$ stabilizes the two connected components of $\CN_1^{[0],\circ}$ when $\alpha\equiv 1\mod \pi$ and interchanges them when $\alpha\equiv -1\mod \pi$. As a consequence, we see that 
$$
g_0^{-1}\wt{\CM}_1^{[0],\pm}=\left\{\begin{array}{ll}
\wt{\CM}_1^{[0],\pm}&\text{if }\alpha\equiv 1\mod \pi,\\
\wt{\CM}_1^{[0],\mp}&\text{if }\alpha\equiv -1\mod \pi.
\end{array}\right.
$$
To prove the claim, note that the automorphism $g_0^{-1}:\CN_2^{[0]}\to \CN_2^{[0]}$ preserves the embedding $\CN_1^{[0]}\hookrightarrow \CN_2^{[0]}$. By \eqref{eq: N circ}, the automorphism $g_0^{-1}:\CN_2^{[0,2]}\to \CN_2^{[0,2]}$ preserves the embedding $\CN_1^{[0],\circ}\hookrightarrow \CN_2^{[0,2]}$. 
When $\alpha\equiv 1\mod \pi$, we have $\kappa(g_0^{-1})=1$, so $g_0^{-1}$ preserves the two connected components $\CN_2^{[0,2]}=\CN_2^{[0,2],+}\amalg \CN_2^{[0,2],-}$, hence maps $\CN_1^{[0],\pm}$ to $\CN_1^{[0],\pm}$, which is in fact an  isomorphism by the commutative diagram \eqref{eq: N circ}.
When $\alpha\equiv -1\mod \pi$, we have $\kappa(g_0^{-1})=-1$, so $g_0^{-1}$ interchanges the two components of $\CN_2^{[0,2]}$ and induces  isomorphisms $\CN_2^{[0,2],\pm}\to \CN_2^{[0,2],\mp}$.
Again, by the  diagram \eqref{eq: N circ}, $g_0^{-1}$ restricts to isomorphisms $\CN_1^{[0],\pm}$ to $\CN_1^{[0],\mp}$. 

We conclude that
$$
\mathrm{length}(\wt{\CM}_1^{[0],+,\spt}\cap g\wt{\CM}_1^{[0],-.\spt})=\left\{\begin{array}{ll}
\mathrm{length}(\wt{\CM}_1^{[0],+}\cap g_1\wt{\CM}_1^{[0],-})&\text{when }\alpha\equiv 1\mod \pi,\\
\mathrm{length}(\wt{\CM}_1^{[0],-}\cap g_1\wt{\CM}_1^{[0],-})&\text{when }\alpha\equiv -1\mod \pi,
\end{array}\right.$$
where $g_1=\begin{pmatrix}a'&b'\\b'&a'
\end{pmatrix}\in \SU(W_1)(F_0)$ corresponds to $h=a'+b'\in D^1\subset F\oplus D^{-}$.

 Passing through the exceptional isomorphism $\CN_2^{[0]}\cong (\CM_{\Gamma_0(\pi_0)})_{O_{\breve{F}}}$, the problem reduces to computing the length of certain closed subschemes in Iwahori level Lubin-Tate moduli spaces:
$$
\mathrm{length}(\wt{\CM}_1^{[0],\pm}\cap h\wt{\CM}_1^{[0],-}).
$$

Recall the description of $\wt{\CM}_1^{[0],\pm}\subset (\CM_{\Gamma_0(\pi_0)})_{O_{\breve{F}}}$ in \eqref{equ:cycle in the iwahori lubin tate} as pull-back of quasi-canonical lifts,
\begin{equation}
\wt{\CM}_1^{[0],+}=p_1^*(\CW_0), \quad \wt{\CM}_1^{[0],-}=p_2^*(\CW_0).
\end{equation}
Recall that the projection map $p_2:\CM_{\Gamma_0(\pi_0)}\to \CM$ is defined by $(Y,Y',\iota,\rho_Y)\mapsto (Y',\rho_{Y'})$ where $\rho_{Y'}=\iota_{\BE}(\pi)\circ\rho_Y$. It follows that $$hp_{1,*}(\CW_0)=p_{1,*}(h \cdot\CW_0),\quad
hp_{2,*}(\CW_0)=p_{2,*}(\pi h\pi^{-1} \cdot\CW_0).$$
By the projection formula, we have
$$
\mathrm{length}(p_i^*\CW_0\cap hp_2^*(\CW_0)=
\mathrm{length}(\CW_0\cap p_{i,*}hp_2^*(\CW_0).
$$
Combining with the discussions above and Lemma \ref{lem:compute-Hecke}, we conclude that
\begin{equation}\label{eq:reduce to gross formula}
\mathrm{length}(\wt{\CM}_1^{[0],+,\spt}\cap g\wt{\CM}_1^{[0],-.\spt})=\left\{\begin{array}{ll}
(q+1)\mathrm{length}(\CW_0\cap (a'+b')\cdot \CW_0)&\text{when }\alpha\equiv 1\mod \pi\\
\mathrm{length}(\CW_0\cap (a'-b')\cdot (\CW_0+\CW_1))&\text{when }\alpha\equiv -1\mod \pi,
\end{array}\right.
\end{equation}
where we note that $h=a'+b'$ and  $\pi h \pi^{-1}=a'-b'$.

To compute the length, recall from \cite[\S 6.2]{RSZ2} that for any $0\neq c\in O_D$, the special divisor $\CT(c)$ is defined as the locus of $\CM$ where the element $c\in \Hom_{O_{F_0}}(\ov{\BE},\BE)$ lifts to a homomorphism $\ov{\CE}\to\CY_0$, where $\CY_0$  is the universal family over $\CM$. By \cite[Prop. 7.1]{RSZ2}, we have
$$
\CW_0=\CT(1),\quad \CW_0+\CW_1=\CT(\pi),
$$
hence
$$
(a'+b')\CW_0=\CT((a'+b')),\quad (a'-b')(\CW_0+\CW_1)=\CT((a'-b')\pi).
$$
Let $h\in O_D$ be any element. From the moduli description, the intersection $\CW_0\cap \CT(h)$ is the closed sublocus of $\CW_0$ where the endomorphism $h\in \Hom_{O_{F_0}}(\BE,\BE)$ lifts to the canonical lift $\CE_0$ over $ \CW_0$. By Gross's formula \cite[Thm. 2.1]{V1}, writing $h=a+b\in F+D^-$, the length of this intersection is $v(b)+1$. The proposition then follows from \eqref{eq:reduce to gross formula}.
\end{proof}

Following the same argument, we compute the other intersection numbers in Conjecture \ref{conj (n odd,t) Y}, (ii). 
\begin{proposition}\label{prop:compute-hecke-2}
Let 
$
g=\begin{pmatrix}
1&\\&\alpha
\end{pmatrix}
\begin{pmatrix}
a'&b'\\ b'&a'
\end{pmatrix}\in \U(W_1)_{\rs}
$
expressed in the presentation \eqref{equ:non-split unitary rank 2}. Then
\begin{equation*}
\left\langle\wt{\CM}_1^{[0],+,\spt} , g\wt{\CM}_1^{[0],+,\spt} \right\rangle_{ \CN_{1}^{[0]} \times \CN_{2}^{[1],\spt} }=\left\{\begin{array}{ll}
(q+1)(v(b')+1)&\text{when }\alpha\equiv -1\mod \pi\\
v(b')+2&\text{when }\alpha\equiv 1\mod \pi,
\end{array}\right.
\end{equation*}\qed
\end{proposition}

\subsection{Proof of Conjecture \ref{conj (n odd,t) Y} for type $(n-1,t)=(0,0)$}\label{sec:proof of 00}
Recall the following germ expansion result. We define the set of semi-simple but irregular elements as
$$
A_S:=\left\{
\begin{pmatrix}
a&\\&d
\end{pmatrix}\in S(F_0)\Biggl|\, a,d\in F^1\right\}.
$$
For $i=0,1$, let $S_{\rs, i}$ be the set of elements in $S_\rs$ matching an element in $\U(W_i)_\rs$. 
\begin{theorem}[\protect{\cite[Thm. 15.1]{RSZ2}}]\label{thm:germ-expan}
Let $\varphi'\in C_c^{\infty}(S_2(F_0))$ transfer to $(f_0,f_1)\in C_c^{\infty}(\U(W_0))\times C_c^{\infty}(\U(W_1))$, and let $\gamma_0=\mathrm{diag}(a_0,d_0)\in A_S$. Let $i\in \{0,1\}$. If $f_i=0$, then there is the following germ expansion for $\gamma=\gamma(a,b)\in S_{\rs,i}$  in a neighborhood of $\gamma_0$
$$
\partial\!\Orb(\gamma,\varphi')=\frac{1}{2}\Orb(\mathrm{diag}(a_0,d_0),f_{1-i})\log |1-Na|+C,
$$
where $C$ is a constant depending on $\gamma_0$, $\varphi'$, and $i$, but not on $\gamma$.
Here $|\,\,|$ denotes the natural extension to $F$ of the normalized absolute value on $F_0$.
\qed
\end{theorem}

We compute the orbital integrals along semisimple but not regular orbits. Recall from \S \ref{ss: Lat(n-1,t) Y} that we have 
$$
 \wt\BM^{[0]}_1=\{(\Lambda^\flat,\Lambda,\Lambda^{[0]}) \in {\rm Vert}^{0}(W_0^\flat)\times {\rm Vert}^{2}(W_0) \times {\rm Vert}^{0}(W_0)\mid \Lambda^\flat \obot \langle u_0 \rangle\supset \Lambda\subset \Lambda^{[0]} \}.
$$
Let $\Lambda^\flat=\langle u^\flat\rangle$ with $(u^\flat,u^\flat)=1$.  Our notation here agrees with the last displayed equation in \cite[p. 1162]{RSZ2}. According to the discussion before \cite[(15.4)]{RSZ2}, there are two $\pi$-modular lattices contained in $\langle u^\flat\rangle \obot \langle u_0 \rangle$, which are given by
$$
\Lambda^\pm=\pi(\langle u^\flat\rangle \obot \langle u_0 \rangle)+O_F(u_0\pm u^\flat).
$$
Hence we have the decomposition
\begin{equation}\label{equ:decomposition}
 \wt\BM^{[0]}_1=\wt\BM^{[0],+}_1\amalg \wt\BM^{[0],-}_1=\{\Lambda^{[0]} \in {\rm Vert}^{0}(W_0)\mid  \Lambda^+\subset \Lambda^{[0]} \}\amalg \{\Lambda^{[0]} \in {\rm Vert}^{0}(W_0)\mid  \Lambda^-\subset \Lambda^{[0]} \}.
\end{equation}
The action of  $\mathrm{diag}(a,d)$  on the pair $\{\Lambda^+,\Lambda^-\}$ is discussed in \cite[Before (15.4)]{RSZ2}. When $\alpha\equiv 1\mod \pi$, then $\mathrm{diag}(a,d)$ preserves the two lattices and when $a\equiv -1\mod \pi$, the two lattices are  interchanged. 
\begin{lemma}\label{lem:lattice count}
For any $g=\mathrm{diag}(a,d)$ with $a,d\in F^1$, we have:
$$
\#(\wt\BM^{[0],+}_1\cap g\wt\BM^{[0],+}_1)_{\BN_1^{[0]}\times \BN_{2}^{[0]}}=\left\{\begin{array}{ll}
q+1&\text{if }a\equiv 1\mod\pi,\\
1&\text{if }a\equiv -1\mod\pi.
\end{array}\right.
$$
\end{lemma}
\begin{proof}
The element $g$ acts on the lattice model 
$$
\wt{\BM}_1^{[0]}\hookrightarrow \BN_{1}^{[0],\circ} \times \BN_{2}^{[0]}
$$
via its action on the $\BN_{2}^{[0]}$, sending $(\Lambda^\flat,\Lambda,\Lambda^{[0]})$ to $(\Lambda^\flat,\Lambda,g\Lambda^{[0]})$. From the description \eqref{equ:decomposition}, we see that
$$
(\wt\BM^{[0],+}_1\cap g\wt\BM^{[0],+}_1)_{\BN_1^{[0]}\times \BN_{2}^{[0]}}=\{(\Lambda^{[0]}\in\mathrm{Vert}^0(W_0)\mid \Lambda^+\subset \Lambda^{[0]},\quad g^{-1}\Lambda^+\subset \Lambda^{[0]}\}.
$$
When $a\equiv 1\mod \pi$, we have $\Lambda^+\cap g^{-1}\Lambda^+=\Lambda^+$, hence the RHS is the set of self-dual lattices $\Lambda^{[0]}\supset \Lambda^+$, which is in one-to-one correspondence with the set of   isotropic lines in the two-dimensional symplectic space $\Lambda^\vee/\Lambda$.  There are exactly $q+1$ such lines.

When $a\equiv -1\mod \pi$, we have $\Lambda^+ + g^{-1}\Lambda^+=\Lambda^++\Lambda^-$. The RHS  parameterizes the set of self-dual lattices $\Lambda^{[0]}\supset \Lambda^++\Lambda^-$. However, $\Lambda^++\Lambda^-$  is a self-dual lattice, hence $\Lambda^{[0]}=\Lambda^++\Lambda^-$ is uniquely determined.
\end{proof}

\begin{lemma}\label{lem:compute-orbit-int}
For any $h_0\in K_1^{[0]}$ and $a,d\in F^1$, the irregular orbital integrals of $\vol(K_1^{[0],\circ})^{-1}h_0\cdot\varphi_{2}^{[2,0]}$ are given by
$$
\Orb(\mathrm{diag}(a,d),\vol(K_1^{[0],\circ})^{-1}h_0\cdot\varphi_{2}^{[2,0]})=\left\{\begin{array}{ll}
2(q+1)&\text{if }h_0a\equiv 1\mod \pi,\\
2&\text{if }h_0a\equiv -1\mod \pi.
\end{array}\right.
$$
\end{lemma}
\begin{proof}
The conjugation action of $H(F_0)$ on diagonal matrices is trivial, hence the orbital integral equals
\begin{align*}
&\Orb(\mathrm{diag}(a,d),\vol(K_1^{[0],\circ})^{-1}h_0\cdot\varphi_{2}^{[2,0]})\\
={}&\vol(K_1^{[0],\circ})^{-1}\int_{h\in F^1}\varphi_2^{[2,0]}\Biggl(\begin{pmatrix}h^{-1}&\\&1\end{pmatrix}\begin{pmatrix}h_0a&\\&d\end{pmatrix}\begin{pmatrix}h&\\&1\end{pmatrix}\Biggr)dh\\
={}&\frac{\vol(F^1)}{\vol(K_1^{[0],\circ})}\varphi_2^{[2,0]}\Biggl(\begin{pmatrix}h_0a&\\&d\end{pmatrix}\Biggr)=2\varphi_2^{[2,0]}\Biggl(\begin{pmatrix}h_0a&\\&d\end{pmatrix}\Biggr).
\end{align*}
The results now follows from the lattice description in \S \ref{ss: Lat(n-1,t) Y} for the orbital integral and the lattice counting in Lemma \ref{lem:lattice count}. For further discussion on the relation to the lattice model, see \cite[\S 4]{LRZ1}.\end{proof}

\begin{proof}[Proof of Conjecture \ref{conj (n odd,t) Y}]
By Corollary \ref{cor:inhomo-version}, the problem reduces to  the inhomogeneous version.
Let $\wt{\varphi}'$ with transfer $(\vol(K_1^{[0],\circ}h_0\cdot \varphi_2^{[2,0]},0)\in C_c^\infty(\U(W_0))\times C_c^{\infty}(\U(W_1))$. 
Let $\phi\in C_c^{\infty}(S(F_0)_{\rs})$ be defined by
$$
\phi(\gamma):=\left\{\begin{array}{ll}
\left\langle\wt\CM^{[t],+,\spt}_{n} , g\wt\CM^{[t],-,\spt}_{n} \right\rangle_{ \CN_{n}^{[n-1]} \times \CN_{n+1}^{[t],\spt} }\cdot\log q+ \del\big(\gamma,  \varphi' \big)&\text{when }\gamma\in S(F_0)_{\rs,1},\\
0&\text{when }\gamma\in S(F_0)_{\rs,0}.
\end{array}\right.
$$
Let $\gamma_0=\mathrm{diag}(a_0,d_0)\in A_S$. By Proposition \ref{prop:compute-hecke-2}, Theorem \ref{thm:germ-expan}, and Lemma \ref{lem:compute-orbit-int}, there is a constant $C_{\gamma_0}$ such that for any $\gamma\in S(F_0)_{\rs,1}$ in a small neighborhood of $\gamma_0$
\begin{equation}\label{eqgamma0}
\phi(\gamma)=\left\{\begin{array}{ll}
q+1+C_{\gamma_0}&\text{if }a_0\equiv -1\mod \pi,\\
2+C_{\gamma_0}&\text{if }a_0\equiv 1\mod \pi.
\end{array}\right.
\end{equation}
and such  that $\phi(\gamma)=0$ when $\gamma\in S(F_0)_{\rs,0}$.

The intersection number is conjugation invariant, and the derivative of the orbital integral of our smooth transfer is conjugation invariant (cf. \cite[Lem. 3.3]{Mih-AT}). Hence  the function $\phi(\gamma)$ is conjugation invariant. By \eqref{eqgamma0} the function $\phi$ satisfies the hypotheses of \cite[Cor. 3.8]{Mih-AT}. Hence there exists $\fp'_{\rm corr}$ such that $\phi(\gamma)=\Orb(\gamma,\fp'_{\rm corr})$ for all $\gamma\in S(F_0)_{\rs}$. By \cite[Lem. 5.13]{RSZ1}, there exists $\wt\varphi'_{\rm corr}$ transferring to $ (0, 0)\in C_c^\infty(\U(W_0))\times C_c^\infty(\U(W_1))$ and such that  $\del(\gamma,\wt\fp'_{\rm corr})=\Orb(\gamma, \fp'_{\rm corr})$ for all $\gamma\in S(F_0)_{\rs}$. 
The function $\varphi'=\wt\varphi'+\wt\varphi'_{\corr}$ then  transfers to  $(\vol(K_1^{[0],\circ})^{-1} h_0\cdot\varphi^{[2,0]}_{2},0)\in C_c^\infty(\U(W_0))\times C_c^\infty(\U(W_1))$ and satisfies the  identity 
\begin{equation*}\label{equ:AT-identity-nbd}
\left\langle\wt\CM^{[0],+,\spt}_{1} , g\wt\CM^{[0],-,\spt}_{1} \right\rangle_{ \CN_{1}^{[0]} \times \CN_{2}^{[0],\spt} }\cdot\log q+\del\big(\gamma,  \varphi' \big)=0
\end{equation*}
for all $\gamma\in S(F_0)_{\rs,1}$.

By Proposition \ref{prop:compute-hecke-2}, Theorem \ref{thm:germ-expan}, and Lemma \ref{lem:compute-orbit-int}, a similar argument applies to the intersection number $\left\langle\wt\CM^{[0],+,\spt}_{1} , g\wt\CM^{[0],+,\spt}_{1} \right\rangle$. 
\end{proof}

\end{document}